\documentclass[10pt]{amsart}

\topskip  \usepackage{amsmath,amsthm,cite}
\usepackage{amssymb,longtable,multido,pstricks}
\usepackage{mathdots}
\usepackage{xspace}

\newtheorem{theorem}{Theorem}
\newtheorem{lemma}[theorem]{Lemma}
\newtheorem{proposition}[theorem]{Proposition}
\newtheorem{corollary}[theorem]{Corollary}
\def\al{\alpha}
\def\be{\beta}
\def\ga{\gamma}
\def\de{\delta}
\def\gf{generating function\xspace}

\def\v{\vert}

\begin{document}
\title{Enumeration of small Wilf classes avoiding 1342 and two other $4$-letter patterns}
\author[D.~Callan]{David Callan}
\address{Department of Statistics, University of Wisconsin, Madison, WI 53706}
\email{callan@stat.wisc.edu}
\author[T.~Mansour]{Toufik Mansour}
\address{Department of Mathematics, University of Haifa, 3498838 Haifa, Israel}
\email{tmansour@univ.haifa.ac.il}
\maketitle

\begin{abstract}
\bigskip
This paper is one of a series whose goal is to enumerate the avoiders, in the sense of classical pattern avoidance, for each triple of 4-letter patterns.
There are 317 symmetry classes of triples of 4-letter patterns, avoiders of 267 of which have already been enumerated. Here we enumerate avoiders for all small Wilf classes that have a representative triple containing the pattern 1342, giving 40 new enumerations and leaving only 13 classes still to be enumerated. In all but one case, we obtain an explicit algebraic generating function that is rational or of degree 2. The remaining one is shown to be algebraic of degree 3. We use standard methods, usually involving detailed consideration of the left to right maxima, and sometimes the initial letters, to obtain an algebraic or functional equation for the generating function.
\medskip

\noindent{\bf Keywords}: pattern avoidance, Wilf-equivalence, generating function, kernel method
\end{abstract}

\maketitle

\section{Introduction}
This paper is a sequel to \cite{AA20170501} and continues the investigation of permutations avoiding a given triple of 4-letter patterns. All large Wilf classes (those consisting of more than one symmetry class) have been enumerated \cite{CMS3patI,CMS3patII}. Here we enumerate
avoiders for all triples that contain the pattern 1342, lie in a small Wilf class, and are not amenable to the automated INSENC algorithm; see Table \ref{tabgf1342}, where the numbering follows that of Table 2 in the appendix to \cite{HYL}. As a consequence, only 13 symmetry classes remain to be enumerated.

{\small\begin{longtable}[c]{|c|c|c|c|}
\caption{Small Wilf classes of three 4-letter patterns not counted by INSENC that include the pattern $1342$\label{tabgf1342}}\\ \hline
\raisebox{-1.5mm}{No.} & \raisebox{-1.5mm}{Pattern set $T$} & \raisebox{-1.5mm}{Generating function $F_T(x)$} & \raisebox{-1.5mm}{Thm./[Ref]} \\[3mm] \hline
\endfirsthead  \hline
\raisebox{-1.5mm}{No.} & \raisebox{-1.5mm}{Pattern set $T$} & \raisebox{-1.5mm}{Generating function $F_T(x)$} & \raisebox{-1.5mm}{Thm./[Ref]}\\[3mm] \hline
\endhead \hline
\endfoot \hline
\endlastfoot
\raisebox{-1.5mm}{77}&\raisebox{-1.5mm}{$\{1243,2314,3412\}$}
&\raisebox{-1.5mm}{$\frac{1-11x+53x^2-145x^3+248x^4-274x^5+192x^6-80x^7+17x^8}{(1-x)^6(1-2x)^3}$}
&\raisebox{-1.5mm}{Thm. \ref{th77a}}\\[3mm]\hline
\raisebox{-1.5mm}{86}&\raisebox{-1.5mm}{$\{3412,4132,1324\}$}
&\raisebox{-1.5mm}{$\frac{1-7x+19x^2-24x^3+16x^4-4x^5-x^6+2x^7}{(1-x)^3(1-2x)(1-3x+x^2)}$}
&\raisebox{-1.5mm}{\cite{AA20170501}}\\[3mm]\hline
\raisebox{-1.5mm}{90}&\raisebox{-1.5mm}{$\{1243,2431,3412\}$}
&\raisebox{-1.5mm}{$\frac{1-11x+51x^2-129x^3+195x^4-183x^5+104x^6-30x^7+3x^8}{(1-x)^4(1-2x)(1-3x+x^2)^2}$}
&\raisebox{-1.5mm}{Thm. \ref{th90a}}\\[3mm]\hline
\raisebox{-1.5mm}{103}&\raisebox{-1.5mm}{$\{1423,2341,3124\}$}
&\raisebox{-1.5mm}{$\frac{1-9x+35x^2-77x^3+107x^4-97x^5+55x^6-17x^7+x^8}{(1-x)^5(1-4x+5x^2-3x^3)}$\footnotesize{$C(x)$}}
&\raisebox{-1.5mm}{Thm. \ref{th103a}}\\[3mm]\hline
\raisebox{-1.5mm}{106}&\raisebox{-1.5mm}{$\{1342,2143,3412\}$}
&\raisebox{-1.5mm}{$\frac{(1-2x)(1-6x+12x^2-9x^3+4x^4)}{(1-x)^3(1-3x)(1-3x+x^2)}$}
&\raisebox{-1.5mm}{Thm. \ref{th106a}}\\[3mm]\hline
\raisebox{-1.5mm}{118}&\raisebox{-1.5mm}{$\{1423,1234,3412\}$}
&\raisebox{-1.5mm}{$\frac{1-12x+64x^2-198x^3+393x^4-521x^5+463x^6-269x^7+95x^8-17x^9}{(1-x)^7(1-2x)^3}$}
&\raisebox{-1.5mm}{Thm. \ref{th118a}}\\[3mm]\hline
\raisebox{-1.5mm}{130}&\raisebox{-1.5mm}{$\{3412,3124,1342\}$}
&\raisebox{-1.5mm}{$\frac{1-9x+32x^2-58x^3+58x^4-33x^5+8x^6}{(1-x)^4(1-2x)(1-4x+2x^2)}$}
&\raisebox{-1.5mm}{Thm. \ref{th130a}}\\[3mm]\hline
\raisebox{-1.5mm}{131}&\raisebox{-1.5mm}{$\{2134,1423,2341\}$}
&\raisebox{-1.5mm}{$\frac{2x^5+x^4-6x^3+7x^2-4x+1}{(1-2x)(1-x)^3}$\footnotesize{$C(x)$}$-\frac{x(2x^4-x^3+x^2-2x+1)}{(1-2x)(1-x)^4}$}
&\raisebox{-1.5mm}{Thm. \ref{th131a}}\\[3mm]\hline
\raisebox{-1.5mm}{133}&\raisebox{-1.5mm}{$\{1342,2143,2314\}$}
&\raisebox{-1.5mm}{$\frac{(1-2x)(1-3x+x^2)}{1-6x+11x^2-7x^3}$}
&\raisebox{-1.5mm}{Thm. \ref{th133a}}\\[3mm]\hline
\raisebox{-1.5mm}{150}&\raisebox{-1.5mm}{$\{4312,4132,1324\}$}
&\raisebox{-1.5mm}{$\frac{1-11x+52x^2-136x^3+214x^4-204x^5+111x^6-28x^7}{(1-x)^3(1-2x)^3(1-3x+2x^2)}$}
&\raisebox{-1.5mm}{\cite{AA20170501}}\\[3mm]\hline
\raisebox{-1.5mm}{151}&\raisebox{-1.5mm}{$\{4312,1324,1423\}$}
&\raisebox{-1.5mm}{$\frac{1-12x+61x^2-169x^3+275x^4-263x^5+136x^6-29x^7+x^8}{(1-3x+x^2)(1-2x)^4(1-x)^2}$}
&\raisebox{-1.5mm}{\cite{AA20170501}}\\[3mm]\hline
\raisebox{-1.5mm}{153}&\raisebox{-1.5mm}{$\{4231,1324,1423\}$}
&\raisebox{-1.5mm}{$\frac{1-10x+41x^2-87x^3+101x^4-61x^5+15x^6-x^7}{(1-x)^2(1-2x)^3(1-3x+x^2)}$}
&\raisebox{-1.5mm}{\cite{AA20170501}}\\[3mm]\hline
\raisebox{-1.5mm}{156}&\raisebox{-1.5mm}{$\{1324,2341,2431\}$}
&\raisebox{-1.5mm}{$\frac{1-8x+23x^2-25x^3+3x^4+7x^5}{(1-2x)^2(1-3x+x^2)(1-2x-x^2)}$}
&\raisebox{-1.5mm}{\cite{AA20170501}}\\[3mm]\hline
\raisebox{-1.5mm}{158}&\raisebox{-1.5mm}{$\{1324,1342,3412\}$}
&\raisebox{-1.5mm}{$\frac{1-10x+40x^2-81x^3+88x^4-50x^5+11x^6}{(1-x)^3(1-2x)(1-3x)(1-3x+x^2)}$}
&\raisebox{-1.5mm}{\cite{AA20170501}}\\[3mm]\hline
\raisebox{-1.5mm}{159}&\raisebox{-1.5mm}{$\{1243,1342,3412\}$}
&\raisebox{-1.5mm}{$\frac{1-11x+48x^2-104x^3+115x^4-61x^5+13x^6}{(1-x)(1-2x)(1-3x)(1-3x+x^2)^2}$}
&\raisebox{-1.5mm}{Thm.~\ref{th159a}}\\[3mm]\hline
\raisebox{-1.5mm}{162}&\raisebox{-1.5mm}{$\{3412,1423,2341\}$}
&\raisebox{-1.5mm}{$\frac{1-7x+18x^2-21x^3+11x^4}{(1-2x)(1-6x+12x^2 -11x^3+3x^4)}$}
&\raisebox{-1.5mm}{Thm.~\ref{th162a}}\\[3mm]\hline
\raisebox{-1.5mm}{163}&\raisebox{-1.5mm}{$\{1342,2314,3412\}$}
&\raisebox{-1.5mm}{$\frac{(1-3x+3x^2)^2C(x)-x(1-x)(1-3x+5x^2-4x^3)}{(1-x)^5(1-2x)}$}
&\raisebox{-1.5mm}{Thm.~\ref{th163a}}\\[3mm]\hline
\raisebox{-1.5mm}{164}&\raisebox{-1.5mm}{$\{1432,2431,3214\}$}
&\raisebox{-1.5mm}{$\frac{(1-x)^4(1-2x)C(x)-x(1-4x+6x^2-5x^3)}{(1-x)(1-2x)(1-4x+5x^2-3x^3)}$}
&\raisebox{-1.5mm}{Thm.~\ref{th164a}}\\[3mm]\hline
\raisebox{-1.5mm}{165}&\raisebox{-1.5mm}{$\{1342,2314,3421\}$}
&\raisebox{-1.5mm}{$\frac{(1-2x)(1-x)^4C(x)-x(1-4x+6x^2-5x^3+x^4)}{(1-x)^4(1-3x+x^2)}$}
&\raisebox{-1.5mm}{Thm.~\ref{th165a}}\\[3mm]\hline
\raisebox{-1.5mm}{172}&\raisebox{-1.5mm}{$\{2143,4132,1324\}$}
&\raisebox{-1.5mm}{$\frac{(2-10x+16x^2-8x^3+x^4)C(x)-1+4x-5x^2+x^3}{(1-x)^2(1-3x+x^2)}$}
&\raisebox{-1.5mm}{\cite{AA20160607}}\\[3mm]\hline
\raisebox{-1.5mm}{175}&\raisebox{-1.5mm}{$\{1423,2341,3142\}$}
&\raisebox{-1.5mm}{$\frac{1-6x+12x^2-11x^3+5x^4}{1-7x+17x^2-20x^3+12x^4-2x^5}$}
&\raisebox{-1.5mm}{Thm.~\ref{th175a}}\\[3mm]\hline
\raisebox{-1.5mm}{176}&\raisebox{-1.5mm}{$\{1342,2431,3412\}$}
&\raisebox{-1.5mm}{$\frac{(1-x)^2(1-4x+6x^2-5x^3+x^4)C(x)-1+6x-14x^2+15x^3-8x^4+x^5}{x(1-3x+x^2)(1-x+x^3)}$}
&\raisebox{-1.5mm}{Thm.~\ref{th176a}}\\[3mm]\hline
\raisebox{-1.5mm}{178}&\raisebox{-1.5mm}{$\{1342,2314,2431\}$}
&\raisebox{-1.5mm}{$\frac{(1-x)^2(1-4x+6x^2-5x^3+x^4)C(x)-1+6x-14x^2+15x^3-8x^4+x^5}{x(1-3x+x^2)(1-x+x^3)}$}
&\raisebox{-1.5mm}{Thm.~\ref{th178a}}\\[3mm]\hline
\raisebox{-1.5mm}{180}&\raisebox{-1.5mm}{$\{1342,2314,4231\}$}
&\raisebox{-1.5mm}{$\frac{1-7x+18x^2-22x^3+16x^4-6x^5+x^6-\left(x-5x^2+8x^3-2x^4-2x^5+x^6\right)C(x)}{(1-2x)(1-x)^2\left(1-5x+4x^2-x^3\right)}$}
&\raisebox{-1.5mm}{\cite{AA20170501}}\\[3mm]\hline
\raisebox{-1.5mm}{182}&\raisebox{-1.5mm}{$\{2314,2431,3412\}$}
&\raisebox{-1.5mm}{$\frac{1+x^2(1-x)C(x)^4}{1-x(1-2x)C(x)^2}$}
&\raisebox{-1.5mm}{Thm.~\ref{th182a}}\\[3mm]\hline
\raisebox{-1.5mm}{184}&\raisebox{-1.5mm}{$\{1324,2431,3241\}$}
&\raisebox{-1.5mm}{$\frac{1-8x+24x^2-32x^3+19x^4-3x^5}{(1-x)(1-2x)(1-3x+x^2)^2}$}
&\raisebox{-1.5mm}{\cite{AA20170501}}\\[3mm]\hline
\raisebox{-1.5mm}{187} & \raisebox{-1.5mm}{$\{1324,2314,2431\}$}
&\raisebox{-1.5mm}{$\frac{1-9x+31x^2-49x^3+34x^4-7x^5}{(1-3x+x^2)^2(1-2x)^2}$}
&\raisebox{-1.5mm}{\cite{AA20170501}}\\[3mm]\hline
\raisebox{-1.5mm}{190}&\raisebox{-1.5mm}{$\{3142,2314,1423\}$}
&\raisebox{-1.5mm}{$\frac{(1-2x)(1-3x+x^2)^2}{(1-x)(1-8x+22x^2-24x^3+8x^4-x^5)}$}
&\raisebox{-1.5mm}{Thm.~\ref{th190a}}\\[3mm]\hline
\raisebox{-1.5mm}{192}&\raisebox{-1.5mm}{$\{1243,1342,2431\}$}
&\raisebox{-1.5mm}{$\frac{(1-5x+9x^2-6x^3)(C(x)-1)-x^3}{x(1-2x)(1-x)^2}$}
&\raisebox{-1.5mm}{Thm.~\ref{th192a}}\\[3mm]\hline
\raisebox{-1.5mm}{193}&\raisebox{-1.5mm}{$\{1324,2431,3142\}$}
&\raisebox{-2mm}{$\frac{x-1+ (x^2-5 x+2)C(x)}{1-3 x+x^2}$}
&\raisebox{-1.5mm}{\cite{AA20170501}}\\[3mm]\hline
\raisebox{-1.5mm}{194}&\raisebox{-1.5mm}{$\{3124,4123,1243\}$}
&\raisebox{-1.5mm}{$\frac{(1-5x+9x^2-8x^3+4x^4)C(x)-(1-5x+9x^2-6x^3+x^4)}{x(1-2x)^2}$}
&\raisebox{-1.5mm}{Thm.~\ref{th194a}}\\[3mm]\hline
\raisebox{-1.5mm}{197}&\raisebox{-1.5mm}{$\{2413,3241,2134\}$}
&\raisebox{-1.5mm}{$\frac{1-5x+9x^2-7x^3+x^4+(1-5x+9x^2-9x^3+3x^4)\sqrt{1-4x}}{(1-x)(1-6x+12x^2-11x^3+3x^4+(1-4x+6x^2-5x^3+x^4)\sqrt{1-4x})}$}
&\raisebox{-1.5mm}{Thm.~\ref{th197a}}\\[3mm]\hline
\raisebox{-1.5mm}{198}&\raisebox{-1.5mm}{$\{1234,1423,2341\}$}
&\raisebox{-1.5mm}{$\frac{(1-7x+18x^2-19x^3+6x^4)C(x)-(1-6x+12x^2-8x^3+x^4)}{x^2(1-x)(1-2x)}$}
&\raisebox{-1.5mm}{Thm.~\ref{th198a}}\\[3mm]\hline
\raisebox{-1.5mm}{199}&\raisebox{-1.5mm}{$\{1243,1423,2341\}$}
&\raisebox{-1.5mm}{$\frac{x(x-1)^2(2x-1)C(x)+3x^4-7x^3+9x^2-5x+1}{(xC(x)-(x-1)^2)(x-1)^2(2x-1)}$}
&\raisebox{-1.5mm}{Thm.~\ref{th199a}}\\[3mm]\hline
\raisebox{-1.5mm}{204}&\raisebox{-1.5mm}{$\{1243,1423,2314\}$}
&\raisebox{-1.5mm}{$\frac{x(2x^2-2x+1)C(x)-(3x^2-3x+1)}{x(2x^2-2x+1)C(x)-(1-x)(3x^2-3x+1)}$}
&\raisebox{-1.5mm}{Thm.~\ref{th204a}}\\[3mm]\hline
\raisebox{-1.5mm}{207}&\raisebox{-1.5mm}{$\{2134,1423,1243\}$}
&\raisebox{-1.5mm}{$\frac{1-x(1-x)C(x)}{(1-x)(2-C(x))+x^2}$}
&\raisebox{-1.5mm}{\cite{AA20160601}}\\[3mm]\hline
\raisebox{-1.5mm}{208}&\raisebox{-1.5mm}{$\{1234,1342,3124\}$}
&\raisebox{-1.5mm}{$\frac{(1-2x)(1-6x+12x^2-10x^3+2x^4)-x^2(1-2x+2x^2)^2C(x)}{1-9x+30x^2-49x^3+38x^4-8x^5-4x^6}$}
&\raisebox{-1.5mm}{Thm.~\ref{th208a}}\\[3mm]\hline
\raisebox{-1.5mm}{210} & \raisebox{-1.5mm}{$\{1243,1324,2431\}$}
&\raisebox{-1.5mm}{$\frac{1-6x+13x^2-11x^3+4x^4}{x^2(1-x)^2}$\footnotesize{$C(x)$}$-\frac{1-6x+12x^2-8x^3+2x^4}{x^2(1-x)(1-2x)}$}
& \raisebox{-1.5mm}{\cite{AA20170501}}\\[3mm]\hline
\raisebox{-1.5mm}{212}&\raisebox{-1.5mm}{$\{1324,2413,2431\}$}
&\raisebox{-1.5mm}{$1+\frac{x(1-4x+4x^2-x^3-x(1-4x+2x^2)C(x))}{(1-3x+x^2)(1-3x+x^2-x(1-2x)C(x))}$} &\raisebox{-1.5mm}{\cite{AA20170501}}\\[3mm]\hline
\raisebox{-1.5mm}{213}&\raisebox{-1.5mm}{$\{2431,1324,1342\}$}
&\raisebox{-1.5mm}{$\frac{(1-5x+8x^2-5x^3)C(x)-1+4x-4x^2+x^3}{x^2(1-2x)}$}
&\raisebox{-1.5mm}{\cite{AA20170501}}\\[3mm]\hline
\raisebox{-1.5mm}{214}&\raisebox{-1.5mm}{$\{1342,2341,3412\}$}
&\raisebox{-2.5mm}{$\frac{(1-2x)\big((1-5x+9x^2-6x^3)\sqrt{1-4x}-(1-9x+29x^2-38x^3+18x^4)\big)}{2(1-x)^2x(1-7x+14x^2-9x^3)}$}
&\raisebox{-1.5mm}{Thm.~\ref{th214a}}\\[4mm]\hline
\raisebox{-1.5mm}{217}&\raisebox{-1.5mm}{$\{4132,1342,1243\}$}
&\raisebox{-1.5mm}{$\frac{(1-x)(1-3x+x^2)\sqrt{1-4x}-(1-8x+20x^2-15x^3+4x^4)}{2x(1-x)(1-5x+4x^2-x^3)}$}
&\raisebox{-1.5mm}{Thm.~\ref{th217a}}\\[3mm]\hline
\raisebox{-1.5mm}{219}&\raisebox{-1.5mm}{$\{1342,2413,3412\}$}
&\raisebox{-1.5mm}{$1+\frac{x(1-x)^2(1-2x)}{(x^2-3x+1)(2x^2-2x+1)-x(1-2x)(1-x)C(x)}$}
&\raisebox{-1.5mm}{Thm.~\ref{th219a}}\\[3mm]\hline
\raisebox{-1.5mm}{220}&\raisebox{-1.5mm}{$\{2431,2314,3142\}$}
&\raisebox{-1.5mm}{$1+\frac{x(1-x)^2(1-2x)}{(1-3x)(1-x)^3-x(1-2x)(1-x+x^2)(C(x)-1)}$}
&\raisebox{-1.5mm}{Thm.~\ref{th220a}}\\[3mm]\hline
\raisebox{-1.5mm}{222}&\raisebox{-1.5mm}{$\{3412,3421,1342\}$}
&\raisebox{-1.5mm}{$\frac{2-11x+13x^2-6x^3+x(1-x)(1-6x+4x^2)(1-4x)^{-1/2}}{2(1-6x+8x^2-4x^3)}$}
&\raisebox{-1.5mm}{Thm.~\ref{th222a}}\\[3mm]\hline
\raisebox{-1.5mm}{223}&\raisebox{-1.5mm}{$\{1243,1342,2413\}$}
&\raisebox{-2mm}{$\frac{(1-2x)(1-2x-\sqrt{1-8x+20x^2-20x^3+4x^4})}{2x(1-4x+5x^2-x^3)}$}
&\raisebox{-1.5mm}{Thm.~\ref{th223a}}\\[3mm]\hline
\raisebox{-1.5mm}{224}&\raisebox{-1.5mm}{$\{4132,1342,1423\}$}
&\raisebox{-1.5mm}{$\frac{2-10x+9x^2-3x^3+x(1-x)(2-x)\sqrt{1-4x}}{2(1-5x+4x^2-x^4)}$}
&\raisebox{-1.5mm}{Thm.~\ref{th224a}}\\[3mm]\hline
\raisebox{-1.5mm}{226}&\raisebox{-1.5mm}{$\{1342,2143,2413\}$}
&\raisebox{-2mm}{$\frac{1-3x+x^2-\sqrt{(1-7x+13x^2-8x^3)(1-3x+x^2)}}{2x(1-x)(1-2x)}$}
&\raisebox{-1.5mm}{Thm.~\ref{th226a}}\\[3mm]\hline
\raisebox{-1.5mm}{231}&\raisebox{-1.5mm}{$\{1324,1342,2341\}$}
&\raisebox{-1.5mm}{$\frac{(1-3x)(1-2x-xC(x))}{(1-4x)(1-3x+x^2)}$}
&\raisebox{-1.5mm}{\cite{AA20170501}}\\[3mm]\hline
\raisebox{-1.5mm}{232}&\raisebox{-1.5mm}{$\{1234,1342,2341\}$}
&\raisebox{-1.5mm}{$\frac{1-4x+2x^2-(1-6x+9x^2)C(x)}{x(1-4x)}$}
&\raisebox{-1.5mm}{Thm.~\ref{th232a}}\\[3mm]\hline
\raisebox{-1.5mm}{242}&\raisebox{-1.5mm}{$\{2341,2431,3241\}$}
&\raisebox{-1.5mm}{\footnotesize{$F_T(x)$}$=1+\frac{xF_T(x)}{1-xF_T^2(x)}$}
&\raisebox{-1.5mm}{Thm.~\ref{th242a}}\\[3mm]\hline
\end{longtable}}

\section{Preliminaries}
For every pattern set $T$ considered, $F_T(x)$ denotes the generating function $\sum_{n\ge 0}\v S_n(T)\v x^n$ for $T$-avoiders and
$G_m(x)$ the generating function for $T$-avoiders $\pi=i_1\pi^{(1)}i_2\pi^{(2)}\cdots i_m\pi^{(m)}\in S_n(T)$ with $m$ left-right maxima
$i_1,i_2,\dots,i_m=n$; thus $F_T(x)=\sum_{m\ge 0}G_m(x)$.
In every case, $G_0(x)=1$ and $G_1(x)=xF_T(x)$, and for most triples $T$, our efforts are directed toward finding an expression for $G_m(x)$, usually distinguising the case $m=2$ and sometimes $m=3$ from larger values of $m$. As usual, $C(x)$ denotes the \gf for the Catalan numbers $\frac{1}{n+1}\binom{2n}{n}$, which counts $\tau$-avoiders for each 3-letter pattern $\tau$ (see \cite{K}).

Given nonempty sets of numbers $S$ and $T$, we will write $S<T$ to mean $\max(S)<\min(T)$ (with the inequality vacuously holding if $S$ or $T$ is empty).  In this context, we will often denote singleton sets simply by the element in question. Also, for a number $k$, $S-k$ means the set $\{s-k:s\in S\}$.

\section{Proofs}
\subsection{Case 77: $\{1243,2314,3412\}$}
We count by number of left-right maxima. First, we suppose there are 2 left-right maxima and distinguish the cases $\pi_1=n-1$ and $\pi_1\le n-2$.
\begin{lemma}\label{lem77a1}
Let $H(x)$ be the generating function for permutations $(n-1)\pi'n\pi''\in S_n(T)$. Then
$$H(x)=\frac{x^2(1-4x+5x^2)}{(1-2x)^3}.$$
\end{lemma}
\begin{proof}
Refine $H(x)$ to $H_d(x)$, the generating function when $\pi''$ has $d$ letters. Clearly, $H_0(x)=x^2F_{\{231,1243\}}(x)$. So, by \cite[Seq. A005183]{Sl}, we have
$$H_0(x)=x^2\left(1+\frac{x(1-3x+3x^2)}{(1-x)(1-2x)^2}\right).$$
Henceforth, assume $d\geq1$.
Then $\pi''$ is decreasing (to avoid 3412), say $\pi''= k_1 k_2\cdots k_d$.
Let $t\ge 1$ be maximal such that the initial $t$ letters, $I_t:=\{\pi_1\pi_2\dots
\pi_t\}$, of $\pi$ are decreasing and $>k_1$. The subsequence of letters $<k_1$
in $\pi'$ is decreasing (else $nk_1$ terminates a 1243). Also, as soon as a
letter in $(n-1)\pi'\,\backslash\, I_t$ is $>k_1$, so are all later letters
(to avoid 2314). Thus $\pi$ has the form
$I_t\, \al\, \be\, n\,k_1\, k_2\,\dots\, k_d$ with $\al<k_1<\be$. Furthermore, also to avoid 2314,
$\be$ decomposes as $\be_1\be_2\dots \be_t$ with $\be_1<\pi_t<\be_2<\pi_{t-1}< \cdots <\be_t<\pi_1=n-1$.

If $\al=\emptyset$, we clearly have a contribution of $x^dH_0(x)$. Henceforth, assume $\al \ne \emptyset$. Then all $\be$'s avoid 231 and 132 (due to $\al$ on the left and $n$ on the right). Also, as soon as a $\be$ is nonempty, all later $\be$'s are increasing (or $\al$ and the nonempty $\be$ give the 12 of a 1243). Now to count contributions: $I_t,n,k_1,\dots,k_d$ give $x^{t+1+d}$, and $\al$ is separated into $d$ decreasing blocks by $k_1,\dots,k_d$, giving $1/(1-x)^d -1$.

When $t=1$, the lone $\be$ contributes $L:=F_{\{132,231\}}(x)=\frac{1-x}{1-2x}$ \cite{SiS}. When $t>1$, either all $\be$'s are empty or the first nonempty $\be$ (a $\{132,231\}$-avoider) is followed by $r \in [0,t-1]$ increasing blocks, for a contribution of $1+(L-1)\big(1+1/(1-x)+ \cdots +1/(1-x)^{t-1}\big)$.
Hence, for $d\ge 1$,
\begin{align*}
H_d(x)&=x^d H_0(x)+x^{d+2}\left(\frac{1}{(1-x)^d} -1\right)L \\
&+\left(\frac{1}{(1-x)^d} -1\right)\sum_{t\ge 2}x^{t+1+d}\left(1+(L-1)\frac{1-1/(1-x)^t}{1-1/(1-x)}\right)\, .
\end{align*}
We have thus determined $H_d(x)$ for all $d$ and, since
$H(x)=\sum_{d\geq0}H_d(x)$, the result follows by routine computation.
\end{proof}

\begin{lemma}\label{lem77a2}
Let $J(x)$ be the generating function for the number of permutations $i\pi'n\pi''\in S_n(T)$ such that $i\leq n-2$. Then
$$J(x)=\frac{x^3(1-4x+9x^2-11x^3+6x^4-2x^5)}{(1-x)^5(1-2x)^2}.$$
\end{lemma}
\begin{proof}
Refine $J(x)$ to $J_d(x)$, the generating function for permutations $i\pi'n\pi''\in S_n(T)$ such that $i\leq n-2$ and $\pi''$ has $d$ letters smaller than $i$.
Since $n-1 \in \pi''$,
we have $\pi'$ decreasing, for else $n\,(n-1)$ are the 43 of a 1243.

If $d=0$, then $\pi''>i$ and, due to the presence of the letter $i$, $\pi''$ avoids
$\{132,2314,3412\}$. Thus
\[
J_0(x)=\frac{x^2}{1-x}\big(K(x)-1\big)
\]
where $K(x)=\frac{1-3x+3x^2}{(1-x)^2(1-2x)}$ is the generating function for $\{132,2314,3412\}$-avoiders (proof omitted).

Henceforth, assume $d\ge 1$. The letters smaller than $i$ in $\pi''$ are decreasing (to avoid 3412); denote them $k_1>k_2>\cdots>k_d$.
No letters $u<v$ after $n$ and $>i $ can be separated by one of the $k$'s (else $iukv$ is a 2314).
Hence $\pi$ decomposes as
\[
\pi=i\,\alpha\, n\,\beta_1 \,k_1\, \be_2\, k_2\, \dots\, \beta_d \,k_d\, \beta_{d+1}
\]
with $\be_1>\be_2> \cdots >\be_d > \be_{d+1}$ and, since $\al$ is decreasing, $\al$ can be separated by the $k$'s into  $\al_1 \al_2 \dots \al_{d+1}$ (see Figure \ref{figc77a}). If some $\be_j$ is nonempty, then all lower-indexed $\be$'s are decreasing since $u<v$ in a lower-indexed $\be$ would be the 34 of a 3412.
\begin{figure}[htp]
\begin{center}
\begin{pspicture}(0,6.3)
\psset{xunit=1.2cm}
\psset{yunit=.6cm}
\psline(0,0)(0,1)(-1,1)(-1,0)(0,0)
\psline(-1,1)(-1,2)(-2,2)(-2,1)(-1,1)
\psline(-1,1)(-1,2)(-2,2)(-2,1)(-1,1)
\psline(-3,3)(-3,4)(-4,4)(-4,3)(-3,3)
\psline(-4,4)(-4,5)(-5,5)(-5,4)(-4,4)
\psline(1,9)(1,10)(0,10)(0,9)(1,9)
\psline(2,8)(2,9)(1,9)(1,8)(2,8)
\psline(4,6)(4,7)(3,7)(3,6)(4,6)
\psline(5,5)(5,6)(4,6)(4,5)(5,5)
\psline[linecolor=lightgray](0,1)(0,9)
\psline[linecolor=lightgray](-4,5)(4,5)
\psline[linecolor=lightgray](-3,4)(1,4)(1,8)
\psline[linecolor=lightgray](-3,3)(2,3)(2,8)
\psline[linecolor=lightgray](-1,2)(3,2)(3,6)
\psline[linecolor=lightgray](0,1)(4,1)(4,5)
\psdots[linewidth=1.5pt](-5,5)(0,10)(1,4)(2,3)(3,2)(4,1)
\rput(-2.5,2.6){\textrm{{\footnotesize $\ddots$}}}
\rput(2.5,7.6){\textrm{{\footnotesize $\ddots$}}}
\rput(2.5,2.6){\textrm{{\footnotesize $\ddots$}}}
\rput(-4.5,4.5){\textrm{{\footnotesize $\al_1\downarrow$}}}
\rput(-3.5,3.5){\textrm{{\footnotesize $\al_2\downarrow$}}}
\rput(-1.5,1.5){\textrm{{\footnotesize $\al_d\downarrow$}}}
\rput(-0.5,0.5){\textrm{{\footnotesize $\al_{d+1}\downarrow$}}}
\rput(.5,9.5){\textrm{{\footnotesize $\be_1$}}}
\rput(1.5,8.5){\textrm{{\footnotesize $\be_2$}}}
\rput(3.5,6.5){\textrm{{\footnotesize $\be_d$}}}
\rput(4.5,5.5){\textrm{{\footnotesize $\be_{d+1}$}}}
\rput(-5.2,5.2){\textrm{{\footnotesize $i$}}}
\rput(-0.3,10.2){\textrm{{\footnotesize $n$}}}
\rput(1.3,4.3){\textrm{{\footnotesize $k_1$}}}
\rput(2.3,3.3){\textrm{{\footnotesize $k_2$}}}
\rput(3.3,2.3){\textrm{{\footnotesize $k_d$}}}
\rput(4.4,1.3){\textrm{{\footnotesize $k_{d+1}$}}}
\end{pspicture}
\caption{The decomposition of $\pi$ for $J_d$ in Lemma \ref{lem77a2}}\label{figc77a}
\end{center}
\end{figure}
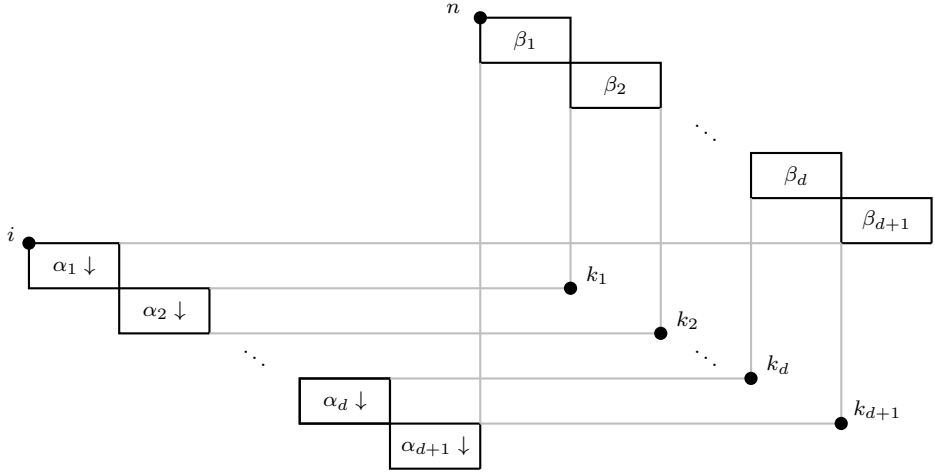

It is convenient to focus on $\al_2 \dots \al_{d+1}$ because a letter in this range is smaller then some $k$, implying restrictions on the $\be$'s.
\begin{itemize}
\item $\al_2 \dots \al_{d+1}=\emptyset$. Let $j\in[1,d+1]$ be maximal such
that $\be_j\ne\emptyset$ ($j$ exists since $n-1$ is in one of the $\be$'s).
Then, as noted, $\be_1,\be_2,\dots,\be_{j-1}$ are all decreasing. Also, $\be_j$
is a $\{132,2314,3412\}$-avoider due to the presence of $i$.
So $i,n$ and the $k$'s contribute $x^{d+2}$,
$\al_1$ contributes $1/(1-x)$, while $\be_1,\be_2,\dots,\be_{j-1}$ contribute
$1/(1-x)^{j-1}$ and $\be_j$ contributes $K(x)-1$, for an overall contribution of
\[
\frac{x^{d+2}}{1-x}\big(K(x)-1\big)\sum_{j=1}^{d+1}\frac{1}{(1-x)^{j-1}}\, .
\]

\item $\al_2 \dots \al_{d+1}\ne\emptyset$. Let $s\in [2,d+1]$ be minimal
such that $\al_s \ne \emptyset$. Then $\be_{s+1}\be_{s+2}\dots \be_{d+1}=
\emptyset$, otherwise for $u\in \al_s,\ uk_{s-1}k_s$ would form the 231 of a 2314.

If $\be_2\be_3 \dots \be_s =\emptyset$, then $\be_1 \ne \emptyset$ and
avoids $\{132,2314,3412\}$.
So $i,n$ and the $k$'s contribute $x^{d+2}$, $\al_1$  contributes $1/(1-x)$, $\al_s$  contributes $x/(1-x)$, later $\al$'s
contribute $1/(1-x)^{d+1-s}$, and $\be_1$ contributes $K(x)-1$, for an overall contribution of
\[
\sum_{s=2}^{d+1}\frac{x^{d+3}}{(1-x)^{d+3-s}}\big(K(x)-1\big)\, .
\]

If $\be_2\be_3 \dots \be_s \ne \emptyset$, then $\be_2\be_3 \dots \be_s$ is increasing since $u>v$ in $\be_2\be_3 \dots \be_s$ would form the 43 of a 1243.
Thus, since $\be_2>\be_3> \cdots > \be_s$, $\be_2\be_3 \dots \be_s \ne \emptyset$ implies the there is a unique $j \in [2,s]$ with $\be_j$ nonempty, and this $\be_j$ is increasing. Also, as noted, $\be_1$ is decreasing. Thus we get a contribution of
\[
\sum_{s=2}^{d+1}\frac{(s-1)x^{d+4}}{(1-x)^{d+5-s}}
\]
\end{itemize}
Hence, for $d\ge 1$,
\begin{eqnarray*}
J_d(x) & = &\frac{x^{d+2}}{1-x}\big(K(x)-1\big)\sum_{j=1}^{d+1}\frac{1}{(1-x)^{j-1}}
+ \\
& & \sum_{s=2}^{d+1}\frac{x^{d+3}}{(1-x)^{d+3-s}}\big(K(x)-1\big) + \sum_{s=2}^{d+1}\frac{(s-1)x^{d+4}}{(1-x)^{d+5-s}}
\end{eqnarray*}
Since $J(x)=\sum_{d\geq0}J_d(x)$, the result follows by routine computation.
\end{proof}

\begin{corollary}\label{cor77a3}
The generating function $G_2(x)$ for $T$-avoiders with exactly $2$ left-right maxima, is given by
$$G_2(x)=\frac{x^2(1-8x+29x^2-58x^3+66x^4-43x^5+15x^6-x^7)}{(1-2x)^3(1-x)^5}.$$
\end{corollary}
\begin{proof}
The result follows from Lemmas \ref{lem77a1} and \ref{lem77a2} because $G_2(x)=H(x)+J(x)$.
\end{proof}

Now, we are ready to find an explicit formula for $F_T(x)$.
\begin{theorem}\label{th77a}
Let $T=\{1243,2314,3412\}$. Then
$$F_T(x)=\frac{1-11x+53x^2-145x^3+248x^4-274x^5+192x^6-80x^7+17x^8}{(1-x)^6(1-2x)^3}.$$
\end{theorem}
\begin{proof}
Let $G_m(x)$ be the generating function for $T$-avoiders with $m$
left-right maxima. Clearly, $G_0(x)=1$ and $G_1(x)=xF_T(x)$. The formula for $G_2(x)$ is given by Corollary \ref{cor77a3}.

To find an equation for $G_3(x)$, let $\pi=i_1\pi^{(1)}i_2\pi^{(2)}i_3\pi^{(3)}\in S_n(T)$ with $3$ left-right maxima. Since $\pi$ avoids $1243$ and $2314$, we see that $\pi^{(3)}<i_2$ and $\pi^{(2)}>i_1$. By considering the four cases whether $\pi^{(2)}$ is empty or not and whether $\pi^{(3)}<i_1$ or not, we obtain
$$G_3(x)=xH(x)+\frac{x^4}{(1-x)(1-2x)}+\frac{x^4}{(1-2x)^2}+\frac{x^5}{(1-x)(1-2x)^2},$$
where $H(x)$ is defined in Lemma \ref{lem77a1}. Hence,
$$G_3(x)=\frac{x^3(1-x)^2}{(1-2x)^3}.$$

Now, suppose $m\geq4$. Let $\pi=i_1\pi^{(1)}\cdots i_m\pi^{(m)}\in S_n(T)$ with $m$ left-right maxima. Since $\pi$ avoids $T$, we have that $\pi^{(3)}=\emptyset$ and
$\pi^{(s)}$ has no letter between $i_2$ and $i_3$ for all $s\geq3$. Hence, $G_m(x)=xG_{m-1}(x)$.
By summing over all $m\geq4$, we obtain
$$F_T(x)-G_0(x)-G_1(x)-G_2(x)-G_3(x)=x\big(F_T(x)-G_0(x)-G_1(x)-G_2(x)\big.$$
Thus, by substituting the expressions above for $G_j(x)$ with $j=0,1,2,3$ and then  solving for $F_T(x)$, we complete the proof.
\end{proof}

\subsection{Case 90: $\{1243,2431,3412\}$}
We consider left-right maxima and begin with the case where they are consecutive integers and $n$ is not the last letter.
\begin{lemma}\label{lem90a1}
For $m\ge 2$, let $H_m(x)$ be the generating function for permutations $\pi=(n+1-m)\pi^{(1)}(n+2-m)\pi^{(2)}\cdots n\pi^{(m)}\in S_n(T)$ such that $\pi^{(m)}$ is not empty. Then
$$\sum_{m\geq2}H_m(x)=\frac{x^3(1-2x)}{(1-3x+x^2)^2}.$$
\end{lemma}
\begin{proof}
Let $i_j=n-m+j,\ 1\le j \le m$, denote the left-right maxima. First, $\pi^{(2)}\dots\pi^{(m)}$ is decreasing (or $i_1 i_2$ is the 34 of a 3412) and contains no gaps (if there is a gap, say between $u>v$, then $v+1 \in \pi^{(1)}$ and $(v+1)i_2uv$ is a 2431). So $\pi$ has the form in Figure \ref{fig90a}a) where the left-right maxima $i_1, i_2,\dots, i_m$ are consecutive integers and $A$ consists of consecutive decreasing integers whose last (and smallest) entry, say $k$, is in $\pi^{(m)}$.
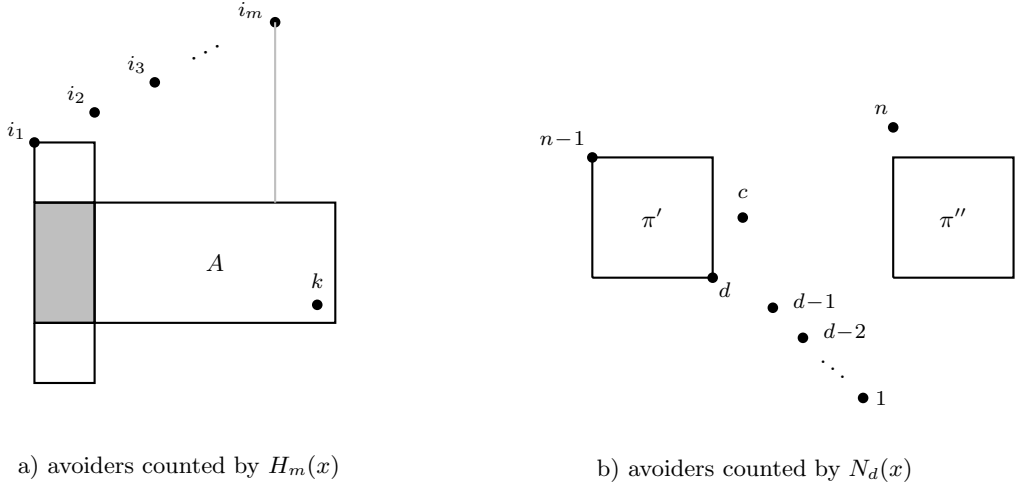
\begin{figure}[htp]
\begin{center}
\begin{pspicture}(1.0,-2.3)(7.8,4.5)
\psset{xunit=.8cm}
\psset{yunit=.8cm}
\rput(2.4,-2.4){\textrm{\small a) avoiders counted by $H_m(x)$}}
\qdisk(0,3){2.0pt}
\qdisk(1,3.5){2.0pt}
\qdisk(2,4){2.0pt}
\qdisk(4,5){2.0pt}
\qdisk(4.7,0.3){2.0pt}
\pspolygon[fillstyle=solid,fillcolor=lightgray](0,0)(1,0)(1,2)(0,2)(0,0)
\psline(0,0)(0,-1)(1,-1)(1,2)(5,2)(5,0)(0,0)
\psline(0,2)(0,3)(1,3)(1,2)(0,2)
\rput(3,1){\textrm{\small $A$}}
\psline[linecolor=lightgray](4,5)(4,2)
\rput[b]{-10}(2.8,4.3){$\iddots$}
\rput(-0.3,3.2){\textrm{\footnotesize $i_1$}}
\rput(.7,3.8){\textrm{\footnotesize $i_2$}}
\rput(1.7,4.3){\textrm{\footnotesize $i_3$}}
\rput(3.6,5.2){\textrm{\footnotesize $i_m$}}
\rput(4.7,.7){\textrm{\footnotesize $k$}}
\end{pspicture}
\begin{pspicture}(0.5,-1.3)(4,4.5)
\psset{xunit=.8cm}
\psset{yunit=.8cm}
\rput(2.7,-1.2){\textrm{\small b) avoiders counted by $N_d(x)$}}
\qdisk(0,4){2.0pt}
\qdisk(5,4.5){2.0pt}
\qdisk(2.5,3){2.0pt}
\qdisk(2,2){2.0pt}
\qdisk(3,1.5){2.0pt}
\qdisk(3.5,1){2.0pt}
\qdisk(4.5,0){2.0pt}
\psline(0,2)(0,4)(2,4)(2,2)(0,2)
\psline(5,2)(7,2)(7,4)(5,4)(5,2)
\rput(1,3){\textrm{\small $\pi'$}}
\rput(6,3){\textrm{\small $\pi''$}}
\rput(4,.6){\textrm{\small $\ddots$}}
\rput(-0.5,4.3){\textrm{\footnotesize $n\!-\!1$}}
\rput(4.8,4.8){\textrm{\footnotesize $n$}}
\rput(2.5,3.4){\textrm{\footnotesize $c$}}
\rput(2.2,1.8){\textrm{\footnotesize $d$}}
\rput(3.7,1.6){\textrm{\footnotesize $d\!-\!1$}}
\rput(4.2,1.1){\textrm{\footnotesize $d\!-\!2$}}
\rput(4.8,-0){\textrm{\footnotesize $1$}}
\end{pspicture}

\caption{$T$-avoiders}\label{fig90a}
\end{center}
\end{figure}

Deleting the letters in $A$ after $i_3$ and before $k$ and standardizing leaves an avoider counted by $H_2(x)$,
and the originals can be recovered by placing $i_3 \dots i_m$ immediately before the last letter, $k$, and then inserting 0 or more dots into each of the $m-2$ interstices between $i_3, \dots, i_m,k$.
Thus,
\[H_m(x)=\frac{x^{m-2}}{(1-x)^{m-2}}H_2(x)\,.\]
This equation says nothing for $m=2$ but the same considerations yield
\[H_2(x)=\frac{1}{1-x}M(x)\,,\]
where $M(x)$ is the \gf for $T$-avoiders in $S_n$ with $n-1$ in first position and $n$ in second to last position. To find $M(x)$, refine it to $M_d(x)$, the \gf for permutations $\pi=(n-1)\pi'n(d+1)\in S_n(T)$.

Clearly, $M_0(x)=x^3K(x)$, where $K(x):=F_{\{132,3412\}}(x)=\frac{1-2x}{1-3x+x^2}$  \cite[Seq. A001519]{Sl}.
For $d\ge 1$, the letters $d,d-1,\dots,1$ occur in decreasing order (or $n(d+1)$ is the 43 of a 1243) and there
is at most one letter $>d$ between $d$ and $1$ (if $u,v$ are 2 such letters, then $u<v$ implies $uv1(d+1)$ is a 3412, while $u>v$ implies $duv1$ is a 2431). Now consider whether there is a letter between $d$ and $d-1$ or not. If not, the contribution is $xM_{d-1}(x)$. Otherwise, there is exactly one such letter, say $c>d$, and $d-1,\dots,1$ are adjacent so that the permutation has the form in Figure \ref{fig90a}b), and we let $N_d(x)$ denote the \gf for these avoiders, so that $M_d(x)=xM_{d-1}(x) +N_d(x)$.

To find $N_d(x)$ we consider 3 cases according to the position of $n-2$.

\begin{itemize}
\item $n-2 \in \pi'$. Here, $n-2$ must be the leftmost letter in $\pi'$ and the contribution is $x N_d(x)$.
\item $c=n-2$. Here, $\pi'=\emptyset$ (or $(n-2)(d-1)(d+1)$ is the 412 of a 3412) and $\pi''$ avoids both 132 (or 1 is part of a 1243) and  3412. So the contribution is $x^{d+4}K(x)$.
\item $n-2 \in \pi''$. Here $\pi'c>$ the letters between $n-2$ and $n$ (or $(n-2)(d+1)$ is the 41 of a 2431) and are decreasing (or $c(n-2)$ is the 34 of a 3412). So the contribution is $x/(1-x)N_d(x)$.
\end{itemize}
Hence, $N_d(x)= xN_d(x)+x/(1-x)N_d(x)+x^{d+4}K(x)$, leading to
\[
N_d(x)=\frac{x^{d+4}(1-x)(1-2x)}{(1-3x+x^2)^2}\,.
\]
From $M_d(x)=xM_{d-1}(x) +N_d(x)$, it now follows by induction that
\[
M_d(x)=\frac{1-2x}{(1-3x+x^2)^2}x^{d+3}\big(1+(d-3)x+(1-d)x^2\big)\,,
\]
and so
\[
M(x)=\sum_{d\ge 0}M_d(x)=x^3\left(\frac{1-2x}{1-3x+x^2}\right)^2\,.
\]
Hence,
$$H(x)=\sum_{m\geq2}H_m(x)=\frac{1-x}{1-2x}H_2(x)=\frac{1}{1-2x}M(x)=\frac{x^3(1-2x)}{(1-3x+x^2)^2},$$
as required.
\end{proof}

\begin{lemma}\label{lem90a2}
Let $L_m(x)$ be the generating function for permutations $\pi=i_1\pi^{(1)}\cdots i_m\pi^{(m)}\in S_n(T)$ with
$m$ left-right maxima such that $\pi^{(m)}$ has a letter between $i_1$ and $i_2$. Then
$$\sum_{m\geq2}L_m(x)=\frac{x^3(1-2x+2x^2-3x^3+x^4)}{(1-x)^4(1-3x+x^2)}.$$
\end{lemma}
\begin{proof}
Suppose $m\geq3$ and $\pi=i_1\pi^{(1)}\cdots i_m\pi^{(m)}\in S_n(T)$ and that
$\pi^{(m)}$ has a letter between $i_1$ and $i_2$. Then $\pi^{(3)}\pi^{(4)}\cdots\pi^{(m)}$ is both decreasing and $>i_1$ (or $i_2i_3$ is the 34 of a 3412) and also $<i_2$ (or $i_1i_2i_3$ is the 124 of a 1243). Now let $\al$ denote the subsequence of letters of $\pi^{(m)}$ that are $<i_1$ and $\be$ those that are $>i_1$.
Then $\be>\pi^{(3)}\pi^{(4)}\cdots\pi^{(m)}$ (if $uv$ is a violator of this inequality,
$i_1ui_3v$ is a 1243),
$\al$ lies to the left of $\be$ (or $i_1i_2$ is the 24 of a 2431), and
$\pi^{(1)}\al$ is decreasing (or, for $u\in \pi^{(m)}$, $i_mu$ is the 43 of a 1243). These observations together imply that $\pi$ has the form shown in Figure \ref{fig90b},
\begin{figure}[htp]
\begin{center}
\begin{pspicture}(0,-0.6)(7,6)
\psset{xunit=.8cm}
\psset{yunit=.8cm}
\psline[linecolor=lightgray](3,5)(2,5)(2,0)
\psline[linecolor=lightgray](3,4)(3,2)(4,2)
\psline[linecolor=lightgray](6,6.5)(6,2)
\psline[linecolor=lightgray](4,5)(4,5.5)
\qdisk(0,2){2.0pt}
\qdisk(2,5){2.0pt}
\qdisk(4,5.5){2.0pt}
\qdisk(6,6.5){2.0pt}
\qdisk(6.7,2.3){2.0pt}
\psline(0,0)(0,2)(3,2)(3,0)(0,0)
\psline(3,5)(4,5)(4,2)(7,2)(7,4)(3,4)(3,5)
\rput(3.5,4.5){\textrm{\small $\be$}}
\rput(2.5,1){\textrm{\small $\al$}}
\rput(1.5,1.6){\textrm{\small $\pi^{(1)}$}}
\rput[b]{-10}(4.9,5.7){$\iddots$}
\rput(-0.3,2.2){\textrm{\footnotesize $i_1$}}
\rput(1.7,5.2){\textrm{\footnotesize $i_2$}}
\rput(3.7,5.7){\textrm{\footnotesize $i_3$}}
\rput(5.7,6.7){\textrm{\footnotesize $i_m$}}
\psline[arrows=->,arrowsize=4pt 2](.3,1.7)(2.7,0.3)
\psline[arrows=->,arrowsize=4pt 2](4.3,3.7)(6.5,2.5)
\end{pspicture}
\caption{A $T$-avoider counted by $L_m(x)$ when $m\ge 3$}\label{fig90b}
\end{center}
\end{figure}
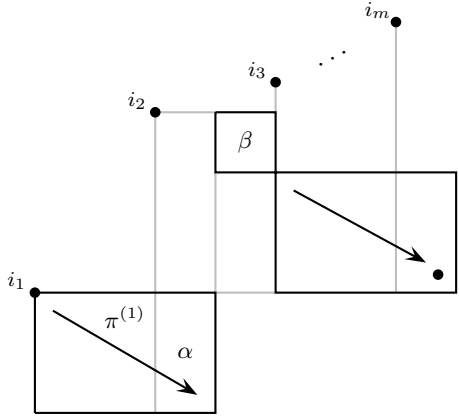
where the down arrows indicate decreasing entries and $\beta$ avoids $132$ and $3412$. It  follows that $L_m(x)=\frac{x^{m+1}}{(1-x)^m}K(x)$ for $m\ge 3$.

To find $L_2(x)$, suppose $\pi=i\pi' n\pi''\in S_n(T)$ has $2$ left-right maxima and $n-1 \in \pi''$. Since $\pi$ avoids $2431$, $\pi''$ can be written as $\alpha\beta$ such that $\alpha<i<\beta$ and $n-1$ belongs to $\beta$. If $\beta$ has a descent then $\pi'\alpha$ is decreasing ($\pi$ avoids $1243$) and $\beta$ avoids $132$ and $3412$. Thus, we have a contribution of $\frac{x^2}{(1-x)^2}(K(x)-1/(1-x))$. Otherwise, $\beta$ is increasing and, by deleting $\be$, we have a contribution of $\frac{1}{1-x}L'_2(x)$, where $L'_2(x)$ is the generating function for permutations $(n-2)\pi'n\pi''(n-1)\in S_n(T)$.

For $L'_2(x)$, if $\pi'=\emptyset$ we have a contribution of
$x^3/(1-x)$. Otherwise, $\pi'$ has $d\ge 1$ letters. Then $\pi'$ is decreasing
(to avoid 1243) say $\pi'=i_1i_2\cdots i_d$. Since $\pi$ avoids 2431,
we have $\pi''=\alpha_d\alpha_{d-1}\cdots\alpha_0$, where
$\alpha_d<i_d<\alpha_{d-1}<i_{d-1}<\cdots\alpha_{1}<i_1<\alpha_{0}$.
Since $\pi$ avoids 3412, we see that there exists at most one $j$ such that
$\alpha_{j }$ is decreasing and, if there is such a $j$, then $\alpha^{(i)}=\emptyset$ for $i\neq j$.
Thus, by considering whether there is such a $j$ or not,
we have a contribution of
$(d+1)x^{d+4}/(1-x)+x^{d+3}.$
So $L'_2(x)=x^3/(1-x)+\sum_{d\ge 1}\big((d+1)x^{d+4}/(1-x)+x^{d+3}\big)
=x^3/(1-x)+x^4/(1-x)^3$.

Hence,
$$L_2(x)=\frac{x^2}{(1-x)^2}\left(K(x)-\frac{1}{1-x}\right)+\frac{x^3}{(1-x)^2}+\frac{x^4}{(1-x)^4}\,,$$
and, finally,
$$L(x)=\sum_{m\geq2}L_m(x)=L_2(x)+\frac{x^4K(x)}{(1-x)^2(1-2x)},$$
which simplifies to the desired expression.
\end{proof}

\begin{theorem}\label{th90a}
Let $T=\{1243,2431,3412\}$. Then
$$F_T(x)=\frac{1-11x+51x^2-129x^3+195x^4-183x^5+104x^6-30x^7+3x^8}{(1-x)^4(1-2x)(1-3x+x^2)^2}.$$
\end{theorem}
\begin{proof}
Let $G_m(x)$ be the generating function for $T$-avoiders with $m$ left-right maxima. Clearly, $G_0(x)=1$ and $G_1(x)=xF_T(x)$.

Now suppose $\pi=i_1\pi^{(1)}\cdots i_m\pi^{(m)}\in S_n(T)$ with $m\geq2$ and let us find the contributions to $G_{m}(x)$.
We have $\pi^{(s)}<i_2$ for all $s=2,3,\ldots,m$ (or $i_1 i_2 i_s$ is the 124 of a 1243). Consider three cases:
\begin{itemize}
\item $\pi^{(m)}=\emptyset$. For this case we have a contribution of $xG_{m-1}(x)$.
\item $\pi^{(m)}$ has a letter between $i_1$ and $i_2$. These are the permutations counted by $L_{m}(x)$ in Lemma \ref{lem90a2}.
\item $\emptyset\neq\pi^{(m)}<i_1$.
Here, we have $\pi^{(s)}<i_1$ for all $s=2,3,\ldots,m$ (or $i_1 i_2$ is the 24 of a 2431) and so the permutations are precisely those counted by $H_{m}(x)$ in Lemma \ref{lem90a1}.
\end{itemize}
Therefore, $G_m(x)=xG_{m-1}(x)+L_m(x)+H_m(x)$ for all $m\geq2$. Summing over all $m\geq2$,
$$F_T(x)-1-xF_T(x)=x(F_T(x)-1)+L(x)+H(x),$$
where $H(x)$ and $L(x)$ are given by Lemmas \ref{lem90a1} and \ref{lem90a2}, respectively. Solve for $F_T(x)$.
\end{proof}

\subsection{Case 103: $\{1423,2341,3124\}$}
Let $G_m(x)$ denote the generating function for $T$-avoiders with $m$ left-right maxima.
Clearly, $G_0(x)=1$ and $G_1(x)=xF_T(x)$.
Recall that every nonempty permutation has a unique decomposition into reverse components, and a
permutation is reverse indecomposable if it has exactly one reverse component.
For example, 564132 has 3 reverse components: 56,\,4,\,132; and the permutations 1, 12, and 2413 are all reverse indecomposable. Note that all the patterns in $T$ contain 123.
\begin{lemma}\label{lem103i}
If $\pi\in S_n(T)$ has only $2$ left-right maxima, then the first reverse component of $\pi$ avoids $123$.
\end{lemma}
\begin{proof}
It suffices to show that if  $\pi\in S_n(T)$ with $2$ left-right maxima contains a 123 pattern, then the minimum 123 pattern $abc$ (the positions of the letters $a,b,c$ in $\pi$ are minimum in lex order) does not lie in the first reverse component. Consider the matrix diagram of $\pi$ as illustrated
\begin{figure}[htp]
\begin{center}
\begin{pspicture}(-5,-0.1)(10,4.5)
\psset{xunit=.45cm}
\psset{yunit=.45cm}
\psline(0,0)(10,0)(10,10)(0,10)(0,1)
\psline(7,7)(10,7)
\psline(0,7)(3,7)(3,10)
\pspolygon[fillstyle=solid,fillcolor=lightgray](0,0)(0,5)(5,5)(5,10)(7,10)(7,7)(3,7)(3,0)(0,0)
\pspolygon[fillstyle=solid,fillcolor=lightgray](0,5)(0,7)(5,7)(5,10)(3,10)(3,5)(0,5)
\rput(1.5,8.5){\textrm{\small $A$}}
\rput(8.5,8.5){\textrm{\small $B$}}
\rput(1.5,2.5){\textrm{\small min}}
\rput(4,6){\textrm{\small min}}
\rput(6,8.5){\textrm{\small min}}
\rput(3.5,3.0){\textrm{\small $a$}}
\rput(5.4,4.7){\textrm{\small $b$}}
\rput(7.3,6.6){\textrm{\small $c$}}
\qdisk(3,3){2pt}
\qdisk(5,5){2pt}
\qdisk(7,7){2pt}
\rput(1.5,5.7){\textrm{\small $\bullet abc$}}
\rput(1.5,6.3){\textrm{\small $3124$}}
\rput(4,8.8){\textrm{\small $a\!\bullet\! bc$}}
\rput(4,8.2){\textrm{\small $1423$}}
\end{pspicture}
\caption{A $T$-avoider with 2 LR max and a 123}\label{fig103m1}
\end{center}
\end{figure}
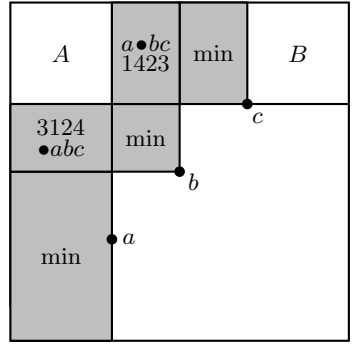
in Figure \ref{fig103m1},
where shaded regions are empty for the reason indicated (min refers to the minimum property of $abc$). If $a$ is a left-right maximum, then so are $b$ and $c$, violating the hypothesis. Hence, $a$ is not a left-right maximum, and so $A\ne \emptyset$. Also, $A>B$ (or $ab$ is the 12 of a 3124). This forces $A$ to consist of a square block at the upper left of the diagram and we are done.
\end{proof}

\begin{lemma}\label{lem103m2} We have
\[
G_2(x)= \big(xC(x)-x\big) F_T(x)\,.
\]
\end{lemma}
\begin{proof}
Suppose $\pi\in S_n(T)$ has $2$ left-right maxima. The first
reverse component of $\pi$ has length $\ge 2$ and, from Lemma \ref{lem103i}, avoids 123.
The rest of $\pi$ is an arbitrary $T$-avoider and so $G_2(x)=R(x)F_T(x)$, where $R(x)$
is the generating function for reverse indecomposable $123$-avoiders of length $\ge 2$.
We have $R(x)=xC(x)-x$: given the matrix diagram of a
reverse indecomposable $\{123\}$-avoider $\pi$ of length $n$, the map ``form lattice path from $(0,n)$ to $(n,0)$ enclosing the right-left maxima of $\pi$'' is a bijection to indecomposable Dyck paths, whose enumeration is well known.
\end{proof}

\begin{lemma}\label{lem103a1} We have
\[
G_3(x)=\frac{x^3}{(1-x)^5}\,,\textrm{and}
\]
$$\sum_{m\geq4}G_m(x)=\frac{1}{1-t-t^3}-1-t-t^2-2t^3-2t^4\,,$$
where $t=x/(1-x)$.
\end{lemma}
\begin{proof}
Suppose $\pi=i_1\pi^{(1)}i_2\pi^{(2)}\cdots i_m\pi^{(m)}\in S_n(T)$
has $m\geq3$ left-right maxima.
We have $\pi^{(s)}>i_{s-2}$ for $s=3,4,\ldots,m$  since $u\in \pi^{(s)}$
with $u<i_{s-2}$ makes $i_{s-2}i_{s-1}i_s u$ a 2341.  Also, $\pi^{(s)}$ is decreasing
for all $s$ (a violator $uv$ makes $i_s uv i_{s+1}$ a 3124 for $s=1,2,\dots,m-1$ and
$i_{m-2}i_m uv$ a 1423 for $s=m$).
So $\pi$ has the form illustrated for $m=4$ in Figure \ref{fig103a}a), where the down
arrows indicate decreasing entries. Moreover, $\al_s\be_{s+1}$ is decreasing for all $s$
(a violator $uv$ makes $i_s uv i_{s+2}$ a 3124 for $s=1,2,\dots,m-2$ and $i_{s-1}i_suv$
a 1423 for $s=2,3,\dots,m-1$, covering all cases) and so $\pi$ decomposes further as in Figure \ref{fig103a}b).
\begin{figure}[htp]
\begin{center}
\begin{pspicture}(-5,-.8)(5.5,5.9)
\psset{xunit=.67cm}
\psset{yunit=.67cm}
\psline(-9,0)(-7,0)(-7,4)(-6,4)(-6,2)(-9,2)(-9,0)
\psline(-6,2)(-5,2)(-5,0)(-6,0)(-6,2)
\psline(-5,4)(-5,6)(-4,6)(-4,2)(-3,2)(-3,4)(-5,4)
\psline(-3,6)(-3,8)(-2,8)(-2,4)(-1,4)(-1,6)(-3,6)
\psline(-2,4)(-2,6)(-1,6)(-1,4)(-2,4)
\psline[linecolor=lightgray](-7,0)(-6,0)
\psline[linecolor=lightgray](-6,4)(-5,4)(-5,2)(-4,2)
\psline[linecolor=lightgray](-4,6)(-3,6)(-3,4)(-2,4)
\psline(1,1)(1,2)(2,2)(2,1)(1,1)
\psline(2,3)(2,4)(3,4)(3,3)(2,3)
\psline(4,5)(4,6)(5,6)(5,5)(4,5)
\psline(6,6)(6,7)(7,7)(7,6)(6,6)
\psline(3,1)(4,1)(4,0)(3,0)(3,1)
\psline(5,2)(5,3)(6,3)(6,2)(5,2)
\psline(7,4)(7,5)(8,5)(8,4)(7,4)
\psline[linecolor=lightgray](2,2)(2,3)
\psline[linecolor=lightgray](3,1)(3,3)(5,3)
\psline[linecolor=lightgray](2,1)(3,1)
\psline[linecolor=lightgray](3,4)(7,4)
\psline[linecolor=lightgray](2,2)(5,2)(5,5)(7,5)(7,6)
\psline[linecolor=lightgray](4,1)(4,5)
\psline[linecolor=lightgray](6,3)(6,6)(5,6)
\rput(-8,1){\textrm{\small $\al_1\downarrow$}}
\rput(-6.5,3){\textrm{\small $\al_2\downarrow$}}
\rput(-4.5,5){\textrm{\small $\al_3\downarrow$}}
\rput(-2.5,7){\textrm{\small $\al_m\downarrow$}}
\rput(-5.5,1){\textrm{\small $\be_2\downarrow$}}
\rput(-3.5,3){\textrm{\small $\be_3\downarrow$}}
\rput(-1.5,5){\textrm{\small $\be_m\downarrow$}}
\rput(-9.3,2.2){\textrm{\small $i_1$}}
\rput(-7.3,4.2){\textrm{\small $i_2$}}
\rput(-5.3,6.2){\textrm{\small $i_3$}}
\rput(-3.4,8.2){\textrm{\small $i_m$}}
\rput(.7,2.2){\textrm{\small $i_1$}}
\rput(1.7,4.2){\textrm{\small $i_2$}}
\rput(3.7,6.2){\textrm{\small $i_3$}}
\rput(5.7,7.2){\textrm{\small $i_m$}}
\rput(-0.5,3){$\longrightarrow$}
\qdisk(-9,2){2pt}
\qdisk(-7,4){2pt}
\qdisk(-5,6){2pt}
\qdisk(-3,8){2pt}
\qdisk(1,2){2pt}
\qdisk(2,4){2pt}
\qdisk(4,6){2pt}
\qdisk(6,7){2pt}
\rput(1.5,1.5){\textrm{\small $\al_1\downarrow$}}
\rput(2.5,3.5){\textrm{\small $\al_2\downarrow$}}
\rput(4.5,5.5){\textrm{\small $\al_3\downarrow$}}
\rput(6.5,6.5){\textrm{\small $\al_m\!\downarrow$}}
\rput(3.5,0.5){\textrm{\small $\be_2\downarrow$}}
\rput(5.5,2.5){\textrm{\small $\be_3\downarrow$}}
\rput(7.5,4.5){\textrm{\small $\be_m\downarrow$}}
\rput(-5,-1){\textrm{a)}}
\rput(4.5,-1){\textrm{b)}}
\end{pspicture}
\caption{A $T$-avoider with $m\ge 3$ LR max}\label{fig103a}
\end{center}
\end{figure}
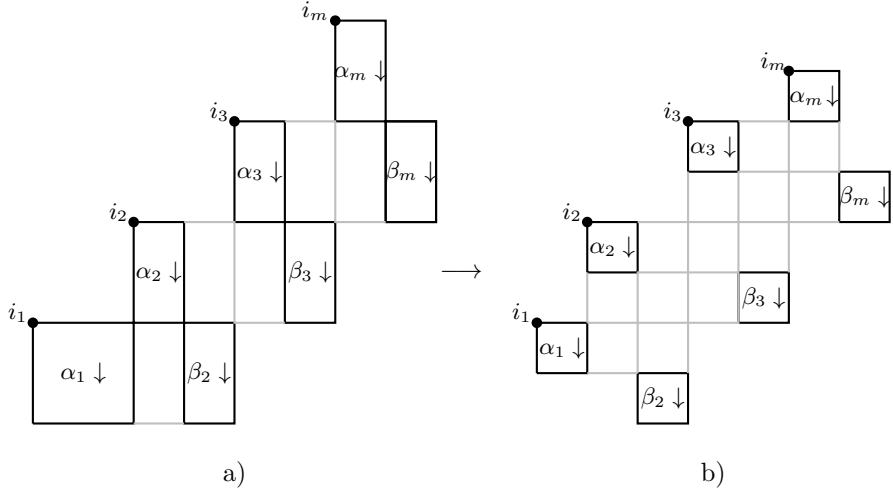

If $m=3$, there are no further restrictions and so there are 5 boxes to be filled each with an arbitrary number of dots, and $G_3(x)=\frac{x^3}{(1-x)^5}$ as required.

But if $m\ge 4$, there is one more restriction: no two
consecutively-indexed $\be$'s, say $\be_s$ and $\be_{s+1}$, are both nonempty ($u\in \be_s$
and $v\in \be_{s+1}$ makes $i_{s-2}i_suv$ a 1423 for $s=3,4,\dots,m-1$ and $i_s uvi_m$
a 3124 for $s=2,3,\dots,m-2$; note that neither condition says anything when $m=3$).
The contribution with $k$ nonempty $\be$'s is $x^{m+k}/(1-x)^{m-k}$.
A specification of empty/nonempty for each $\be$ corresponds to a binary string $w_1 w_2 \dots w_m$ with $w_1=0$ and $w_s=1$ if and only if $\be_s\ne \emptyset$ for $s\ge 1$. Let $H_m(t)$ denote the \gf for binary strings $w_1 w_2 \dots w_m$ with first bit 0 and no two consecutive 1's, where $t$ marks number of 1's. Then $G_m(x)=\big(\frac{x}{1-x}\big)^m H_m\big( \frac{x}{1-x}\big)$ for $m\ge 4$.

To find $\sum_{m\geq4}G_m(x)$, set $H(t,y)=\sum_{m\ge1}H_m(t)y^m$. We have the recurrence $H_m(t)=H_{m-1}(t)+t H_{m-2}(t)$ for $m\ge 3$ with initial conditions $H_1(t)=1$ and $H_2(t)=1+t$. It follows routinely that $H(t,y)=y(1+ty)/(1-y-ty^2)$.
Now, with $t:=x/(1-x)$, $\sum_{m\geq4}G_m(x)=\sum_{m\geq4}t^m H_m(t)=H(t,t)-\sum_{m=1}^3t^m H_m(t)$, which simplifies to the stated expression.
\end{proof}

Lemmas \ref{lem103m2} and \ref{lem103a1} now give an expression for the right side in the
identity $F_T(x)= \sum_{m\ge0}G_m(x)$. Solving for $F_T(x)$ yields the following result.

\begin{theorem}\label{th103a}
Let $T=\{1423,2341,3124\}$. Then
$$F_T(x)=\frac{1-9x+35x^2-77x^3+107x^4-97x^5+55x^6-17x^7+x^8}{(1-x)^5(1-4x+5x^2-3x^3)}C(x).$$
\end{theorem}

\subsection{Case 106: $\{1342,2143,3412\}$} We consider $G_m(x)$, the \gf for $T$-avoiders with $m$ left-right maxima. Most of the work is in finding $G_2(x)$.
\begin{lemma}\label{lem106a0}
Suppose $\pi=\pi_1\pi_2\cdots \pi_n$ is a $T$-avoider. Then $\pi$ has exactly $2$ left-right maxima if and only if $\pi_1=n-1$ or $\pi_2=n$  or both.
\end{lemma}
\begin{proof}
If $\pi$ has exactly $2$ left-right maxima but neither of the conditions holds, then $n-1$ occurs after $n$ and $\pi_1\pi_2n(n-1)$ is a 2143. The other direction is clear.
\end{proof}

Define $H(x)$ to be the generating function for permutations $\pi\in S_n(T)$ with $\pi_1=n-1$.
\begin{lemma}\label{lem106a2} We have 
$$H(x)=\frac{x^2\big(x^2+(1-x)^3F_T(x)\big)}{(1-x)^4}.$$
\end{lemma}
\begin{proof}
Refine $H(x)$ to $H_d(x)$, the generating function for permutations $\pi=(n-1)\pi'n\pi''\in S_n(T)$ such that $\pi''$ has exactly $d$ letters. Clearly, $H_0(x)=x^2 F_T(x)$. For $d\ge 1$,
say $\pi''=j_1j_2\cdots j_d$. Then the $j$'s are decreasing (or $(n-1)n$ is the 34 of a 3412) and $\pi$ has the form shown in Figure \ref{fig106m1},
\begin{figure}[htp]
\begin{center}
\begin{pspicture}(-3,-0.3)(10,3.8)
\psset{xunit=.8cm}
\psset{yunit=.5cm}
\psline(0,5)(0,6)(1,6)(1,4)(3,4)(3,3)(2,3)(2,5)(0,5)
\psline(4,2)(5,2)(5,0)(6,0)(6,1)(4,1)(4,2)
\psline[linecolor=lightgray](2,5)(6.5,5)
\psline[linecolor=lightgray](3,4)(7,4)
\psline[linecolor=lightgray](5,2)(8,2)
\psline[linecolor=lightgray](6,1)(8.5,1)
\psline[linecolor=lightgray](6,1)(6,6.5)
\rput(.5,5.5){\textrm{\small $\al_1$}}
\rput(1.5,4.5){\textrm{\small $\al_2$}}
\rput(2.5,3.5){\textrm{\small $\al_3$}}
\rput(4.5,1.5){\textrm{\small $\al_d$}}
\rput(5.5,0.5){\textrm{\small $\al_{d+1}$}}
\rput(-0.6,6.2){\textrm{\small $n\!-\!1$}}
\rput(6.4,6.7){\textrm{\small $n$}}
\rput(6.9,5.2){\textrm{\small $j_1$}}
\rput(7.4,4.2){\textrm{\small $j_2$}}
\rput(8.6,2.2){\textrm{\small $j_{d-1}$}}
\rput(8.9,1.2){\textrm{\small $j_d$}}
\rput(3.4,2.6){$\ddots$}
\rput(7.5,3.1){$\ddots$}
\qdisk(0,6){2pt}
\qdisk(6,6.5){2pt}
\qdisk(6.5,5){2pt}
\qdisk(7,4){2pt}
\qdisk(8,2){2pt}
\qdisk(8.5,1){2pt}
\end{pspicture}
\caption{A $T$-avoider with first entry $n-1$ and $n$ not the last entry}\label{fig106m1}
\end{center}
\end{figure}
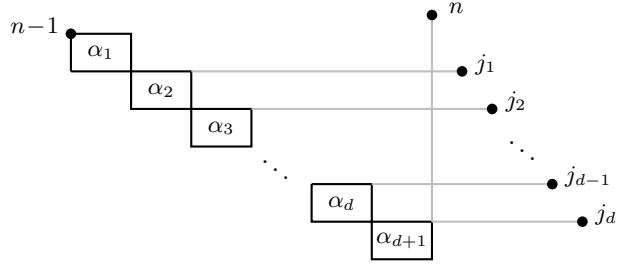
where $\al_s$ lies to the left of $\al_{s+1}$ for all $s$ (a violator $uv$ makes $uvnj_s$ a 1342) and $\alpha^{(2)}\cdots \alpha^{(d)}\alpha^{(d+1)}$ is increasing (a violator $uv$ makes $uvnj_1$ a 2143). So, in particular, at most one of
$\alpha^{(2)},\dots, \alpha^{(d)},\alpha^{(d+1)}$ is nonempty.
Two cases:
\begin{itemize}
\item $\alpha_1$ has an ascent.
Here $\alpha_2=\cdots= \alpha_d=\alpha_{d+1}=\emptyset$
(take $ab$ in $\al_1$ with $a<b$, then
$u\in \alpha_2\cdots \alpha_d \alpha_{d+1}$ makes $abuj_1$ a $3412$), and
the contribution is $$x^{d+2}\big(F_T(x)-\frac{1}{1-x}\big).$$

\item $\al_1$ is decreasing. Here, the
contributions are $x^{d+2}\frac{1}{1-x}$ (for $\alpha_2,\dots, \alpha_d,\alpha_{d+1}$ all empty) and  $dx^{d+2}\frac{x}{(1-x)^2}$ (for one of $\alpha_2,\dots, \alpha_d,\alpha_{d+1}$ nonempty).
\end{itemize}

Therefore,
\[
H_d(x)=x^{d+2}\left(F_T(x)-\frac{1}{1-x}\right)+x^{d+2}\left(\frac{1}{1-x}+d\frac{x}{(1-x)^2}\right)
\]
for all $d\geq1$.
Since $H(x)=\sum_{d\ge 0}H_d(x)$, the result follows by summing over $d$.
\end{proof}
Now define $J(x)$ to be the generating function for permutations $\pi\in S_n(T)$ with $\pi_2=n$.
\begin{lemma}\label{lem106a3} We have
$$J(x)=\frac{x^2(1-2x)}{(1-x)(1-3x)}.$$
\end{lemma}
\begin{proof}
Refine $J(x)$, to $J_d(x)$, the generating function for permutations $\pi=(d+1)n\pi'\in S_n(T),\ d\ge 0$. If $d=0$ so that $\pi_1=1$, then $\pi'$ avoids 231 (or else $\pi_1$ is the ``1'' of a 1342). Hence  $J_0(x)=x^2K(x)$ where $K(x)=\frac{1-2x}{1-3x+x^2}$ is the generating function for $\{231,2143\}$ avoiders \cite[Seq. A001519]{Sl}. Now suppose $d\ge 1$.
The letters $\le d$ are decreasing in $\pi$ (else $(d+1)\,n$ forms the 34 of a 3412).

The contribution for $n=d+2$ is $x^{d+2}$ since, here, $\pi_1=d+1$ and  $\pi_2=d+2$. So suppose $n>d+2$. Then $\pi$ has the form shown in Figure \ref{fig106m2},
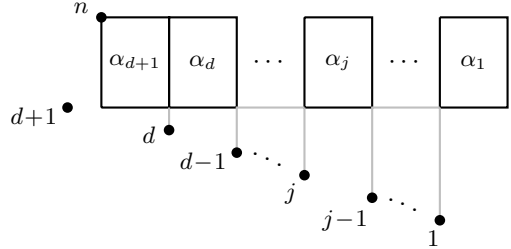
\begin{figure}[htp]
\begin{center}
\begin{pspicture}(-1,-1.5)(6,1.2)
\psset{xunit=.9cm}
\psset{yunit=.6cm}
\psline(0,0)(0,2)(1,2)(1,0)(2,0)(2,2)(1,2)(1,0)(0,0)
\psline(3,2)(4,2)(4,0)(3,0)(3,2)
\psline(5,2)(6,2)(6,0)(5,0)(5,2)
\psline[linecolor=lightgray](1,0)(1,-.5)
\psline[linecolor=lightgray](3,0)(2,0)(2,-1)
\psline[linecolor=lightgray](3,0)(3,-1.5)
\psline[linecolor=lightgray](5,0)(4,0)(4,-2)
\psline[linecolor=lightgray](5,0)(5,-2.5)
\rput(-1,-0.2){\textrm{\small $d\!+\!1$}}
\rput(-.3,2.2){\textrm{\small $n$}}
\rput(.7,-.6){\textrm{\small $d$}}
\rput(1.5,-1.2){\textrm{\small $d\!-\!1$}}
\rput(2.8,-1.9){\textrm{\small $j$}}
\rput(3.6,-2.5){\textrm{\small $j\!-\!1$}}
\rput(4.9,-2.9){\textrm{\small $1$}}
\rput(5.5,1){\textrm{\small $\al_1$}}
\rput(3.5,1){\textrm{\small $\al_j$}}
\rput(1.5,1){\textrm{\small $\al_d$}}
\rput(0.5,1){\textrm{\small $\al_{d+1}$}}
\qdisk(-.5,0){2pt}
\qdisk(0,2){2pt}
\qdisk(1,-.5){2pt}
\qdisk(2,-1){2pt}
\qdisk(3,-1.5){2pt}
\qdisk(4,-2){2pt}
\qdisk(5,-2.5){2pt}
\rput(2.5,1){$\cdots$}
\rput(4.5,1){$\cdots$}
\rput[b]{15}(2.5,-1.4){$\ddots$}
\rput[b]{15}(4.5,-2.4){$\ddots$}
\end{pspicture}
\caption{A $T$-avoider with second entry $n$}\label{fig106m2}
\end{center}
\end{figure}
where $n-1$ occurs in one of the $\al$'s. We consider three cases according to the position of $n-1$:
\begin{itemize}
\item  $n-1$ does not lie in $\al_{d+1}$. Say $n-1 \in \al_j\ 1\le j \le d$. Here, $n-1$ is the last entry in $\al_j$ and all $\al$'s to the right of $\al_j$ are empty (or else $(d+1)d(n-1)$ is the 214 of a 2143). Thus $\pi$ ends with $(n-1)(j-1)(j-2)\cdots 1$. Deleting these entries, we get a contribution of $x^j J_{d-(j-1)}$.
\item $n-1$ lies in $\al_{d+1}$ and there is a letter $>d+1$ on the right of $n-1$.
Here, $\pi$ has the form shown in Figure \ref{fig106m3},
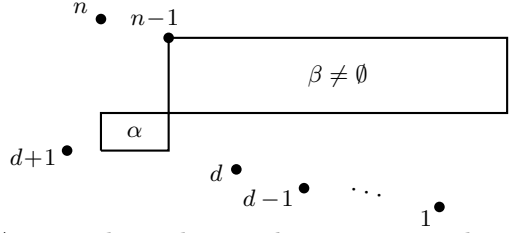
\begin{figure}[htp]
\begin{center}
\begin{pspicture}(-2,-.5)(6,2.5)
\psset{xunit=.9cm}
\psset{yunit=.5cm}
\psline(0,0)(0,1)(6,1)(6,3)(1,3)(1,0)(0,0)
\rput(-1,-0.2){\textrm{\small $d\!+\!1$}}
\rput(-.3,3.8){\textrm{\small $n$}}
\rput(.8,3.5){\textrm{\small $n\!-\!1$}}
\rput(1.7,-.6){\textrm{\small $d$}}
\rput(2.5,-1.3){\textrm{\small $\!d-\!1$}}
\rput(4.8,-1.8){\textrm{\small $1$}}
\rput(.5,.5){\textrm{\small $\al$}}
\rput(3.5,2){\textrm{\small $\be\ne\emptyset$}}
\qdisk(-.5,0){2pt}
\qdisk(0,3.5){2pt}
\qdisk(1,3){2pt}
\qdisk(2,-.5){2pt}
\qdisk(3,-1){2pt}
\qdisk(5,-1.5){2pt}
\rput[b]{25}(4,-1.4){$\ddots$}
\end{pspicture}
\caption{A $T$-avoider with $\pi_1=d+1,\:\pi_2=n$ and $n-1$ before $d$}\label{fig106m3}
\end{center}
\end{figure}
where $\al <\be$ ($a\in \al$ and $b\in \be$ with $a>b$ makes $(d+1)a(n-1)b$ a 1342) and $\al$ is increasing ($a_1 a_2$ in $\al$ with $a_1>a_2$ makes $a_1a_2(n-1)b$ a 2143 for any $b\in \be$). Deleting $\al$ and $n-1$, the contribution is $$\frac{x}{1-x}\big(J_d(x)-x^{d+2}\big)$$ since $J_d(x)$ allows for $\be=\emptyset$ and we must correct for the overcount.

\item $n-1$ lies in $\al_{d+1}$ and there is no letter $>d+1$ on the right of $n-1$. Here, $\pi$
has the form $(d+1)n \al (n-1)d(d-1) \cdots 1$ where $\al$ avoids $231$ (or $d+1$ is the ``1'' of a 1342). So the contribution is $x^{d+3}K(x)$ where, again, $K(x)=F_{ \{231,2143\} }(x)$.
\end{itemize}

Thus, for $d\geq1$,
$$J_d(x)=x^{d+2}+\sum_{j=1}^dx^jJ_{d+1-j}(x)+\frac{x}{1-x}\left(J_d(x)-x^{d+2}\right)+x^{d+3}K(x),$$
Since $J(x)=\sum_{d\ge0}J_d(x)$ and $J_0(x)=x^2K(x)$, summing over $d\geq1$ yields
$$J(x)-x^2K(x)=\frac{x^3}{1-x}+\frac{2x}{1-x}\left(J(x)-x^2K(x)\right)-\frac{x^4}{(1-x)^2}+\frac{x^4}{1-x}K(x)\,.$$
The result follows by solving for $J(x)$.
\end{proof}

\begin{theorem}\label{th106a}
Let $T=\{1342,2143,3412\}$. Then
$$F_T(x)=\frac{(1-2x)(1-6x+12x^2-9x^3+4x^4)}{(1-x)^3(1-3x)(1-3x+x^2)}.$$
\end{theorem}
\begin{proof}
Let $G_m(x)$ be the generating function for $T$-avoiders with $m$ left-right maxima.
Clearly, $G_0(x)=1$ and $G_1(x)=xF_T(x)$. From Lemmas \ref{lem106a2} and \ref{lem106a3},
we have  $G_2(x)=H(x)+J(x)-\frac{x^2}{1-x}$ because the $T$-avoiders $\pi\in S_n$ with
$\pi_1=n-1$ and $\pi_2=n$ are counted by both $H(x)$ and $J(x)$.
Now suppose $m\ge 3$. Then $\pi$ has the form illustrated in Figure \ref{fig106m4},
\begin{figure}[htp]
\begin{center}
\begin{pspicture}(-1,-.2)(6,2.6)
\psset{xunit=.75cm}
\psset{yunit=.55cm}
\psline(0,0)(5,0)(5,5)(4,5)(4,4)(3,4)(3,3)(2,3)(2,2)(1,2)(1,1)(0,1)(0,0)
\psline(5,4)(4,4)(4,3)(3,3)(3,2)(2,2)(2,0)(1,0)(1,1)
\psline[linecolor=lightgray](0,0)(0,-1)
\psline[linecolor=lightgray](1,0)(1,-1)
\psline[linecolor=lightgray](2,0)(2,-1)
\psline[linecolor=lightgray](3,0)(3,-1)
\psline[linecolor=lightgray](4,0)(4,-1)
\psline[linecolor=lightgray](5,0)(5,-1)
\pspolygon[fillstyle=solid,fillcolor=lightgray](5,4)(4,4)(4,3)(3,3)(3,2)(2,2)(2,1)(5,1)(5,4)
\rput(.5,-.5){\textrm{\small $\pi^{(1)}$}}
\rput(1.5,-.5){\textrm{\small $\pi^{(2)}$}}
\rput(2.5,-.5){\textrm{\small $\pi^{(3)}$}}
\rput(3.5,-.5){\textrm{\small $\pi^{(4)}$}}
\rput(4.5,-.5){\textrm{\small $\pi^{(m)}$}}
\rput(-.3,1.2){\textrm{\small $i_1$}}
\rput(.7,2.2){\textrm{\small $i_2$}}
\rput(1.7,3.2){\textrm{\small $i_3$}}
\rput(2.7,4.2){\textrm{\small $i_4$}}
\rput(3.7,5.2){\textrm{\small $i_m$}}
\qdisk(0,1){2pt}
\qdisk(1,2){2pt}
\qdisk(2,3){2pt}
\qdisk(3,4){2pt}
\qdisk(4,5){2pt}
\rput(4,2){\textrm{\small $1\,3\,4\,2$}}
\end{pspicture}
\caption{A $T$-avoider with $m\ge3$ left-right maxima}\label{fig106m4}
\end{center}
\end{figure}
where the shaded region is empty to avoid the indicated pattern. Three cases:
\begin{itemize}
\item $\pi^{(m)}=\emptyset$. Here, the contribution is $xG_{m-1}(x)$.

\item $\pi^{(m)}\ne\emptyset$ and $\pi^{(m)}$ has a letter in the interval $[i_{m-1},i_m]$.
Here, $\pi^{(1)},\pi^{(2)},\dots,\pi^{(m-1)}$ are all empty (or $i_m$ is the 4 of a 2143). So $i_1,\dots,i_{m-1}$ are in consecutive positions; deleting $i_1,\dots,i_{m-2}$ and standardizing leaves an avoider with $\pi_2=n$ and $\pi_1 \ne n-1$, counted by $J(x)-\frac{x^2}{1-x}$. Thus, the contribution is $$x^{m-2}\big(J(x)-\frac{x^2}{1-x}\big).$$

\item  $\pi^{(m)}\ne\emptyset$ and $\pi^{(m)}$ has no letter in the interval $[i_{m-1},i_m]$.
Here, $i_1 \pi^{(1)}\cdots i_{m-1}\pi^{(m-1)}>\pi^{(m)}$ since for $v \in \pi^{(m)}$ and $u \in \pi^{(s)}$ with $1\le s \le m-2,\ u<v$ makes $ui_{m-1}i_m v$ a 1342, while $u\in i_{m-1}$ and $u<v$ makes $i_{m-2}i_{m-1}uv$ a 3412. So the contribution is $$\frac{x^2}{1-x}G_{m-1}(x).$$
\end{itemize}

Adding the three contributions, we have, for $m\geq3$,
$$G_m(x)=\frac{x}{1-x}G_{m-1}(x)+x^{m-2}\left(J(x)-\frac{x^2}{1-x}\right)\,.$$
Summing over $m\geq3$, we obtain
$$F_T(x)-G_0(x)-G_1(x)-G_2(x)=\frac{x}{1-x}\big(F_T(x)-G_0(x)-G_1(x)\big)+\frac{x}{1-x}\left(J(x)-\frac{x^2}{1-x}\right)\,.$$
Substitute for $G_0(x),G_1(x),G_2(x),J(x)$, and solve for $F_T(x)$ to complete the proof.
\end{proof}

\subsection{Case 118: $\{1423,1234,3412\}$}
We deal with right-left maxima and begin with the case where they are contiguous.
For $m\ge 1$, define $J_m=J_m(x)$ to be the generating function for $T$-avoiders of the form $\pi i_mi_{m-1}\cdots i_1\in S_n$ with $m$ contiguous right-left maxima $n=i_m>i_{m-1}>\cdots>i_1\geq1$, and set $J_0=1$.
\begin{proposition}\label{prop118a1}
For $m\geq2$,
\begin{align*}
J_m = xJ_m + (2x-x^2)J_{m-1} -x^2J_{m-2}
+\frac{x^{m+2}}{(1-x)^{m-1}(1-2x)}+\frac{x^{m+3}}{(1-x)^2(1-2x)}\,.
\end{align*}
\end{proposition}
\begin{proof}
Let $\pi=\pi' i_mi_{m-1}\cdots i_1\in S_n(T)$ be an avoider counted by $J_m$.

If $i_1=1$, the  contribution is $xJ_{m-1}$. 

If $i_1>1$ and $i_{m-1} = n-1$, the contribution is $xJ_{m-1}-x^2 J_{m-2}$ (needs $J_0=1$ when $m=2$).

Now suppose $i_1>1$ and $i_{m-1} <n-1$ so that $\pi$ has the form
$\pi=\alpha (n-1)\beta i_mi_{m-1}\cdots i_1$ with $n=i_m>n-1>i_{m-1}>\cdots>i_1>1$,
where $\al$ is decreasing (or $(n-1)n$ is the 34 of a 1234).

If $\alpha=\emptyset$, the contribution is $x(J_m-xJ_{m-1})$.
Otherwise, we count by the position of 1.
\begin{itemize}
\item $1\in \al$. Here, $\pi$ has the form in Figure \ref{fig118J}a),
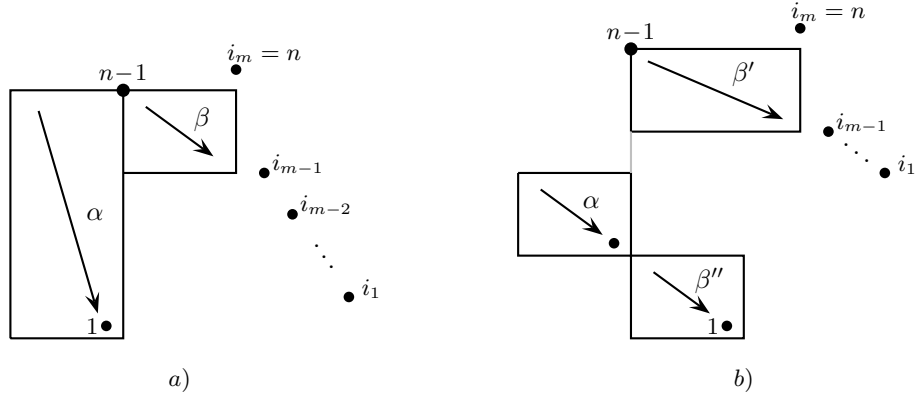
\begin{figure}[htp]
\begin{center}
\begin{pspicture}(-1.87,-.6)(12,5)
\psset{xunit=.75cm}
\psset{yunit=.55cm}
\psline(0,0)(0,6)(2,6)(2,0)(0,0)
\psline(2,6)(4,6)(4,4)(2,4)
\rput(1.5,3){\textrm{\normalsize $\al$}}
\rput(3.4,5.3){\textrm{\normalsize $\be$}}
\rput(5.1,4.2){\textrm{\small $i_{m-1}$}}
\rput(5.6,3.2){\textrm{\small $i_{m-2}$}}
\rput(6.4,1.2){\textrm{\small $i_{1}$}}
\rput(2,6.4){\textrm{\small $n\!-\!1$}}
\rput(4.5,6.9){\textrm{\small $i_m=n$}}
\rput(1.45,.3){\textrm{\small $1$}}
\rput(12.45,.3){\textrm{\small $1$}}
\rput(3,-1){\textrm{\small $a)$}}
\psline[arrows=->,arrowsize=4pt 2](.5,5.5)(1.55,0.6)
\psline[arrows=->,arrowsize=4pt 2](2.4,5.6)(3.6,4.4)
\rput(13,-1){\textrm{\small $b)$}}
\qdisk(2,6){2.5pt}
\qdisk(1.7,.3){2pt}
\qdisk(4,6.5){2pt}
\qdisk(4.5,4){2pt}
\qdisk(5,3){2pt}
\qdisk(6,1){2pt}
\qdisk(11,7){2.5pt}
\qdisk(10.7,2.3){2pt}
\qdisk(14,7.5){2pt}
\qdisk(14.5,5){2pt}
\qdisk(15.5,4){2pt}
\qdisk(12.7,.3){2pt}
\psline(9,4)(11,4)(11,0)(13,0)(13,2)(9,2)(9,4)
\psline(11,5)(11,7)(14,7)(14,5)(11,5)
\psline[linecolor=lightgray](11,4)(11,5)
\rput(11,7.4){\textrm{\small $n\!-\!1$}}
\rput(14.5,7.9){\textrm{\small $i_m=n$}}
\rput(10.3,3.3){\textrm{\normalsize $\al$}}
\rput(13,6.4){\textrm{\normalsize $\be'$}}
\rput(12.4,1.4){\textrm{\normalsize $\be''$}}
\rput(15.1,5.2){\textrm{\small $i_{m-1}$}}
\rput(15.9,4.2){\textrm{\small $i_{1}$}}
\psline[arrows=->,arrowsize=4pt 2](9.4,3.6)(10.5,2.5)
\psline[arrows=->,arrowsize=4pt 2](11.4,1.6)(12.4,0.6)
\psline[arrows=->,arrowsize=4pt 2](11.3,6.7)(13.7,5.3)
\rput[b]{0}(15.0,4.2){$\ddots$}
\rput[b]{-17}(5.5,1.7){$\ddots$}
\end{pspicture}
\caption{A $T$-avoider with $i_1>1$ and $i_{m-1} <n-1$  and $\alpha\ne \emptyset$}\label{fig118J}
\end{center}
\end{figure}
\noindent where the down arrows indicate decreasing and the bullets indicate mandatory entries;
$\be>i_{m-1}$ since $b\in \be$ with $b<i_{m-1}$ makes $1(n-1)bi_{m-1}$ a 1423,
and $\be$ is decreasing (or $1n$ is the 14 of a 1234).
This reduces matters to a balls-in-boxes problem and the contribution is
$\frac{x^{m+2}}{(1-x)^{m-1}(1-2x)}$.

\item $1\in \be$. Here, $\pi$ has the form in Figure \ref{fig118J}b)
where $\beta'$ denote the letters $>i_1$ in $\beta$, and $\beta''$ the letters $<i_1$ in $\beta$.
Note that $\al<i_1$ since $a\in \al,\ a>i_1$ makes $a(n-1)1i_1$ a 3412, and $\al>\be''$ since a violator $ab$ makes $a(n-1)bi_1$ a 1423.
Also, $\be'>i_{m-1}$ since $a\in \al$ and $b\in \be'$ with $b<i_{m-1}$ makes $a(n-1)bi_{m-1}$ a 1423.
Again, this is a balls-in-boxes problem and the contribution is $\frac{x^{m+3}}{(1-x)^2(1-2x)}$.
\end{itemize}
The result now follows by summing contributions.
\end{proof}

\begin{corollary}\label{lem118a2}
$$\sum_{m\geq2}J_m=\frac{x^2(1-7x+22x^2-35x^3+29x^4-13x^5)}{(1-x)^4(1-2x)^3}\,.$$
\end{corollary}
\begin{proof} Define $J=\sum_{m\geq0}J_m$. Summing the recurrence of Proposition \ref{prop118a1} over $m\ge 2$ yields
\[
J-J_0-J_1=(3x-2x^2)J-x J_1-(3x-x^2)J_0
+\frac{x^4}{(1-2x)^2}+\frac{x^5}{(1-x)^3(1-2x)}\,.
\]
Recall $J_0=1$. Clearly, $J_1=xF_{\{123,3412\}}(x)$, and $F_{\{123,3412\}}(x)=
 1+\frac{x(1-4x+7x^2-5x^3+2x^4)}{(1-x)^4(1-2x)} $ \cite{wikipermpatt}.
The result follows by solving for $J$ and computing $J-J_0-J_1$.
\end{proof}
\begin{theorem}\label{th118a}
Let $T=\{1423,1234,3412\}$. Then
$$F_T(x)=\frac{1-12x+64x^2-198x^3+393x^4-521x^5+463x^6-269x^7+95x^8-17x^9}{(1-x)^7(1-2x)^3}.$$
\end{theorem}
\begin{proof}
Let $G_m(x)$ be the generating function for $T$-avoiders with $m$
right-left maxima. Clearly, $G_0(x)=1$ and $G_1(x)=xF_T(x)$. To find an equation for $G_m(x)$ with $m\ge 2$, suppose $\pi=\pi^{(m)}i_m\pi^{(m-1)}i_{m-1}\cdots \pi^{(1)}i_1\in S_n(T)$ with $m\ge 2$ right-left maxima.
If $i_1=1$, we have a contribution of $xG_{m-1}(x)$.
Otherwise, there is a maximal $j \in [\kern .1em 1,m\kern .1em]$  such that $\pi^{(j)}$
has a smaller letter than $i_1$.

If $j=1$, then
$\pi^{(m)}=\pi^{(m-1)}=\cdots= \pi^{(2)}=\emptyset$ (or $(i_1-1)i_1$ is the 12 of a 3412) and the contribution is $x^m\big(F_{\{123,3412\}}(x)-1\big)$.

If $j\ge 2$, then $\pi$ has the form in Figure \ref{fig118G},
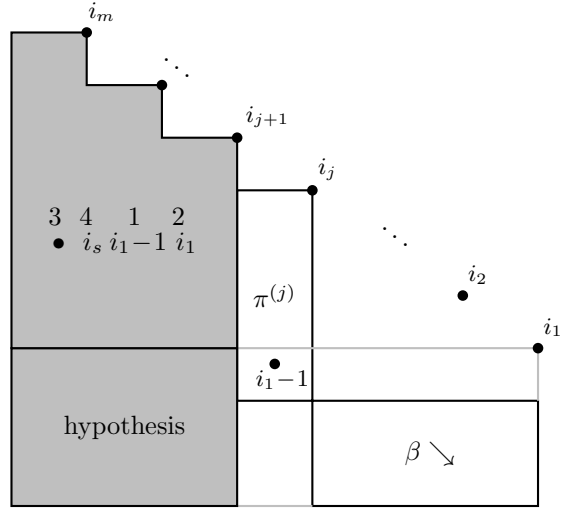
\begin{figure}[htp]
\begin{center}
\begin{pspicture}(-1,-0.3)(7,7)
\psset{xunit=1cm}
\psset{yunit=.7cm}
\psline(3,6)(4,6)(4,0)(7,0)(7,2)(3,2)(3,6)
\psline[linecolor=lightgray](3,0)(4,0)
\psline[linecolor=lightgray](3,3)(7,3)(7,2)
\pspolygon[fillstyle=solid,fillcolor=lightgray](0,0)(3,0)(3,3)(0,3)(0,0)
\pspolygon[fillstyle=solid,fillcolor=lightgray](0,3)(0,9)(1,9)(1,8)(2,8)(2,7)(3,7)(3,3)(0,3)
\rput(5.5,1){\textrm{ $\be\searrow$}}

\rput[b]{-0}(5.1,4.9){$\ddots$}
\rput[b]{0}(2.2,8.1){$\ddots$}
\rput(1.5,1.5){\textrm{hypothesis}}
\rput(1.4,5.5){\textrm{$3\ \ 4\ \quad 1\, \quad 2$}}
\rput(1.5,5){\textrm{$\bullet\ \ i_s\: i_{1}\!-\!1\ i_1$}}
\rput(3.5,4){\textrm{$\pi^{(j)}$}}
\qdisk(1,9){2pt}
\qdisk(2,8){2pt}
\qdisk(3,7){2pt}
\qdisk(4,6){2pt}
\qdisk(6,4){2pt}
\qdisk(7,3){2pt}
\qdisk(3.5,2.7){2pt}
\rput(3.6,2.4){\textrm{\small $i_{1}\!-\!1$}}
\rput(1.2,9.4){\textrm{\small $i_m$}}
\rput(3.4,7.4){\textrm{\small $i_{j+1}$}}
\rput(4.2,6.4){\textrm{\small $i_j$}}
\rput(6.2,4.4){\textrm{\small $i_2$}}
\rput(7.2,3.4){\textrm{\small $i_1$}}
\end{pspicture}
\caption{A $T$-avoider with $m$ RL max, $i_1\ge 2$, and $j\ge 2$}\label{fig118G}
\end{center}
\end{figure}
where shaded regions are empty for the reason indicated and $\pi^{(j)}>\be:=\pi^{(j-1)}\cdots\pi^{(2)}\pi^{(1)}$ since $a\in \al,\,b\in \be_s$ with $a<b$ implies $ai_j b i_s$ is a 1423. Hence, $i_1-1\in \pi^{(j)}$ since $\pi^{(j)}$ contains a letter $<i_1$. Also, $\be$ is decreasing (or $(i_1-1)i_j$ is the 34 of a 3412).

Now, if $\be=\emptyset$, the contribution is $x^{m-j}(J_j-xJ_{j-1})$, observing that $J_j$ is an overcount because $\al:=\{$letters $<i_1$ in $\pi^{(j)}\}$ is required to be nonempty.

On the other hand, if $\be \ne \emptyset$, then $\pi^{(j)}<i_1$ ($a\in \pi^{(j)}$ with $a>i_1$ and $b\in \be$ makes $ai_jbi_1$ a 3412). So $\pi^{(j)}$ contributes $F_{\{123,3412\}}(x)-1$ and $\be$ contributes $\frac{1}{(1-x)^{j-1}}-1$ for an overall contribution of $x^m \big(\frac{1}{(1-x)^{j-1}}-1\big)\big(F_{\{123,3412\}}(x)-1\big)$.

We have shown that, for fixed $j\in [2,m\kern .1em]$, the contribution to $G_m(x)$ is
$$x^{m-j}\left(J_j(x)-xJ_{j-1}(x)+x^j\Big(\frac{1}{(1-x)^{j-1}}-1\Big)(K-1)\right),$$
where $K=F_{\{123,3412\}}(x)$.
Thus, for $m\ge 2$,
$$G_m(x)=xG_{m-1}(x)+\sum_{j=2}^m x^{m-j}\left(J_j(x)-xJ_{j-1}(x)+x^j\Big(\frac{1}{(1-x)^{j-1}}-1\Big)(K-1)\right)+x^m(K-1)\,.$$
Summing over $m\geq2$,
\begin{align*}
F_T(x)-1-G_1(x)&=x(F_T(x)-1)+\frac{x^2}{1-x}(K-1)\\
&+\sum_{m\geq2}J_m(x) -\frac{x}{1-x}J_1(x)+\sum_{m\geq2}x^m\sum_{j=1}^{m-1}\left(\frac{1}{(1-x)^j-1}\right)(K-1)\,.
\end{align*}
Substitute $xF_T(x)$ for $G_1(x)$, evaluate the first sum using Corollary \ref{lem118a2}, the second sum by computer algebra, and solve for $F_T(x)$ using the expressions  for $J_1$ and $K$ to complete the proof.
\end{proof}
\subsection{Case 130: $T=\{1342,3124,3412\}$}
We deal with left-right maxima and begin with the case $\pi_2=n$. So
let $H=H(x)$ be the generating function for $T$-avoiders of the form $\pi=in\pi'\in S_n$.
\begin{lemma}\label{lem130a1}
$$H=\frac{x^2(x^2-3x+1)}{(1-x)(1-4x+2x^2)}.$$
\end{lemma}
\begin{proof}
Suppose $\pi=in\pi'\in S_n(T)$.
If $i=n-1$, then $\pi'$ is decreasing (to avoid 3412) and the contribution is $\frac{x^2}{1-x}$.

If $i<n-1$, then $\pi=in\al(n-1)\be $ decomposes as in Figure \ref{fig130HK}a) where $\al_1<\be_1$ (1342) and $\al_2>\be_2$ since, in fact, $\al_2\be_2$ is decreasing (3412).

\begin{figure}[htp]
\begin{center}
\begin{pspicture}(-1.5,-2.5)(10,2)
\psset{xunit=.6cm}
\psset{yunit=.4cm}
\psline[linecolor=lightgray](1,4.5)(1,2)
\psline[linecolor=lightgray](.5,0)(1,0)
\psline[linecolor=lightgray](12,-2)(12,-1.5)
\psline(1,0)(3,0)(3,-4)(5,-4)(5,-2)(1,-2)(1,0)
\psline(3,0)(3,4)(5,4)(5,2)(1,2)(1,0)

\qdisk(0.5,0){2pt}
\qdisk(1,4.5){2pt}
\qdisk(3,4){2pt}

\rput(2,1){\textrm{ $\al_1 $}}
\rput(2,-1){\textrm{ $\al_2 $}}
\rput(4,3){\textrm{ $\be_1 $}}
\rput(4,-3){\textrm{ $\be_2$}}

\rput(.2,0.2){\textrm{\small $i$}}
\rput(.6,4.8){\textrm{\small $n$}}
\rput(2.5,4.6){\textrm{\small $n\!-\!1$}}

\psline[linecolor=lightgray](10,3)(10,.5)
\psline[linecolor=lightgray](9.5,-1.5)(10,-1.5)
\psline(12,-2)(14,-2)(14,-4)(12,-4)(12,-2)
\psline(10,-1.5)(12,-1.5)(12,2.5)(14,2.5)(14,.5)(10,.5)(10,-1.5)

\qdisk(9.5,-1.5){2pt}
\qdisk(10,3){2pt}
\qdisk(14,2.5){2pt}
\qdisk(12,-2){2pt}

\rput(11,-.5){\textrm{$\al$}}
\rput(13,1.5){\textrm{$\be_1$}}
\rput(13,-3){\textrm{$\be_2$}}

\rput(9.2,-1.2){\textrm{\small $i$}}
\rput(9.6,3.3){\textrm{\small $n$}}
\rput(14.2,3.1){\textrm{\small $n\!-\!1$}}
\rput(11.3,-2){\textrm{\small $i\!-\!1$}}

\rput(3,-5.5){\textrm{ $a)\ H$}}
\rput(12,-5.5){\textrm{ $b)\ K$}}

\end{pspicture}
\caption{$T$-avoiders counted by $H$ and $K$}\label{fig130HK}
\end{center}
\end{figure}
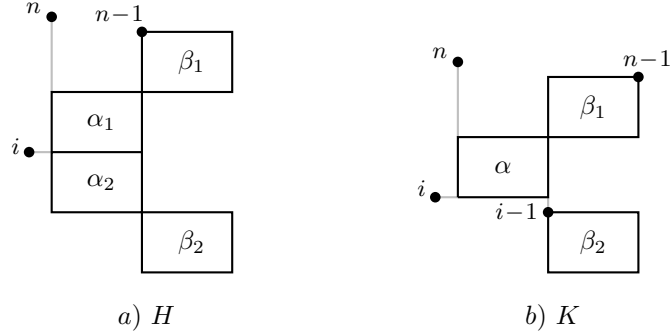

Now $\pi$ can be recovered from St($in \al (n-1)$) and St($in \be$),
and $in \be$ is counted by $H$. Thus,
the contribution is $\frac{K H}{x^2}$, where $K$ is the \gf for $T$-avoiders $\pi\in S_n$
with $\pi_2=n$ and $\pi_n =n-1$.

To find $K$, suppose $in \al (n-1)\in S_n(T)$.
If $i=1$, then $\al$ avoids 231 and 312, and the contribution is $x^3 L$ where
$L=L(x)=F_{\{231,132\}}(x)=\frac{1-x}{1-2x}$ \cite{SiS}.
If $i>1$,
then $\pi=in\al(i-1)\be (n-1)$ decomposes as in Figure \ref{fig130HK}b) since all letters $<i$ in $\pi$ are decreasing (3412) and $\al<\be_1$ (1342). Also, $\al$ avoids 231 and 312.
The contribution is then $xLK$ since $in \be (n-1)$ is counted by $K$.
Hence, $K=x^3 L + xLK$, yielding $K=\frac{x^3(1-x)}{1-3x+x^2}$.

We have shown that $H=\frac{x^2}{1-x}+\frac{K H}{x^2}$.
Substitute for $K$ and solve for $H$ to obtain the result.
\end{proof}

\begin{theorem}\label{th130a}
Let $T=\{1342,3124,3412\}$. Then
$$F_T(x)=\frac{1-9x+32x^2-58x^3+58x^4-33x^5+8x^6}{(1-x)^4(1-2x)(1-4x+2x^2)}.$$
\end{theorem}
\begin{proof}
Let $G_m(x)$ be the generating function for $T$-avoiders with $m$ left-right maxima. Clearly, $G_0(x)=1$ and $G_1(x)=xF_T(x)$.

To find a formula for $G_2(x)$, suppose $\pi=i\pi'n\pi''\in S_n(T)$ with $2$ left-right maxima.
Then $\pi$ decomposes as in Figure \ref{fig130G2}a),
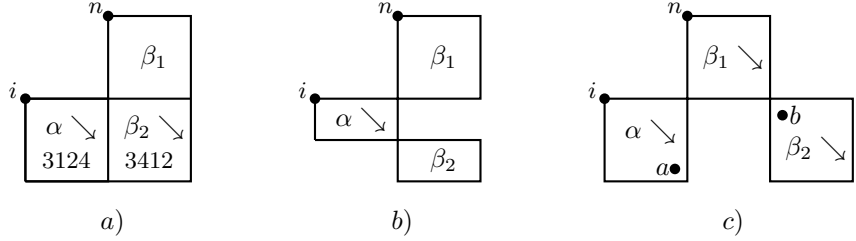
\begin{figure}[htp]
\begin{center}
\begin{pspicture}(-1.5,-0.6)(12,2.8)
\psset{xunit=.55cm}
\psset{yunit=.55cm}
\psline(0,0)(0,2)(2,2)(2,4)(4,4)(4,0)(0,0)
\psline(0,0)(0,2)(2,2)(2,0)(0,0)
\psline(2,2)(4,2)
\qdisk(0,2){2pt}
\qdisk(2,4){2pt}
\rput(1,1.3){\textrm{ $\al\searrow$}}
\rput(3,1.3){\textrm{ $\be_2\searrow$}}
\rput(3,3){\textrm{ $\be_1$}}
\rput(1,.5){\textrm{\small $3124$}}
\rput(3,.5){\textrm{\small $3412$}}
\rput(-.3,2.2){\textrm{\small $i$}}
\rput(1.7,4.2){\textrm{\small $n$}}
\qdisk(7,2){2pt}
\qdisk(9,4){2pt}
\rput(6.7,2.2){\textrm{\small $i$}}
\rput(8.7,4.2){\textrm{\small $n$}}
\psline(7,1)(7,2)(11,2)(11,4)(9,4)(9,0)(11,0)(11,1)(7,1)
\rput(8,1.5){\textrm{ $\al\searrow$}}
\rput(10,.5){\textrm{ $\be_2$}}
\rput(10,3){\textrm{ $\be_1$}}
\qdisk(14,2){2pt}
\qdisk(16,4){2pt}
\qdisk(15.7,.3){2pt}
\qdisk(18.3,1.6){2pt}
\psline(14,2)(14,0)(16,0)(16,4)(18,4)(18,0)(20,0)(20,2)(14,2)
%
\rput(15.3,.3){\textrm{ $a$}}
\rput(18.5,1.6){\textrm{ $b$}}
\rput(13.7,2.2){\textrm{\small $i$}}
\rput(15.7,4.2){\textrm{\small $n$}}
\rput(15,1.2){\textrm{ $\al\searrow$}}
\rput(17,3){\textrm{ $\be_1\searrow$}}
\rput(19,.8){\textrm{ $\be_2\searrow$}}
\rput(2,-1){\textrm{ $a)$}}
\rput(9,-1){\textrm{ $b)$}}
\rput(17,-1){\textrm{ $c)$}}
\end{pspicture}
\caption{$T$-avoiders with $2$ left-right maxima}\label{fig130G2}
\end{center}
\end{figure}
where $\al$ and $\be_2$ are decreasing for the reason indicated. If $\al\be_2$ is decreasing, then we have Figure \ref{fig130G2}b), where $in\be$ is counted by $H$ and so the contribution is $\frac{1}{1-x}H$.
Otherwise, the smallest letter $a$ in $\al$ is less than the smallest letter $b$ in $\be_2$ and $\pi$ decomposes as in Figure \ref{fig130G2}c), where $\be_1$ is decreasing (or $ab$ is the 12 of a 1342) and $\be_1$ lies to the left of $\be_2$ (or $iab$ is the 312 of a 3124). Here, $\al\be_2$ contributes $ \frac{1}{1-2x}-\frac{1}{(1-x)^2}$ where $\frac{1}{1-2x}$ counts placements of letters in 2 boxes and $-\frac{1}{(1-x)^2}$ corrects the overcount of decreasing letters in 2 boxes. Hence the contribution is
\[
\frac{x^2}{1-x}\left(\frac{1}{1-2x}-\frac{1}{(1-x)^2}\right)\,,
\]
and the result follows by addding the two contributions.

To find $G_m(x)$ with $m\geq3$, let $\pi=i_1\pi^{(1)}\cdots i_m\pi^{(m)}\in S_n(T)$ with $m\ge 3$ left-right maxima.
Since $\pi$ avoids $3124$, we see that $\pi^{(s)}$ is decreasing for $s=1,2,\ldots,m-1$.
Since $\pi$ avoids $3412$, all the letters that are smaller than $i_1$ in $\pi$ are decreasing.
Since $\pi$ avoids $1342$, the letters that are greater than $i_1$ in $\pi^{(s)}$ are $>i_{s-1}$ for $s=3,4,\dots,m$.
Thus,
$\pi$ decomposes as in Figure \ref{fig130Gm},
\begin{figure}[htp]
\begin{center}
\begin{pspicture}(-1.5,-0)(9,8)
\psset{xunit=.7cm}
\psset{yunit=.5cm}
\psline(0,5)(0,6)(2,6)(2,7)(1,7)(1,5)(0,5)
\psline(2,4)(3,4)(3,5)(2,5)(2,4)
\psline(3,7)(4,7)(4,8)(3,8)(3,7)
\psline(4,3)(5,3)(5,4)(4,4)(4,3)
\psline(6,10)(7,10)(7,11)(6,11)(6,10)
\psline(7,2)(8,2)(8,0)(9,0)(9,1)(7,1)(7,2)
\psline(8,11)(9,11)(9,12)(8,12)(8,11)
\qdisk(0,6){2pt}
\qdisk(1,7){2pt}
\qdisk(3,8){2pt}
\qdisk(6,11){2pt}
\qdisk(8,12){2pt}
\rput(.5,5.5){\textrm{ $\searrow$}}
\rput(1.5,6.5){\textrm{ $\searrow$}}
\rput(2.5,4.5){\textrm{ $\searrow$}}
\rput(3.5,7.5){\textrm{ $\searrow$}}
\rput(4.5,3.5){\textrm{ $\searrow$}}
\rput(6.5,10.5){\textrm{ $\searrow$}}
\rput(7.5,1.5){\textrm{ $\searrow$}}
\rput(-.4,6.2){\textrm{\small $i_1$}}
\rput(0.6,7.2){\textrm{\small $i_2$}}
\rput(2.6,8.2){\textrm{\small $i_3$}}
\rput(5.3,11.2){\textrm{\small $i_{m-1}$}}
\rput(7.6,12.2){\textrm{\small $i_m$}}
\psline[linecolor=lightgray](8,2)(8,11)(7,11)
\psline[linecolor=lightgray](7,2)(7,10)
\psline[linecolor=lightgray](3,4)(4,4)(4,7)
\psline[linecolor=lightgray](3,5)(3,7)(2,7)
\psline[linecolor=lightgray](1,5)(2,5)(2,6)
\rput[b]{0}(4.9,8.9){$\iddots$}
\rput[b]{10}(6,2){$\ddots$}
\end{pspicture}
\caption{A $T$-avoider with $m\ge 3$ left-right maxima}\label{fig130Gm}
\end{center}
\end{figure}
where there are $2m-3$ independently decreasing regions as indicated by the down arrows, and $i_1i_m\pi^{(m)}$ is counted by $H$.
Thus$$G_m(x)=\frac{x^{m-2}}{(1-x)^{2m-3}}H\,.$$
Summing over $m\geq3$, we obtain $$F_T(x)-G_0(x)-G_1(x)-G_2(x)=\frac{xH}{(1-x)(1-3x+x^2)}.$$
Substitute for $H$, $G_0(x)$, $G_1(x)$ and $G_2(x)$, and then solve for $F_T(x)$ to complete the proof.
\end{proof}

\subsection{Case 131: $\{2134,1423,2341\}$}
We focus on the first two letters of an avoider.
Set $a(n)=|S_n(T)|$ and define
$a(n;i_1,i_2,\dots, i_m)$ to be the number of permutations $\pi=\pi_1\pi_2\cdots\pi_n$ in
$S_n(T)$ such that $\pi_1\pi_2\cdots\pi_m=i_1i_2\cdots i_m$. Clearly, $a(n;1)=|S_{n-1}(\{312,2134,2341\})|$.
Let $ g(x)=\sum_{n\geq0}a(n;1)x^n$ and $\ell_i=|S_i(\{213,1423,2341\})|$.
It is known \cite{wikipermpatt} that
$g(x)=\frac{x^4-x^3+4x^2-3x+1}{(1-x)^4}$
and that $\ell_i=2^i-i$.
Set $b(n;i)=a(n;i,n-1)$ and $b'(n;i)=a(n;i,n)$. We have the following recurrences (proof omitted).
\begin{lemma}\label{lem131a}
We have $\ell_i=|S_i(\{213,1423,2341\})|=2^i-i$. Then
\begin{align*}
a(n;i,j)&=a(n-1;i,j),\qquad\mbox{if $2\leq i<j\leq n-2$},\\
a(n;i,j)&=a(n-1;i,j)+\sum_{k=1}^{j-1}a(n-1;j,k),\qquad\mbox{if $1\leq j<i\leq n-1$ and $(i,j)\neq(n-1,1)$},\\
a(n;n,i)&=a(n-1;i),\qquad\mbox{if $1\leq i\leq n-1$},\\
b(n;i)&=b(n-1;i)+\ell_{i-1},\qquad\mbox{$1\leq i\leq n-2$},\\
b'(n;i)&=b'(n-1;i)+\sum_{j=1}^{i-1}a(n-1;i,j),\qquad\mbox{$1\leq i\leq n-1$},\\
a(n;n-1,n)&=a(n-2),\\
a(n;n-1,1)&=a(n-1;1).\hspace*{100mm} \qed
\end{align*}
\end{lemma}

Define $a'(n;i)=\sum_{j=1}^{i-1}a(n;i,j)$ and $a''(n;i)=\sum_{j=i+1}^na(n;i,j)$. Define \begin{align*}
B_n(v)&=\sum_{i=1}^{n-2}b(n;i)v^{i-1}, &&B'_n(v)=\sum_{i=1}^{n-1}b'(n;i)v^{i-1},\\  A'_n(v)&=\sum_{i=1}^na'(n;i)v^{i-1}, &&A''_n(v)=\sum_{i=1}^na''(n;i)v^{i-1},\\ A_n(v)&=\sum_{i=1}^na(n;i)v^{i-1}.  
\end{align*}
Define generating functions $B(x,v)=\sum_{n\geq2}B_n(v)x^n$, $B'(x,v)=\sum_{n\geq2}B'_n(v)x^n$,
$A(x,v)=1+x+\sum_{n\geq2}A_n(v)x^n$, $A'(x,v)=\sum_{n\geq2}A'_n(v)x^n$ and $A''(x,v)=\sum_{n\geq2}A''_n(v)x^n$.

\begin{lemma}\label{lem131b}
We have
\begin{align*}
B(x,v)&=\frac{(3v^2x^2-3vx+1)x^3}{(1-x)^2(1-vx)^2(1-2vx)},\textrm{ and}\\
B'(x,v)&=\frac{x^2}{1-x}+\frac{x}{1-x}A'(x,v).
\end{align*}
\end{lemma}
\begin{proof}
By Lemma \ref{lem131a},
$$b'(n;i)=a'(n-1;i)+b'(n-1;i),\quad 1\leq i\leq n-1.$$
Multiplying by $v^{i-1}$ and summing over $i=1,2,\ldots,n-1$,
$$B'_n(v)=A'_{n-1}(v)+B'_{n-1}(v)$$
with $B'_2(v)=1$. Thus, the generating function for $B'_n(v)$ satisfies $B'(x,v)-x^2=xA'(x,v)+xB'(x,v)$, which leads to $B'(x,v)=\frac{x^2}{1-x}+\frac{x}{1-x}A'(x,v)$.

Also by Lemma \ref{lem131a},
$$b(n;i)=b(n-1;i)+\ell_{i-1},\quad 1\leq i\leq n-2.$$
Multiplying by $v^{i-1}$ and summing over $i=1,2,\ldots,n-2$, we have
$$B_n(v)=B_{n-1}(v)+\sum_{i=1}^{n-2}\ell_{i-1}v^{i-1},$$
which leads to
$B(x,v)=\frac{1}{1-x}\sum_{n\geq2}\sum_{i=1}^{n-2}(2^{i-1}-(i-1))v^{i-1}x^n$ and the first assertion.
\end{proof}

\begin{lemma}\label{lem131c}
We have
\begin{align*}
A'(x,v)&=\frac{x}{1-v}\left(A'(x,v)-\frac{1}{v}A'(vx,1)\right)+
\frac{x(1+v)}{v}(A(xv,1)-1)-x^2A(xv,1),\textrm{ and}\\
A''(x,v)&=xA''(x,v)+(1-x)x\big(g(x)-1\big)+B(x,v)-B(x,0)+(1-x)(B'(x,v)-B'(x,0))\,.
\end{align*}
\end{lemma}
\begin{proof}
Lemma \ref{lem131a} gives
$$a'(n;i)=a'(n-1;1)+a'(n-1;2)+\cdots+a'(n-1;i)$$
with $a'(n;n)=A_{n-1}(1)$ and $a'(n;n-1)=A_{n-1}(1)-A_{n-2}(1)$. Multiplying by $v^{i-1}$ and summing over $i=1,2,\ldots,n-2$, we have
\begin{align*}
A'_n(v)&=\frac{1}{1-v}(A'_{n-1}(v)-A'_{n-1}(1)v^{n-2})+(A_{n-1}(1)-A_{n-2}(1))v^{n-2}+A_{n-1}(1)v^{n-1}
\end{align*}
with $A'_1(v)=1$. The first assertion follows by multiplying by $x^n$ and summing over $n\geq2$.

Lemma \ref{lem131a} also gives
$$a''(n;i)=\sum_{j=i+1}^{n-2}a(n;i,j)+b(n;i)+b'(n;i)=\sum_{j=i+1}^{n-2}a(n-1;i,j)+b(n;i)+b'(n;i),$$
which leads to $a''(n;i)=a''(n-1;i)+b(n;i)+b'(n;i)-b'(n-1;i)$ for all $i=2,3,\ldots,n-1$, where $a''(n;1)=a(n;1)=|S_{n-1}(\{312,2134,2341\})|$. Thus, multiplying by $v^{i-1}$ and summing over $i=2,3,\ldots,n-1$, we have
$$A''_n(v)=A''_{n-1}(v)+A''_n(0)-A''_{n-1}(0)+B_n(v)-B_n(0)+B'_n(v)-B'_n(0)-B'_{n-1}(v)-B'_{n-1}(0)$$
with $A''_1(v)=0$ and $B'_1(v)=B'_2(v)=1$. Multiplying by $x^n$ and summing over $n\geq2$, we obtain
$$A''(x,v)=xA''(x,v)+(1-x)A''(x,0)+B(x,v)-B(x,0)+(1-x)\big(B'(x,v)-B'(x,0)),$$
where $A''(x,0)=x\sum_{n\geq1}|S_n(\{312,2134,2341\})|x^n=x\big(g(x)-1\big)$, and the second assertion follows.
\end{proof}

From Lemmas \ref{lem131b} and \ref{lem131c}, we have
\begin{align*}
\left(1-\frac{x}{v(1-v)}\right)A'\Big(\frac{x}{v},v\Big)&=-\frac{x}{v^2(1-v)}A'(x,1)+\frac{x(1+v)}{v^2}\big(A(x,1)-1\big)-\frac{x^2}{v^2}A(v,1).
\end{align*}
By substituting $v=1/C(x)$, we obtain
\begin{align}\label{eqf131A}
A'(x,1)=x(1-x)C(x)A(x,1)-xC(x)-x.
\end{align}
Moreover, we have
\begin{align}\label{eqf131B}
(1-x)A''(x,1)&=x(1-x)\big(g(x)-1\big)+B(x,1)-B(x,0)+(1-x)\big(B'(x,1)-B'(x,0)\big),
\end{align}
where $B(x,1)=\frac{(3x^2-3x+1)x^3}{(1-x)^4(1-2x)}$, $B(x,0)=\frac{x^3}{(1-x)^2}$,
$B'(x,1)=\frac{x^2}{1-x}+\frac{x}{1-x}A'(x,1)$ and $B'(x,0)=\sum_{n\geq2}a(n;1,n)x^n=\frac{x^2}{1-x}$.
By solving the three equations \eqref{eqf131A}, \eqref{eqf131B} and $A(x,1)=1+x+A'(x,1)+A''(x,1)$ for $A(x,1),A'(x,1),A''(x,1)$, we obtain the following result.

\begin{theorem}\label{th131a}
Let $T=\{2134,1423,2341\}$. Then
$$F_T(x)=\frac{2x^5+x^4-6x^3+7x^2-4x+1}{(1-2x)(1-x)^3}C(x)-\frac{x(2x^4-x^3+x^2-2x+1)}{(1-2x)(1-x)^4}.$$
\end{theorem}

\subsection{Case 133: $\{1342,2143,2314\}$}
\begin{theorem}\label{th133a}
Let $T=\{1342,2143,2314\}$. Then
$$F_T(x)=\frac{(1-2x)(1-3x+x^2)}{1-6x+11x^2-7x^3}.$$
\end{theorem}
\begin{proof}
Let $G_m(x)$ be the generating function for $T$-avoiders with $m$
left-right maxima. Clearly, $G_0(x)=1$ and $G_1(x)=xF_T(x)$.

To find $G_2(x)$, consider $\pi=i\pi'n\pi''\in S_n(T)$ with left-right maxima $i$ and $n$. If $\pi'$ is not empty then $i=n-1$ (to avoid 2143) and $i$ can be safely deleted, leaving nonempty $T$-avoiders with maximum entry not in first position. Hence, the contribution is $x\big(F_T(x)-1-xF_T(x)\big)$.
If $\pi'$ is empty, then $\pi_2=n$ can be safely deleted and the contribution is $x\big(F_T(x)-1\big)$.
Thus,
$$G_2(x)=x\big(F_T(x)-1-xF_T(x)\big)+x\big(F_T(x)-1\big)\,.$$

To find $G_m(x)$ with $m\geq3$, consider a $T$-avoider $\pi=i_1\pi^{(1)}\cdots i_m\pi^{(m)}$ with $m$ left-right maxima. It decomposes as in Figure \ref{fig133},
\begin{figure}[htp]
\begin{center}
\begin{pspicture}(0,-0.5)(5,4)
\psset{xunit=1cm}
\psset{yunit=.7cm}
\psline(0,1)(0,2)(2,2)(2,3)(1,3)(1,1)(0,1)
\psline(3,4)(4,4)(4,6)(5,6)(5,5)(3,5)(3,4)
\psline(5,0)(6,0)(6,1)(5,1)(5,0)
\qdisk(0,2){2pt}
\qdisk(1,3){2pt}
\qdisk(3,5){2pt}
\qdisk(4,6){2pt}
\rput(.5,1.5){\textrm{$\al_1 $}}
\rput(1.5,2.5){\textrm{$\al_2 $}}
\rput(3.5,4.5){\textrm{$\al_{m-1} $}}
\rput(4.5,5.5){\textrm{$\al_m $}}
\rput(5.5,0.5){\textrm{$\be_m $}}
\rput(-.2,2.2){\textrm{\small $i_1$}}
\rput(.8,3.2){\textrm{\small $i_2$}}
\rput(2.5,5.2){\textrm{\small $i_{m-1}$}}
\rput(3.7,6.2){\textrm{\small $i_m$}}
\pspolygon[fillstyle=solid,fillcolor=lightgray](1,0)(1,2)(2,2)(2,4)(4,4)(4,0)(1,0)
\pspolygon[fillstyle=solid,fillcolor=lightgray](4,1)(4,5)(6,5)(6,1)(4,1)
\rput(3,2.1){\textrm{\small $2\,3\,1\,4$}}
\rput(3.1,1.7){\textrm{\small $\bullet$}}
\rput(5,3.1){\textrm{\small $1\,3\,4\,2$}}
\rput(5.3,2.7){\textrm{\small $\bullet$}}
\rput(1.8,4.4){\textrm{$\iddots$}}
\end{pspicture}
\caption{A $T$-avoider with $m\ge 3$ left-right maxima}\label{fig133}
\end{center}
\end{figure}
where the shaded regions are empty for the reason indicated and $\al_m$ lies
to the left of $\be_m$ (or $i_1 i_{m-1}$ is the 23 of a 2314).
If $\al_j\neq\emptyset$ for some $j\in [1,m-1\kern .1em]$,
then $\al_i=\emptyset$ for all $i\neq j$ (2143), $\al_j$ avoids 231 and 2143,
and $\be_m$ avoids $T$,
and so the contribution is $x^m\big(K(x)-1\big)F_T(x)$ where
$K(x)=\frac{1-2x}{1-3x+x^2}$ is the generating function for $\{231,2143\}$-avoiders
\cite[Seq. A001519]{Sl}.
On the other hand, if $\al_i=\emptyset$ for all $i\in [1,m-1\kern .1em]$, then $\al_m$ avoids 231 and 2143, and $\be_m$ avoids $T$, giving $x^mK(x)F_T(x)$.
Hence,
$$G_m(x)=(m-1)x^m\big(K(x)-1\big)F_T(x)+x^mK(x)F_T(x)\,.$$
Summing over $m\ge 3$ and substituting for $G_0(x),G_1(x)$ and $G_2(x)$, we obtain
\begin{align*}
F_T(x)=1+xF_T(x)+x\big(F_T(x)-1-xF_T(x)\big)+x\big(F_T(x)-1\big)+\frac{x^3(1+x)F_T(x)}{1-3x+x^2}\,,
\end{align*}
and solving for $F_T(x)$ completes the proof.
\end{proof}

\subsection{Case 159: $\{1243,1342,3412\}$}
We consider left-right maxima and begin with a particular case. Let $J=J(x)$
denote the generating function for $T$-avoiders $\pi=\pi_1\pi_2\cdots \pi_n$ with $\pi_2=n$.
\begin{lemma}\label{lem159a1} We have
$J=\frac{x^2(1-2x)}{(1-x)(1-3x)}$.
\end{lemma}
\begin{proof}
We refine $J$ to $J_d,\ d\ge 1$, the generating function for permutations $\pi=dn\pi''\in S_n(T)$. Since $\pi$ avoids $3412$, the letters $< d$ are decreasing in $\pi$.
We count by position of $n-1$.
If $n=d+1$ (so $n-1$ is in first position), then $\pi=(n-1)n(n-2)(n-3)\cdots 1$, and the contribution to $J_d$ is $x^{d+1}$. Now suppose $n>d+1$ and let $j\in [\kern .1em 1,d\kern .1em]$ be maximal such that $d\pi''$ begins with $j$ consecutive decreasing integers $d(d-1)(d-2)\cdots(d-j+1),$ and consider the next letter, $\pi_{j+2}$, in $\pi$.
\begin{itemize}
\item $\pi_{j+2}=n-1$. Here, we can safely delete the $j$ letters $\pi_2\pi_3 \cdots \pi_{j+1}=n(d-1)(d-2)\cdots (d-(j-1))$, and the contribution is $x^jJ_{d-(j-1)}$.
\item $\pi_{j+2}<n-1$.
Here, $\pi$ decomposes as in Figure \ref{fig159J},
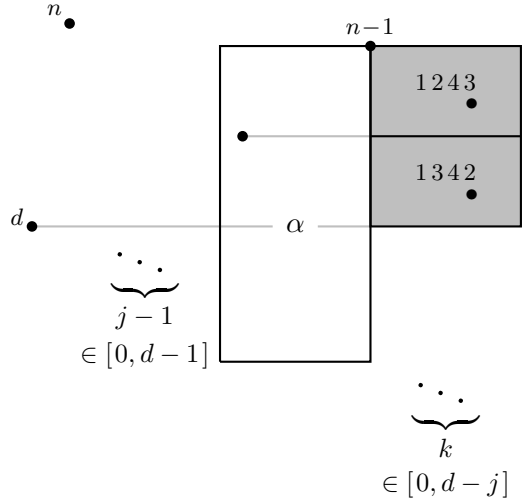
\begin{figure}[htp]
\begin{center}
\begin{pspicture}(-1.5,-2.2)(7,4.3)
\psset{xunit=1cm}
\psset{yunit=.6cm}
\psline[linecolor=lightgray](-.5,2)(2.7,2)
\psline[linecolor=lightgray](3.3,2)(4,2)
\psline[linecolor=lightgray](2.3,4)(4,4)
\pspolygon[fillstyle=solid,fillcolor=lightgray](4,2)(4,6)(6,6)(6,2)(4,2)
\qdisk(-.5,2){2pt}
\qdisk(0,6.5){2pt}
\qdisk(4,6){2pt}
\qdisk(2.3,4){2pt}
\psline(2,-1)(2,6)(4,6)(4,-1)(2,-1)
\psline(4,4)(6,4)
\rput(5,5.2){\textrm{\small 1\,2\,4\,3}}
\rput(5.35,4.7){$\bullet$}
\rput(5,3.2){\textrm{\small 1\,3\,4\,2}}
\rput(5.35,2.7){$\bullet$}
\rput(3,2){\large $\al$}
\rput(-.7,2.2){\textrm{\small $d$}}
\rput(-.2,6.8){\textrm{\small $n$}}
\rput(4,6.4){\textrm{\small $n\!-\!1$}}
\rput[b]{15}(1,.8){\huge $\ddots$}
\rput[b]{15}(5,-2.1){\huge $\ddots$}
\rput(1,0){$\underbrace{\phantom{ha\hspace*{5mm}}}_{\stackrel{\raisebox{0mm}{$j-1$}}{\raisebox{-4mm}{$\in [\kern .1em 0,d-1\kern .1em]$}}}$}
\rput(5,-2.9){$\underbrace{\phantom{ha\hspace*{5mm}}}_{\stackrel{\raisebox{0mm}{$k$}}{\raisebox{-4mm}{$\in [\kern .1em 0,d-j\kern .1em]$}}}$}
\end{pspicture}
\caption{A $T$-avoider counted by $J_d$}\label{fig159J}
\end{center}
\end{figure}
where the shaded regions are empty for the indicated reason. For fixed $j,k$, the $j-1+k$ letters above the braces (decreasing dots), along with $n-1$, may be safely deleted, leaving $dn \al$, counted by $J_{d-(k+j-1)} - xJ_{d-(k+j)}$, taking into account that $\al$ begins with a letter larger than the first entry. (This requires $J_0=x$ for $k=d-j$.)
\end{itemize}
Adding the contributions,
\[
J_d=x^{d+1}+\sum_{j=1}^d x^j J_{d+1-j}+\sum_{j=1}^d\sum_{k=0}^{d-j}x^{k+j}(J_{d+1-k-j}-x J_{d-k-j)})\,,
\]
for all $d\geq1$. The sum over $k$ telescopes and
summing over $d\geq1$, we have
$$J=\frac{x^2}{1-x}+\frac{2x}{1-x}J-\frac{x^3}{(1-x)^2},$$
an identity for $J$ with solution as stated.
\end{proof}
Now let $G_m(x)$ be the generating function for $T$-avoiders with $m$ left-right maxima.
Clearly, $G_0(x)=1$ and $G_1(x)=xF_T(x)$.
\begin{lemma}\label{lem159a2}
\[
G_2(x)=\frac{x^3(1-4x+3x^2+x^3)}{(1-x)^2(1-2x)(1-3x)(1-3x+x^2)}+\frac{x^2}{1-x}F_T(x)\,.
\]
\end{lemma}
\begin{proof}
We refine $G_2(x)$ to $G_2(x;d)$, the generating function for permutations of the form $\pi=i\pi'n\pi''\in S_n(T)$ with $2$ left-right maxima in which $\pi''$ has $d$ letters $<i$. Since $\pi$ avoids $3412$, the letters $<i$ in $\pi''$, say $j_1,j_2,\dots,j_d$, are decreasing and in the matrix diagram of $\pi$, these $d$ letters split $\pi''$ into vertically divided segments $\be_1,\be_2,\dots,\be_{d+1}$ (left to right) and $\pi'$ into horizontally divided segments $\al_{d+1},\al_d,\dots,\al_1$ (top to bottom).
We distinguish two cases:
\begin{itemize}
\item $\alpha_d=\cdots=\alpha_1=\emptyset$. Here, if $i=n-1$, the contribution is $x^{d+2}F_T(x)$. Otherwise, $\alpha_{d+1}$ is decreasing since $\pi$ avoids $1243$, which implies a contribution of $\frac{J_{d+1}-x^{d+2}}{1-x}$ by Lemma \ref{lem159a1}.

\item Otherwise, there is a minimal $s\in [\kern .1em 1,d\kern .1em]$ such that $\al_s$ is nonempty. Here, $\pi$ decomposes as in Figure \ref{fig159G2},
\begin{figure}[htp]
\begin{center}
\begin{pspicture}(-6,-1)(8,4)
\psset{xunit=.55cm}
\psset{yunit=.65cm}
\psline[linecolor=lightgray](9.5,-1)(9.5,4)
\psline[linecolor=lightgray](7.5,-.5)(7.5,4)
\psline[linecolor=lightgray](0,0)(5.5,0)(5.5,4)
\psline[linecolor=lightgray](0,1)(3.5,1)(3.5,4)
\psline[linecolor=lightgray](-6,4)(0,4)
\psline[linecolor=lightgray](-4,3)(2.5,3)
\psline[linecolor=lightgray](-4,2)(3,2)
\qdisk(-8,4){2pt}
\qdisk(0,5.5){2pt}
\qdisk(-.3,.3){2pt}
\qdisk(2.5,3){2pt}
\qdisk(3,2){2pt}
\qdisk(3.5,1){2pt}
\qdisk(5.5,0){2pt}
\qdisk(7.5,-.5){2pt}
\qdisk(9.5,-1){2pt}
\psline(-2,0)(0,0)(0,1)(-2,1)(-2,0)
\psline(-8,4)(-6,4)(-6,2)(-4,2)(-4,3)(-8,3)(-8,4)
\psline(0,4)(2,4)(2,5.5)(0,5.5)(0,4)
\psline(3.5,4)(5.5,4)(5.5,5.5)(3.5,5.5)(3.5,4)
\psline(7.5,4)(9.5,4)(9.5,5.5)(11.5,5.5)(11.5,4)(9.5,4)(9.5,5.5)(7.5,5.5)(7.5,4)
\rput(6.5,4.9){$\dots$}
\psline[arrows=->,arrowsize=4pt 2](3.8,4.3)(11.2,5.2)
\rput(-7,3.5){\small $\al_{d+1}\!\searrow$}
\rput(-5,2.5){\small $\al_{d}\searrow$}
\rput(-1.2,0.6){\small $\al_{s}\!\searrow$}
\rput(1,4.7){\small $\be_1\searrow$}
\rput(4.5,4.9){\small $\be_{d+2-s}$}
\rput(8.6,4.5){\small $\be_{d}$}
\rput(10.5,4.5){\small $\be_{d+1}$}
\rput(-8.3,4.2){\textrm{\small $i$}}
\rput(-.3,5.7){\textrm{\small $n$}}
\rput(3,3.2){\textrm{\small $j_1$}}
\rput(3.5,2.2){\textrm{\small $j_2$}}
\rput(4.5,1.2){\textrm{\small $j_{d+1-s}$}}
\rput(6.5,0.2){\textrm{\small $j_{d+2-s}$}}
\rput(8.5,-.3){\textrm{\small $j_{d-1}$}}
\rput(10.5,-.8){\textrm{\small $j_d$}}
\rput[b]{15}(-2.9,1.2){$\ddots$}
\rput[b]{-25}(3.2,1.3){\small $\ddots$}
\rput[b]{25}(6.6,-.5){$\ddots$}
\end{pspicture}
\caption{A $T$-avoider counted by $G_2(x;d)$}\label{fig159G2}
\end{center}
\end{figure}
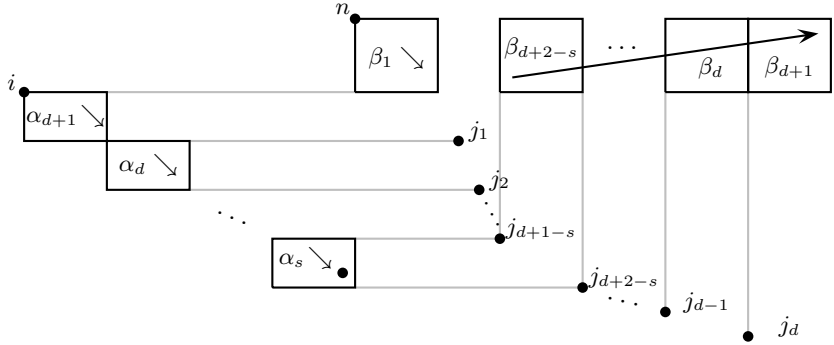
where $\al_{t+1}$ lies to the left of $\al_t$ for all $t$ (a violator is the 13 of a 1342),
$\be_1$ is decreasing (a violator is the 34 of a 1342),
$\be_{2}=\be_3=\cdots=\be_{d+1-s}=\emptyset$ (a violator is the 4 of a 1342),
$\be_{d+2-s} \cdots \be_{d} \be_{d+1}$ is increasing (a violator is the 43 of a 1243),
$\al_{d+1}$ is decreasing (a violator is the 34 of a 3412),
and $\al_{d},\al_{d-1}, \dots,\al_s$  are all decreasing (a violator is the 12 of a 1243).

The $\be$'s contribute $\frac{(1-x)^{s-1}}{(1-2x)^s}$ and the overall
contribution is $\frac{x^{d+3}}{(1-x)^{d+3-s}}\frac{(1-x)^s}{(1-2x)^s}$.
\end{itemize}
Hence, for all $d\geq0$,
$$G_2(x;d)=\sum_{s=1}^d\frac{x^{d+3}}{(1-x)^{d+3-2s}(1-2x)^s}+x^{d+2}F_T(x)+\frac{J_{d+1}-x^{d+2}}{1-x}.$$
The result follows by summing over $d\geq0$, using Lemma \ref{lem159a1} and the identity $J=\sum_{d\geq1}J_d$.
\end{proof}

\begin{lemma}\label{lem159a3} For $m\ge 2$, $G_m(x)=\left(\frac{x}{1-x}\right)^{m-2}G_2(x)$.
\end{lemma}
\begin{proof}
Suppose $\pi=i_1\pi^{(1)}\cdots i_m\pi^{(m)}\in S_n(T)$ has $m\ge 3$ left-right maxima. Since $\pi$ avoids $1342$ and $1243$, we see that $\pi^{(j)}<i_1$ for all $j=3,4,\ldots,m$. Since $\pi$ avoids $1342$ and $3412$, we see that $\pi^{(s)}>\pi^{(m)}$, and $\pi^{(m)}$ is decreasing. Hence $G_m(x)=\frac{x}{1-x}G_{m-1}(x)$.
\end{proof}
Now sum the result of Lemma \ref{lem159a3} over $m\ge2 $ to get an equation for $F_T(x)$ using Lemma \ref{lem159a2}, whose solution yields the following Theorem.

\begin{theorem}\label{th159a}
Let $T=\{1243,1342,3412\}$. Then
$$F_T(x)=\frac{1-11x+48x^2-104x^3+115x^4-61x^5+13x^6}{(1-x)(1-2x)(1-3x)(1-3x+x^2)^2}.$$
\end{theorem}

\subsection{Case 162: $T=\{3412,1423,2341\}$}
Let $G_m(x)$ be the generating function for the number of $T$-avoiders with exactly $m$ left-right maxima. Clearly, $G_0(x)=1$ and $G_1(x)=xF_T(x)$.
\begin{lemma}\label{lem162a}
$$G_2(x)=\frac{x^6}{(1-2x)^2(1-x)^4}+\frac{x^2(-x^3+4x^2-3x+1)}{(1-2x)(1-x)^3}F_T(x)\,.$$
\end{lemma}
\begin{proof}
We refine $G_2(x)$ to $G_2(x;d)$, the generating function for $T$-avoiders $\pi=i\pi'n\pi''$ with  $2$ left-right maxima such that $\pi''$ has $d$ letters smaller than $i$, say $j_1,j_2,\dots,j_d$. If $d=0$ so that $\pi''>i$, then $\pi''$ is decreasing (or $in$ is the 14 of a 1423). Hence $G_2(x;0)=\frac{x^2}{1-x}F_T(x)$.

If $d\ge 1$, then $\pi$ decomposes as in Figure \ref{fig162g2},
\begin{figure}[htp]
\begin{center}
\begin{pspicture}(0,-0.3)(9,6.3)
\psset{xunit=1.1cm}
\psset{yunit=.6cm}
\psline(0,0)(4,0)(4,1)(0,1)(0,0)
\psline(4,1)(4,2)(3,2)(3,1)
\psline(0,5)(1,5)(1,3)(2,3)(2,4)(0,4)(0,5)
\psline(4,10)(5,10)(5,8)(6,8)(6,9)(4,9)(4,10)
\psline(7,7)(8,7)(8,5)(9,5)(9,6)(7,6)(7,7)
\psline[linecolor=lightgray](1,5)(8,5)
\psline[linecolor=lightgray](2,4)(5,4)(5,8)
\psline[linecolor=lightgray](2,3)(6,3)(6,8)
\psline[linecolor=lightgray](4,2)(7,2)(7,6)
\psline[linecolor=lightgray](4,1)(8,1)(8,5)
\psline[linecolor=lightgray](4,2)(4,9)
\rput(.5,4.5){\textrm{\small $\al_1\!\searrow$}}
\rput(1.5,3.5){\textrm{\small $\al_2\!\searrow$}}
\rput(3.5,1.5){\textrm{\small $\al_d\!\searrow$}}
\rput(2,0.5){\textrm{\small $\al_{d+1}$}}
\rput(4.5,9.5){\textrm{\small $\be_1\!\searrow$}}
\rput(5.5,8.5){\textrm{\small $\be_2\!\searrow$}}
\rput(7.5,6.5){\textrm{\small $\be_d\!\searrow$}}
\rput(8.5,5.5){\textrm{\small $\be_{d+1}\!\!\searrow$}}
\rput(-0.2,5.2){\textrm{\small $i$}}
\rput(3.75,10.2){\textrm{\small $n$}}
\rput(5.3,4.2){\textrm{\small $j_1$}}
\rput(6.3,3.2){\textrm{\small $j_2$}}
\rput(7.4,2.2){\textrm{\small $j_{d-1}$}}
\rput(8.3,1.2){\textrm{\small $j_d$}}
\rput(6.5,2.6){$\ddots$}
\rput(6.5,7.6){$\ddots$}
\rput(2.5,2.6){$\ddots$}
\qdisk(0,5){2pt}
\qdisk(4,10){2pt}
\qdisk(5,4){2pt}
\qdisk(6,3){2pt}
\qdisk(7,2){2pt}
\qdisk(8,1){2pt}
%
\end{pspicture}
\caption{A $T$-avoider with 2 left-right maxima}\label{fig162g2}
\end{center}
\end{figure}
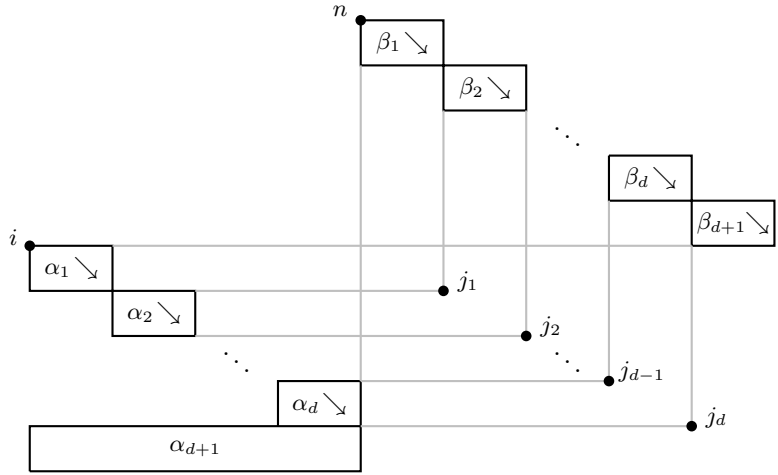
where $j_1 j_2 \cdots j_d$ is decreasing (or $in$ is the 34 of a 3412), $\be_1 \be _2 \cdots \be_d \be_{d+1}$ is decreasing (or $in$ is the 14 of a 1423), and $\al_1\al_2 \cdots \al_d$ is decreasing (or $nj_d$ is the 41 of a 2341). If $\pi'=\emptyset$ (all $\al$'s are empty), the contribution is $\frac{x^{d+2}}{(1-x)^{d+1}}$. Otherwise, let $s\in [\kern .1em 1,d+1\kern .1em]$ be maximal such that $\alpha^{(s)}\neq\emptyset$. Note that $s>1$ implies $\beta_2\cdots\beta_{d+1}=\emptyset$ since $a\in \al_s,\ b\in
\beta_2\cdots\beta_{d+1}$ makes $anj_1b$ a $1423$.
\begin{itemize}
\item $s=1$. Here, $\al_1\ne \emptyset$ and $\al_2= \cdots =\al_d =\al_{d+1} = \emptyset$ and the contribution is $\frac{x^{d+3}}{(1-x)^{d+2}}$.
\item $2\leq s\leq d$. Here, the contribution is $\frac{x^{d+3}}{(1-x)^{s+1}}$ since $\beta_2\cdots\beta_{d+1}=\emptyset$.
\item $s=d+1$. Here, again, $\beta_2\cdots\beta_{d+1}=\emptyset$.
If $\alpha_{d+1}$ lies to the right of $\al_1 \cdots\al_d$, the contribution is
$\frac{x^{d+2}}{(1-x)^{d+1}}\big(F_T(x)-1\big)$ since $\al_{d+1}$ avoids $T$ and is nonempty.

Otherwise, $\pi$ has the form shown in Figure \ref{fig162g2a},
\begin{figure}[htp]
\begin{center}
\begin{pspicture}(-2,-0.3)(9,5.5)
\psset{xunit=1cm}
\psset{yunit=.6cm}
\psline(1,3)(1,6)(6,6)(6,3)(1,3)
\psline(2,1)(2,2)(4,2)(4,0)(6,0)(6,1)(2,1)
\psline(6,8)(6,9)(7,9)(7,8)(6,8)
\psline(-1,8)(0,8)(0,7)(-1,7)(-1,8)
\psline[linecolor=lightgray](2,2)(2,5)
\psline[linecolor=lightgray](4,2)(4,4.5)
\psline[linecolor=lightgray](6,6)(7.5,6)
\psline[linecolor=lightgray](6,3)(8,3)
\psline[linecolor=lightgray](0,8)(6,8)
\psline[linecolor=lightgray](4,2)(8.5,2)
\psline[linecolor=lightgray](0,7)(7,7)(7,8)
\psline[linecolor=lightgray](6,0)(6,8)
\rput(6.5,8.5){\textrm{$\be_1\!\searrow$}}
\rput(5,.5){\textrm{$\de \searrow$}}
\rput[b]{10}(5.5,3.3){\textrm{$\ddots$}}
\rput(-1.2,8.2){\textrm{\small $i$}}
\rput(5.75,9.2){\textrm{\small $n$}}
\rput(8,6.2){\textrm{\small $j_{t-1}$}}
\rput(8.3,3.2){\textrm{\small $j_t$}}
\rput(8.8,2.2){\textrm{\small $j_d$}}
\rput(7.3,7.2){\textrm{\small $j_1$}}
\rput(2.4,5.4){\textrm{\small $k_{p-1}$}}
\rput(4.2,4.8){\textrm{\small $k_{p}$}}
\rput(5.4,4.3){\textrm{\small $k_{p+1}$}}
\rput(2.4,1.5){\textrm{\small $u$}}
\rput(3,1.5){\textrm{$\ga$}}
\rput(3,4){\textrm{\large $\al_t$}}
\rput(-.5,7.5){\textrm{$\al_1\!\searrow$}}
\rput[b]{10}(.5,6.3){$\ddots$}
\rput[b]{-15}(7.25,6.25){$\ddots$}
\rput[b]{10}(1.5,5.3){$\ddots$}
\rput[b]{-15}(8.2,2.25){$\ddots$}
\qdisk(-1,8){2pt}
\qdisk(6,9){2pt}
\qdisk(2,5){1pt}
\qdisk(4,4.5){2pt}
\qdisk(5,4){1pt}
\qdisk(7.5,6){2pt}
\qdisk(8,3){2pt}
\qdisk(8.5,2){2pt}
\qdisk(2.2,1.5){2pt}
\qdisk(7,7){2pt}
\end{pspicture}
\caption{A $T$-avoider with 2 left-right maxima and $\al_{d+1}$ not on right of $\al_1 \cdots\al_d$}\label{fig162g2a}
\end{center}
\end{figure}
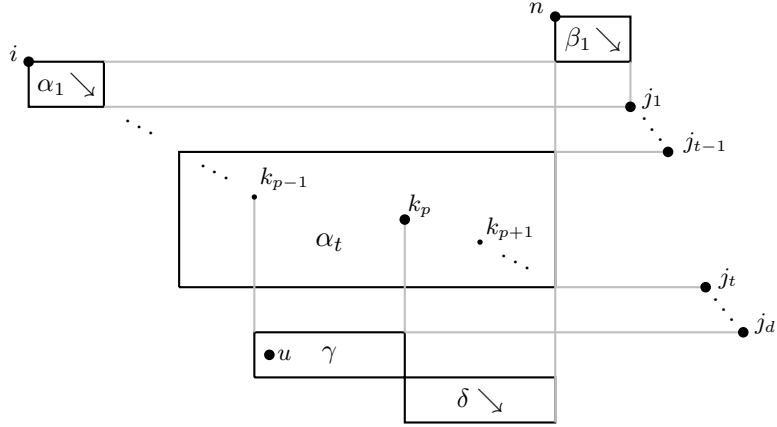
where (i) $\al_{d+1}=\ga\de$ and the leftmost letter $u$ of $\al_{d+1}$ lies to the left of the $p$-th letter $k_p$ ($p\ge 1$) of some $\al_t$ ($t \in [\kern .1em 1,d\kern .2em]$), (ii) $\al_{t+1}=\cdots =\al_d= \emptyset$ (a violator $a$ makes $uk_paj_t$ a 1423), (iii) $\ga>\de$
(or $k_pj_t$ is the 43 of a 1423), and (iv) $\de$ is decreasing (or $uk_p$ is the 34 of a 3412).

\noindent If $\de=\emptyset$, the contribution is $\frac{x^{d+3}}{(1-x)^{t+2}}\big(F_T(x)-1\big)$ since
$\ga$ avoids $T$ and is nonempty, and $k_p$ splits $\al_t$ into two decreasing sequences.

\noindent If $\de\ne \emptyset$, then $\ga$ is decreasing (a violator would be the 23 of a 2341). Thus, the two decreasing sequences $k_{p+1}k_{p+2}\cdots$ and $\de\ne \emptyset$ together contribute $\frac{1}{1-2x}-\frac{1}{1-x}=
\frac{x}{(1-x)(1-2x)}$ to an overall contribution of $\frac{x^{d+4}}{(1-x)^{t+2}}\times \frac{x}{(1-x)(1-2x)}$.
\end{itemize}
By summing all the contributions, we obtain
\begin{align*}
G_2(x;d)&=\frac{x^{d+2}}{(1-x)^{d+1}}+\frac{x^{d+3}}{(1-x)^{d+2}}+\sum_{s=2}^{d}\frac{x^{d+3}}{(1-x)^{s+1}}\\
&+\frac{x^{d+2}}{(1-x)^{d+1}}(F_T(x)-1)
+\sum_{t=1}^{d}\left(\frac{x^{d+3}}{(1-x)^{t+2}}\big(F_T(x)-1\big)+\frac{x^{d+5}}{(1-x)^{t+3}(1-2x)}\right),
\end{align*}
valid for $d\ge 1$, and also for $d=0$. Summing over $d\geq0$ gives the stated expression for $G_2(x)$.
\end{proof}

\begin{lemma}\label{lem162b}
$$G_3(x)=\frac{x}{1-x}G_2(x)+\frac{x^4}{(1-x)^3}F_T(x)+
\frac{x^5}{(1-x)^4(1-2x)}\,.$$
\end{lemma}
\begin{proof}
Suppose $\pi=i_1\pi^{(1)}i_2\pi^{(2)}i_3\pi^{(3)}$ is a permutation in $S_n(T)$ with $3$ left-right maxima. Since $\pi$ avoids $2341$, we see that $\pi^{(3)}>i_1$. If $\pi^{(3)}>i_2$, then, since $\pi^{(3)}$ is decreasing (to avoid $1423$), the contribution to $G_3(x)$ is $\frac{x}{1-x}G_2(x)$. Otherwise, $\pi^{(3)}$ has a letter between $i_1$ and $i_2$, and $\pi$ decomposes either as in Figure \ref{fig162g3}a) (when $\pi^{(2)}>i_1$) or as in Figure \ref{fig162g3}b)
(when $\pi^{(2)}\not> i_1$)
\begin{figure}[htp]
\begin{center}
\begin{pspicture}(-1,-1.5)(11,2.5)
\psset{xunit=1cm}
\psset{yunit=.6cm}
\psline(0,0)(1,0)(1,1)(0,1)(0,0)
\psline(1,3)(3,3)(3,4)(2,4)(2,2)(1,2)(1,3)
\psline(3,2)(4,2)(4,1)(3,1)(3,2)
\psline(10,1)(11,1)(11,2)(10,2)(10,1)
\psline(7,0)(9,0)(9,-1)(8,-1)(8,1)(7,1)(7,0)
\psline(8,2)(9,2)(9,4)(10,4)(10,3)(8,3)(8,2)
\psline[linecolor=lightgray](2,2)(3,2)(3,3)
\psline[linecolor=lightgray](8,1)(10,1)
\psline[linecolor=lightgray](1,2)(1,1)(3,1)
\psline[linecolor=lightgray](9,0)(9,2)(10,2)(10,3)
\psline[linecolor=lightgray](8,1)(8,2)
\rput[b]{10}(1.5,2.3){\textrm{$\searrow$}}
\rput[b]{10}(2.5,3.3){\textrm{$\searrow$}}
\rput[b]{10}(3.6,1.2){\textrm{$\searrow$}}
\rput[b]{10}(7.5,0.3){\textrm{$\searrow$}}
\rput[b]{10}(8.5,2.3){\textrm{$\searrow$}}
\rput[b]{10}(9.5,3.3){\textrm{$\searrow$}}
\rput[b]{10}(8.6,-0.8){\textrm{$\searrow$}}
\rput[b]{10}(10.6,1.2){\textrm{$\searrow$}}
\rput(-0.2,1.2){\textrm{\small $i_1$}}
\rput(.8,3.2){\textrm{\small $i_2$}}
\rput(1.8,4.2){\textrm{\small $i_3$}}
\rput(6.8,1.2){\textrm{\small $i_1$}}
\rput(7.8,3.2){\textrm{\small $i_2$}}
\rput(8.8,4.2){\textrm{\small $i_3$}}
\rput(.5,.5){\textrm{\small Av $T$}}
\qdisk(0,1){2pt}
\qdisk(1,3){2pt}
\qdisk(2,4){2pt}
\qdisk(3.2,1.8){2pt}
\qdisk(7,1){2pt}
\qdisk(8,3){2pt}
\qdisk(9,4){2pt}
\qdisk(10.2,1.8){2pt}
\qdisk(8.2,-0.2){2pt}
\rput(2,-2.2){\textrm{a) $\pi^{(2)}>i_1$}}
\rput(9,-2.2){\textrm{b) $\pi^{(2)}\not>i_1$}}
\end{pspicture}
\caption{A $T$-avoider with 3 left-right maxima and $\pi^{(3)}\not>i_2$}\label{fig162g3}
\end{center}
\end{figure}
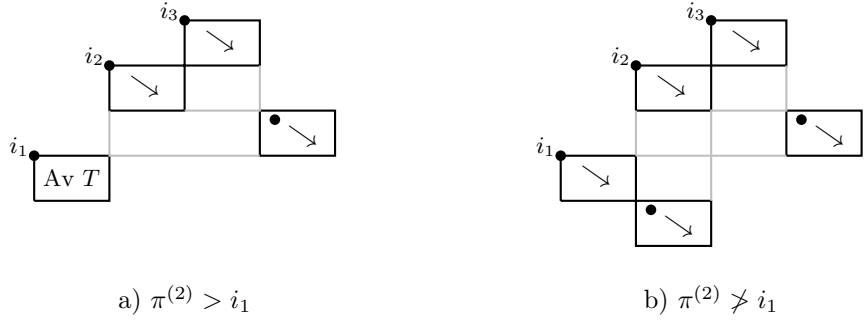
where the down arrows indicate decreasing entries,
with respective contributions of $\frac{x^4}{(1-x)^3}F_T(x)$ and $\frac{x^4}{(1-x)^3} \big(\frac{1}{1-2x}-\frac{1}{1-x}\big)$.
The result follows by adding contributions.
\end{proof}

\begin{lemma}\label{lem162c} For $m\geq4$,
$$G_m(x)=\frac{x}{1-x}G_{m-1}(x)+\frac{x^3}{(1-x)^3}G_{m-2}(x)\,.$$
\end{lemma}
\begin{proof}
Let $\pi=i_1\pi^{(1)}i_2\pi^{(2)}\cdots i_m\pi^{(m)}\in S_n(T)$ with $m\ge 4$ left-right maxima. Then $\pi^{(m)}>i_{m-2}$ (to avoid 2341).
If $\pi^{(m)}>i_{m-1}$, the contribution is $\frac{x}{1-x}G_{m-1}(x)$.
If $\pi^{(m)}$ has a letter between $i_{m-2}$ and $i_{m-1}$, then $\pi^{(m-1)}>i_{m-2}$ (or $i_{m-3}i_{m-1}$ is the 14 of a 1423) and, similar to Figure \ref{fig162g3}a), the contribution is $\frac{x^3}{(1-x)^3}G_{m-2}(x)$.
\end{proof}
By summing the recurrence of Lemma \ref{lem162c} over $m\geq4$, we obtain
\begin{align*}
&F_T(x)-1-xF_T(x)-G_2(x)-G_3(x)\\
&=\frac{x}{1-x}\big(F_T(x)-1-xF_T(x)-G_2(x)\big)+\frac{x^3}{(1-x)^3}\big(F_T(x)-1-xF_T(x)\big).
\end{align*}
Substituting the expressions for $G_2(x)$ and $G_3(x)$ from Lemmas \ref{lem162a} and Lemma \ref{lem162b} and then solving for $F_T(x)$, we get the following result.

\begin{theorem}\label{th162a}
Let $T=\{3412,1423,2341\}$. Then
$$F_T(x)=\frac{1-7x+18x^2-21x^3+11x^4}{(1-2x)(1-6x+12x^2 -11x^3+3x^4)}\,.$$
\end{theorem}

\subsection{Case 163: $\{1342,2314,3412\}$}
Note that all three patterns contain 231.
\begin{lemma}\label{lem163a1}
The generating function for $T$-avoiders with $2$ left-right maxima is given by
$$H(x)=\frac{x\big((1-4x+7x^2-7x^3+4x^4)C(x)-1+4x-8x^2+9x^3-4x^4\big)}{(1-x)^4(1-2x)}.$$
\end{lemma}
\begin{proof}
Let $H_d(x)$ be the generating function for  $T$-avoiders $i\pi'n\pi''$ with $2$ left-right maxima where $\pi''$ has $d$ letters smaller than $i$.
If $d=0$, then $\pi'$ and $\pi''$ independently avoid 231, and so $H_0(x)=x^2C(x)^2$.
Now let $d\ge 1$ and $j_1,j_2,\dots,j_d$ be the letters in $\pi''$ smaller than $i$. These letters occur in decreasing order (to avoid 3412).
Since $\pi$ avoids $1342$, we can write $\pi$ as
$$\pi=i\alpha_0\alpha_1\cdots\alpha_d n\beta_0j_1\beta_1\cdots j_d\beta_d \,, $$
where $i>\alpha_0>j_1>\alpha_1>\cdots>j_d>\alpha_d$. Since $\pi$ avoids $2314$, we also have $\beta_0>\beta_1>\cdots>\beta_d$.

By considering the cases (i) $\alpha_d=\beta_d=\emptyset$, (ii) $\alpha_d\neq\emptyset,\beta_d=\emptyset$, (iii) $\alpha_d=\emptyset,\beta_d\neq\emptyset$, and (iv) $\alpha_d\neq\emptyset,\beta_d\neq\emptyset$, we obtain the respective contributions
$xH_{d-1}(x)$, $\frac{x^{d+2}}{(1-x)^{d+1}}\big(C(x)-1\big)$, $\frac{x^{d+2}}{(1-x)^{d+1}}\big(C(x)-1\big)$, and $\frac{x^{d+2}}{(1-x)^2}\big(C(x)-1\big)^2$. Thus,
$$H_d(x)=xH_{d-1}(x)+\frac{2x^{d+2}}{(1-x)^{d+1}}\big(C(x)-1\big)+\frac{x^{d+2}}{(1-x)^2}\big(C(x)-1\big)^2.$$
Summing over $d\geq1$ and using the expression for $H_0(x)$, we obtain
$$H(x)-x^2C(x)^2=xH(x)+\frac{2x^3}{(1-x)(1-2x)}\big(C(x)-1\big)+\frac{x^3}{(1-x)^3}\big(C(x)-1\big)^2\,,$$
and the result follows by solving for $H(x)$.
\end{proof}

\begin{theorem}\label{th163a}
Let $T=\{1342,2314,3412\}$. Then
$$F_T(x)=\frac{(1-3x+3x^2)^2C(x)-x(1-x)(1-3x+5x^2-4x^3)}{(1-x)^5(1-2x)}.$$
\end{theorem}
\begin{proof}
Let $G_m(x)$ be the generating function for $T$-avoiders with $m$
left-right maxima. Clearly, $G_0(x)=1$, $G_1(x)=xF_T(x)$, and Lemma \ref{lem163a1} gives $G_2(x)$. For $G_m(x)$ with $m\geq3$, suppose $\pi=i_1\pi^{(1)}i_2\pi^{(2)}\cdots i_m\pi^{(m)}\in S_n(T)$ has $m\ge 3$ left-right maxima. Since $\pi$ avoids $1342$, we see that $\pi^{(s)}>i_{s-1}$ for all $s=2,3,\ldots,m-1$, and $\pi^{(m)}$ can be written as $\alpha\beta$ with $\alpha>i_{m-1}$ and $\pi^{(1)}>\beta$ (to avoid $1342$), and $\beta$ is decreasing (to avoid 3412). Note that $\pi$ avoids $T$ if and only if each of $\pi^{(1)},\ldots,\pi^{(m-1)},\alpha$ avoids $231$. Hence, $G_m(x)=\frac{x^mC(x)^m}{1-x}$. Summing over $m\geq3$, we obtain
$$F_T(x)-1-xF_T(x)-G_2(x)=\frac{x^3C(x)^3}{(1-x)\big(1-xC(x)\big)}=\frac{x^3C(x)^4}{1-x}.$$
Substituting for $G_2(x)$ and solving for $F_T(x)$, we complete the proof.
\end{proof}

\subsection{Case 164: $\{1432,2431,3214\}$}
We count by initial letters and define $a(n)=|S_n(T)|$ and  $a(n;i_1,i_2,\dots, i_m)$ to be the number of $T$-avoiders in $S_n$ whose first $m$ letters are $i_1,i_2,\dots, i_m$.
Clearly, $a(n;n)=a(n-1)$. Note that all three patterns in $T$ contain 321.

\begin{lemma}\label{lem164a1}
The following two tables give a recurrence for $a(n;i,j)$ according as $i<j$ or $i>j$, valid whenever they make sense:
\begin{center}
\begin{tabular}{|l|l|rcl|}\hline
$i=1$     &  & $a(n;1,j)$ & = & ${\displaystyle \frac{j-1}{n-1}\binom{2n-2-j}{n-2} }$ \vphantom{${\displaystyle\sum_i^{\sum_A^B x^i}} $}  \\[2mm] \hline
          & $j=i+1$    & $a(n;i,i+1)$ & = & $a(n-1;i)$ \vphantom{$ {\displaystyle \sum}$}  \\ \cline{2-5}
          & $j=i+2$    & $a(n;i,i+2)$ & = & ${\displaystyle a(n;i-1,i+2)}$ \vphantom{$ {\displaystyle \sum}$} \\   \cline{2-5}
\raisebox{5mm}{$i\ge 2$}   & $j\ge i+3$ & $a(n;i,j)$   & = & ${\displaystyle \sum_{k=j-1}^{n-1}a(n-1;i,k) }$ \vphantom{${\displaystyle\sum_i^{\sum}} $}\\[5mm] \hline
\end{tabular}

\centerline{Recurrence for $a(n;i,j)$ when $i<j$}

\vspace*{3mm}
\begin{tabular}{|l|l|rcl|}\hline
$j=1$     &  & $a(n;i,1)$ & = & $a(n;1,i) +2^{i-2}+1-i$  \vphantom{$ {\displaystyle \sum}$} \\  \hline
          & $i\le n-2$  & $a(n;i,j)$ & = & $a(n-1;i,j)$   \vphantom{$ {\displaystyle \sum}$} \\ \cline{2-5}
$j\ge 2$  & $i=n-1$     & $a(n;n-1,j)$ & = & $a(n-1;n-2,j)+2^{n-3-j}$  \vphantom{$ {\displaystyle \sum}$} \\   \cline{2-5}
          & $i=n$       & $a(n;n,j)$   & = & $a(n-1;j)$  \vphantom{$ {\displaystyle \sum}$} \\  \hline
\end{tabular}

\vspace*{2mm}
\centerline{Recurrence for $a(n;i,j)$ when $i>j$}
\end{center}
\end{lemma}
\begin{proof}
We prove the first entry in each Table and leave the other proofs to the reader. An avoider $\pi=1j\pi'$ is counted by $a(n;1,j)$. Since $\pi$ avoids 1432, St($j\pi'$) avoids 321, has length $n-1$ and first letter $j-1$. Such permutations are known to be counted by the ``Catalan triangle'' and so $a(n;1,j)=\frac{j-1}{n-1}\binom{2n-2-j}{n-2}$, the first item in the top Table.

Now, consider $a(n;i,1)$. Let $\pi=i1\pi'\in S_n(T)$. Either there is no occurrence of $321$ in $\pi$ that starts with $i$, or there is such an occurrence.
In the first case, the map $i1\pi'\rightarrow1i\pi'$ is a bijection,
so we have a contribution of $a(n;1,i)$. Thus,
$a(n;i,1)=a(n;1,i)+b(n,i)$, where $b(n,i)$ is the number of permutations $i1\pi'\in S_n(T)$ containing an occurrence of $321$ that starts with $i$.

Now let us find a formula for $b(n,i)$. Let $\pi=i1\pi'\in S_n(T)$ with $i\pi_p\pi_q$ an occurrence of $321$ where $p$ is minimal and $p+q$ is minimal.
Say $\pi_p=u$ and $\pi_q=v$. Thus $iuv$ is the leftmost (minimal) occurrence of $321$ in $\pi$, and
$\pi$ has the form shown in Figure \ref{fig164},
\begin{figure}[htp]
\begin{center}
\begin{pspicture}(-1,-0.3)(6,4.4)
\psset{xunit=.8cm}
\psset{yunit=.5cm}
\psline(0,6)(2,6)
\psline(2,0)(0,0)(0,2)
\psline(2,6)(2,8)(4,8)(4,4)
\pspolygon[fillstyle=solid,fillcolor=lightgray](0,2)(0,8)(2,8)(2,4)(4,4)(4,0)(2,0)(2,2)(0,2)
\pspolygon[fillstyle=solid,fillcolor=lightgray](4.4,0)(6.4,0)(6.4,2)(4.4,2)(4.4,0)
\pspolygon[fillstyle=solid,fillcolor=lightgray](4.4,6)(6.4,6)(6.4,8)(4.4,8)(4.4,6)
\psline(4.4,2)(4.4,6)
\psline(6.4,2)(6.4,6)
\psline(0,6)(2,6)
\rput(1,1){\textrm{$\al\nearrow$}}
\rput(3,6){\textrm{$\be\nearrow$}}
\rput(5.4,4){\textrm{$\ga\nearrow$}}
\rput(3.9,8.5){\textrm{$n$}}
\rput(-.7,6.2){\textrm{$i$}}
\rput(-.3,-.2){\textrm{$1$}}
\rput(1.7,4.3){\textrm{$u$}}
\rput(4.2,2.3){\textrm{$v$}}
\qdisk(0,0){2pt}
\qdisk(-.4,6){2pt}
\qdisk(4,8){2pt}
\qdisk(2,4){2pt}
\qdisk(4.4,2){2pt}
\rput(2,3){\textrm{\small min}}
\rput(1,6.7){\textrm{\small $i\!\bullet\! uv$}}
\rput(1,7.3){\textrm{\small $1432$}}
\rput(5.4,6.7){\textrm{\small $iuv\bullet$}}
\rput(5.4,7.3){\textrm{\small $3214$}}
\rput(5.4,0.7){\textrm{\small $1uv\bullet$}}
\rput(5.4,1.3){\textrm{\small $1432$}}
\end{pspicture}
\caption{A $T$-avoider counted by $b(n,i)$}\label{fig164}
\end{center}
\end{figure}
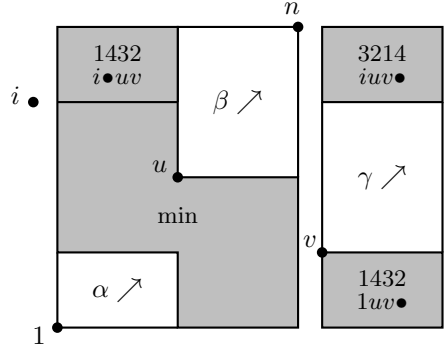
where the shaded regions are empty for the reason indicated (min refers to the minimal property
of $iuv$), $n$ occurs before $v$ (or $n$ is the 4 of a 3214) and in fact immediately before $v$ (or $n$ is the 4 of a 1432), $\al$ is increasing by
the minimal property of $iuv$, $\be$ is increasing (or $v$ is the 2 of a 1432), and
$\ga$ is increasing (or $n$ is the 4 of a 1432). The \gf for the part of $\pi$ below $i$ is $$\frac{x^3}{(1-x)^2(1-2x)}$$ and for the part at or above $i$ is $x^2/(1-x)$. Hence,
$$b(n,i)=[x^{i-1}]\frac{x^3}{(1-x)^2(1-2x)}\times [x^{n-i+1}]\frac{x^2}{1-x},$$
which implies that $b(n,i)=2^{i-2}+1-i$, as required.
\end{proof}

Define 
\begin{align*}
A_n^+(i)&=\sum_{j=i+1}^na(n;i,j),&& A_n^+=\sum_{i=1}^{n-1}A_n^+(i),&&A^+(x)=\sum_{n\geq2}A_n^+x^n.
\end{align*}
Similarly, define 
\begin{align*}
A_n^-(i)&=\sum_{j=1}^{i-1}a(n;i,j),&&A_n^-=\sum_{i=2}^nA_n^-(i),&&A^-(x)=\sum_{n\geq2}A_n^-x^n.
\end{align*} 
Thus, with $A(x)=\sum_{n\geq0}a(n)x^n$, we have
$$A(x)=1+x + A^+(x) + A^-(x).$$

From Lemma \ref{lem164a1}, we have
$$a(n;i,j)=\frac{j-1}{n-1}\binom{2n-2-j}{n-2},$$
for all $i \le j-2$ (independent of $i$), and consequently,
$$\sum_{j=i+2}^na(n;i,j)= \frac{i + 2}{n}\binom{2 n - i - 3}{n - 1}\,.$$
Hence, for $i\le n-1$,
$$A_n^+(i)=a(n-1;i)+ \frac{i + 2}{n}\binom{2 n - i - 3}{n - 1}\,.$$
Summing over $i$,
$$A_n^+=a(n-1)+ \frac{n-2}{2n-1}C_n$$
for $n\geq2$. Multiplying by $x^n$ and summing over $n\geq2$, we find
\begin{align}
A^+(x)=xA(x)+ (1-2x)C(x)-1\,.\label{eq164a3}
\end{align}

Finding $A^-(x)$ is a little more tedious. By Lemma \ref{lem164a1}, we have
\begin{align*}
A_n^-(n-1)&=2^{n-3}+\sum_{j=2}^{n-3}a(n;n-1,j)+a(n-3)\\
&=2^{n-3}+a(n-3)+\sum_{j=2}^{n-3}a(n-1;n-2;j)+(2^0+2^1+\cdots+2^{n-5})\\
&=2^{n-3}+a(n-3)+A_{n-1}^-(n-2)-2^{n-4}+2^{n-4}-1\\
&=A_{n-1}^-(n-2)+a(n-3)+2^{n-3}-1,
\end{align*}
and induction on $n$ implies
\begin{align}\label{eq164a1}
A_n^-(n-1)=\sum_{j=0}^{n-3}a(j)+2^{n-2}-n+1\,.
\end{align}
Now for $n\ge 6$, using Lemma \ref{lem164a1},
\begin{align*}
A_n^-&=A_n^-(2)+A_n^-(3)+\sum_{i=4}^{n-2}A_n^-(i)+A_n^-(n-1)+A_n^-(n)\\
&=C_{n-2}+(C_{n-2}+1)+\sum_{i=4}^{n-2}\sum_{j=1}^{i-1}a(n;i,j)+A_n^-(n-1)+a(n-1)\\
&=2C_{n-2}+1+a(n-1)+A_n^-(n-1)+\sum_{i=4}^{n-2}\sum_{j=1}^{i-1}a(n-1;i,j)+\sum_{i=4}^{n-2}(a_n(n;i,1)-a(n-1;i,1))\\
&=2C_{n-2}+1+a(n-1)+A_n^-(n-1)+A_{n-1}^--A_{n-1}^-(2)-A_{n-1}^-(3)-A_{n-1}^-(n-1)\\
&\quad +\sum_{i=4}^{n-2}\big(a_n(n;i,1)-a(n-1;i,1)\big)\\
&=2C_{n-2}+1+a(n-1)+A_n^-(n-1)+A_{n-1}^--2C_{n-3}-1-a(n-2)\\
&\quad +\sum_{i=4}^{n-2}\big(a_n(n;1,i)-a(n-1;1,i)\big)\\
&=2C_{n-2}+a(n-1)+A_n^-(n-1)+A_{n-1}^--2C_{n-3}-a(n-2)\\
&+C_{n-1}-3C_{n-2}+2C_{n-3}-n+2\\
&=A_{n-1}^-+a(n-1)-a(n-2)+A_n^-(n-1)+C_{n-1}-C_{n-2}-n+2.
\end{align*}
Thus, by \eqref{eq164a1}, we have
\begin{equation}\label{eq164m}
A_n^-=A_{n-1}^- -2a(n-2)+\sum_{j=0}^{n-1}a(j) +C_{n-1}-C_{n-2}+2^{n-2}-2n+3\,,
\end{equation}
where $A_1^-=0$ and $A_2^-=1$, and this formula also holds for $n=3,4,5$.

Multiplying (\ref{eq164m})  by $x^n$ and summing over $n\geq 3$, we get
\begin{align*}
A^-(x)-x^2=\ & xA^-(x)-2x^2\big(A(x)-1\big)+\frac{x}{1-x}\big(A(x)-1-2x\big) +
x\big(C(x)-1-x\big) \\
&-x^2\big(C(x)-1\big)-\frac{x^3(1-3x)}{(1-x)^2(1-2x)}\,,
\end{align*}
 which leads to
\begin{align}\label{eq164a2}
A^-(x)&=\frac{-x}{1-x}+\frac{x}{(1-x)^2}A(x)-\frac{2x^2}{1-x}A(x)+x\big(C(x)-1\big)-\frac{x^3(1-3x)}{(1-x)^3(1-2x)}.
\end{align}
From \eqref{eq164a3} and \eqref{eq164a2} we have
\begin{align*}
A(x)-1-x=&\ \frac{-x}{1-x}+\frac{x}{(1-x)^2}A(x)-\frac{2x^2}{1-x}A(x)+x\big(C(x)-1\big)-\frac{x^3(1-3x)}{(1-x)^3(1-2x)}\\
&+xA(x)+(1-2x)C(x)-1.
\end{align*}
Solving this equation for $A(x)=F_T(x)$, we have the following result.
\begin{theorem}\label{th164a}
Let $T=\{1432,2431,3214\}$. Then
$$F_T(x)=\frac{(1-x)^4(1-2x)C(x)-x(1-4x+6x^2-5x^3)}{(1-x)(1-2x)(1-4x+5x^2-3x^3)}.$$
\end{theorem}

\subsection{Case 165: $\{1342,2314,3421\}$} Note that all three patterns contain 231.
\begin{theorem}\label{th165a}
Let $T=\{1342,2314,3421\}$. Then
$$F_T(x)=\frac{(1-2x)(1-x)^4C(x)-x(1-4x+6x^2-5x^3+x^4)}{(1-x)^4(1-3x+x^2)}.$$
\end{theorem}
\begin{proof}
Let $G_m(x)$ be the generating function for $T$-avoiders with $m$
left-right maxima. Clearly, $G_0(x)=1$ and $G_1(x)=xF_T(x)$.

Let us write an equation for $G_2(x)$. Let $\pi=i\pi'n\pi''\in S_n(T)$ with $2$
left-right maxima.  Say there are $k$ letters $j_1,j_2,\dots, j_k$ in $\pi''$ that are
smaller than $i$. Since $\pi$ avoids $3421$, we see that $j_1<j_2<\cdots<j_k$.
Since $\pi$ avoids $2314$ and $1342$, we can write $\pi$ as
$$i\alpha^{(1)}\alpha^{(2)}\cdots\alpha^{(k+1)}n\beta^{(1)}j_1\beta^{(2)}\cdots j_k\beta^{(k+1)}$$
such that $i>\alpha^{(1)}>j_1>\alpha^{(2)}>j_2>\cdots>\alpha^{(k)}>j_k>\alpha^{(k+1)}$ and $n>\beta^{(1)}>\beta^{(2)}>\cdots>\beta^{(k+1)}>i$.
Furthermore, each of $\alpha^{(1)},\alpha^{(k+1)},\beta^{(1)},\beta^{(k+1)}$ avoids $231$ and
all other $\al$'s and $\be$'s avoid $12$. Three cases:
\begin{itemize}
\item If $\alpha^{(1)}$ has a rise then $\alpha^{(j)}=\beta^{(j)}=\emptyset$
for all $j=2,3,\ldots,k+1$.
So we have a contribution of $x^{k+2}C(x)\big(C(x)-1/(1-x)\big)$.
\item If $\alpha^{(1)}$ is decreasing and $\beta^{(1)}$ has a rise, then
$\alpha^{(j)}=\beta^{(j)}=\emptyset$ for all $j=2,3,\ldots,k+1$.
So we have a contribution of $x^{k+2}/(1-x)\, C(x)\big(C(x)-1/(1-x)\big)$.
\item If $\alpha^{(1)}$ and $\beta^{(1)}$ are decreasing then $\alpha^{(j)},\beta^{(j)}$ are decreasing for all $j=2,3,\ldots,k$.
So we have a contribution of $x^{k+2}/(1-x)^{2k}\, C(x)^2$.
\end{itemize}
Hence,
$$G_2(x)=\sum_{k\geq1}x^{k+2}C(x)\left(C(x)-\frac{1}{1-x}\right)+\sum_{k\geq1}\frac{x^{k+2}}{1-x}C(x)\left(C(x)-\frac{1}{1-x}\right)+\sum_{k\geq0}\frac{x^{k+2}}{(1-x)^{2k}}C(x)^2\,,$$
which implies
$$G_2(x)=\frac{x^2((1-x)(1-3x+2x^2)C(x)^2-x(1-3x+x^2))}{(1-x)^3(1-3x+x^2)}.$$

For $G_m(x)$ with $m\geq3$, a $T$-avoider $\pi$ decomposes as in Figure \ref{fig133} in Case 133 since that case also avoids 1342 and 2314, and furthermore, $\al_1,\dots,\al_{m-1}$ all avoid 231
(or $i_m$ is the 4 of a 2314), $\al_m$ avoids 231 (or $i_1$ is the 1 of a 1342), and $\be_m$
is increasing (or $i_1i_m$ is the 34 of a 3421). Hence,
$$G_m(x)=\frac{x^m}{1-x}C(x)^m.$$

Summing over $m\geq3$ and using the expressions for $G_0(x),G_1(x)$ and $G_2(x)$, we obtain
$$F_T(x)=1+xF_T(x)+\frac{x^2\big((1-x)(1-3x+2x^2)C(x)^2-x(1-3x+x^2)\big)}{(1-x)^3(1-3x+x^2)}+\frac{x^3C(x)^3}{(1-x)\big(1-xC(x)\big)}\,.$$
Solve for $F_T(x)$ and use the identity $C(x)=1+xC(x)^2$ repeatedly to complete the proof.
\end{proof}

\subsection{Case 175: $\{1423,2341,3142\}$}
The first and last patterns contain 312 and $\{312,2341\}$-avoiders have \gf $L(x)$ given by \cite [A116703]{Sl}
\[
L(x)=\frac{(1 - x)^3}{1 - 4 x + 5 x^2 - 3 x^3}\,.
\]
Let $L_m(x)$ denote the generating function for $\{312,2341\}$-avoiders with $m$ left-right maxima so that $L(x)=\sum_{m\geq0}L_m(x)$.
\begin{theorem}\label{th175a}
Let $T=\{1423,2341,3142\}$. Then
$$F_T(x)=\frac{1-6x+12x^2-11x^3+5x^4}{1-7x+17x^2-20x^3+12x^4-2x^5}\,.$$
\end{theorem}
\begin{proof}
Let $G_m(x)$ be the generating function for $T$-avoiders with $m$
left-right maxima. Clearly, $G_0(x)=1$ and $G_1(x)=xF_T(x)$.

For $G_2(x)$, define $G_2(x;r)$ to be the generating function for $T$-avoiders $\pi=(n-r)\pi'n\pi''$ so that $G_2(x)=\sum_{r\ge 1}G_2(x;r)$. Since $\pi$ avoids $1423$,
we see that $n-1,n-2,\dots,n-r+1$ occur in that order and so $\pi$ has the form 
$\pi=(n-r)\alpha_1 n\alpha_2 (n-1)\cdots\alpha_r (n-r+1)\alpha_{r+1}.$
Since $\pi$ avoids $3142$, we see that $\alpha_1>\alpha_2>\cdots>\alpha_r$.

If $\alpha_{r+1}$ is not empty, then $\alpha_j$ is decreasing for $j=1,2,\ldots,r$
since $\pi$ avoids $2341$, and $\alpha_{r+1}$ avoids $T$. So we have a contribution of $\frac{x^{r+1}}{(1-x)^{r}}\big(F_T(x)-1\big)$.
If $\alpha_{r+1}$ is empty, then by removing the letter $n-r+1$, we have a contribution of $x\,G_2(x;r-1)$.
Thus, for $r\ge 2$,
\begin{equation}\label{eq175}
G_2(x;r)=x\,G_2(x;r-1)+\frac{x^{r+1}\big(F_T(x)-1\big)}{(1-x)^{r}}\,.
\end{equation}
Considering whether $\pi''$ is empty or not, we find that $G_2(x;1)=x^2F_T(x)+\frac{x^2}{1-x}\big(F_T(x)-1\big)$.

Summing (\ref{eq175}) over $r\geq2$, we obtain
$$G_2(x)=\frac{x^2}{1-x}F_T(x)+\frac{x^2}{(1-x)^2}F_T(x)+\frac{x^3}{(1-x)^2(1-2x)}\big(F_T(x)-1\big)\,.$$

For $G_m(x)$ with $m\geq3$, suppose $\pi=i_1\pi^{(1)}i_2\pi^{(2)}\cdots i_m\pi^{(m)}\in S_n(T)$ has  $m\ge 3$ left-right maxima. Since $\pi$ avoids $2341$, certainly $\pi^{(j)}>i_1$ for  $j\ge 3$.

If $\pi^{(2)}>i_1$, then $i_2\pi^{(2)}\cdots i_m\pi^{(m)}$ avoids 312 (or $i_1$ is the 1 of a 1423) and the contribution is $xF_T(x)L_{m-1}(x)$.

If $\pi^{(2)}\not>i_1$, then
$i_1>1$ and $\pi^{(j)}>i_2$ for $j\ge 3$ and $1 \in \pi^{(2)}$ (or $i_11i_2$ is the 314 of a 3142).
Thus, $i_1\pi^{(1)}i_2\pi^{(2)}$ and $i_3\pi^{(3)}\cdots i_m\pi^{(m)}$ respectively contribute  factors of $G_2(x)-xF_T(x)L_1(x)$ and $L_{m-2}(x)$.

 Hence, for all $m\geq3$,
$$G_m(x)=xF_T(x)L_{m-1}(x)+\big(G_2(x)-xF_T(x)L_1(x)\big)L_{m-2}(x).$$

Summing over $m\geq3$,
$$F_T(x)-G_2(x)-G_1(x)-1=xF_T(x)\big(L(x)-L_1(x)-1\big)+\big(G_2(x)-xF_T(x)L_1(x)\big)\big(L(x)-1\big)
\,.$$

Clearly, $L_1(x)=\frac{x}{1-x}$.
Substitute for $G_1,G_2,L,L_1$ and solve for $F_T(x)$.
\end{proof}

\subsection{Case 176: $\{1342,2431,3412\}$}
Note that all three patterns contain 231, and the first two contain 132.
\begin{theorem}\label{th176a}
Let $T=\{1342,2431,3412\}$. Then
\[
F_T(x)=\frac{(1 - x)^2 (1 - 4 x + 6 x^2 - 5 x^3 + x^4)\, C(x)- 1 + 6 x - 14 x^2 +
 15 x^3 - 8 x^4 + x^5}{x (1 - 3 x + x^2) (1 - x + x^3)}\,.
\]
\end{theorem}
\begin{proof}
Let $G_m(x)$ be the generating function for $T$-avoiders with $m$
left-right maxima. Clearly, $G_0(x)=1$ and $G_1(x)=xF_T(x)$.

For $G_2(x)$, suppose $\pi=i\pi'n\pi''\in S_n(T)$ has 2 left-right maxima.
If $\pi''>i$, then $\pi''$ avoids 231
(or $i$ is the 1 of a 1342) while $\pi'$ avoids $T$, and the contribution is $x^2F_T(x)C(x)$.
Otherwise, $\pi''$ has $d\ge 1$ letters smaller than $i$
and these letters are decreasing left to right (or $in$ is the 34 of a 3412)
and form an interval of integers
(or $n$ is the 4 of a 2431). So $\pi$ decomposes as in Figure \ref{fig176g2},
\begin{figure}[htp]
\begin{center}
\begin{pspicture}(-1,-0.3)(4,2.3)
\psset{xunit=1cm}
\psset{yunit=.6cm}
\psline(0,2)(1,2)(1,3)(0,3)(0,2)
\psline(1,0)(2,0)(2,2)(3,2)(3,1)(1,1)(1,0)
\psline(3,3)(4,3)(4,4)(3,4)(3,3)
\psline[linecolor=lightgray](3,4)(2,4)(2,2)
\psline[linecolor=lightgray](2,2)(1,2)(1,1)
\psline[linecolor=lightgray](1,3)(3,3)
\rput(0.5,2.5){\textrm{{$\al$}}}
\rput(1.5,0.5){\textrm{{$\be$}}}
\rput(2.5,1.4){\textrm{{$\ga\!\searrow$}}}
\rput(3.5,3.5){\textrm{{$\de$}}}
\rput(-0.2,3.2){\textrm{{$i$}}}
\rput(1.8,4.2){\textrm{{$n$}}}
\qdisk(0,3){2pt}
\qdisk(2,4){2pt}
\qdisk(2.2,1.8){2pt}
\end{pspicture}
\caption{A $T$-avoider $i\pi'n\pi''$ with $2$ left-right maxima and $\pi''\not>i$}\label{fig176g2}
\end{center}
\end{figure}
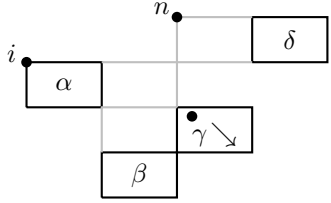
where $\ga$ is nonempty decreasing and $\ga$ is on the left of $\de$ (or $in$ is the 24 of a 2431) and $\al$ is on the left of $\be$ (or $nj_1$ is the 42 of a 1342).
Also, $\de$ avoids $231$ (or $i$ is the 1 of a 1342), while $\alpha$ avoids both $132$
(or $j_1$ is the 4 of a 2431) and $3412$ and $\be$ avoids $T$.

If $\alpha$ is decreasing, the contribution is $\frac{x^3}{(1-x)^2}C(x)F_T(x)$.
If $\alpha$ is not decreasing, then $\be$ is decreasing (to avoid 3412) and the contribution
is $\frac{x^3}{1-x}C(x)\big(K(x)-\frac{1}{1-x}\big)$, where $K(x)=\frac{1-2x}{1-3x+x^2}$
is the generating function for $\{132,3412\}$-avoiders \cite[Seq. A001519]{Sl}. Hence,
$$G_2(x)=x^2C(x)F_T(x)+\frac{x^3}{(1-x)^2}C(x)F_T(x)+\frac{x^3}{1-x}C(x)\left(K(x)-\frac{1}{1-x}\right)\,.$$

For $G_m(x)$ with $m\geq3$, suppose $\pi=i_1\pi^{(1)}i_2\pi^{(2)}\cdots i_m\pi^{(m)}\in S_n(T)$ has $m\ge 3$ left-right maxima.
If $\pi^{(m)}>i_{m-1}$, then $\pi$ avoids $T$ if and only if
$i_1\pi^{(1)}\cdots i_{m-1}\pi^{(m-1)}$ avoids $T$ and $\pi^{(m)}$ avoids $231$, which gives a contribution of $xG_{m-1}(x)C(x)$.

If $\pi^{(m)}\not>i_{m-1}$, then $\pi$ decomposes as in Figure \ref{fig176g3},
\begin{figure}[htp]
\begin{center}
\begin{pspicture}(-1,-0.3)(8,5.8)
\psset{xunit=1.3cm}
\psset{yunit=.6cm}
\psline(0,4)(2,4)(2,3)(1,3)(1,5)(0,5)(0,4)
\psline(3,1)(5,1)(5,0)(4,0)(4,2)(3,2)(3,1)
\psline(5,8)(6,8)(6,9)(5,9)(5,8)
\psline[linecolor=lightgray](5,1)(5,8)
\psline[linecolor=lightgray](4,2)(4,5)
\psline[linecolor=lightgray](4,8)(5,8)
\psline[linecolor=lightgray](4,8)(4,9)(5,9)
\pspolygon[fillstyle=solid,fillcolor=lightgray](1,5)(4,5)(4,8)(3,8)(3,7)(2,7)(2,6)(1,6)(1,5)
\pspolygon[fillstyle=solid,fillcolor=lightgray](4,5)(6,5)(6,8)(4,8)(4,5)
\rput(2.15,7.85){\textrm{{\footnotesize $\iddots$}}}
\rput[b]{10}(2.5,2.3){\textrm{{\footnotesize $\ddots$}}}
\rput(0.5,4.5){\textrm{{$\al_1$}}}
\rput(1.5,3.5){\textrm{{$\al_2\searrow$}}}
\rput(3.5,1.5){\textrm{{$\al_{m-1}\!\!\searrow$}}}
\rput(4.5,0.4){\textrm{{$\:\al_m\!\searrow$}}}
\rput(5.5,8.5){\textrm{{$\de$}}}
\rput(-0.2,5.3){\textrm{{$i_1$}}}
\rput(.8,6.3){\textrm{{$i_2$}}}
\rput(1.8,7.3){\textrm{{$i_3$}}}
\rput(2.65,8.3){\textrm{{$i_{m-1}$}}}
\rput(3.7,9.3){\textrm{{$i_m$}}}
\rput(3.07,5.8){\textrm{\small $\bullet$}}
\rput(3,6.3){\textrm{\small $2\,4\,3\,1$}}
\rput(5.25,6.3){\textrm{\small $\bullet$}}
\rput(5,6.8){\textrm{\small $1\,3\,4\,2$}}
\qdisk(0,5){2pt}
\qdisk(1,6){2pt}
\qdisk(2,7){2pt}
\qdisk(3,8){2pt}
\qdisk(4,9){2pt}
\qdisk(4.2,0.8){2pt}
\end{pspicture}
\caption{A $T$-avoider with $m\ge 3$ left-right maxima and $\pi^{(m)}\not>i_{m-1}$}\label{fig176g3}
\end{center}
\end{figure}
where the shaded regions are empty for the reason indicated, $\al_m$ is left of $\de$
(or $i_{m-1}i_m$ is the 24 of a 2431), $\al_1>\al_2$ (a violator $uv$ and $a$ in $\al_m$
makes $ui_2va$ a 2431), and $\al_2\cdots \al_m$ is decreasing (or $i_1i_2$ is the 34 of a 3412).
Also, $\al_1$ avoids both 132 (since $\al_m\ne \emptyset$) and 3412, and $\de$ avoids 231.
Thus, we have a contribution of $\frac{x^{m+1}}{(1-x)^{m-1}}K(x)C(x)$.

Hence, for $m\ge 3$,
$$G_m(x)=xG_{m-1}(x)C(x)+\frac{x^{m+1}}{(1-x)^{m-1}}K(x)C(x).$$
Summing this recurrence over $m\geq3$, we obtain
$$F_T(x)=1+xF_T(x)+G_2(x)+xC(x)\big(F_T(x)-1-xF_T(x)\big)+\frac{x^4}{(1-x)^2}\left(K(x)-\frac{1}{1-x}\right)C(x)\,,$$
and, substituting for $G_2(x)$, the result follows by solving for $F_T(x)$.
\end{proof}

\subsection{Case 178: $\{1342,2314,2431\}$} Note that all three patterns contain 231.
\begin{theorem}\label{th178a}
Let $T=\{1342,2314,2431\}$. Then
\[
F_T(x)=\frac{(1 - x)^2 (1 - 4 x + 6 x^2 - 5 x^3 + x^4) C(x) - 1 + 6 x - 14 x^2 +
   15 x^3 - 8 x^4 + x^5}{x (1 - 3 x + x^2) (1 - x + x^3)}\,.
\]
\end{theorem}
\begin{proof}
Let $G_m(x)$ be the generating function for $T$-avoiders with $m$
left-right maxima. Clearly, $G_0(x)=1$ and $G_1(x)=xF_T(x)$.

For $G_2(x)$, suppose $\pi=i\pi'n\pi''\in S_n(T)$. If $i=n-1$, the contribution is  $x(F_T(x)-1)$.
Otherwise, we denote the contribution by $H$. So $G_2(x)=x\big(F_T(x)-1\big)+H$.
Now let us write a formula for $H$.
Here, $i<n-1$ and $\pi$ decomposes as in Figure \ref{fig178g2}a),
\begin{figure}[htp]
\begin{center}
\begin{pspicture}(-6.2,-1)(8,2.5)
\psset{xunit=.47cm}
\psset{yunit=.38cm}
\psline(-10,0)(-6,0)(-6,4)(-4,4)(-4,2)(-10,2)(-10,0)
\psline(-8,0)(-8,2)
\psline(0,0)(4,0)(4,4)(6,4)(6,2)(0,2)(0,0)
\psline(6,4.5)(8,4.5)(8,6.5)(6,6.5)(6,4.5)
\psline[linecolor=lightgray](-8,2)(-8,4)(-6,4)
\rput(-9,1){\textrm{{$\pi'$}}}
\rput(-7,1){\textrm{{$\al$}}}
\rput(-5,2.7){\textrm{{$\be$}}}
\rput(2,1){\textrm{{$\ga$}}}
\rput(5,3){\textrm{{$\de$}}}
\rput(7,5.2){\textrm{{$\be$}}}
\rput(-10.4,2.2){\textrm{\small {$i$}}}
\rput(-8.5,4.2){\textrm{{\small $n$}}}
\qdisk(-10,2){2pt}
\qdisk(-8,4){2pt}
\qdisk(-5,3.6){2pt}
\psline[linecolor=lightgray](0,2)(0,4.5)(6,4.5)
\qdisk(0,4.5){2pt}
\qdisk(2,6.5){2pt}
\qdisk(7,6.1){2pt}
\qdisk(4,4){2pt}
\rput(-.4,4.7){\textrm{{\small $i$}}}
\rput(1.5,6.8){\textrm{{\small $n$}}}
\rput(3.1,4.1){\textrm{{\small $i-1$}}}
\rput(-7,-1.5){\textrm{a) general form}}
\rput(4,-1.5){\textrm{b) $i>1$ and $i-1$ after $n$}}
\end{pspicture}
\caption{A $T$-avoider $i\pi'n\pi''$ with $2$ left-right maxima and $i<n-1$}\label{fig178g2}
\end{center}
\end{figure}
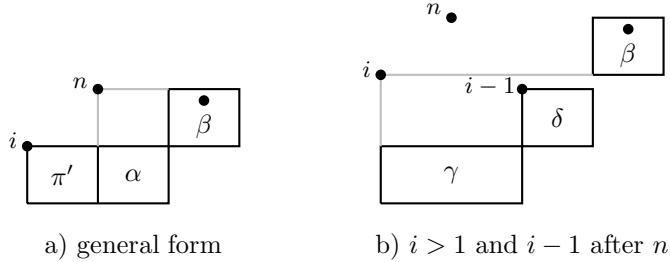
where $\alpha$ is left of $\beta$ (or $in$ is the 24 of a 2431),
$\beta\neq\emptyset$ and avoids $231$ (or $i$ is the 1 of a 1342), and $\pi'\alpha$ also avoids $231$ (or $n$ is the 4 of a 2314). Four cases:
\begin{itemize}
\item If $i=1$, then $\pi'\alpha=\emptyset$ and the contribution is $x^2\big(C(x)-1\big)$.
\item If $i>1$ and $i-1$ occurs in second position, the contribution is $xH$.
\item If $i>1$ and $i-1$ occurs before $n$ but not in second position, let $a$ denote the smallest letter that occurs before $i-1$. We have $a=1$ (or $a(i-1)1b$ is a 2314 for $b$ in $\be$) and
$\alpha=\emptyset$ (or $a(i-1)n$ is the 134 of a 1342). So $\pi'$ is a 231-avoider of
length $\ge 2$ in which 1 precedes its maximal letter, and $\be>i>\pi'$ is a 231-avoider
of length $\ge 1$, giving a contribution of
$x^2\Big( \big(C(x)-1-x\big)-x\big(C(x)-1\big) \Big) \big(C(x)-1\big)=x^5C(x)^5$ in compact form.
\item If $i>1$ and $i-1$ occurs after $n$, then $\pi$ has the form in Figure \ref{fig178g2}b)
where $\ga<\de$ because $\ga(i-1)\de=\pi'\al$ avoids 231. Since $(i-1)\de$ is counted by $xC(x)$, the contribution is $xC(x)H$.
\end{itemize}
Thus
$$H=x^2(C(x)-1)+xH+x^5C(x)^5+xC(x)H\,,$$
with solution
\[
H=\frac{x^2(C(x)-1)+x^5C(x)^5}{1-x-xC(x)}=x(C(x)-1)^2 +x^4 C(x)^5(C(x)-1)\,.
\]

For $G_m(x)$ with $m\geq3$, a $T$-avoider $\pi=i_1\pi^{(1)}i_2\pi^{(2)}\cdots i_m\pi^{(m)}\in S_n$ decomposes as illustrated
\begin{figure}[htp]
\begin{center}
\begin{pspicture}(1,-.3)(5,3.8)
\psset{xunit=1cm}
\psset{yunit=.7cm}
\pspolygon[fillstyle=solid,fillcolor=lightgray](2,1)(6,1)(6,4)(4,4)(4,3)(3,3)(3,2)(2,2)(2,1)
\pspolygon[fillstyle=solid,fillcolor=lightgray](1,0)(4,0)(4,1)(1,1)(1,0)
\psline(0,0)(1,0)(1,2)(3,2)(3,4)(4,4)(4,3)(2,3)(2,1)(0,1)(0,0)
\psline(4,0)(5,0)(5,1)
\psline(5,4)(6,4)(6,5)(5,5)(5,4)
\psline[linecolor=lightgray](4,4)(4,5)(5,5)
\rput(4.5,.5){\textrm{{$\al$}}}
\rput(5.5,4.5){\textrm{{$\be$}}}
\rput(-0.3,1.2){\textrm{{$i_1$}}}
\rput(.7,2.2){\textrm{{$i_2$}}}
\rput(1.7,3.2){\textrm{{$i_3$}}}
\rput(2.7,4.2){\textrm{{$i_4$}}}
\rput(3.7,5.2){\textrm{{$i_m$}}}
\qdisk(0,1){2pt}
\qdisk(1,2){2pt}
\qdisk(2,3){2pt}
\qdisk(3,4){2pt}
\qdisk(4,5){2pt}
\rput(.5,.5){\textrm{{$\pi^{(1)}$}}}
\rput(1.5,1.5){\textrm{{$\pi^{(2)}$}}}
\rput(2.5,2.5){\textrm{{$\pi^{(3)}$}}}
\rput(3.5,3.5){\textrm{{$\pi^{(4)}$}}}
\rput(4.2,2.2){\textrm{{\small $1\,3\,4\,2$}}}
\rput(4.5,1.8){\textrm{{\small $\bullet$}}}
\rput(2.5,.7){\textrm{{\small $2\,3\,1\,4$}}}
\rput(2.6,.3){\textrm{{\small $\bullet$}}}
\end{pspicture}
\caption{A $T$-avoider with $m\ge 3$ left-right maxima}\label{fig178g3}
\end{center}
\end{figure}
in Figure \ref{fig178g3} for $m=5$, where $\alpha$ is left of $\beta$ (or $i_{m-1}i_m$ is the 24 of a 2431).

If $\pi^{(2)}=\pi^{(3)}=\cdots=\pi^{(m)}=\emptyset$, then $\pi^{(1)}$ avoids $231$, and
the contribution is $x^mC(x)$.

Otherwise, there is a maximal $p\in [\kern .1em 2,m\kern .1em]$ such that  $\pi^{(p)}\neq\emptyset$.
Two cases:
    \begin{itemize}
    \item $1\leq p\leq m-1$. Here, $\pi$ avoids $T$ if and only if $\pi^{(j)}$ avoids $231$ for all $j=1,2,\ldots,p$, giving a contribution of $x^mC(x)^{p-1}\big(C(x)-1\big)$.

    \item $p=m$. Here, $\pi^{(m)}=\al\be\ne\emptyset$.  Hence, since $a\in\al$ and $b \in \be$ makes $i_1i_2ab$ a 2314, exactly one of $\al$ and $\be$ is nonempty.

If $\alpha\neq\emptyset$, then $\pi^{(2)} \pi^{(3)}\cdots \pi^{(m-1)}=\emptyset$ (to avoid 2431),
$\pi^{(1)}>\alpha$ (to avoid 1342), $\pi^{(1)}$ avoids $\{132,231\}$, and $\alpha$ is nonempty and avoids $T$, giving a contribution $x^m L\big(F_T(x)-1\big)$, where $L=\frac{1-x}{1-2x}$ is the generating function for $\{132,231\}$-avoiders \cite{SiS}.

If $\be\neq\emptyset$, then $\pi^{(j)}$ avoids $231$ for all $j=1,2,\ldots,m$, and we have a contribution of $x^mC(x)^{m-1}\big(C(x)-1\big)$.
\end{itemize}
Adding all the contributions, we have
$$G_m(x)=x^mC(x)+x^m\, L\,\big(F_T(x)-1\big)+\sum_{p=2}^m x^mC(x)^{p-1}\big(C(x)-1\big).$$
Summing  over $m\geq3$ and using the expressions for $G_0,G_1,G_2,L$, we obtain
$$F_T(x)=1-x+2xF_T(x)+x\big(C(x)-1\big)^2+x^4C(x)^5\big(C(x)-1\big)
+\frac{x^3(F_T(x)-1)}{1-2x}+x^3C^4(x)\,.$$
Solving for $F_T(x)$ gives an expression which can be written as in the statement of the Theorem.
\end{proof}

\subsection{Case 182: $\{2314,2431,3412\}$}
Note that all three patterns in $T$ contain 231.
Let $G_m(x)$ be the generating function for $T$-avoiders with $m$ left-right maxima. As usual,
$G_0(x)=1$ and $G_1(x)=xF_T(x)$. For $G_2(x)$, we begin with $H(x)$, the generating function
for $T$-avoiders of the form $\pi=(n-1)\pi'n\pi''$.
\begin{lemma}\label{lem182H}
$$H(x)=xC(x)-x.$$
\end{lemma}
\begin{proof}
Suppose $\pi=(n-1)\pi'n\pi''$ is counted by $H(x)$.
If $\pi'=\pi''=\emptyset$, the contribution is $x^2$. Otherwise $n\ge 3$ and we consider the
position of $n-2$. If $n-2$ occurs after $n$, then $\pi''$ is decreasing (to avoid $3412$) and $\pi'<\pi''$ (or $n(n-2)$ is the $43$ of a $2431$), and $\pi'$ avoids $231$. So the contribution is $\frac{x^3}{1-x}C(x)$. If $n-2$ occurs before $n$ and is adjacent to $n$, then by removing $n-2$, the contribution is $xH(x)$. If $n-2$ occurs before $n$ but is not adjacent to $n$, then $\pi$ has the form $(n-1)\alpha(n-2)\beta n\pi''$ with $\beta\ne \emptyset$, and $\al<\be$ (or $(n-2)n$ is the $34$ of a $2314$)  and $\al<\pi''$ (or, for $b\in \be$, $(n-2)b$ is the $43$ of a $2431$), and $\alpha$ avoids $231$ and $(n-1)\beta n\pi''$ avoids $T$. So the contribution is $xC(x)\left(H(x)-\frac{x^2}{1-x}\right)$. Hence, $H(x)=x^2+\frac{x^3}{1-x}C(x)+xH(x)+xC(x)\left(H(x)-\frac{x^2}{1-x}\right)$, leading to
$H(x)=xC(x)-x\,.$ (Is there a direct bijection from $T$-avoiders counted by $H$ to a manifestation of the Catalan numbers?)
\end{proof}
\begin{lemma}\label{lem182a1}
$$G_2(x)=xC(x)-x+x^2C(x)\big(F_T(x)-1\big)+\frac{x^3C(x)^2}{1-x}\big(F_T(x)-1\big).$$
\end{lemma}
\begin{proof}
Suppose $\pi=i\pi'n\pi''\in S_n(T)$ has $2$ left-right maxima.

If $i=n-1$, then we have a contribution of $H(x)$.

So we assume that $i\leq n-2$. If $\pi''>i$ then $\pi$ avoids $T$ if and only if $\pi'$ avoids $231$ and $\pi''\ne \emptyset$ avoids $T$, giving a contribution of
$x^2C(x)\big(F_T(x)-1\big)$.

Otherwise, $\pi$ decomposes as in Figure \ref{fig182g2},
\begin{figure}[htp]
\begin{center}
\begin{pspicture}(0,-.3)(6,6)
\psset{xunit=1.1cm}
\psset{yunit=.6cm}
\psline(0,0)(1,0)(1,2)(2,2)(2,1)(0,1)(0,0)
\psline(3,3)(4,3)(4,5)(6,5)(6,6)(5,6)(5,4)(3,4)(3,3)
\psline(6,8)(7,8)(7,9)(6,9)(6,8)
\psline[linecolor=lightgray](0,1)(0,8)(6,8)
\psline[linecolor=lightgray](5,6)(5,9)
\psline[linecolor=lightgray](1,2)(1,7.5)
\psline[linecolor=lightgray](2,2)(2,7)
\psline[linecolor=lightgray](4,5)(4,6)
\rput(.5,.5){\textrm{$\al_1$}}
\rput(1.5,1.5){\textrm{$\al_2$}}
\rput(3.5,3.5){\textrm{$\al_s$}}
\rput(4.5,4.5){\textrm{$\al_{s+1}$}}
\rput(5.5,5.4){\textrm{$\be\searrow$}}
\rput(6.5,8.4){\textrm{$\ga$}}
\rput(.75,7.7){\textrm{$p_1$}}
\rput(1.75,7.2){\textrm{$p_2$}}
\rput(3.75,6.2){\textrm{$p_s$}}
\rput(-.2,8.2){\textrm{$i$}}
\rput(4.8,9.2){\textrm{$n$}}
\qdisk(0,8){2pt}
\qdisk(5,9){2pt}
\qdisk(6.5,8.8){2pt}
\qdisk(5.6,5.8){2pt}
\qdisk(1,7.5){2pt}
\qdisk(2,7){2pt}
\qdisk(4,6){2pt}
\rput[b]{20}(2.9,6.5){$\ddots$}
\rput[b]{-10}(2.5,2.3){$\iddots$}
\end{pspicture}
\caption{A $T$-avoider $i\pi'n\pi''$ with  $\pi''\not< i$ and $\pi''\not> i$}\label{fig182g2}
\end{center}
\end{figure}
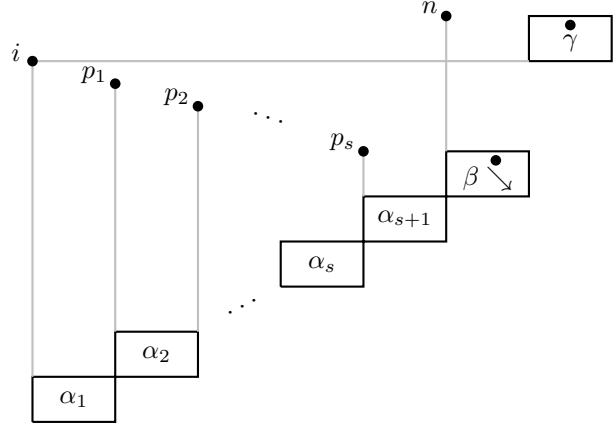
where $\be$ consists of the letters in in $\pi''$ that are $<i$ so that both $\be$ and $\ga$ are
nonempty by assumption and $p_1,\dots,p_s,\ s\ge 0$, are the letters in $\pi'$ that are $>\be$.
The letters in $\be$ are decreasing (or $in$ is the $34$ of a $3412$) and consecutive integers
(a vertical gap in $\be$ makes $c\in \ga$ the $4$ of a $2314$), and $\be$ is left of $\ga$
(or $in$ is the $24$ of a $2431$). Also, the $p$'s are decreasing (or $b\in \be$ and $c \in \ga$
make $bc$ the $14$ of a $2314$) and $\al_k<\al_{k+1}$ for all $k$ (a violator $uv$ makes $up_k vn$ a $2314$). Each $\alpha_k,\ 1\le k \le s+1$, avoids $231$ and $\beta\ne \emptyset$ avoids $T$. So, for given $s$, the contribution is $\frac{x^3}{1-x}\big(F_T(x)-1\big)x^sC(x)^{s+1}$,
Thus, summing over $s\ge 0$, the overall contribution is
$\frac{x^3}{1-x} \big(F_T(x)-1\big)\frac{C(x)}{1-xC(x)} = \frac{x^3C(x)^2}{1-x}\big(F_T(x)-1\big)$. Now add the three contributions to get the result.
\end{proof}

\begin{theorem}\label{th182a}
Let $T=\{2314,2431,3412\}$. Then
$$F_T(x)=\frac{1+x^2(1-x)C(x)^4}{1-x(1-2x)C(x)^2}.$$
\end{theorem}
\begin{proof}
We have found $G_0,G_1,G_2$. To write an equation for $G_m(x)$ with $m\geq3$, it is convenient to define $P_m(x)$ (respectively, $L_m(x)$) to be the generating function for $T$-avoiders $i_1\pi^{(1)}\cdots i_m\pi^{(m)}$
with  $m$ left-right maxima and $\pi^{(1)}\neq\emptyset$ (respectively,
$\pi^{(2)}=\cdots=\pi^{(m-1)}=\emptyset$).
Now suppose a $T$-avoider $\pi=i_1\pi^{(1)}\cdots i_m\pi^{(m)}$ has $m\ge 3$ left-right maxima.
Then $\pi$ decomposes as illustrated in Figure \ref{fig182P} a) for $m=5$.
\begin{figure}[htp]
\begin{center}
\begin{pspicture}(0,-.8)(13,4)
\psset{xunit=1cm}
\psset{yunit=.7cm}
\pspolygon[fillstyle=solid,fillcolor=lightgray](1,0)(4,0)(4,3)(3,3)(3,2)(2,2)(2,1)(1,1)(1,0)
\psline(1,0)(0,0)(0,1)(1,1)(1,2)(2,2)(2,3)(3,3)(3,4)(4,4)
\psline(4,3)(4,5)(5,5)(5,0)(4,0)
\rput(.5,.5){\textrm{{$\pi^{(1)}$}}}
\rput(1.5,1.5){\textrm{{$\pi^{(2)}$}}}
\rput(2.5,2.5){\textrm{{$\pi^{(3)}$}}}
\rput(3.5,3.5){\textrm{{$\pi^{(4)}$}}}
\rput(4.5,2.5){\textrm{{$\pi^{(m)}$}}}
\rput(3,1.2){\textrm{{\small $2\,3\,1\,4$}}}
\rput(3.1,.8){\textrm{{\small $\bullet$}}}
\rput(-0.3,1.2){\textrm{{$i_1$}}}
\rput(.7,2.2){\textrm{{$i_2$}}}
\rput(1.7,3.2){\textrm{{$i_3$}}}
\rput(2.7,4.2){\textrm{{$i_4$}}}
\rput(3.7,5.2){\textrm{{$i_m$}}}
\qdisk(0,1){2pt}
\qdisk(1,2){2pt}
\qdisk(2,3){2pt}
\qdisk(3,4){2pt}
\qdisk(4,5){2pt}
\rput(12.1,.5){\textrm{{$\al\  \searrow$}}}
\pspolygon[fillstyle=solid,fillcolor=lightgray](8,1)(12,1)(12,5)(11,5)(11,4)(10,4)(10,3)(9,3)(9,2)(8,2)(8,1)
\pspolygon[fillstyle=solid,fillcolor=lightgray](8,0)(11,0)(11,1)(8,1)(8,0)
\pspolygon[fillstyle=solid,fillcolor=lightgray](12,1)(13,1)(13,5)(12,5)(12,1)
\psline(12,4)(13,4)
\psline(11,0)(13,0)(13,1)
\rput(7.35,1.2){\textrm{{$i_1$}}}
\rput(7.7,2.2){\textrm{{$i_2$}}}
\rput(8.7,3.2){\textrm{{$i_3$}}}
\rput(9.7,4.2){\textrm{{$i_4$}}}
\rput(10.7,5.2){\textrm{{$i_m$}}}
\qdisk(7.6,1){2pt}
\qdisk(8,2){2pt}
\qdisk(9,3){2pt}
\qdisk(10,4){2pt}
\qdisk(11,5){2pt}
\qdisk(12,.8){2pt}
\rput(9.5,.7){\textrm{{\small $2\,3\,1\,4$}}}
\rput(9.6,.3){\textrm{{\small $\bullet$}}}
\rput(12.5,4.7){\textrm{{\small $2\,3\,1\,4$}}}
\rput(12.6,4.3){\textrm{{\small $\bullet$}}}
\rput(10.6,2.6){\textrm{{\small $2\,4\,3\,1$}}}
\rput(10.7,2.2){\textrm{{\small $\bullet$}}}
\rput(12.5,2.7){\textrm{{\small $3\,4\,1\,2$}}}
\rput(12.82,2.3){\textrm{{\small $\bullet$}}}
\rput(2.5,-.8){\textrm{a) general form}}
\rput(10.5,-.8){\textrm{b) $\pi^{(1)}=\emptyset,\ i_1>1$}}
\end{pspicture}
\caption{A $T$-avoider with $m\ge 3$ left-right maxima}\label{fig182P}
\end{center}
\end{figure}
If $\pi^{(2)}=\cdots=\pi^{(m-1)}=\emptyset$, then by definition, we have a contribution of $L_m(x)$. Otherwise, there is a minimal $j\in [\kern .1em 2,m-1\kern .1em]$
with $\pi^{(j)}\neq\emptyset$, and then $\pi^{(m)}>i_{j-1}$ (or $i_{j-1}i_j$ is the 24 of a $2431$) and $\pi^{(1)}$ avoids $231$. So we have a contribution of $x^{j-1}C(x)P_{m-(j-1)}(x)$. Hence, for $m\ge 3$,
$$G_m(x)=L_m(x)+ C(x) \sum_{j=2}^{m-1}x^{j-1}P_{m+1-j}(x)\,.$$

To find $P_m(x)$, observe that $P_m(x)=G_m(x)-Q_m(x)$ where $Q_m(x)$ is the \gf for $T$-avoiders with $m$ left-right maxima and $\pi^{(1)}=\emptyset$. To find $Q_m(x)$ for $m\ge 3$, if $\pi^{(m)}>i_1$, the contribution is $xG_{m-1}(x)$. Otherwise, $\pi$ decomposes as in Figure \ref{fig182P} b), where $\al\ne\emptyset$ is decreasing (to avoid $3412$), and the contribution is $\frac{x^{m+1}}{1-x}$.

Thus,
\begin{equation}\label{eq182P1}
\begin{aligned}
P_m(x)&=G_m(x)-xG_{m-1}(x)-\frac{x^{m+1}}{1-x}, \quad m\geq3, \\
&=G_2(x)-\frac{x^2}{1-x}F_T(x), \hspace*{15mm}m=2,
\end{aligned}
\end{equation}
the $m=2$ case by a similar decomposition.

It remains to find an equation for $L_m(x)$.
Suppose $m\ge 3$ and $\pi$ is counted by $L_m(x)$, so $\pi^{(s)}=\emptyset$ for $s=2,3,\ldots,m-1$.
If $\pi^{(m)}$ contains a letter smaller than $i_1$, then $\pi^{(m)}<i_1$ just as in Figure \ref{fig182P} b).
Deleting $i_2,\dots,i_{m-1}$, we have a contribution of $x^{m-2}\big(H(x)-x^2C(x)\big)$.
So we can assume that $\pi^{(m)}>i_1$. If $\pi^{(m)}$ contains a letter between $i_{j-1}$
and $i_j$ for some $j\in [\kern .1em 2,m-1\kern .1em]$, then $\pi^{(m)}=(i_j-1)(i_j-2)\cdots(i_{j-1}+1)$
and the contribution is $\frac{x^{m+1}}{1-x}C(x)$. Otherwise, $\pi^{(m)}>i_{m-2}$,
and in this case $\pi^{(m)}=(i_{m-1}-1)\cdots(i_{m-2}+1)\beta$ with $\beta>i_{m-1}$
a $T$-avoider, which implies a contribution of $\frac{x^m}{1-x}C(x)F_T(x)$. Hence,
\begin{align}
L_m(x)=x^{m-2}\big(H(x)-x^2C(x)\big)+(m-2)\frac{x^{m+1}}{1-x}C(x)+\frac{x^m}{1-x}C(x)F_T(x).\label{eq182L1}
\end{align}

Summing \eqref{eq182P1} over $m\ge 2$,
$$P(x):=\sum_{m\geq2}P_m(x)=(1-x)\big(F_T(x)-1-xF_T(x)\big)-\frac{x^2}{1-x}F_T(x)-\frac{x^4}{(1-x)^2}.$$
Summing \eqref{eq182L1} over $m\ge 3$ and using Lemma \ref{lem182H}, we have
$$L(x):=\sum_{m\geq3}L_m(x)=x^2 C(x) -\frac{x^2}{1-x} +
\frac{x^4}{(1-x)^3}C(x)+\frac{x^3}{(1-x)^2}C(x)F_T(x)\,.$$

Hence, by computing $\sum_{m\geq3}G_m(x)$, we obtain
$$F_T(x)-1-xF_T(x)-G_2(x)=L(x)+\frac{x}{1-x}C(x)P(x)\,.$$
The result follows by substituting for $L,P$, and $G_2$ and  solving for $F_T(x)$.
\end{proof}

\subsection{Case 190: $\{3142,2314,1423\}$}
\begin{theorem}\label{th190a}
Let $T=\{3142,2314,1423\}$. Then
$$F_T(x)=\frac{(1-2x)(1-3x+x^2)^2}{(1-x)(1-8x+22x^2-24x^3+8x^4-x^5)}\,.$$
\end{theorem}
\begin{proof}
Let $G_m(x)$ be the generating function for $T$-avoiders with $m$
left-right maxima. Clearly, $G_0(x)=1$ and $G_1(x)=xF_T(x)$.

We can reduce the $m\ge 3 $ case to the $m=2$ case. A $T$-avoider $\pi=i_1\pi^{(1)}i_2\pi^{(2)}
\cdots i_m\pi^{(m)}$ with $m\ge 3$ left-right maxima has $\pi^{(j)}>i_{j-1}$
for $j\in[\kern .1em 2,m-1\kern .1em ]$
(to avoid 2314), and so $\pi$ decomposes as in Figure \ref{fig190g3},
\begin{figure}[htp]
\begin{center}
\begin{pspicture}(0,-.5)(12,5.8)
\psset{xunit=1.0cm}
\psset{yunit=.6cm}
\psline(0,0)(1,0)(1,1)(0,1)(0,0)
\psline(1,2)(2,2)(2,3)(1,3)(1,2)
\psline(4.5,7)(6,7)(6,9)(7,9)(7,8)(4.5,8)(4.5,7)
\psline(2,4)(3,4)(3,5)(2,5)(2,4)
\psline(7,6)(8.5,6)(8.5,7)(7,7)(7,6)
\psline(10,3)(11,3)(11,4)(10,4)(10,3)
\psline(11,0)(12,0)(12,2)(11,2)(11,0)
\psline[linecolor=lightgray](2,2)(11,2)
\psline[linecolor=lightgray](3,4)(10,4)
\psline[linecolor=lightgray](1,0)(11,0)
\rput(.5,.5){\textrm{$\al_1$}}
\rput(1.5,2.5){\textrm{$\al_2\!\!\searrow$}}
\rput(2.5,4.5){\textrm{$\al_3\!\!\searrow$}}
\rput(5.25,7.5){\textrm{$\al_{m-1}\!\searrow$}}
\rput(6.5,8.5){\textrm{$\be_m\!\!\searrow$}}
\rput(7.75,6.5){\textrm{$\be_{m-1}\!\searrow$}}
\rput(10.5,3.5){\textrm{$\be_3\!\searrow$}}
\rput(11.5,1){\textrm{$\be_2$}}
\rput(-.3,1.2){\textrm{$i_1$}}
\rput(.7,3.2){\textrm{$i_2$}}
\rput(1.7,5.2){\textrm{$i_3$}}
\rput(4.0,8.2){\textrm{$i_{m-1}$}}
\rput(5.7,9.2){\textrm{$i_m$}}
\qdisk(0,1){2pt}
\qdisk(1,3){2pt}
\qdisk(2,5){2pt}
\qdisk(4.5,8){2pt}
\qdisk(6,9){2pt}
\rput[b]{0}(3.5,6){$\iddots$}
\rput[b]{0}(9.2,4.8){$\ddots$}
\end{pspicture}
\caption{A $T$-avoider with $m\ge 3$ left-right maxima}\label{fig190g3}
\end{center}
\end{figure}
where $i_j \al_j$ is decreasing for $j\in[\kern .1em 2,m-1\kern .1em ]$ (1423)
and consists of consecutive integers (a gap would imply a 3142), $\be_m\be_{m-1}\cdots\be_3$ is decreasing and lies to the left of $\be_2$ (1423), and $i_1\al_1 i_2\be_2$ avoids $T$ and has 2 left-right maxima. Hence,
for $m\ge 3$ (and even for $m=2$),
\begin{equation}\label{eqn190gm}
G_m(x)=\frac{x^{m-2}}{(1-x)^{2m-4}}G_2(x)\,.
\end{equation}

Now, we focus on the case $m=2$. Fix $k\ge 0$ and let $\pi=(n-1-k)\pi'n\pi''\in S_n(T)$ with $2$ left-right maxima. Then $\pi$ decomposes as in Figure \ref{fig190g2},
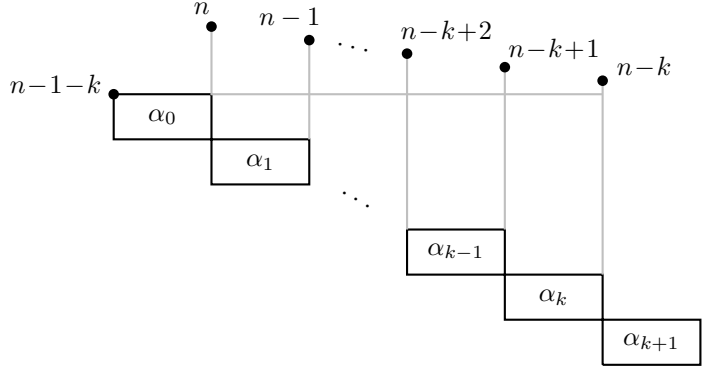
\begin{figure}[htp]
\begin{center}
\begin{pspicture}(0,-.5)(6.5,5)
\psset{xunit=1.3cm}
\psset{yunit=.6cm}
\psline(0,6)(1,6)(1,4)(2,4)(2,5)(0,5)(0,6)
\psline(5,0)(6,0)(6,1)(4,1)(4,3)(3,3)(3,2)(5,2)(5,0)
\psline[linecolor=lightgray](3,3)(3,6.9)
\psline[linecolor=lightgray](4,3)(4,6.6)
\psline[linecolor=lightgray](5,2)(5,6.3)
\psline[linecolor=lightgray](1,7.5)(1,6)(5,6)
\psline[linecolor=lightgray](2,5)(2,7.2)
\rput(.5,5.5){\textrm{$\al_0$}}
\rput(1.5,4.5){\textrm{$\al_1$}}
\rput(3.5,2.5){\textrm{$\al_{k-1}$}}
\rput(4.5,1.5){\textrm{$\al_k$}}
 \rput(5.5,.5){\textrm{$\al_{k+1}$}}
\rput(-.6,6.2){\textrm{$n\!-\!1\!-\!k$}}
\rput(.9,7.9){\textrm{$n$}}
\rput(1.8,7.7){\textrm{$n-1$}}
\rput(3.4,7.4){\textrm{$n\!-\!k\!+\!2$}}
\rput(4.5,7){\textrm{$n\!-\!k\!+\!1$}}
\rput(5.4,6.6){\textrm{$n\!-\!k\!$}}
\qdisk(0,6){2pt}
\qdisk(1,7.5){2pt}
\qdisk(2,7.2){2pt}
\qdisk(3,6.9){2pt}
\qdisk(4,6.6){2pt}
\qdisk(5,6.3){2pt}

\rput[b]{25}(2.5,6.8){$\ddots$}
\rput[b]{10}(2.5,3.4){$\ddots$}
\end{pspicture}
\caption{A $T$-avoider with 2 left-right maxima}\label{fig190g2}
\end{center}
\end{figure}
where $\al_j>\al_{j+1}$ for all $j$ (a violator $a_j<a_{j+1}$ makes $(n-1-k)a_j (n-j)a_{j+1}$ a 3142).
If $\alpha^{(0)},\alpha^{(1)},\cdots,\alpha^{(k-1)}$ are all decreasing, we see that
$\pi$ avoids $T$ if and only if $\alpha^{(k)}$ avoids both $231$ and $1423$ and $\alpha^{(k+1)}$ avoids $T$. Thus, we have a contribution of $\frac{x^{k+2}}{(1-x)^k}K(x)F_T(x)$, where $K(x)=\frac{1-2x}{1-3x+x^2}$ is the generating function for $\{231,1423\}$-avoiders \cite[Seq. A001519]{Sl}.

Otherwise, there is a minimal $j\in[\kern .1em 0,k-1\kern .1em]$ such that $\al_j$ is not decreasing.
Here, $\al_0,\al_1,\dots,\al_{j-1}$ are all decreasing; $\al_j$ is not decreasing and avoids $231$
and $1423$; $\alpha^{(i)}=\emptyset$ for $i\in[\kern .1em j+1,k\kern .1em]$ (since $a\in \al_i$ makes $a(n-k)$
the 14 of a 2314); and $\alpha^{(k+1)}$ avoids $T$. Hence, we have a contribution of $\frac{x^{k+2}}{(1-x)^j}\big(K(x)-\frac{1}{1-x}\big)F_T(x)$.

Summing contributions over $k$ and $j$,
\begin{align*}
G_2(x)&=\sum_{k\geq0}\left(\frac{x^{k+2}}{(1-x)^k}K(x)F_T(x)+\sum_{j=0}^{k-1}\frac{x^{k+2}}{(1-x)^j}
\left(K(x)-\frac{1}{1-x}\right)F_T(x)\right) \\
&=\frac{x^2(1-4x+5x^2-x^3)}{(1-3x+x^2)(1-2x)(1-x)}F_T(x)\,.
\end{align*}

Now sum (\ref{eqn190gm}) over $m\geq3$ and use the expressions for $G_0,G_1,G_2$ to obtain
$$F_T(x)=1+xF_T(x)+\frac{x^2(1-x)(1-4x+5x^2-x^3)}{(1-2x)(1-3x+x^2)^2}F_T(x)\,,$$
and the result follows by solving for $F_T(x)$.
\end{proof}

\subsection{Case 192: $\{1243,1342,2431\}$}
\begin{theorem}\label{th192a}
Let $T=\{1243,1342,2431\}$. Then
$$F_T(x)=\frac{(1 - 5 x + 9 x^2 - 6 x^3)\big(C(x) - 1\big)  - x^3}{x(1 - 2 x) (1 - x)^2}\,.$$
\end{theorem}
\begin{proof}
Let $G_m(x)$ be the generating function for $T$-avoiders with $m$
left-right maxima. Clearly, $G_0(x)=1$ and $G_1(x)=xF_T(x)$.

For $G_2(x)$, let $\pi=i\pi'n\pi''\in S_n(T)$ with $2$ left-right maxima. If $\pi''<i$ then we have a contribution of $x\big(F_T(x)-1\big)$. Otherwise, $\pi''$ has a letter greater than $i$ and, since $\pi$ avoids $1243$, we have that $\pi'$ is decreasing, say $j_1>j_2>\cdots>j_d$. Then, since $\pi$ avoids $2431$, we can express $\pi$ as
$$\pi=ij_1j_2\cdots j_dn\alpha_{d+1}\alpha_{d}\cdots\alpha_{0}$$
where $\alpha_{d+1}<j_d<\alpha_{d}<\cdots<j_1<\alpha_{1}<i<\alpha_0<n$ and $\alpha_0\ne \emptyset$.
If $\alpha_s$ is increasing for all $s=0,1,\ldots,d$, then $\al_{d+1}$ avoids $T$ and
we have a contribution of $\frac{x^{d+3}}{(1-x)^{d+1}}F_T(x)$.
Otherwise, there is a minimal $k\in[\kern .1em 0,d\kern .1em]$ such that $\alpha_k$ is not increasing.
Then $\alpha_s$ is increasing for $s\in[\kern .1em 0,k-1\kern .1em]$;  avoids both 132 and 231 for $s=k$; is empty for $s\in[\kern .1em k+1,d\kern .1em]$; and is decreasing for $s=d+1$. Thus,
with $L(x)=\frac{1-x}{1-2x}$ denoting the generating function for $\{132,231\}$-avoiders \cite{SiS},
the contribution is $\frac{x^{d+2}}{1-x}\big(L(x)-\frac{1}{1-x}\big)$ for $k=0$ and $\frac{x^{d+3}}{(1-x)^{k+1}}\big(L(x)-\frac{1}{1-x}\big)$ for $k\in[\kern .1em 1,d\kern .1em]$.

Summing the contributions,
\begin{align*}
G_2(x)& =x(F_T(x)-1)+\sum_{d\geq0}\frac{x^{d+3}}{(1-x)^{d+1}}F_T(x)\\
&\hspace*{5mm}+\sum_{d\geq0}\left(\frac{x^{d+2}}{1-x}+\sum_{k=1}^d\frac{x^{d+3}}{(1-x)^{k+1}}\right)\left(L(x)-\frac{1}{1-x}\right)\\
&= x\big(F_T(x)-1\big) + \frac{x^{3}}{1-2x}F_T(x) + \frac{x^{4}}{(1-x)(1-2x)^2}\,.
\end{align*}

For $G_m(x)$ with $m\geq3$, let $\pi=i_1\pi^{(1)}i_2\pi^{(2)}\cdots i_m\pi^{(m)}\in S_n(T)$ with $m$ left-right maxima. If $\pi^{(m)}=\emptyset$ then we have a contribution of $xG_{m-1}(x)$.  Otherwise, $\pi$ decomposes as in Figure \ref{fig192g3},
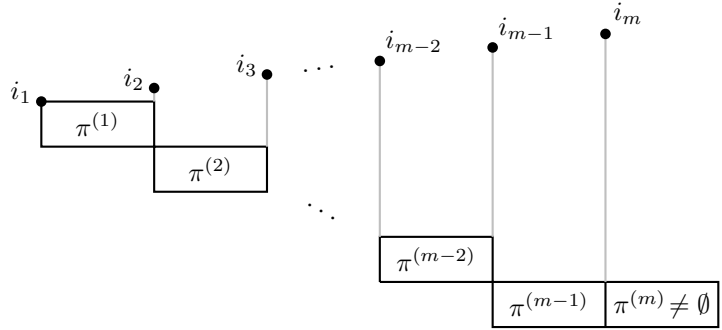
\begin{figure}[htp]
\begin{center}
\begin{pspicture}(0,-.5)(8,4.5)
\psset{xunit=1.5cm}\psset{yunit=.6cm}
\psline(0,5)(1,5)(1,3)(2,3)(2,4)(0,4)(0,5)\psline(3,1)(6,1)(6,0)(4,0)(4,2)(3,2)(3,1)\psline(5,0)(5,1)
\psline[linecolor=lightgray](5,1)(5,6.5)\psline[linecolor=lightgray](4,2)(4,6.2)
\psline[linecolor=lightgray](3,2)(3,5.9)\psline[linecolor=lightgray](2,4)(2,5.6)
\psline[linecolor=lightgray](1,5)(1,5.3)
\rput(.5,4.5){\textrm{$\pi^{(1)}$}}\rput(1.5,3.5){\textrm{$\pi^{(2)}$}}
\rput(3.5,1.5){\textrm{$\pi^{(m-2)}$}}\rput(4.5,0.5){\textrm{$\pi^{(m-1)}$}}
\rput(5.5,0.5){\textrm{$\pi^{(m)}\!\ne \emptyset$}}
\rput(-.18,5.2){\textrm{$i_1$}}\rput(.82,5.5){\textrm{$i_2$}}
\rput(1.82,5.8){\textrm{$i_3$}}\rput(3.3,6.3){\textrm{$i_{m-2}$}}
\rput(4.3,6.6){\textrm{$i_{m-1}$}}\rput(5.2,6.95){\textrm{$i_m$}}
\qdisk(0,5){2pt}\qdisk(1,5.3){2pt}\qdisk(2,5.6){2pt}\qdisk(3,5.9){2pt}
\qdisk(4,6.2){2pt}\qdisk(5,6.5){2pt}
\rput[b]{-30}(2.4,5.5){$\iddots$}\rput[b]{10}(2.5,2.3){$\ddots$}
\end{pspicture}
\caption{A $T$-avoider with $m\ge 3$ left-right maxima and $\pi^{(m)}\!\ne \emptyset$}\label{fig192g3}
\end{center}
\end{figure}
where $\pi^{(j)}>\pi^{(j+1)}$ for $j=1,\dots, m-2$ (a violator $a_j<a_{j+1}$ and $b\in \pi^{(m)}$ makes $a_j i_j a_{j+1} b$ a $2431$); $\pi^{(j)}$ avoids $1324$ for $j=1,\dots,m-2$ (or $\pi^{(m)}$ contains the $1$ of a $2431$); $\pi^{(m-1)}i_m \pi^{(m)}$ is a $T$-avoider with max entry not in last position. Hence, for $m\ge 3$,
$$G_m(x)=xG_{m-1}(x) + x^{m-1}C(x)^{m-2}\big(F_T(x)-1-xF_T(x)\big)\,.$$

Summing this recurrence over $m\ge 3$ and using the expressions for $G_0,G_1,G_2$, we obtain an equation for $F_T(x)$ with solution as stated.
\end{proof}

\subsection{Case 194: $\{3124,4123,1243\}$}
We define $a(n)=|S_n(T)|$ and define $a(n;i_1,i_2,\dots, i_m)$ to be the number of permutations $\pi=\pi_1\pi_2\cdots\pi_n$ in
$S_n(T)$ such that $\pi_1\pi_2\cdots\pi_m=i_1i_2\cdots i_m$. Clearly, $a(n;1)=|S_{n-1}(\{3124,4123,132\})|$, and 
$H(x):=\sum_{n\geq0}|S_n(\{132,3124,4123\})|x^n$ is given by $H(x)=1+\frac{x(1-x)^2}{(1-2x)^2}.$

Set $b(n;i)=a(n;i,i+1)$ and $b'(n;i)=a(n;i,n)$. As in other cases, one can obtain the following relations.
\begin{lemma}\label{lem194a}
For $n\geq4$,
\begin{align*}
a(n;i,n)&=a(n;i,n-1),\qquad\mbox{if $1\leq i\leq n-2$},\\
a(n;i,j)&=a(n-1;i,j)+b(n;i-1),\qquad\mbox{if $2\leq i<j\leq n-1$},\\
a(n;i,j)&=\sum_{k=1}^ja(n-1;i-1,k),\qquad\mbox{if $1\leq j<i-1\leq n-2$},\\
a(n;i,i-1)&a(n-1,i-1),\qquad\mbox{if $2\leq i\leq n$},\\
b(n;i)&=\sum_{k=1}^ib(n-1;k),\qquad\mbox{$1\leq i\leq n-1$},\\
a(n;n-1)&=a(n-1),\\
a(n;n)&=a(n-1),\\
b(n;n)&=0.
\end{align*}
\end{lemma}

Define $A^-(x;w,v)=\sum_{n\geq2}\sum_{i=1}^n\sum_{j=1}^{i-1}a(n;i,j)w^iv^{j-1}$.

\begin{proposition}\label{pro194a}
$$A^-(x;1,1)=xC(x)\big(F_T(x)-1\big).$$
\end{proposition}
\begin{proof}
By Lemma \ref{lem194a}, we have $a(n;i,j)=\sum_{k=1}^ja(n-1;i-1,k)$, for $1\leq j<i-1\leq n-1$. Define $A^-_{n,i}(v)=\sum_{j=1}^{i-1}a(n;i,j)v^{j-1}$. Multiplying the last recurrence by $v^{j-1}$ and summing over $j=1,2,\ldots,i-2$, we obtain
$$A_{n,i}^-(v)-a(n;i,i-1)v^{i-2}
=\frac{1}{1-v}(A^-_{n-1,i-1}(v)-v^{i-2}A^-_{n-1,i-1}(1)),$$
which, by Lemma \ref{lem194a}, implies
$$A_{n,i}^-(v)=\frac{1}{1-v}(A^-_{n-1,i-1}(v)-v^{i-2}A^-_{n-1,i-1}(1))+a(n;i-1)v^{i-2}$$
with $A_{n,1}^-(v)=0$.

Define $A^-_{n}(w,v)=\sum_{i=1}^nA^-_{n;i}(v)w^i$ and $A_n(v)=\sum_{i=1}^na(n;i)v^{i-1}$. Multiplying the last recurrence
by $w^i$ and summing over $i=2,3,\ldots,n$, we obtain
$$A_n^-(w,v)=\frac{w}{v(1-v)}(vA^-_{n-1}(w,v)-A^-_{n-1}(wv,1))+w^2A_{n-1}(wv)$$
with $A_1^-(w,v)=0$. Hence,
$$A^-(x;w,v)=\frac{wx}{v(1-v)}(vA^-(x;w,v)-A^-(x;wv,1))+xw^2(A(x;wv)-1).$$
By taking $w=(1-v)/x$, we obtain
$$A^-(x;(1-v)v/x,1)=v(1-v)^2/x(A(x;(1-v)v/x)-1),$$
which, by taking $v=1/C(x)$, leads to $A^-(x;1,1)=xC(x)\big(A(x;1)-1\big)$, as required.
\end{proof}

Define $B'_n(v)=\sum_{i=1}^{n-1}b'(n;i)v^{i-1}$ and $B_n(v)=\sum_{i=1}^{n-1}b(n;i)v^{i-1}$, and their generating functions by
$B'(x;v)=\sum_{n\geq3}B'_{n}(v)x^n$ and $B(x;v)=\sum_{n\geq2}B_n(v)x^n$.

\begin{lemma}\label{lem194b}
We have
$$B(x;v)=\frac{x^2\big(1-vC(xv)\big)}{1-x-v}$$
and
$$B'(x;v)=\frac{2x^3+vB(x;v)-x^2(v+x)C(xv)}{1-2x}.$$
\end{lemma}
\begin{proof}
By Lemma \ref{lem194a}, we have $b(n;n)=0$ and $b(n;i)=b(n-1;1)+\cdots+b(n-1;i)$. Multiplying by $v^{i-1}$ and summing over $i=1,2,\ldots,n-1$, we obtain
$$B_n(v)=\frac{1}{1-v}(B_{n-1}(v)-v^{n-1}B_{n-2}(1))$$
with $B_2(v)=1$. Hence,
$$B(x/v;v)=x^2/v^2+\frac{x}{v(1-v)}(B(x/v;v)-B(x;1)).$$
By taking $v=1/C(x)$, we have that $B(x;1)=x^2C^2(x)=x(C(x)-1)$, and then
$$B(x;v)=\frac{x^2(1-vC(xv))}{1-x-v}.$$

By Lemma \ref{lem194a}, we have $b'(n;n-1)=b'(n;n-2)=C_{n-2}$, and
$$b'(n;i)=a(n;i,n)=a(n;i,n-1)=a(n-1;i,n-1)+a(n-1;i,n-2)+b(n;i-1),$$
which gives that
$$b'(n;i)=2b'(n-1;i)+b(n;i-1)$$
with $b'(n;n-1)=b'(n;n-2)=C_{n-2}$. By multiplying the last recurrence by $v^{i-1}$ and summing over $i=1,2,\ldots,n-3$, we obtain
\begin{align*}
B'_n(v)&=C_{n-2}(v^{n-2}+v^{n-3})+2(B'_{n-1}(v)-C_{n-3}v^{n-3})\\
&+v(B_n(v)-b(n;n-3)v^{n-4}-b(n;n-2)v^{n-3}-b(n;n-1)v^{n-2}).
\end{align*}
By the first part of the proof, we see that $b(n;n-1)=b(n;n-2)=C_{n-2}$ and $b(n;n-3)=C_{n-2}-C_{n-3}$. Thus, $B'_3(v)=1+v$ and
\begin{align*}
B'_n(v)&=2B'_{n-1}(v)+vB_n(v)-(C_{n-3}+C_{n-2}v^2)v^{n-3}.
\end{align*}
Hence,
$$B'(x;v)=(1+v)x^3+2xB'(x;v)+v(B(x;v)-(1+v)x^3-x^2)-x^3(C(xv)-1)-vx^2(C(xv)-1-xv),$$
which leads to
$$B'(x;v)=\frac{2x^3+vB(x;v)-x^2(v+x)C(xv)}{1-2x},$$
as claimed.
\end{proof}

Define $A^+(x;v)=\sum_{n\geq2}\sum_{i=1}^n\sum_{j=i+1}^na(n;i,j)v^{j-1-i}$.

\begin{proposition}\label{pro194b}
We have
$$A^+(x;1)=\frac{(x^4-2x^3+5x^2-4x+1)C(x)-x^3-2x^2+3x-1}{(1-2x)^2}.$$
\end{proposition}
\begin{proof}
By Lemma \ref{lem194a}, we have $a(n;i,n)=b'(n;i)$ and
$$a(n;i,j)=a(n-1;i,j)+a(n-1;i,j-1)+b(n;i-1),$$
for all $2\leq i<j\leq n-1$. Define $A^+_{n;i}(v)=\sum_{j=i+1}^na(n;i,j)v^{j-1-i}$. Thus,
$$A^+_{n;i}(v)-b'(n;i)v^{n-1-i}=A^+_{n-1;i}(v)+v\left(A^+_{n-1;i}(v)-b'(n-1;i)v^{n-2-i}\right)+\frac{1-v^{n-1-i}}{1-v}b(n;i-1),$$
which leads to
\begin{align}\label{eq194a1}
A^+_{n;i}(v)=b'(n;i)v^{n-1-i}-b'(n-1;i)v^{n-1-i}+(1+v)A^+_{n-1;i}(v)+\frac{1-v^{n-1-i}}{1-v}b(n;i-1),
\end{align}
for all $i=2,3,\ldots,n-2$. Note that $A^+_{n;n-1}(v)=a(n;n-1,n)=C_{n-2}$. Moreover, $A^+_{n;n}(v)=0$.

Define $A^+_{n}(v)=\sum_{i=1}^nA^+_{n;i}(v)$. By summing \eqref{eq194a1} over $i=2,3,\ldots,n-2$, using $b'(n;1)=2^{n-3}$ and $b'(n;n-1)=b(n;n-2)=C_{n-2}$ (see Lemma \ref{lem194b}), we have
\begin{align*}
A^+_n(v)&=A^+_{n;1}(v)+\left(B'_n(1/v)-2^{n-3}\right)v^{n-2}-\left(B'_{n-1}(1/v)-2^{n-4}\right)v^{n-2}\\
&+(1+v)\left(A^+_{n-1}(v)-A^+_{n-1;1}(v)\right)+\frac{1}{1-v}\left(B_n(1)-v^{n-3}B_n(1/v)\right)+C_{n-2}/v,
\end{align*}
Let $\sum_{n\geq2}A^+_{n;1}(v)x^n=G(x;v)$. Multiplying by $x^n$ and summing $n\geq4$, we obtain
\begin{align*}
\big(1-(1+v)x\big)A^+(x;v)&=\big(1-(1+v)x\big)G(x;v)+\frac{1-vx}{v^2}B'(vx;1/v)-x^3v-\frac{v^2x^4}{1-2vx}\\
&+\frac{1}{1-v}(B(x;1)-\frac{1}{v^3}B(vx;1/v))+\frac{x^2}{v}C(x).
\end{align*}
By Lemma \ref{lem194b}, we have
$$\lim_{v\rightarrow1}\frac{1}{1-v}(B(x;1)-\frac{1}{v^3}B(vx;1/v))+x^2C(x)=(1-3x+x^2)C(x)-1+2x.$$
Hence, by the fact that $G(x;1)=xH(x)-x-x^2$, we have
\begin{align*}
A^+(x;1)&=xH(x)+\frac{1-x}{1-2x}B'(x;1)-\frac{x(1-x)^2}{1-2x}-\frac{x^4}{(1-2x)^2}\\
&+\frac{1-3x+x^2}{1-2x}C(x)-1.
\end{align*}
By Lemma \ref{lem194b} and the formula for $H(x)$, we complete the proof.
\end{proof}

\begin{theorem}\label{th194a}
Let $T=\{3124,4123,1243\}$. Then
\[
F_T(x)= \frac{  (1 - 5 x + 9 x^2 - 8 x^3 + 4 x^4)C(x) - (1 - 5 x + 9 x^2 - 6 x^3 + x^4)}{x (1 - 2 x)^2}
\]

\end{theorem}
\begin{proof}
By Propositions \ref{pro194a} and \ref{pro194b}, we have that
\begin{align*}
F_T(x)-1-x&=xC(x)\big(F_T(x)-1\big)\\
&+\frac{(x^4-2x^3+5x^2-4x+1)C(x)-x^3-2x^2+3x-1}{(1-2x)^2}.
\end{align*}
By solving for $F_T(x)$ we complete the proof.
\end{proof}

\subsection{Case 197: $\{2413,3241,2134\}$} Note that all three patterns contain 213.
\begin{theorem}\label{th197a}
Let $T=\{2413,3241,2134\}$. Then
$$F_T(x)=\frac{1-5x+9x^2-7x^3+x^4+(1-5x+9x^2-9x^3+3x^4)\sqrt{1-4x}}{(1-x)(1-6x+12x^2-11x^3+3x^4+(1-4x+6x^2-5x^3+x^4)\sqrt{1-4x})}.$$
\end{theorem}
\begin{proof}
Let $G_m(x)$ be the generating function for $T$-avoiders with $m$
left-right maxima. Clearly, $G_0(x)=1$ and $G_1(x)=xF_T(x)$.
Now suppose $m\geq2$ and $\pi=i_1\pi^{(1)}i_2\pi^{(2)}\cdots i_m\pi^{(m)}\in S_n(T)$ has $m$ left-right maxima.
We have $\pi^{(j)}=\emptyset$ for all $j=1,2,\ldots,m-2$ (or $i_{m-1}i_m$ the $34$ of a $2134$). We consider two cases:
\begin{itemize}
\item $\pi^{(m-1)}=\emptyset$. Here, $\pi$ has the form illustrated in Figure \ref{fig197c1},
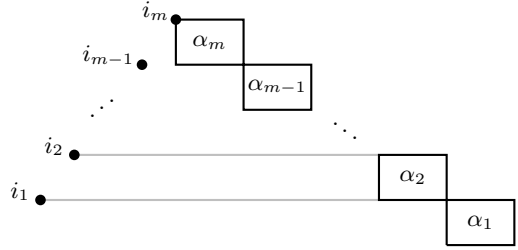
\begin{figure}[htp]
\begin{center}
\begin{pspicture}(0,-.4)(6,3.3)
\psset{xunit=.9cm}
\psset{yunit=.6cm}
\psline[linecolor=lightgray](5,1)(0,1)
\psline[linecolor=lightgray](.5,2)(5,2)
\qdisk(0,1){2pt}
\qdisk(.5,2){2pt}
\qdisk(1.5,4){2pt}
\qdisk(2,5){2pt}
\psline(2,5)(3,5)(3,3)(4,3)(4,4)(2,4)(2,5)
\psline(6,0)(7,0)(7,1)(5,1)(5,2)(6,2)(6,0)
\rput(2.5,4.5){\textrm{\small $\al_m$}}
\rput(3.5,3.5){\textrm{\small $\al_{m-1}$}}
\rput(5.5,1.5){\textrm{\small $\al_2$}}
\rput(6.5,0.5){\textrm{\small $\al_1$}}
\rput(.9,3.2){\textrm{\small $\iddots$}}
\rput(4.5,2.7){\textrm{\small $\ddots$}}
\rput(-.3,1.2){\textrm{\small $i_1$}}
\rput(.2,2.2){\textrm{\small $i_2$}}
\rput(1.0,4.2){\textrm{\small $i_{m-1}$}}
\rput(1.7,5.2){\textrm{\small $i_m$}}
\end{pspicture}
\caption{A $T$-avoider with $\pi^{(m-1)}=\emptyset$}\label{fig197c1}
\end{center}
\end{figure}
\noindent where $\al_{j+1}$ lies to the left of $\al_{j},\ j=1,2,\dots,m-1$ (or $i_j i_{j+1}$ is the 24 of a 2413).
If $\alpha_{m-1}=\cdots=\alpha_{1}=\emptyset$, then we have a contribution of $x^mF_T(x)$. Otherwise, let $j\in\{1,2,\dots,m-1\}$ be the minimal index such that $\alpha_{j}\neq\emptyset$.
Then $\al_1= \cdots=\al_{j-1}=\emptyset$ while $\al_j$
avoids $T$ and  $\al_{j+1}, \dots,\al_m$ all avoid $213$ (or $u\in \al_j$ is the ``1'' of a 3241).
So we have a contribution of $x^mC(x)^{m-j}\big(F_T(x)-1\big)$. Hence, the total contribution for this case is
$$x^mF_T(x)+x^m(F_T(x)-1)\sum_{j=1}^{m-1}C(x)^j.$$

\item $\pi^{(m-1)}\neq\emptyset$.
Since  $\pi^{(m-1)}$ is nonempty, there is a maximal $j\in\{1,2,\dots,m-1\}$ such that $\pi^{(m-1)}$ has a letter between $i_{j-1}$ and $i_j$ ($i_0:=0$) and then $\pi$ has the form illustrated in Figure \ref{fig197c2},
\begin{figure}[htp]
\begin{center}
\begin{pspicture}(0,-.4)(7.5,6.3)
\psset{xunit=1cm}
\psset{yunit=.7cm}
\psline[linecolor=lightgray](1.5,5)(7,5)(7,7)
\psline[linecolor=lightgray](3,4)(6,4)(6,7)
\psline[linecolor=lightgray](6,1)(6,4)(7,4)
\psline[linecolor=lightgray](2,7)(6,7)
\psline[linecolor=lightgray](.5,1)(5,1)
\qdisk(.5,1){2pt}
\qdisk(1,3){2pt}
\qdisk(1.5,5){2pt}
\qdisk(2,7){2pt}
\qdisk(6,8){2pt}
\qdisk(2.5,3.8){2pt}
\psline(2,4)(3,4)(3,2)(4,2)(4,3)(2,3)(2,4)
\psline(6,8)(7,8)(7,7)(6,7)(6,8)
\psline(7,5)(8,5)(8,4)(7,4)(7,5)
\psline(5,0)(6,0)(6,1)(5,1)(5,0)
\pspolygon[fillstyle=solid,fillcolor=lightgray](6,7)(8,7)(8,5)(6,5)(6,7)
\pspolygon[fillstyle=solid,fillcolor=lightgray](6,4)(8,4)(8,0)(6,0)(6,4)
\pspolygon[fillstyle=solid,fillcolor=lightgray](7,7)(8,7)(8,8)(7,8)(7,7)
\pspolygon[fillstyle=solid,fillcolor=lightgray](6,4)(6,5)(7,5)(7,4)(6,4)
\rput(6.5,7.5){\textrm{\small $\al'\searrow$}}
\rput(7.5,4.5){\textrm{\small $\al''$}}
\rput(2.5,3.4){\textrm{\small $\be_j$}}
\rput(3.5,2.5){\textrm{\small $\be_{j-1}$}}
\rput(5.5,0.5){\textrm{\small $\be_1$}}
\rput(7,2.3){\textrm{\small $i_j\be_ji_m\bullet$}}
\rput(7,1.7){\textrm{\small $3\ 2\ 4\ 1$}}
\rput(6.3,1.7){\textrm{\small $=$}}
\rput(7.1,6.3){\textrm{\small $i_ji_{m-1}\be_j\bullet$}}
\rput(7.1,5.7){\textrm{\small $2\quad 4\quad 1\  3$}}
\rput(6.2,5.7){\textrm{\small $=$}}
\rput(4.5,1.6){\textrm{\small $\ddots$}}
\rput(.2,1.2){\textrm{\small $i_1$}}
\rput(.6,3.2){\textrm{\small $i_{j-1}$}}
\rput(1.2,5.2){\textrm{\small $i_j$}}
\rput(1.5,7.2){\textrm{\small $i_{m-1}$}}
\rput(5.7,8.2){\textrm{\small $i_m$}}
\rput[b]{25}(.8,1.8){$\iddots$}
\rput[b]{25}(1.8,5.8){$\iddots$}
\end{pspicture}
\caption{A $T$-avoider with $\pi^{(m-1)}\ne\emptyset$}\label{fig197c2}
\end{center}
\end{figure}
\noindent where $\be_{k}$ lies to the left of $\be_{k-1},\ k=2,3,\dots,j$ (or $i_{k-1} i_k$ is the 24 of a 2413), $\al'$
lies to the left of $\al''$ (or $i_j i_m $ is the 24 of a 2413), $\al'$ is decreasing (or $i_ju$ with $u\in \be_j$ forms the 21 of a 2143), and shaded regions are empty for the reason indicated.

\begin{itemize}
\item If $\beta_1=\cdots=\beta_{j-1}=\emptyset$ and $\beta_j$ is increasing, then $\pi$ avoids $T$ if and only if $\alpha''$ avoids $T$. So, we have a contribution of
    \[
    \frac{x^{m+1}}{(1-x)^2}F_T(x)\,.
    \]
\item If $\beta_1=\cdots=\beta_{j-1}=\emptyset$ and $\beta_j$ is not increasing, then $\pi$ avoids $T$ if and only if $\alpha''$ is decreasing and $\beta_j$ avoids $213$. Thus, we have a contribution of
    \[
    \frac{x^m}{(1-x)^2}\left(C(x)-\frac{1}{1-x}\right)\,.
    \]
\item If $\beta_1\cdots\beta_{j-1}\neq\emptyset$ (and so $j\ge 2$), then $\pi$ avoids $T$ if and only if $\alpha''=\emptyset$ and $\beta_k$ avoids $213$ for all $k=1,2,\ldots,j$. Thus, we have a contribution of
    \[
    \frac{x^m}{1-x}\big(C(x)^{j-1}-1\big)\big(C(x)-1\big)\,.
    \]
\end{itemize}
Since $j$ can be any integer in $\{1,2,\dots,m-1\}$, the total contribution in case $\pi^{(m-1)}\neq\emptyset$ is given by
$$\frac{(m-1)x^{m+1}}{(1-x)^2}F_T(x)+\frac{(m-1)x^m}{(1-x)^2}\left(C(x)-\frac{1}{1-x}\right)
    +\frac{x^m\big(C(x)-1\big)}{1-x}\sum_{j=2}^{m-1}\big(C(x)^{j-1}-1\big).$$
\end{itemize}
We have now obtained an expression for $G_m(x)$ in terms of $F_T(x)$ and $C(x)$ for all $m$. The identity $F_T(x)=\sum_{m\ge 0}G_m(x)$ yields the following equation for $F_T(x)$,
$$F_T(x)=\frac{1-5x+9x^2-9x^3+3x^4+x^3(1-x)C(x)}{(1-x)^5}+\frac{xC(x)(1-3x+4x^2-x^3-x^3C(x))}{(1-x)^4}F_T(x),$$
with the stated solution.
\end{proof}

\subsection{Case 198: $\{1234,1423,2341\}$} We consider both left-right maxima and right-left maxima. Let $G_m(x)$ be the generating function for $T$-avoiders with $m$ left-right maxima. Since $T$ contains 1234, $G_m(x)=0$ for all $m\geq4$. Most of the work is in finding $G_2(x)$. For this purpose, it is convenient to introduce, for $d\ge 1$, the auxiliary generating function $H_d(x)$ that counts $T$-avoiders $\pi$ of the form $(n-1-d)n\pi'$. Set $H(x)=\sum_{d\ge 1}H_d(x)$. By deleting the second letter, namely $n$, we see that $H(x)/x=G_2(x)$. So our efforts will be directed toward finding an expression for $H(x)$.

In a $T$-avoider $(n-1-d)n\pi'$ counted by $H_d(x)$, the $d+1$ letters $n,n-1,\dots,n-d$ occur in that order (to avoid 1423) and hence are all right-left maxima, and we refine $H_d(x)$ to $H_{d,e}(x),\ e\ge 0$, where $d+1+e$ is the number of right-left maxima.
We also need for $d,e\ge 1$, $J_{d,e}(x)$, the generating function for $T$-avoiders $\pi=(n-1-d)n\pi'\in S_n(T)$ with $d+1+e$ right-left maxima and at least one letter between $n$ and $n-1$ smaller than the last letter of $\pi$, and $J_d(x)=\sum_{e\geq1}J_{d,e}(x)$.

\begin{lemma}\label{lem198a}
For $d\geq2$ and $e\geq1$,
$$H_{d,e}(x)=xH_{d-1,e}(x)+\sum_{j=0}^{e-1}\frac{x^2}{(1-x)^{j+1}}H_{d-1+j,e-j}(x)+J_{d,e}(x).$$
\end{lemma}
\begin{proof}
Let us write an equation for $H_{d,e}(x)$. Let $\pi=(n-1-d)n\alpha(n-1)\pi''\in S_n(T)$ where the $d+1+e$ right-left maxima of $\pi$ are $j_{e+1}=0<j_e<j_{e-1}<\cdots<j_1<j_0=n-d<\cdots<n-1<n$. If $\alpha$ is empty then we have a contribution of $xH_{d-1,e}(x)$. Thus, we can assume that $\alpha$ has a letter between $j_s$ and $j_{s-1}$, where $1\leq s\leq e+1$. Note that $\alpha$ is decreasing because $d\geq2$. For $1\leq s\leq e$, since $\pi$ avoids $T$, we see that there is no letter greater than $j_s$ on the right side of $n-1$. Thus, the number of permutations $\pi$ is the number of permutations of the form $(n-d-t)n\pi''$, for some $t$, that avoid $T$ and contain $d+e$ right-left maxima. Hence, for given $t$, we have a contribution $\frac{x^2}{(1-x)^t}H_{d-2+t,e+1-t}(x)$. If $s=e+1$, the contribution is $J_{d,e}(x)$. Summing all contributions,
$$H_{d,e}(x)=xH_{d-1,e}(x)+\sum_{j=0}^{e-1}\frac{x^2}{(1-x)^{j+1}}H_{d-1+j,e-j}(x)+J_{d,e}(x),$$
as required.
\end{proof}

\begin{lemma}\label{lem198b}
For $d\geq2$,
$$H_{d,0}(x)=x^{d+2}C(x)+\sum_{j=1}^{d-1}\frac{x^{d+2}}{(1-x)^j}\big(C(x)-1\big)\,.$$
\end{lemma}
\begin{proof}
For $d=1$, the avoiders are $(n-2)n\al (n-1)$ where $\al$ avoids 123,
and so $H_{1,0}(x)=x^3C(x)$.
For $d\geq2$, an avoider can be written as $\pi=(n-d-1)n \alpha^{(1)}(n-1)\cdots\alpha^{(d)}(n-d)$. The contribution when $\alpha^{(d)}=\emptyset$ is $xH_{d-1,0}(x)$. If $\alpha^{(d)}\neq\emptyset$, then $\alpha^{(1)}\cdots\alpha^{(d-1)}$ is decreasing, giving a contribution of $\frac{x^{d+2}}{(1-x)^{d-1}}(C(x)-1)$. Hence,
$$H_{d,0}(x)=xH_{d-1,0}(x)+\frac{x^{d+2}}{(1-x)^{d-1}}(C(x)-1)\,,$$
and the lemma follows by induction on $d$.
\end{proof}

\begin{lemma}\label{lem198c}
For $d,e\geq1$,
\begin{align*}
J_{d,e}(x)&=J'_{d,e}(x)+\sum_{j=1}^{d+e-1}\frac{x^{d+e+3}}{(1-x)^{j+e}}(C(x)-1),\textrm{\ where}\\
J'_{d,e}(x)&=J''_{d,e}(x)+\frac{x^{d+e+2}}{(1-x)^e}(C(x)-1),\textrm{\ and}\\
J''_{d,e}(x)&=\sum_{j=1}^e\frac{x^{d+e+4}}{(1-x)^{j}(1-2x)}C(x)\,.
\end{align*}
\end{lemma}
\begin{proof}
To write a recurrence for $J_{d,e}(x)$, suppose $\pi=(n-d-1)n \pi'(n-1)\pi''$ where $\pi'$ has a
letter smaller than the rightmost letter $p$ of $\pi''$. Then $\pi''$ has the form $\beta^{(2)}(n-2)\cdots\beta^{(d)}(n-d)\beta^{(d+1)}j_1\cdots\beta^{(d+e)}j_e$, where $(n-2)>\cdots>n-d>j_1>\cdots>j_e$
are the right-left maxima of $\pi''$. We denote the contribution of the case $\pi''>p$ by $J'_{d,e}(x)$.
Thus, we can assume that $\pi''$ has a letter smaller than $p$, so there exists $j$ such that
$\beta^{(j)}\neq\emptyset$ and $\beta^{(i)}=\emptyset$ for all $i=j+1,\ldots,d+e$.
Since $\pi$ avoids $T$, we see that $\pi'>\beta^{(2)}>\cdots>\beta^{(j)}$;
$\pi',\beta^{(2)},\ldots,\beta^{(j-1)}$ are decreasing; $\beta^{(j)}$ avoids $123$;
and $\pi'$ is partitioned into $e+1$ decreasing subsequences.
Thus, we have a contribution of $\frac{x^{d+e+3}}{(1-x)^{j+e}}(C(x)-1)$.
Adding the contributions, we obtain
\begin{align*}
J_{d,e}(x)&=J'_{d,e}(x)+\sum_{j=1}^{d+e-1}\frac{x^{d+e+3}}{(1-x)^{j+e}}(C(x)-1)\,.
\end{align*}

To find $J'_{d,e}(x)$, suppose $\pi=(n-d-1)n \pi'(n-1)\pi''$ where $\pi'$ has a letter smaller than the
rightmost letter $p$ of $\pi''$ and $\pi''>p$. If $\pi'$ can be decomposed as $\alpha\alpha'$ with $\alpha>p>\alpha'\neq\emptyset$, then we have a contribution of  $\frac{x^{d+e+2}}{(1-x)^e}\big(C(x)-1\big)$.
Otherwise, we denote the contribution by $J''_{d,e}(x)$. By adding over all contributions, we obtain
\begin{align*}
J'_{d,e}(x)&=J''_{d,e}(x)+\frac{x^{d+e+2}}{(1-x)^e}(C(x)-1).
\end{align*}

Lastly, let us write a recurrence for $J''_{d,e}(x)$. If $\pi$ is counted by $J''_{d,e}(x)$, then there exists $t$, $1\leq t\leq e$, such that $$\pi=(n-d-1)n\alpha^{(1)}\alpha^{(2)}\cdots\alpha^{(t)}\beta^{(1)}k_1\beta^{(2)}k_2\cdots\beta^{(r)}k_r\beta^{(r+1)}
(n-1)(n-2)\cdots(n-d)j_1j_2\cdots j_e,$$
where $\beta^{(1)}\neq\emptyset$; $\alpha^{(1)}\cdots\alpha^{(t)}k_1k_2\cdots k_r$ is decreasing with $j_{t-1}>\alpha^{(t)}k_1k_2\cdots k_r>j_t$ ($j_0=n-d-1$); $\beta^{(1)}\cdots\beta^{(r-1)}$ is decreasing;
and $\beta^{(r)}$ avoids 123. Hence,
\begin{align*}
J''_{d,e}(x)&=\sum_{t=1}^e\frac{x^{d+e+4}}{(1-x)^t(1-2x)}C(x),
\end{align*}
as claimed.
\end{proof}

Recall $H_d(x)=\sum_{e\geq0}H_{d,e}(x)$ and $J_d(x)=\sum_{e\geq1}J_{d,e}(x)$ and define
$H(x;v)=\sum_{d\geq1}H_d(x)v^d$, $J(x;v)=\sum_{d\geq1}J_d(x)v^d$ and $E(x;v)=\sum_{d\geq1}H_{d,0}(x)v^d$.

\begin{proposition}\label{pro198a}
We have
\begin{align*}
H_1(x)&=x^2\big(F_T(x)-1\big)\,; \\[2mm]
J_d(x)&=\frac{x^{d+4}C(x)}{(1-2x)^2}-\frac{x^{d+4}C(x)}{(1-x)(1-2x)}
+\frac{x^{d+4}C(x)^2}{(1-x)^{d-1}(1-3x+x^2)} \textrm{\quad for $d\geq1$}\,; \\[2mm]
J(x;v)&=\frac{vx^5C(x)}{(1-2x)^2(1-xv)}-\frac{vx^5C(x)}{(1-x)(1-xv)(1-2x)}+
\frac{vx^5(1-x)C(x)^2}{(1-3x+x^2)(1-x-vx)}\,; \\[2mm]
E(x;v)&=\frac{vx^3\big((1-x)C(x)-vx\big)}{(1-xv)(1-x-xv)}\,.
\end{align*}
\end{proposition}
\begin{proof}
By definition, $H_1(x)$ counts $T$-avoiders in $S_n$ whose first two letters are $(n-2)n$. Deleting these letters, we see that $H_1(x)=x^2\big(F_T(x)-1\big)$.
Lemma \ref{lem198c} gives
$$J_{d,e}(x)=\frac{x^{d+e+3}C(x)}{(1-x)^{e}(1-2x)}
-\frac{x^{d+e+3}C(x)}{1-2x}+\frac{x^{d+e+3}C(x)^2}{(1-x)^{2e+d-1}}\,,$$
and the second assertion follows by summing over $e\ge 1$.
The third assertion is a routine computation and the fourth follows from Lemma \ref{lem198b}.
\end{proof}


\begin{proposition}\label{pro198b}
$$H(x;1)=\frac{x(x^3-7x^2+5x-1)}{(1-x)(1-2x)}-\frac{x(3x^3-9x^2+6x-1)}{(1-x)^2}C(x)+x^2C(x)F_T(x).$$
\end{proposition}
\begin{proof}
By Lemma \ref{lem198a}, we have that for all $d\geq1$,
$$H_d(x)-H_{d,0}(x)=\frac{x}{1-x}\big(H_{d-1}(x)-H_{d-1,0}(x)\big)+J_d(x)+\frac{x^2}{(1-x)^2}
\sum_{j\geq d}\frac{1}{(1-x)^j}\big(H_j(x)-H_{j,0}(x)\big).$$
Multiplying by $v^d$ and summing over $d\geq2$, we obtain
\begin{equation}\label{eq198a1}
\begin{aligned}
H(x;v)&=H_1(x)v+E(x;v)-H_{1,0}(x)v+\frac{xv}{1-x}\big(H(x;v)-E(x;v)\big)+J(x;v)-J_1(x)v \\
&\quad +\frac{vx^2\left(-H(x;v)+v(1-x)H\big(x;\frac{1}{1-x}\big)+E(x;v)-v(1-x)E\big(x;\frac{1}{1-x}\big)\right)}
{(1-x)(1-v+xv)}\,.
\end{aligned}
\end{equation}
Taking $v=C(x)$, equation (\ref{eq198a1}) reduces to
\begin{align*}
H\big(x;\frac{1}{1-x}\big)&=H_1(x)C(x)^2+E\big(x;C(x)\big)-H_{1,0}(x)C(x)^2-\frac{xC(x)^2}{1-x}E\big(x;C(x)\big)\\
&\quad \ +J\big(x;C(x)\big)C(x)-J_1(x)C(x)^2\\
&\qquad -\frac{1}{(1-x)C(x)}\left(E\big(x;C(x)\big)-(1-x)E\big(x;\frac{1}{1-x}\big)C(x)\right) \\
&=x^2C(x)^2F_T(x)-\frac{2x^2-4x+1}{1-x}-\frac{x^5-16x^4+36x^3-28x^2+9x-1}{(1-2x)^2(1-x)}C(x)\,.
\end{align*}
By \eqref{eq198a1} with $v=1$, we have
\begin{align*}
H(x;1)&=H_1(x)+E(x;1)-H_{1,0}(x)+J(x;1)-J_1(x) 
+x\left(H\big(x;\frac{1}{1-x}\big)-E\big(x;\frac{1}{1-x}\big)\right).
\end{align*}
Using the expression for $H\big(x;\frac{1}{1-x}\big)$, $H_{1,0}(x)=x^3C(x)$ and Proposition \ref{pro198a}, we obtain
$$H(x;1)=\frac{x(x^3-7x^2+5x-1)}{(1-x)(1-2x)}-\frac{x(3x^3-9x^2+6x-1)}{(1-x)^2}C(x)+x^2C(x)F_T(x).$$
\end{proof}

\begin{theorem}\label{th198a}
We have
\begin{align*}
F_T(x)&=\frac{ (1 - 7 x + 18 x^2 - 19 x^3 + 6 x^4) C(x)- (1 - 6 x + 12 x^2 -
 8 x^3 + x^4)}{x^2(1 - x)(1 - 2 x)}\,.
\end{align*}
\end{theorem}
\begin{proof}
By Proposition \ref{pro198b}, we have $$G_2(x)=\frac{H(x)}{x}=\frac{H(x;1)}{x}=\frac{x^3-7x^2+5x-1}{(1-x)(1-2x)}-\frac{3x^3-9x^2+6x-1}{(1-x)^2}C(x)+xC(x)F_T(x).$$

Now let us write a formula for $G_3(x)$. Let $\pi=i_1\pi^{(1)}i_2\pi^{(2)}i_3\pi^{(3)}\in S_n(T)$ with exactly $3$ left-right maxima. Since $\pi$ avoids $2341$, we see that $\pi^{(3)}>i_1$. Since $\pi$ avoids $1423$, we see that $\pi^{(3)}$ is decreasing and the subsequence of the letters between $i_1$ and $i_2$ is decreasing. If $\pi^{(3)}>i_2$ then $\pi$ avoids $T$ if and only if $i_1\pi^{(1)}i_2\pi^{(i_2)}$ avoids $123$. Thus, we have a contribution of $\frac{x}{1-x}\big(C(x)-1-xC(x)\big)$, where $C(x)-1-xC(x)$ counts the number permutations in $\pi'\in S_n(\{123\})$ with $n\geq2$ such that the first letter is not $n$. If $\pi^{(3)}$ has a letter between $i_1$ and $i_2$, then $\pi^{(1)}$ is decreasing and $\pi^{(2)}$ can be decomposed as $\alpha^{(1)}k_1\alpha^{(2)}k_2\cdots\alpha^{(m)}k_m\alpha^{(m+1)}$ such that $\alpha^{(1)}>k_1>\cdots\alpha^{(m)}>k_m$, where $\alpha^{(1)},\ldots,\alpha^{(m)}$ are decreasing and $\alpha^{(m+1)}$ avoids $123$. Thus, we have a contribution of $\frac{x^4}{(1-x)^3}\frac{1-x}{1-2x}C(x)$. Hence,
$$G_3(x)=\frac{x}{1-x}\big(C(x)-1-xC(x)\big)+\frac{x^4}{(1-x)^3}\frac{1-x}{1-2x}C(x).$$
Since $G_0(x)=1,\,G_1(x)=xF_T(x)$, and $G_m(x)=0$ for all $m\geq4$, we have
\begin{align*}
F_T(x)&=1+xF_T(x)+\frac{x^3-7x^2+5x-1}{(1-x)(1-2x)}-\frac{3x^3-9x^2+6x-1}{(1-x)^2}C(x)+xC(x)F_T(x)\\
&+\frac{x}{1-x}\big(C(x)-1-xC(x)\big)+\frac{x^4}{(1-x)^3}\frac{1-x}{1-2x}C(x)\,,
\end{align*}
and solving for $F_T(x)$ completes the proof.
\end{proof}

\subsection{Case 199: $\{1243,1423,2341\}$}
For this case, we count by the first few entries in an avoider, turn the resulting recurrences into an equation for the \gf, and then solve the equation using the kernel method.

Define
$a(n)=|S_n(T)|$ and $a(n;i_1,i_2,\dots, i_m)=|\{\pi \in S_n(T):$ the first $m$ entries in $\pi$ are $i_1,i_2,\dots, i_m\}|$ in that order. Then, $a(n;1)=|S_{n-1}(\{132,312,2341\})|=\binom{n-1}{2}+1$. Set $b(n;i)=a(n;i,n-1)$ and $b'(n;i)=a(n;i,n)$.

It is convenient to define the sequence $w_{n,j}$ via the recurrence relation $w_{n,j}=w_{n-1,j}+w_{n-1,j-1}+1$, for $n\geq1$ and $0\leq j\leq n$, with the initial condition $w_{0,j}=\delta_{j=0}$. Define $W_n(v)=\sum_{i=0}^nw_{n,i}v^i$ and $W(x,v)=\sum_{n\geq0}W_n(v)x^n$. Then
$$W_n(v)=(1+v)W_{n-1}(v)+1+v+\cdots+v^n=(1+v)W_{n-1}(v)+\frac{1-v^{n+1}}{1-v}$$
with $W_0(v)=1$. Thus, $W(x,v)-1=(1+v)xW(x,v)+\frac{x(1+v)-x^2v}{(1-x)(1-xv)}$, which implies
\begin{align}
W(x,v)=\frac{1}{(1-x)(1-xv)(1-(1+v)x)}.\label{eq199aW1}
\end{align}

Since $\pi$ avoids $T$, we have the following recurrences.
\begin{lemma}\label{lem199a}
We have $a(n;1)=\binom{n-1}{2}+1$ and
\begin{align*}
a(n;i,j)&=a(n-1;i,j),\qquad\mbox{if $1\leq i<j\leq n-2$},\\
a(n;n-1,i)&=a(n-1;i),\qquad\mbox{if $1\leq i\leq n-2$},\\
a(n;i,i-1)&=a(n-1,i-1),\qquad\mbox{$2\leq i\leq n$},\\
a(n;i,i+1)&=a(i-1),\qquad\mbox{$1\leq i\leq n-1$},\\
a(n;i,j)+w_{i-3,j-1}&=a(n;i+1,j),\qquad\mbox{$1\leq j<i-1<n-1$},\\
a(n;n-1,n)&=a(n-2),\\
a(n;n;i)&=a(n-1,i),\qquad\mbox{if $1\leq i\leq n-1$}.
\end{align*}

Also, $b'(n;i)=\sum_{j=1}^ib'(n-1;j)$ for all $i=1,2,\ldots,n-3$ while $b'(n;n-1)=b'(n;n-2)=a(n-2)$ and $b'(n;n)=0$.

Lastly, $b(n;i)=b'(n-1;i)+\binom{n-3}{i}$ for all $i=1,2,\ldots,n-1$. \qed
\end{lemma}

Define $a'(n;i)=\sum_{j=1}^{i-1}a(n;i,j)$ and $a''(n;i)=\sum_{j=i+1}^na(n;i,j)$. Define $A'_n(v)=\sum_{i=1}^na'(n;i)v^{i-1}$, $A''_n(v)=\sum_{i=1}^na''(n;i)v^{i-1}$, and $A_n(v)=\sum_{i=1}^na(n;i)v^{i-1}$. Define
$A(x,v)=1+x+\sum_{n\geq2}A_n(v)x^n$, $A'(x,v)=\sum_{n\geq2}A'_n(v)x^n$ and $A''(x,v)=\sum_{n\geq2}A''_n(v)x^n$.

\begin{lemma}\label{lem199b}
We have
\begin{align*}
A'(x,v)=\frac{vx}{1-v}A(x,v)-\frac{xv}{1-v}A(vx,1)+\frac{(vx^2+vx-1)v^2x^5}{(1-2vx)(1-(1+v)x)(1-vx)^2(1-x)^2}.
\end{align*}
\end{lemma}
\begin{proof}
Lemma \ref{lem199a} gives $a(n;i,j)=a(n;i+1,j)-w_{i-3,j-1}$, which leads to
$$a(n;i,j)=a(n;n-1,j)-w_{n-5,j-1}-\cdots-w_{i-2,j-1}-w_{i-3;j-1}.$$
Thus, since $a(n;n-1,j)=a(n-1;j)$, we have that for all $1\leq j<i-1<n-1$,
\begin{align}
a(n;i,j)=a(n-1;j)-w_{n-5,j-1}-\cdots-w_{i-2,j-1}-w_{i-3;j-1}.\label{eq199bw1}
\end{align}

Then Lemma \ref{lem199a} and \eqref{eq199bw1} yield $a'(n;1)=0$, $a'(n;2)=a(n-1;1)$, $a'(n;n-1)=a(n-1)-a(n-2)$, $a'(n;n)=a(n-1)$ and
\begin{align*}
a'(n;i+1)-a'(n;i)
&=a(n-1;i)+a(n;i+1,i-1)-a(n;i,i-1)+\sum_{j=1}^{i-2}a(n;i+1,j)-a(n;i,j)\\
&=a(n-1;i)-\sum_{i'=i-2}^{n-5}w_{i',i-2}+\sum_{j=0}^{i-2}w_{i-3,j},
\end{align*}
so for all $i=2,3,\ldots,n-2$,
\begin{align*}
a'(n;i+1)=a'(n;i)+a(n-1;i)-\sum_{i'=i-2}^{n-5}w_{i',i-2}+\sum_{j=0}^{i-3}w_{i-3,j}
\end{align*}
with $a'(n;1)=0$, $a'(n;2)=a(n-1;1)$, $a'(n;n-1)=a(n-1)-a(n-2)$ and $a'(n;n)=a(n-1)$.
Multiplying by $v^i$ and summing over $i=2,3,\ldots,n-2$, we have
\begin{align*}
(1-v)A'_n(v)&=vA_{n-1}(v)-A_{n-1}(1)v^n-v^2\sum_{j=0}^{n-5}W_j(v)+\sum_{j=0}^{n-5}W_j(1)v^{j+3}.
\end{align*}
Multiplying by $x^n$, summing over $n\geq2$, and using $A'_1(v)=0$, we obtain
\begin{align*}
A'(x,v)=\frac{vx}{1-v}A(x,v)-\frac{xv}{1-v}A(vx,1)-\frac{x^5v^2}{(1-x)(1-v)}W(x,v)+\frac{v^3x^5}{(1-x)(1-v)}W(vx,1).
\end{align*}
Substituting for $W$ from \eqref{eq199aW1}, we complete the proof.
\end{proof}

Now define $B(x,v)=\sum_{n\geq2}B_n(v)x^n$ and $B'(x,v)=\sum_{n\geq2}B'_n(v)x^n$.
\begin{lemma}\label{lem199c}
\begin{align*}
A''(x,v)&=\frac{1}{1-x}B(x,v)+B'(x,v),\\
B'(x,v)&=x^2A(vx,1)+\frac{x^2}{v}(A(vx,1)-1)+\frac{x}{1-v}(B'(x,v)-\frac{1}{v^2}B'(vx,1)),\\
B(x,v)&=xB'(x,v)+\frac{x^4}{(1-(1+v)x)(1-x)}.
\end{align*}
\end{lemma}
\begin{proof}
By Lemma \ref{lem199a}, we have
\begin{align*}
a''(n;i)&=\sum_{j=i+1}^na(n;i,j)=b(n;i)+b'(n;i)+\sum_{j=i+1}^{n-2}a(n-1;i,j),\\
a''(n-1;i)&=\sum_{j=i+1}^{n-1}a(n-1;i,j)=b'(n-1;i)+\sum_{j=i+1}^{n-2}a(n-1;i,j),
\end{align*}
which leads to
$$a''(n,i)-a''(n-1,i)=b(n,i)+b'(n;i)-b'(n-1,i),$$
where
\begin{align*}
b'(n;i)&=\sum_{j=1}^ib'(n-1;j),\\
b(n;i)&=b'(n-1;i)+\binom{n-3}{i} ,
\end{align*}
for $1\le i \le n-3 $, while $b'(n;n-1)=b'(n;n-2)=a(n-2)$ and $b'(n;n)=0$.

Multiplying the above recurrences by $v^{i-1}$ and summing over $i$, we obtain
\begin{align*}
A''_n(v)&=A''_{n-1}(v)+B_n(v)+B'_n(v)-B'_{n-1}(v),\\
B'_n(v)&=A_{n-2}(1)v^{n-2}+A_{n-2}(1)v^{n-3}+\frac{1}{1-v}(B'_{n-1}(v)-v^{n-3}B'_{n-1}(1)),\\
B_n(v)&=B'_{n-1}(v)+\frac{(1+v)^{n-3}-1}{v}.
\end{align*}
with $A''_1(v)=0$, $A''_2(v)=1$, $B'_2(v)=1$ and $B_2(v)=0$.
Multiply by $x^n$ and sum over $n\geq3$ to complete the proof.
\end{proof}

From Lemmas \ref{lem199b} and \ref{lem199c} and the identity $A(x,v)=A'(x,v)+A''(x,v)+1+x$, we have
\begin{align}
&\left(1-\frac{x}{1-v}\right)A(\frac{x}{v},v)=1+\frac{x}{v}-\frac{x}{1-v}A(x,1)+\frac{(x^2+vx-v)x^5}{v(1-2x)(v-(1+v)x)(1-x)^2(v-x)^2}\label{eqf199a}\\
&\qquad\qquad\qquad\qquad\qquad+\frac{v}{v-x}B'(x/v,v)+\frac{x^4}{v(v-(1+v)x)(v-x)^2},\notag\\
&\left(1-\frac{x}{v(1-v)}\right)B'(\frac{x}{v},v)=\frac{x^2}{v^2}A(x,1)+\frac{x^2}{v^3}(A(x,1)-1)-\frac{x}{v^3(1-v)}B'(x,1).\label{eqf199b}
\end{align}
By substituting $v=1-x$ into \eqref{eqf199a} and \eqref{eqf199b} and solving, for $B'(x/(1-x),1-x)$ and $B'(x,1)$, we obtain
\begin{align}\label{eqf199c}
B'(x,1)=\frac{x(x-1)^2(2x-1)(x^2-x+1)A(x,1)+x(3x^4-7x^3+9x^2-5x+1)}{(x-1)^2(2x-1)}.
\end{align}
By substituting $v=\frac{1}{C(x)}$ into \eqref{eqf199b}, we obtain
\begin{align}\label{eqf199d}
B'(x,1)=x^2A(x,1)(1+C(x))-x^2C(x).
\end{align}
Comparing \eqref{eqf199c} and \eqref{eqf199d} and solving for $A(x,1)$, we obtain the following result.
\begin{theorem}\label{th199a}
Let $T=\{1243,1423,2341\}$. Then
$$F_T(x)=\frac{x(x-1)^2(2x-1)C(x)+3x^4-7x^3+9x^2-5x+1}{(xC(x)-(x-1)^2)(x-1)^2(2x-1)}.$$
\end{theorem}

\subsection{Case 204: $\{1243,1423,2314\}$}
As in Case 198, $T$ contains the pattern 1423 and, again, $H_{d,e}(x)$ denotes the generating function for  permutations $\pi=(n-1-d)n\pi'\in S_n(T)$ with $d+1+e$ right-left maxima.

\begin{lemma}\label{lem204a}
For $d\geq2$ and $e\geq1$,
\begin{equation}\label{eq204H}
H_{d,e}(x)=xH_{d-1,e}(x)+\sum_{j=0}^e\frac{x^2}{(1-x)^{j+1}}H_{d-1+j,e-j}(x)\,.
\end{equation}
\end{lemma}
\begin{proof}
Let us write an equation for $H_{d,e}(x)$. Let $\pi=(n-1-d)n\alpha(n-1)\pi''\in S_n(T)$ where the $d+1+e$
right-left maxima of $\pi$ are $j_{e+1}=0<j_e<j_{e-1}<\cdots<j_1<j_0=n-d<\cdots<n-1<n$. If $\alpha$ is
empty then we have a contribution of $xH_{d-1,e}(x)$. Thus, we can assume that $\alpha$ has a letter
between $j_s$ and $j_{s-1}$, where $1\leq s\leq e+1$. Note that $\alpha$ is decreasing because $d\geq2$.
Since $\pi$ avoids $T$, we see that there is no letter greater than $j_s$ on the right side of $n-1$.
Thus, the number of permutations $\pi$ is the same as the number of permutations $(n-d-t)n\pi''$ with $t\ge1$
that avoid $T$ and contain  $d+e$ right-left maxima. Hence, for fixed $t\ge 1$, we have a contribution
$x^2/(1-x)^t H_{d-2+t,e+1-t}(x)$. The result follows by adding contributions.
\end{proof}

Let $J(x)$ denote the generating function for $\{231,1423,1243\}$-avoiders.
By the main result of \cite{MV},
\[
J(x)=\frac{1-3x+3x^2}{(1-x)^2(1-2x)}\,.
\]
As in Case 198, set $H_d(x)=\sum_{e\ge 0}H_{d,e}(x)$ and $H(x;v)=\sum_{d\geq1}H_d(x)v^d$.
\begin{lemma}\label{lem204b}
We have that $H_1(x)=x^2\big(F_T(x)-1\big)$ and $H_2(x)=x^3C(x)\big(F_T(x)-1\big)$.
Also, $H(x;1)=x^2C(x)\big(F_T(x)-1\big)$, and
$H\big(x;\frac{x}{1-x}\big)=x^2C(x)^2\big(F_T(x)-1\big)$.
\end{lemma}
\begin{proof}
Suppose $\pi=(n-1-d)n\pi'\in S_n(T)$ has $d+1$ right-left-maxima, hence counted by $H_{d,0}(x)$.
Then $\pi$ decomposes as
$$\pi=(n-1-d)n\alpha^{(1)}(n-1)\cdots\alpha^{(d)}(n-d),$$
where $\alpha^{(1)}>\cdots>\alpha^{(d)}$, \ $\alpha^{(1)}\cdots\alpha^{(d-1)}$ is decreasing, and $\alpha^{(d)}$ avoids $\{231,1423,1243\}$. Hence,
\begin{align}\label{eq204a1}
H_{d,0}(x)=\frac{x^{d+2}}{(1-x)^{d-1}}J(x)\,.
\end{align}

Summing (\ref{eq204H}) over $e$, we obtain for $d\geq2$,
$$H_d(x)-H_{d,0}(x)=\frac{x}{1-x}H_{d-1}(x)-\frac{x}{1-x}H_{d-1,0}(x)+
\sum_{j\geq0}\frac{x^2}{(1-x)^{j+2}}H_{d+j}(x)\,.$$
By \eqref{eq204a1}, we see that $H_{d,0}(x)=\frac{x}{1-x}H_{d-1,0}(x)$.
Thus, for $d\geq2$, $$H_d(x)=\frac{x}{1-x}H_{d-1}(x)+\sum_{j\geq0}\frac{x^2}{(1-x)^{j+2}}H_{d+j}(x)\,.$$
By looking at the difference $\frac{1}{1-x}H_d(x)-H_{d-1}(x)$, we obtain
$$H_d(x)=H_{d-1}(x)-xH_{d-2}(x),\quad d\geq3\,.$$
Multiplying by $v^d$ and summing over $d\geq3$, we have
$$H(x;v)-H_2(x)v^2-H_1(x)v=v(H(x;v)-H_1(x)v)-xv^2H(x;v)\,,$$
which implies
\begin{align}\label{eq204a2}
(1-v+xv^2)H(x;v)=H_2(x)v^2+H_1(x)v(1-v)\,.
\end{align}
On the other hand, $H_1(x)=x^2\big(F_T(x)-1\big)$ and, from \eqref{eq204a2} with $v=C(x)$,
$$H_2(x)=xC(x)H_1(x)=x^3C(x)\big(F_T(x)-1\big)\,.$$
Hence, by \eqref{eq204a2} with $v=1$, we have $H(x;1)=x^2C(x)(F_T(x)-1)$.
Now, by substituting $v=1/(1-x)$ into \eqref{eq204a2}, we derive that $H\big(x;\frac{x}{1-x}\big)=C(x)^2H_1(x)=x^2C(x)^2\big(F_T(x)-1\big)$.
\end{proof}

\begin{lemma}\label{lem204c}
We have
$$G_2(x)=xC(x)(F_T(x)-1).$$
\end{lemma}
\begin{proof}
Refine $K(x):=G_2(x)$ to $K_{d,p}(x)$, $d\ge 0,\,p\ge 0$, the generating function for permutations $(n-1-d)\pi'n\pi''\in S_n(T)$ where $\pi'$ has $p$ letters, and set $K_d(x)=\sum_{p\geq0}K_{d,p}(x)$.
Clearly, $K_0(x)=x\big(F_T(x)-1\big)$.
Now suppose $d\ge 1$.  Then, since we avoid 1243, the letters of $\pi'$ are decreasing.
By considering the position of $n-2-d$, we see that, for $p\ge 1$,
$$K_{d,p}(x)=xK_{d,p-1}(x)+K_{d+1,p}(x)\,.$$
Summing over $p\geq1$,
$$K_d(x)-K_{d,0}(x)=xK_d(x)+K_{d+1}(x)-K_{d+1,0}(x)\,.$$
Since $K_{d,0}(x)=H_d(x)$, we have for $d\geq1$,
\begin{align}\label{eq204a3}
K_d(x)=\frac{1}{1-x}K_{d+1}(x)+\frac{1}{1-x}H_d(x)-\frac{1}{1-x}H_{d+1}(x)\,.
\end{align}
Summing over $d\geq1$,
$$K(x)-K_0(x)=\frac{1}{1-x}\big(K(x)-K_0(x)-K_1(x)\big)+\frac{1}{1-x}H(x;1)-\frac{1}{1-x}\big(H(x;1)-H_1(x)\big)\,,$$
and hence, using Lemma \ref{lem204b},
$$K(x)=K_0(x)+\frac{1}{x}K_1(x)-x\big(F_T(x)-1\big)\,.$$
Since $K_0(x)=x\big(F_T(x)-1\big)$, the last equation implies
\[
K(x)=\frac{1}{x}K_1(x)\,.
\]

On other hand, by iterating \eqref{eq204a3}, we have that for all $d\geq1$,
$$K_d(x)=\sum_{j\geq0}\frac{H_{d+j}(x)-H_{d+1+j}(x)}{(1-x)^{j+1}}\,.$$
Using this equation with $d=1$ and recalling that $H(x;v)=\sum_{d\geq1}H_d(x)v^d$, we have
$K_1(x)=H\big(x;\frac{1}{1-x}\big)-(1-x) \big(H(x;\frac{1}{1-x})- \frac{H_1(x)}{1-x} \big)$.
Now Lemma \ref{lem204b} implies $K_1(x)=x^2C(x)\big(F_T(x)-1\big)$.
Since $G_2(x)=K(x)=\frac{1}{x}K_1(x)$, the result follows.
\end{proof}

\begin{lemma}\label{lem204d}
We have
$$G_3(x)=x^2\big(F_T(x)-1\big)+\frac{x}{1-x}\big(G_2(x)-x(F_T(x)-1)\big)+\frac{x^3}{(1-x)^2}\big(G_2(x)+x\big)\,.$$
\end{lemma}
\begin{proof}
Define $M(x)$ be the generating function for permutations $\pi=(n-1)\pi'n\pi''\in S_n(T)$ such that $\pi'$ is decreasing. Let us write an equation for $M(x)$. The contribution for the case $\pi'=\emptyset$ is $x^2F_T(x)$. By looking at the maximal element of $\pi'$, we see that $\pi'n\pi''$ is counted by $M(x)+K_1(x)+K_2(x)+K_3(x)+\cdots$, where $K_d(x)$ is as in Lemma \ref{lem204c}, and Lemma \ref{lem204c} implies
$$M(x)=\frac{x^2F_T(x)+x(G_2(x)-x(F_T(x)-1))}{1-x}=\frac{x}{1-x}(G_2(x)+x)\,.$$

Now, let us write a formula for $G_3(x)$. Let $\pi=i_1\pi^{(1)}i_2\pi^{(2)}i_3\pi^{(3)}\in S_n(T)$.
Then $\pi^{(2)}>i_1$ (2314) and $\pi^{(3)}<i_2$ (1243). Considering the three cases,
(i) $\pi^{(2)}=\emptyset$ and $\pi^{(3)}<i_1$, (ii) $\pi^{(2)}\neq\emptyset$ (decreasing) and $\pi^{(3)}<i_1$, (iii) $\pi^{(3)}$ has a letter between $i_1$ and $i_2$, we get the respective contributions (i) $x^2(F_T(x)-1)$, (ii) $\frac{x}{1-x}(G_2(x)-x(F_T(x)-1))$, (iii) $\frac{x^2}{1-x}M(x)$.
\end{proof}

\begin{theorem}\label{th204a}
Let $T=\{1243,1423,2314\}$. We have
$$F_T(x)=\frac{x(2x^2-2x+1)C(x)-(3x^2-3x+1)}{x(2x^2-2x+1)C(x)-(1-x)(3x^2-3x+1)}\,.$$
\end{theorem}
\begin{proof}
Let us write an equation for $G_m(x)$ with $m\geq4$. Let $\pi=i_1\pi^{(1)}\cdots i_m\pi^{(m)}\in S_n(T)$
with $m\ge 4$ left-right maxima. Then $\pi^{(m)}<i_2$, $\pi^{(j)}=\emptyset$ for all $j=3,4,\ldots,m-1$, and $\pi^{(2)}>i_1$. Hence, $G_m(x)=xG_{m-1}(x)$. Summing over $m\geq4$, we obtain
$$F_T(x)-1-G_1(x)-G_2(x)-G_3(x)=x(F_T(x)-1-G_1(x)-G_2(x)).$$
Using Lemmas \ref{lem204c} and \ref{lem204d},
\begin{align*}
F_T(x)&=1+xF_T(x)+xC(x)\big(F_T(x)-1\big)+x^2\big(F_T(x)-1\big)\\
&+\frac{x}{1-x}\Big(xC(x)\big(F_T(x)-1\big)-x\big(F_T(x)-1\big)\Big)\\
&+\frac{x^2}{(1-x)^2}\Big(x^2C(x)\big(F_T(x)-1\big)+x^2\Big)+x\Big(-xF_T(x)-xC(x)\big(F_T(x)-1\big)+F_T(x)-1\Big)\,.
\end{align*}
Solving for $F_T(x)$ completes the proof.
\end{proof}

\subsection{Case 208: $\{1234,1342,3124\}$} We consider the left-right minima, left factors that are decreasing (and hence left-right minima), and the position of $n$ relative to the left-right minima.

Define $f_{m,p}(x),\ 0\leq p\leq m-1$, to be the generating function for  permutations $\pi=\pi_1\pi_2\cdots\pi_n\in S_n(T)$ with $m$ left-right minima such that the first $p+1$ letters are all left-right minima, that is, $\pi_1>\pi_2>\cdots>\pi_p>\pi_{p+1}$.
Say the left-right minima are $\pi_1,\pi_2,\cdots,\pi_{p+1}=a_{p+1},a_{p+2},\ldots,a_m$.
Refine $f_{m,p}(x)$ to $f_{m,p;i}(x)$, $1\le i \le m-p$, to count avoiders where $n$ appears between the letters $a_{p+i}$ and $a_{p+i+1}$.
For convenience, set $f_{m,-1}(x)=f_{m,0}(x)$ and $f_{m,-1;i}(x)=f_{m,0;i}(x)$ for all $m,i$.

\begin{lemma}\label{lem208a1}
For  $0\leq p\leq m-1$\,,
$$f_{m,p}(x)=xf_{m-1,p-1}(x)+\sum_{i=1}^{m-p}f_{m,p;i}(x),$$
and for $1\leq i\leq m-p$,
\begin{align*}
f_{m,p;i}(x)&=xf_{m,p+i-1}+\frac{x^{p+1}}{(1-x)^2}\sum_{j=-1}^{i-2}x^{i-j}f_{m-i+j+1-p;j}(x)\\
&+\frac{x^{3+p}}{(1-x)^3}\sum_{k=0}^{i-1}\sum_{j=k}^{m-i-1+k-p}\frac{x^{i-k}}{(1-x)^{j-k}}f_{m-i+k-p,j}(x)\,.
\end{align*}
\end{lemma}
\begin{proof}
Suppose $\pi\in S_n(T)$ is counted by $f_{m,p}(x)$.
If $\pi_1=n$, the contribution is $xf_{m-1,p-1}(x)$.
Otherwise, the letter $n$ appears in $\pi$ between
$a_{p+i}$ and $a_{p+i+1}$ for some $i\in [\kern .1em1, m-p\kern .2em]$. Thus
$$f_{m,p}(x)=xf_{m-1,p-1}(x)+\sum_{i=1}^{m-p}f_{m,p;i}(x).$$

Now, suppose $\pi=\pi_1\pi_2\cdots\pi_{p+1}\pi^{(1)}a_{p+2}\pi^{(2)}\cdots a_m\pi^{(m-p)}\in S_n(T)$ with
$\pi^{(i)}=\alpha^{(i)}n\beta^{(i)}$ is counted by $f_{m,p;i}(x)$.
Our decomposition is based on looking at the letters that are greater than $\pi_1$ in $\pi$.
Set $\pi'=\pi^{(1)}\cdots\pi^{(i)}$ and $\pi''=\pi^{(i+1)}\cdots\pi^{(m-p)}$.
If $\pi'=\emptyset$, the contribution is $xf_{m,p+i-1}(x)$.
If $\pi'\neq\emptyset$, then, since $\pi$ avoids $1234$ we have that $\pi'$ and $\beta^{(i)}\pi''$ are decreasing, and since $\pi$ avoids $3124$, there exists $r\in[\kern .1em 1,i \kern .2em]$ such that $\pi^{(r)}\neq\emptyset$.
We consider two cases:
\begin{itemize}
\item $\pi''<\pi_1$. In this case, we see that $\pi^{(j)}$ has no letters between $a_{p+r}$ and $\pi_1$. Thus, we have a contribution of $\frac{x^{2+r+p}}{(1-x)^2}f_{m-r-p,i-1-r}(x)$.
\item there exists $r'\in[i+1, m-p]$ such that $\pi^{(r')}\neq\emptyset$ and $\pi^{(r'+1)}\cdots\pi^{(m-p)}=\emptyset$. Here we see that $\pi^{(1)}\cdots\pi^{(r'-1)}$ has no letter between $\pi_1$ and $a_{p+r'}$. Moreover, $\pi^{(r')}\cdots\pi^{(m-p)}$ has no letter between $\pi_1$ and $a_{p+r}$. Hence, we have a contribution of
    $\frac{x^{3+r+p}}{(1-x)^{r'-1}}f_{m-r-p,r'-1-r}(x)$.
\end{itemize}
Add all the contributions.
\end{proof}

By Lemma \ref{lem208a1}, we have for  $0\leq p\leq m-1$,
\begin{equation}\label{eq208rec}
\begin{aligned}
f_{m,p}(x)&=xf_{m-1,p-1}(x)+x\sum_{i=1}^{m-p}f_{m,p+i-1}+\frac{x^{p+1}}{(1-x)^2}\sum_{i=1}^{m-p}\sum_{j=-1}^{i-2}x^{i-j}f_{m-i+j+1-p;j}(x)\\
&+\frac{x^{3+p}}{(1-x)^3}\sum_{i=1}^{m-p}\sum_{k=0}^{i-1}\sum_{j=k}^{m-i-1+k-p}\frac{x^{i-k}}{(1-x)^{j-k}}f_{m-i+k-p,j}(x).
\end{aligned}
\end{equation}

To solve (\ref{eq208rec}), we define $F_m(x;v)=\sum_{p=0}^{m-1}f_{m,p}(x)v^p$.  Then (\ref{eq208rec}) can be written as
\begin{align*}
f_{m,p}(x)&=xf_{m-1,p-1}(x)+x\sum_{i=p}^{m-1}f_{m,i}(x)\\
&+\frac{x^{p+3}}{(1-x)^2}\left(\sum_{i=0}^{m-1-p}x^iF_{m-1-p-i}(x;0)+\sum_{i=0}^{m-2-p}x^iF_{m-1-p-i}(x;1)\right)\\
&+\frac{x^{p+3}}{(1-x)^2}\left(\sum_{i=0}^{m-2-p}\frac{x^i}{1-x}F_{m-1-p-i}\left(x;\frac{1}{1-x}\right)-\sum_{i=0}^{m-2-p}x^iF_{m-1-p-i}(x;1)\right).
\end{align*}
Multiplying by $v^p$ and summing over $p=0,1,\ldots,m-1$, we obtain for $m\geq2$,
\begin{align*}
F_m(x;v)&=x\big(F_{m-1}(x;0)+vF_{m-1}(x;v)\big)+\frac{x}{1-v}\big(F_m(x;1)-vF_m(x;v)\big)\\
&+\frac{x^2}{(1-x)^2}\left(\sum_{j=1}^m\frac{x^j(1-v^j)}{1-v}F_{m-j}(x;0)+\sum_{j=1}^{m-1}\frac{x^j(1-v^j)}{1-v}F_{m-j}(x;1)\right)\\
&+\frac{x^2}{(1-x)^2}\left(\sum_{j=1}^{m-1}\frac{x^j(1-v^j)}{(1-x)(1-v)}F_{m-j}\left(x;\frac{1}{1-x}\right)-\sum_{j=1}^{m-1}\frac{x^j(1-v^j)}{1-v}F_{m-j}(x;1)\right),
\end{align*}
where $F_0(x;v)=1$ and $F_1(x;v)=xK=x\sum_{n\geq0}|S_n(\{123,231\})|x^n=x\left(\frac{1}{1-x}+\frac{x^2}{(1-x)^3}\right)$.
Now, we define $F(x,u;v)$ to be the generating function for the sequence $F_m(x;v)$, that is, $F(x,u;v)=\sum_{m\geq0}F_m(x;v)y^m$.
Multiplying the displayed recurrence by $u^m$ and summing over $m\geq2$,
\begin{align*}
F(x,u;v)&=1+xuK+xu\big(F(x,u;0)-1+vF(x,u;v)-v\big)\\
&+\frac{x}{1-v}\big(F(x,u;1)-vF(x,u;v)\big)-x-ux^2K\\
&+\frac{ux^3}{(1-x)^2}\left(\frac{F(x,u;0)+F(x,u;1)-1}{(1-ux)(1-uvx)}-1\right)\\
&+\frac{ux^3}{(1-x)^2(1-ux)(1-vux)}\left(\frac{1}{1-x}F\Big(x,u;\frac{1}{1-x}\Big)-\frac{1}{1-x}-
F(x,u;1)+1\right),
\end{align*}
which implies
\begin{align*}
&\left(1-vux-\frac{xv}{1-v}\right)F(x,u;v)\\
&=1-x+x(1-x)uK+xuF(x,u;0)-xu-vux+\frac{x}{1-v}F(x,u;1)\\
&+\frac{ux^3}{(1-x)^2}\left(\frac{F(x,u;0)+F(x,u;1)-1}{(1-ux)(1-uvx)}-1\right)\\
&+\frac{ux^3}{(1-x)^2(1-ux)(1-vux)}\left(\frac{1}{1-x}F\Big(x,u;\frac{1}{1-x}\Big)-\frac{x}{1-x}-
F(x,u;1)\right).
\end{align*}
Substituting (i) $v=0,u=1$, (ii) $v=u=1$, (iii) $v=C(x),u=1$ yields three equations with unknowns
$F(x,1;1)$, $F(x,1;0)$ and $F(x,1;1/(1-x))$. Solving this system, we obtain the following result.
\begin{theorem}\label{th208a}
Let $T=\{1234,1342,3124\}$. Then
\[
F_T(x)=\frac{(1 - 2 x) (1 - 6 x + 12 x^2 - 10 x^3 + 2 x^4) -
x^2 (1 - 2 x + 2 x^2)^2 C(x)}{1 - 9 x + 30 x^2 - 49 x^3 + 38 x^4 - 8 x^5 - 4 x^6}
\]
\end{theorem}

\subsection{Case 214: $\{1342,2341,3412\}$}
We treat left-right maxima and let $G_m(x)$ be the generating function for $T$-avoiders with $m$ left-right maxima.
We have $G_0(x)=1,\ G_1(x)=xF_T(x)$ and most of the work is in finding an equation for $G_2(x)$.
\begin{lemma}\label{lem214a1}
\begin{align*}
G_2(x)=x^2F_T(x)C(x)+\frac{x^3C(x)^2\big(F_T(x)-1\big)}{1-2x}+\frac{x^3C(x)^2}{1-2x}-\frac{x^3C(x)^2}{(1-x)^2}+\frac{x^3C(x)}{(1-x)\big(1-x-xC(x)\big)}\,.
\end{align*}
\end{lemma}
\begin{proof}
Refine $G_2(x)$ to $G_2(x;d)$, the generating function for permutations $\pi=i\pi'n\pi''\in S_n(T)$ with $2$ left-right maxima and such that $\pi''$ has $d\ge 0$ letters smaller than $i$.
For $d=0, \ \pi''>i$ and $\pi''$ avoids 231 (or $i$ is the ``1'' of a 1342), while $\pi'$ avoids $T$.
Hence, $G_2(x;0)=x^2F_T(x)C(x)$. So assume $d\geq1$.

The letters in $\pi''$  smaller than $i$, say $j_1,j_2,\dots,j_d$, are decreasing (or $in$ is the 34 of a 3412) and $\pi$ has the form illustrated in Figure \ref{fig214m2a},
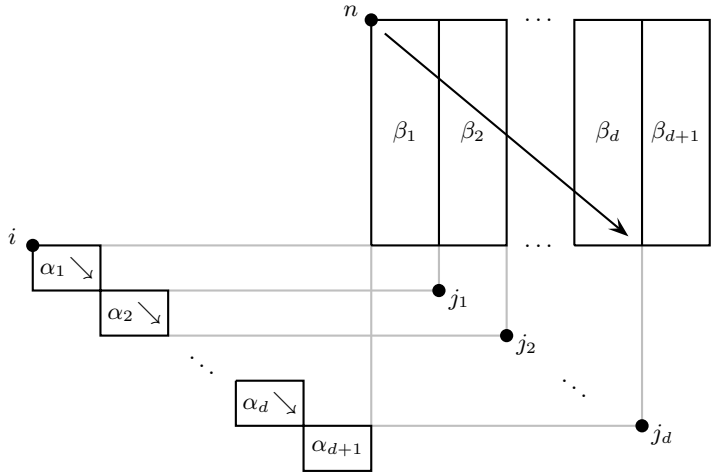
\begin{figure}[htp]
\begin{center}
\begin{pspicture}(-5,0)(5,6)
\psset{xunit=.9cm}
\psset{yunit=.6cm}
\psline[linecolor=lightgray](-4,5)(0,5)(0,1)
\psline[linecolor=lightgray](-3,4)(1,4)(1,5)
\psline[linecolor=lightgray](-3,3)(2,3)(2,5)
\psline[linecolor=lightgray](0,1)(4,1)(4,5)
\qdisk(-5,5){2.5pt}
\qdisk(0,10){2.5pt}
\qdisk(1,4){2.5pt}
\qdisk(2,3){2.5pt}
\qdisk(4,1){2.5pt}
\psline(-5,5)(-4,5)(-4,3)(-3,3)(-3,4)(-5,4)(-5,5)
\psline(-2,2)(-1,2)(-1,0)(0,0)(0,1)(-2,1)(-2,2)
\psline(0,5)(0,10)(2,10)(2,5)(0,5)
\psline(1,5)(1,10)
\psline(3,5)(3,10)(5,10)(5,5)(3,5)
\psline(4,5)(4,10)
\rput(0.5,7.5){\textrm{\small $\be_1$}}
\rput(1.5,7.5){\textrm{\small $\be_2$}}
\rput(3.5,7.5){\textrm{\small $\be_d$}}
\rput(4.5,7.5){\textrm{\small $\be_{d+1}$}}
\rput(-4.5,4.5){\textrm{\small $\al_1\!\searrow$}}
\rput(-3.5,3.5){\textrm{\small $\al_2\!\searrow$}}
\rput(-1.5,1.5){\textrm{\small $\al_d\!\searrow$}}
\rput(-0.5,0.5){\textrm{\small $\al_{d+1}$}}
\psline[arrows=->,arrowsize=4pt 2](.2,9.7)(3.8,5.2)
\rput(-2.5,2.5){\textrm{\small $\ddots$}}
\rput(3,2){\textrm{\small $\ddots$}}
\rput(2.5,5){\textrm{\small $\dots$}}
\rput(2.5,10){\textrm{\small $\dots$}}
\rput(-5.3,5.2){\textrm{\small $i$}}
\rput(-.3,10.2){\textrm{\small $n$}}
\rput(1.3,3.8){\textrm{\small $j_1$}}
\rput(2.3,2.8){\textrm{\small $j_2$}}
\rput(4.3,0.8){\textrm{\small $j_d$}}
\end{pspicture}
\caption{A general $T$-avoider with $2$ left-right maxima}\label{fig214m2a}
\end{center}
\end{figure}
where $\al_1\al_2\dots\al_d$ is decreasing (or $nj_d$ is the 41 of a 2341) and $\al_m$ lies to the left of $\al_{m+1},\ m=1,2,\dots,d$ (or $nj_m$ is the 42 of a 1342) and $\be_1\be_2\dots\be_d$ is decreasing (or $ij_d$ is the 21 of a 2341).

If $\al_2\dots\al_d\al_{d+1}\ne \emptyset$, then $\be_2\dots\be_d=\emptyset$ (since $u\in \al_2\dots\al_d$ and $v \in \be_2\dots\be_d$ makes $uj_1vj_d$ a 2341, while $u\in \al_{d+1}$ and $v \in \be_2\dots\be_d$ makes $uj_1vj_d$ a 1342), and then, with $\be_1=k_1k_2\dots k_p,\ p\ge 0$, $\pi$ has the form illustrated in Figure \ref{fig214m2b},
\begin{figure}[htp]
\begin{center}
\begin{pspicture}(-5,0)(8.5,6)
\psset{xunit=.9cm}
\psset{yunit=.6cm}
\psline[linecolor=lightgray](-4,5)(4,5)
\psline[linecolor=lightgray](2,6)(4,6)
\psline[linecolor=lightgray](4,1)(4,5)
\psline[linecolor=lightgray](0,1)(0,10)
\psline[linecolor=lightgray](1,8)(7,8)
\psline[linecolor=lightgray](.5,9)(7,9)
\psline[linecolor=lightgray](0,10)(8,10)
\qdisk(-5,5){2.5pt}
\qdisk(0,10){2.5pt}
\qdisk(.5,9){2.5pt}
\qdisk(1,8){2.5pt}
\qdisk(2,6){2.5pt}
\qdisk(2.5,4){2.5pt}
\qdisk(3,3){2.5pt}
\qdisk(4,1){2.5pt}
\psline(-5,5)(-4,5)(-4,3)(-3,3)(-3,4)(-5,4)(-5,5)
\psline(-2,2)(-1,2)(-1,0)(0,0)(0,1)(-2,1)(-2,2)
\psline(4,5)(5,5)(5,7)(6,7)(6,6)(4,6)(4,5)
\psline(7,8)(8,8)(8,10)(9,10)(9,9)(7,9)(7,8)
\rput(8.5,9.5){\textrm{\small $\ga_1$}}
\rput(7.5,8.5){\textrm{\small $\ga_2$}}
\rput(5.5,6.5){\textrm{\small $\ga_p$}}
\rput(4.5,5.5){\textrm{\small $\ga_{p+1}$}}
\rput(-4.5,4.5){\textrm{\small $\al_1\!\searrow$}}
\rput(-3.5,3.5){\textrm{\small $\al_2\!\searrow$}}
\rput(-1.5,1.5){\textrm{\small $\al_d\!\searrow$}}
\rput(-0.5,0.5){\textrm{\small $\al_{d+1}$}}
\rput(-2.5,2.5){\textrm{\small $\ddots$}}
\rput(3.6,2){\textrm{\small $\ddots$}}
\rput(1.4,7.2){\textrm{\small $\ddots$}}
\rput(6.5,7.5){\textrm{\small $\iddots$}}
\rput(-5.3,5.2){\textrm{\small $i$}}
\rput(-.3,10.2){\textrm{\small $n$}}
\rput(2.8,3.8){\textrm{\small $j_1$}}
\rput(3.3,2.8){\textrm{\small $j_2$}}
\rput(4.3,.8){\textrm{\small $j_d$}}
\rput(.3,8.8){\textrm{\small $k_1$}}
\rput(.7,7.8){\textrm{\small $k_2$}}
\rput(1.7,5.8){\textrm{\small $k_p$}}
\end{pspicture}
\caption{A $T$-avoider with $2$ left-right maxima and $\al_2\dots\al_d\al_{d+1}\ne \emptyset$}\label{fig214m2b}
\end{center}
\end{figure}
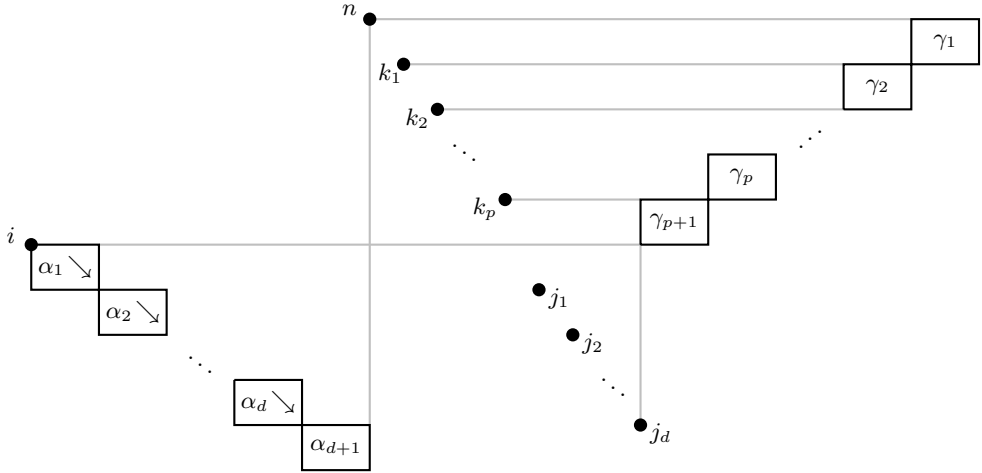
where $\ga_{m+1}$ lies to the left of $\ga_m,\ m=1,2,\dots,p$ (or $ik_m$ is the 13 of a 1342) and all the $\ga$'s avoid 231 (or $i$ is the ``1'' of a 1342).

Now, we consider three cases (Figure \ref{fig214m2b} applies to the first two of them):
\begin{itemize}
\item $\alpha_{d+1}\neq\emptyset$. In this case, we get a contribution of $x^{2+d+p}$ from $i,n$ and the $j$'s and $k$'s, of $F_T(x)-1$ from $\al_{d+1}$, of $1/(1-x)^d$ from the other $\al$'s, and of $C(x)^{p+1}$ from the $\ga$'s. Summing over $p\ge 0$  gives a total contribution of
\[
\frac{x^{2+d}\big(F_T(x)-1\big)C(x)}{(1-x)^d\big(1-xC(x)\big)}\, .
\]
\item $\alpha_m \neq\emptyset$ and $\alpha_{m+1}=\cdots=\alpha_{d+1}=\emptyset$ for some $m\in\{2,3,\ldots,d\}$. This case is similar to the previous one except that we have $m$ decreasing $\alpha$'s to consider and the last
of these is nonempty. Thus, we get a contribution of
\[
\frac{x^{3+d}C(x)}{(1-x)^m(1-xC(x))}\,.
\]
\item $\alpha_{2}=\cdots=\alpha_{d+1}=\emptyset$. In this case, $\pi$ has the form
$i(i-1)\cdots (d+1)n \beta_{1}d\cdots\beta_{d}1\beta_{d+1},$
where $\beta_{1}\cdots \beta_{d} =k_{1}k_{2}\cdots k_{p}$ is decreasing, and
$\beta_{d+1}=\gamma_{p+1}\gamma_{p}\cdots\gamma_{1}$ with $\ga$'s separated by the $k$'s,
and $\gamma_s$ avoids $231$ for all $s=1,2,\ldots,p+1$.
This leads, by a similar analysis, to a contribution of
\[
\frac{x^{2+d}C(x)}{(1-x)\big(1-xC(x)\big)^d}\,.
\]
\end{itemize}
Hence, for $d\ge 1$,
\begin{align*}
G_2(x;d)&=\frac{x^{2+d}(F_T(x)-1)C(x)}{(1-x)^d\big(1-xC(x)\big)}+\left(\sum_{m=2}^d\frac{x^{3+d}C(x)}{(1-x)^m\big(1-xC(x)\big)}\right)
+\frac{x^{2+d}C(x)}{(1-x)\big(1-xC(x)\big)^d}\\
&= \frac{x^{2+d}(F_T(x)-1)C(x)^2}{(1-x)^d}+\left(\sum_{m=2}^d\frac{x^{3+d}C(x)^2}{(1-x)^m)}\right)
+\frac{x^{2+d}C(x)}{(1-x)\big(1-xC(x)\big)^d}
\end{align*}
Since $G_2(x)=\sum_{d\geq0}G_2(x;d)$ and $G_2(x;0)=x^2F_T(x)C(x)$, the result now follows by evaluating geometric sums.
\end{proof}

\begin{lemma}\label{lem214a2}
For $m\ge 3$,
\[
G_m(x)=x^{m-2}C(x)^{m-2}G_2(x)\,.
\]
\begin{proof}
Suppose $m\ge 3$ and $\pi=i_1\pi^{(1)}i_2\pi^{(2)}\cdots i_m\pi^{(m)}\in S_n(T)$ has $m$ left-right maxima.
Since $\pi$ avoids $2341$ and $1342$, we have that $\pi^{(s)}>i_{s-1}$ for all $s=3,4,\ldots,m$.
Thus, $\pi$ avoids $T$ if and only if (i) $i_1\pi^{(1)}i_2\pi^{(2)}$ avoids $T$ and has exactly $2$ left-right maxima, and (ii) $\pi^{(s)}$ avoids $231$ for all $s=3,4,\ldots,m$. The result follows.
\end{proof}
\end{lemma}
\begin{theorem}\label{th214a}
Let $T=\{1342,2341,3412\}$. Then
\[
F_T(x)=\frac{(1 - 2 x) \big( (1 - 5 x + 9 x^2 - 6 x^3)\sqrt{1 - 4 x} - (1 - 9 x + 29 x^2 - 38 x^3 + 18 x^4)\big)}{2 (1 -  x)^2 x (1 - 7 x + 14 x^2 - 9 x^3)}\,.
\]
\end{theorem}
\begin{proof}
Using Lemma \ref{lem214a2} and summing over $m\geq3$ leads to
$$F_T(x)-1-xF_T(x)=G_2(x)+\frac{xC(x)G_2(x)}{1-xC(x)}=\frac{G_2(x)}{1-xC(x)}=G_2(x)C(x)\,.$$
Lemma \ref{lem214a1} now implies
\begin{align*}
&F_T(x)-1-xF_T(x)\\
&=x^2F_T(x)C(x)^2+\frac{x^3C(x)^3(F_T(x)-1)}{1-2x}+\frac{x^3C(x)^3}{1-2x}-\frac{x^3C(x)^3}{(1-x)^2}+\frac{x^3C(x)^4}{1-x}\,.
\end{align*}
Solving for $F_T(x)$ and simplifying the result completes the proof.
\end{proof}

\subsection{Case 217: $\{4132,1342,1243\}$}
\begin{theorem}\label{th217a}
Let $T=\{4132,1342,1243\}$. Then
$$F_T(x)=\frac{(1 - x)(1 - 3 x + x^2)\sqrt{1 - 4 x} - (1 - 8 x + 20 x^2 - 15 x^3 + 4 x^4)}{2x(1 - x)(1 - 5 x + 4 x^2 - x^3)}.$$
\end{theorem}
\begin{proof}
Let $G_m(x)$ be the generating function for $T$-avoiders with $m$
right-left maxima. Clearly, $G_0(x)=1$ and $G_1(x)=xF_T(x)$. In order to find an explicit formula for $G_m(x)$, we define four types of generating functions:
\begin{itemize}
\item $H(x)$ the generating function for the number of permutations in $\pi'n\pi''(n-1)\in S_n(T)$ with two right-left maxima such that $\pi'\neq\emptyset$;
\item  $J(x)$ the generating function for the number of permutations in $\pi'n\pi'' i\in S_n(T)$ with two right-left maxima such that $\pi'$ has a letter smaller than $i$ and $i\leq n-2$;
\item  $J'(x)$ the generating function for the number of permutations in $\pi'''i_3\pi''i_2\pi'i_1\in S_n(T)$ with three right-left maxima ($i_1<i_2<i_3=n$) such that $\pi'''$ has a letter between $i_1$ and $i_2$ and has a letter smaller than $i_1$;
\item  $H'(x)$ the generating function for the number of permutations in $\pi'n\pi''(n-2)\pi'''(n-1)\in S_n(T)$ with two right-left maxima such that $\pi'\neq\emptyset$.
\end{itemize}
\begin{figure}[htp]
\begin{center}
\begin{pspicture}(0,-.3)(8,2)
\psline(0,0)(2,0)(2,1)(0,1)\psline(0,0)(0,2)(1,2)(1,0)
\psline[fillstyle=solid,fillcolor=lightgray](0,0)(0,1)(1,1)(1,0)(0,0)\put(.5,-.5){$xC(x)G_1(x)$}
\qdisk(1,2){2.5pt}\qdisk(2,1){2.5pt}\put(1.1,2){$n$}\put(2.1,1){$i$}
\put(3,0){
\psline(0,0)(2,0)(2,1)(0,1)\psline(0,0)(0,2)(1,2)(1,0)
\put(.3,.3){$\neq\emptyset$}
\psline[fillstyle=solid,fillcolor=lightgray](0,1)(0,2)(1,2)(1,1)(0,1)\put(.5,-.5){$H(x)$}
\qdisk(1,2){2.5pt}\qdisk(2,1){2.5pt}\put(1.1,2){$n$}\put(2.1,1){$i$}}
\put(6,0){
\psline(0,0)(2,0)(2,1)(0,1)\psline(0,0)(0,2)(1,2)(1,0)
\put(.3,.3){$\neq\emptyset$}\put(.3,1.3){$\neq\emptyset$}\put(.5,-.5){$J(x)$}
\qdisk(1,2){2.5pt}\qdisk(2,1){2.5pt}\put(1.1,2){$n$}\put(2.1,1){$i$}}
\end{pspicture}
\caption{$T$-avoiders with two right-left maxima.}\label{fig217a1}
\end{center}
\end{figure}
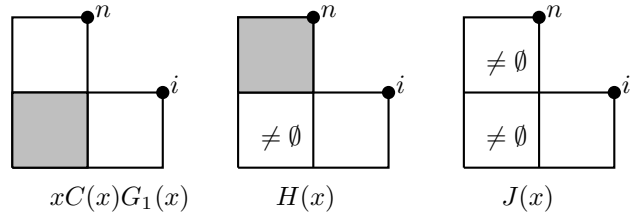
Clearly, $G_2(x)=xC(x)G_1(x)+H(x)+J(x)$ (see Figure \ref{fig217a1}) and $$G_m(x)=xC(x)G_{m-1}+x^{m-2}H(x)+x^{m-3}J'(x),$$
for all $m\geq3$. Thus,
\begin{align}\label{eq217a1}
F_T(x)-1-xF_T(x)=xC(x)(F_T(x)-1)+\frac{1}{1-x}H(x)+J(x)+\frac{1}{1-x}J'(x).
\end{align}
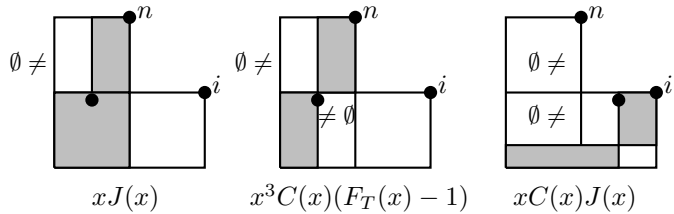
\begin{figure}[htp]
\begin{center}
\begin{pspicture}(0,-.3)(8,2)
\put(0,0){
\psline(0,0)(2,0)(2,1)(0,1)\psline(0,0)(0,2)(1,2)(1,0)
\put(-.6,1.3){\small$\emptyset\neq$}\put(.5,-.5){$xJ(x)$}
\psline[fillstyle=solid,fillcolor=lightgray](0.5,1)(1,1)(1,2)(.5,2)(.5,1)
\qdisk(1,2){2.5pt}\qdisk(2,1){2.5pt}\put(1.1,2){$n$}\put(2.1,1){$i$}
\psline[fillstyle=solid,fillcolor=lightgray](0,0)(0,1)(1,1)(1,0)(0,0)
\qdisk(.5,.9){2.5pt}}
\put(3,0){
\psline(0,0)(2,0)(2,1)(0,1)\psline(0,0)(0,2)(1,2)(1,0)
\put(-.6,1.3){\small$\emptyset\neq$}\put(.5,.6){\small$\neq\emptyset$}
\put(-.4,-.5){$x^3C(x)(F_T(x)-1)$}
\psline[fillstyle=solid,fillcolor=lightgray](0.5,1)(1,1)(1,2)(.5,2)(.5,1)
\qdisk(1,2){2.5pt}\qdisk(2,1){2.5pt}\put(1.1,2){$n$}\put(2.1,1){$i$}
\psline[fillstyle=solid,fillcolor=lightgray](0,0)(0,1)(.5,1)(.5,0)(0,0)
\qdisk(.5,.9){2.5pt}}
\put(6,0){
\psline(0,0)(2,0)(2,1)(0,1)\psline(0,0)(0,2)(1,2)(1,0)
\psline[fillstyle=solid,fillcolor=lightgray](0,0)(1.5,0)(1.5,.3)(0,.3)(0,0)
\psline[fillstyle=solid,fillcolor=lightgray](1.5,.3)(2,.3)(2,1)(1.5,1)(1.5,.3)
\put(.3,1.3){\small$\emptyset\neq$}\put(.3,.6){\small$\emptyset\neq$}
\put(.1,-.5){$xC(x)J(x)$}
\qdisk(1,2){2.5pt}\qdisk(2,1){2.5pt}\put(1.1,2){$n$}\put(2.1,1){$i$}
\qdisk(1.5,.9){2.5pt}}
\end{pspicture}
\caption{$T$-avoiders $\pi'n\pi''i$ with two right-left maxima and $\pi'$ has a letter smaller than $i$ and a letter greater than $i$.}\label{fig217a2}
\end{center}
\end{figure}
Moreover, by considering the position of the letter $i-1$ in permutations $\pi'n\pi'' i\in S_n(T)$ with two right-left maxima such that $\pi'$ has a letter smaller than $i$ and $i\leq n-2$ (see Figure \ref{fig217a2}), we obtain that the contributions are $xJ(x)$, $x^3C(x)(F_T(x)-1)$ and $xC(x)J(x)$, where $\pi'=(i-1)\pi'''$ with $\pi'''\neq\emptyset$, $\pi'=i-1$, and $i-1$  belongs to $\pi''$, respectively. Thus,
$$J(x)=x(J(x)+x^2C(x)(F_T(x)-1))+xC(x)J(x),$$
which implies $J(x)=x^3C(x)(F_T(x)-1)/(1-x-xC(x))=x^3C^3(x)(F_T(x)-1)$.

\begin{figure}[htp]
\begin{center}
\begin{pspicture}(0,-.3)(11,3)
\put(0,0){
\psline[fillstyle=solid,fillcolor=lightgray](0,2)(1,2)(1,3)(0,3)(0,2)
\psline[fillstyle=solid,fillcolor=lightgray](1,0)(2,0)(2,2)(1,2)(1,0)
\psline[fillstyle=solid,fillcolor=lightgray](0,0)(.5,0)(.5,1)(0,1)(0,0)
\psline[fillstyle=solid,fillcolor=lightgray](.5,1)(1,1)(1,2)(.5,2)(.5,1)
\psline(0,0)(3,0)(3,1)(0,1)(0,0)\psline(0,0)(2,0)(2,2)(0,2)(0,0)\psline(0,0)(1,0)(1,3)(0,3)(0,0)
\put(0,1.3){\small$\emptyset\neq$}\put(.5,.3){\small$\emptyset\neq$}\put(.5,-.5){$xJ'(x)$}
\qdisk(1,3){2.5pt}\qdisk(2,2){2.5pt}\qdisk(3,1){2.5pt}\put(1.1,3){$i_3$}\put(2.1,2){$i_2$}\put(3.1,1){$i_1$}
\qdisk(.5,.9){2.5pt}}
\put(4,0){
\psline[fillstyle=solid,fillcolor=lightgray](0,2)(1,2)(1,3)(0,3)(0,2)
\psline[fillstyle=solid,fillcolor=lightgray](1,0)(2,0)(2,2)(1,2)(1,0)
\psline[fillstyle=solid,fillcolor=lightgray](0,0)(1,0)(1,1)(0,1)(0,0)
\psline[fillstyle=solid,fillcolor=lightgray](.5,1)(1,1)(1,2)(.5,2)(.5,1)
\psline(0,0)(3,0)(3,1)(0,1)(0,0)\psline(0,0)(2,0)(2,2)(0,2)(0,0)\psline(0,0)(1,0)(1,3)(0,3)(0,0)
\psline{->}(0,2)(.5,1)\put(0,1.3){\small$\emptyset\neq$}\put(.5,-.5){$x^5C(x)/(1-x)$}
\qdisk(1,3){2.5pt}\qdisk(2,2){2.5pt}\qdisk(3,1){2.5pt}\put(1.1,3){$i_3$}\put(2.1,2){$i_2$}\put(3.1,1){$i_1$}
\qdisk(.5,.9){2.5pt}}
\put(8,0){
\psline[fillstyle=solid,fillcolor=lightgray](0,2)(1,2)(1,3)(0,3)(0,2)
\psline[fillstyle=solid,fillcolor=lightgray](1,0)(2,0)(2,2)(1,2)(1,0)
\psline[fillstyle=solid,fillcolor=lightgray](0,0)(.5,0)(.5,1)(0,1)(0,0)
\psline[fillstyle=solid,fillcolor=lightgray](.5,1)(1,1)(1,2)(.5,2)(.5,1)
\psline[fillstyle=solid,fillcolor=lightgray](.5,0)(2.5,0)(2.5,.3)(.5,.3)(.5,0)
\psline[fillstyle=solid,fillcolor=lightgray](2.5,.3)(3,.3)(3,1)(2.5,1)(2.5,.3)
\psline(0,0)(3,0)(3,1)(0,1)(0,0)\psline(0,0)(2,0)(2,2)(0,2)(0,0)\psline(0,0)(1,0)(1,3)(0,3)(0,0)
\put(0,1.3){\small$\emptyset\neq$}\put(.5,.7){\small$\emptyset\neq$}\put(.5,-.5){$xC(x)J'(x)$}
\qdisk(1,3){2.5pt}\qdisk(2,2){2.5pt}\qdisk(3,1){2.5pt}\put(1.1,3){$i_3$}\put(2.1,2){$i_2$}\put(3.1,1){$i_1$}
\qdisk(2.5,.9){2.5pt}}
\end{pspicture}
\caption{$T$-avoiders $\pi'''i_3\pi''i_2\pi'i_1\in S_n(T)$ with three right-left maxima ($i_1<i_2<i_3=n$) such that $\pi'''$ has a letter between $i_1$ and $i_2$ and has a letter smaller than $i_1$.}\label{fig217a3}
\end{center}
\end{figure}
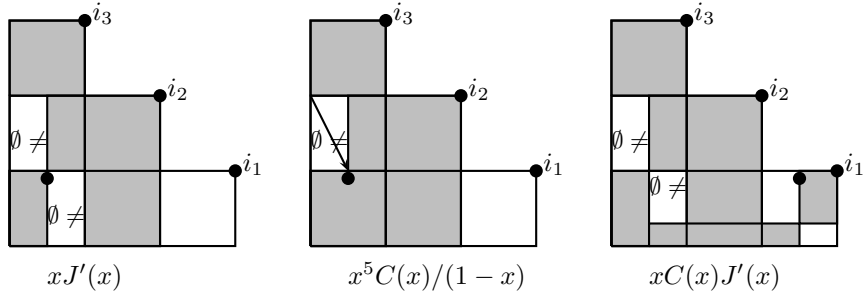
By considering the position of the letter $i_1-1$ in permutations $\pi'''i_3\pi''i_2\pi'i_1\in S_n(T)$ with three right-left maxima ($i_1<i_2<i_3=n$) such that $\pi'''$ has a letter between $i_1$ and $i_2$ and has a letter smaller than $i_1$, we obtain that (see Figure \ref{fig217a3})
$$J'(x)=x(J'(x)+x^4C(x)/(1-x))+xC(x)J'(x),$$
which implies $J'(x)=x^5C(x)/((1-x)(1-x-xC(x)))=x^5C^3(x)/(1-x)$.
\begin{figure}[htp]
\begin{center}
\begin{pspicture}(0,-.3)(8,1.2)
\put(0,0){\put(.2,-.3){$H(x)$}\psline(0,0)(2,0)(2,1)(0,1)(0,0)\psline(1,0)(1,1)\qdisk(1,1){2.5pt}\qdisk(2,1){2.5pt}
\put(1.8,1.1){\small$n-1$}\put(.9,1.1){\small$n$}\put(.3,.3){\small$\neq\emptyset$}\put(2.2,.3){$=$}
\put(3,0){\put(.2,-.3){$H''(x)$}\psline(0,0)(2,0)(2,1)(0,1)(0,0)\psline(1,0)(1,1)\qdisk(1,1){2.5pt}\qdisk(2,1){2.5pt}
\put(1.8,1.1){\small$n-1$}\put(.9,1.1){\small$n$}\qdisk(.1,.9){2.5pt}\put(-.1,1.1){\small$n-2$}\put(2.2,.3){$+$}}
\put(6,0){\put(.2,-.3){$H'(x)$}\psline(0,0)(2,0)(2,1)(0,1)(0,0)\psline(1,0)(1,1)\qdisk(1,1){2.5pt}\qdisk(2,1){2.5pt}
\put(1.8,1.1){\small$n-1$}\put(.9,1.1){\small$n$}\put(.3,.3){\small$\neq\emptyset$}\qdisk(1.5,.9){2.5pt}\put(1.3,1.1){\small$n-2$}}}
\end{pspicture}
\caption{Equation for $H(x)$.}\label{fig217a4}
\end{center}
\end{figure}
By considering the letter $n-2$ either in $\pi'$ or $\pi''$ in the permutation $\pi'n\pi''(n-1)\in S_n(T)$ with two right-left maxima such that $\pi'\neq\emptyset$, we obtain that (see Figure \ref{fig217a4})
$$H(x)=H'(x)+H''(x),$$
where $H''(x)$ is the generating function for the number of permutations $(n-2)\pi'n\pi''(n-1)\in S_n(T)$ with two right-left maxima.

\begin{figure}[htp]
\begin{center}
\begin{pspicture}(0,-.3)(8,1.2)
\put(0,0){\put(.2,-.3){$H''(x)$}\psline(0,0)(2,0)(2,1)(0,1)(0,0)\psline(1,0)(1,1)\qdisk(1,1){2.5pt}\qdisk(2,1){2.5pt}
\put(1.8,1.1){\small$n-1$}\put(.9,1.1){\small$n$}\qdisk(0,1){2.5pt}\put(-.1,1.1){\small$n-2$}\put(2.2,.3){$=x^3+$}}
\put(4,0){\put(.2,-.3){$xH''(x)$}\psline(0,0)(2,0)(2,1)(0,1)(0,0)\psline(1,0)(1,1)\qdisk(1,1){2.5pt}\qdisk(2,1){2.5pt}
\put(1.8,1.1){\small$n-1$}\put(.9,1.1){\small$n$}\qdisk(0,1){2.5pt}\qdisk(.2,.9){2.5pt}
\put(0,1){\small$n-2$}\put(0,.55){\small$n-3$}\put(2.2,.3){$+$}}
\put(7,0){\psline[fillstyle=solid,fillcolor=lightgray](0,0)(1.5,0)(1.5,.3)(0,.3)(0,0)
\psline[fillstyle=solid,fillcolor=lightgray](1.5,.3)(2,.3)(2,1)(1.5,1)(1.5,.3)
\put(.2,-.3){$xC(x)H''(x)$}\psline(0,0)(2,0)(2,1)(0,1)(0,0)\psline(1,0)(1,1)\qdisk(1,1){2.5pt}\qdisk(2,1){2.5pt}
\put(1.8,1.1){\small$n-1$}\put(.9,1.1){\small$n$}\qdisk(0,1){2.5pt}
\put(-.1,1.1){\small$n-2$}\qdisk(1.5,.9){2.5pt}\put(1.2,.6){\small$n-3$}}
\end{pspicture}
\caption{Equation for $H''(x)$.}\label{fig217a5}
\end{center}
\end{figure}
By considering the position of $(n-3)$ (if it exists) in the permutation $(n-2)\pi'n\pi''(n-1)\in S_n(T)$ (see Figure \ref{fig217a5}), we see that $H''(x)=x^3+xH''(x)+xC(x)H''(x)$, which leads to $H''(x)=x^3/(1-x-xC(x))=x^3C^2(x)$.

Now, we find a formula for $H'(x)$. As above, by considering position of $(n-3)$ in the permutations $\pi'n\pi''(n-2)\pi'''(n-1)\in S_n(T)$ with $\pi'\neq\emptyset$, we see that
$H'(x)=H'''(x)+xC(x)H'(x)+xH'(x)$, which gives $H'(x)=H'''(x)/(1-x-xC(x))=H'''(x)C^2(x)$, where
$H'''(x)$ is the generating function for the number of permutations $(n-3)\pi'n\pi''(n-2)\pi'''(n-1)\in S_n(T)$ with two right-left maxima. Again, by looking at position of $(n-4)$ (if exists), we get that (similar to case of $H''(x)$)
$H'''(x)=x^4+xH'''(x)+x^2C^2(x)H''(x)+x^2C(x)H''(x)$, which leads to
$H'''(x)=(x^4+x^5C^4(x)+x^5C^3(x))/(1-x)$. Hence,
$$H'(x)=(x^4+x^5C^4(x)+x^5C^3(x))C^2(x)/(1-x),$$
which implies that
$$H(x)=(x^4+x^5C^4(x)+x^5C^3(x))C^2(x)/(1-x)+x^3C^2(x).$$
Hence, by \eqref{eq217a1}, we have
\begin{align*}
F_T(x)-1-xF_T(x)&=xC(x)(F_T(x)-1)+\frac{1}{1-x}\left(\frac{(x^4+x^5C^4(x)+x^5C^3(x))C^2(x)}{1-x}+x^3C^2(x)\right)\\
&+x^3C^3(x)(F_T(x)-1)+\frac{x^5C^3(x)}{(1-x)^2}.
\end{align*}
By solving for $F_T(x)$, we complete the proof.
\end{proof}

\subsection{Case 219: $\{1342,2413,3412\}$}
\begin{theorem}\label{th219a}
Let $T=\{1342,2413,3412\}$. Then
$$F_T(x)=1+\frac{x(1-x)^2(1-2x)}{(x^2-3x+1)(2x^2-2x+1)-x(1-2x)(1-x)C(x)}.$$
\end{theorem}
\begin{proof}
Let $G_m(x)$ be the generating function for $T$-avoiders with $m$
left-right maxima. Clearly, $G_0(x)=1$ and $G_1(x)=xF_T(x)$.

To find an equation for $G_m(x)$ for $m\ge 3$, suppose $\pi=i_1\pi^{(1)}i_2\pi^{(2)}\cdots i_m\pi^{(m)}\in S_n(T)$ has $m\ge 3$ left-right maxima.
Let $a$ be the smallest letter in $i_1\pi^{(1)}i_2\pi^{(2)}\cdots i_{m-1}\pi^{(m-1)}$. We claim no letter $u$ in
$\pi^{(m)}$ lies in the interval of integers $[a,i_{m-1}]$. To see this, suppose such a $u$ exists. If $a$ is in
$i_1\pi^{(1)} \cdots i_{m-2}\pi^{(m-2)}$, then $a\,i_{m-1}i_m\,u$ is a 1342. Otherwise, $a<i_1$ and $a\in \pi^{(m-1)}$, and if $u>i_1$, then $i_1\,i_{m-1}a\,u$ is a 2413, while if $u<i_1$, then $i_1\,i_2\,a\,u$ is a 3412.
\begin{figure}[htp]
\begin{center}
\begin{pspicture}(-1.5,0)(8,5)
\psset{xunit=1.2cm}
\psset{yunit=.8cm}
\qdisk(0,2){2.5pt}
\qdisk(1,3){2.5pt}
\qdisk(3,5){2.5pt}
\qdisk(4,6){2.5pt}
\psline(0,1)(0,2)(1,2)(1,1)(0,1)
\psline(1,1)(2,1)(2,3)(1,3)(1,2)
\psline(3,1)(4,1)(4,6)(5,6)(5,5)(3,5)(3,1)
\psline(5,1)(6,1)(6,0)(5,0)(5,1)
\rput(0.5,1.5){\textrm{\small $\pi^{(1)}$}}
\rput(1.5,2){\textrm{\small $\pi^{(2)}$}}
\rput(3.5,3){\textrm{\small $\pi^{(m-1)}$}}
\rput(4.5,5.5){\textrm{\small $\al$}}
\rput(5.5,0.5){\textrm{\small $\be$}}
\psline[linecolor=lightgray](4,1)(5,1)(5,5)
\psline[linecolor=lightgray](2,1)(3,1)
\rput(1.5,4){\textrm{\footnotesize $\iddots$}}
\rput(2.5,1.5){\textrm{\footnotesize $\dots$}}
\rput(-0.3,2.2){\textrm{\footnotesize $i_1$}}
\rput(0.7,3.2){\textrm{\footnotesize $i_2$}}
\rput(2.5,5.2){\textrm{\footnotesize $i_{m-1}$}}
\rput(3.7,6.2){\textrm{\footnotesize $i_m$}}
\end{pspicture}
\caption{A $T$-avoider with $m\ge 3$ left-right maxima}\label{fig219m3}
\end{center}
\end{figure}
Thus $\pi$ has the form illustrated in Figure \ref{fig219m3} where $\al$ lies to the left of $\be$ because $uv$ in $\pi^{(m)}$ with $u \in \be$ and $v \in \al$ makes $i_{m-1}i_m\,u\,v$ a 2413. Also, $\al$ avoids 231 (or $i_1$ is the ``1'' of a 1342) and $\be$ is decreasing because $uv \in \be$ with $u<v$ makes $i_{m-1}i_m\,u\,v$ a 3412, and $i_1\pi^{(1)}\cdots i_{m-1}\pi^{(m-1)}$ avoids $T$. Hence, for $m\ge 3$,
\begin{equation}\label{rec219}
G_m(x)=\frac{x}{1-x}C(x)G_{m-1}(x)\,.
\end{equation}
Now suppose that $m=2$ and that $\pi$ has $d$ letters in $\pi^{(m)}=\pi^{(2)}$  smaller than $i:=i_1$, say $j_1,j_2,\dots,j_d$ left to right.
\begin{figure}[htp]
\begin{center}
\begin{pspicture}(1,-0.3)(8,5)
\psset{xunit=1.2cm}
\psset{yunit=.8cm}
\psline(0,5)(1,5)(1,3)(2,3)(2,4)(0,4)(0,5)
\psline(3,2)(4,2)(4,0)(5,0)(5,1)(3,1)(3,2)
\psline(5,6)(6,6)(6,5)(5,5)(5,6)
\rput(0.5,4.5){\textrm{\small $\al_1$}}
\rput(1.5,3.5){\textrm{\small $\al_2$}}
\rput(3.5,1.5){\textrm{\small $\al_d$}}
\rput(4.5,0.5){\textrm{\small $\al_{d+1}$}}
\rput(5.5,5.5){\textrm{\small $\be$}}
\psline[linecolor=lightgray](5,1)(5,5)(1,5)
\psline[linecolor=lightgray](2,4)(6.5,4)
\psline[linecolor=lightgray](2,3)(7,3)
\psline[linecolor=lightgray](5,1)(8,1)
\qdisk(0,5){2.5pt}
\qdisk(5,6){2.5pt}
\qdisk(6.5,4){2.5pt}
\qdisk(7,3){2.5pt}
\qdisk(8,1){2.5pt}
\rput(2.5,2.5){\textrm{\footnotesize $\ddots$}}
\rput(7.5,2.1){\textrm{\footnotesize $\ddots$}}
\rput(6.8,4.2){\textrm{\footnotesize $j_1$}}
\rput(7.3,3.2){\textrm{\footnotesize $j_2$}}
\rput(8.3,1.2){\textrm{\footnotesize $j_d$}}
\rput(-0.2,5.2){\textrm{\footnotesize $i$}}
\rput(4.8,6.2){\textrm{\footnotesize $n$}}
\end{pspicture}
\caption{A $T$-avoider with $m=2$ left-right maxima}\label{fig219m2}
\end{center}
\end{figure}
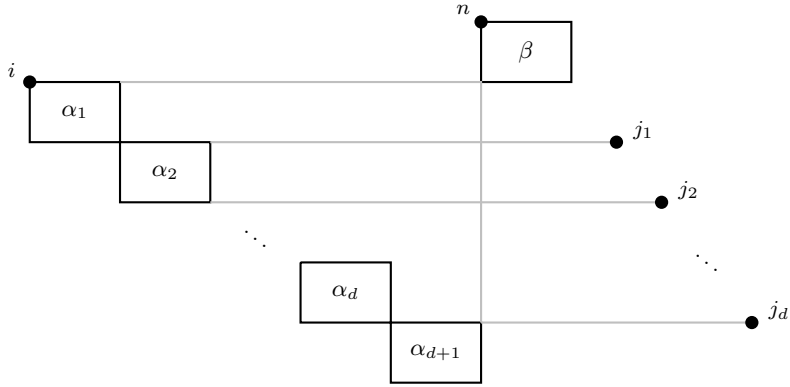
Then the $j$'s are decreasing (to avoid 3412) and $\pi$ has the form illustrated in Figure \ref{fig219m2} where $\be$ lies to the left of the $j$'s (or $in$ is the 24 of a 2413) and $\al_k$ lies to the left of $\al_{k+1},\ 1\le k \le d$ (since $u \in \al_{k+1},\,v \in \al_k,\, u<v\Rightarrow u\,v\,n\,j_k$ is a 1342).

We consider two cases:
\begin{itemize}
\item $\beta$ is decreasing. If $\alpha_s$ is decreasing for all $s \in [1,d\,]$, then $\al_{d+1}$ is a $T$-avoider, and we have a contribution of $x^{d+2}F_T(x)/(1-x)^{d+1}$. Otherwise, let $k$ be minimal such that $\alpha_k$ is not decreasing. In this case, $\alpha_s$ is decreasing for $s=1,2,\ldots,k-1$ and $\alpha_s=\emptyset$ for all $s=k+1,k+2,\dots,d+1$ ($u\in \al_s,\ s>k\, \Rightarrow abuj_{s-1}$ is a 3412 for $a<b$ in $\al_k$).
Thus, we have a contribution of $x^{d+2}/(1-x)^{k}\big(F_T(x)-1/(1-x)\big)$.

\item $\beta$ is not decreasing. Here $\beta$ avoids $231$ (or $i$ is the ``1'' of a 1342), $\al_1$ avoids $T$, and $\al_2\dots \al_{d+1}=\emptyset$ ($u\in \al_2\dots \al_{d+1}\Rightarrow uabj_1$ is a 1342 for $a<b$ in $\be$).
Thus, we have a contribution of $x^{d+2}F_T(x)\big(C(x)-1/(1-x)\big)$.
\end{itemize}
Summing over $d$ yields
\begin{align*}
G_2(x)&=\sum_{d\geq0}\left(\frac{x^{d+2}}{(1-x)^{d+1}}F_T(x)+
\sum_{k=1}^d\frac{x^{d+2}}{(1-x)^k}\left(F_T(x)-\frac{1}{1-x}\right) +x^{d+2}F_T(x)\left(C(x)-\frac{1}{1-x}\right)\right) \\
 &=\frac{x^2}{1-2x}F_T(x)+\frac{x^3}{(1-x)(1-2x)}\left(F_T(x)-\frac{1}{1-x}\right)
+\frac{x^2}{1-x}F_T(x) \left(C(x)-\frac{1}{1-x}\right)\, .
\end{align*}


Summing recurrence (\ref{rec219}) over $m\geq3$ yields
$$F_T(x)-G_0(x)-G_1(x)-G_2(x)=\frac{xC(x)}{1-x}\big(F_T(x)-G_0(x)-G_1(x)\big)\,.$$
Substitute for $G_0,G_1,G_2$ and solve for $F_T(x)$ to complete the proof.
\end{proof}

\subsection{Case 220: $\{2431,2314,3142\}$}
\begin{theorem}\label{th220a}
Let $T=\{2431,2314,3142\}$. Then
$$F_T(x)=1+\frac{x(1-x)^2(1-2x)}{(1-3x)(1-x)^3-x(1-2x)(1-x+x^2)\big(C(x)-1\big)}.$$
\end{theorem}
\begin{proof}
Let $G_m(x)$ be the generating function for $T$-avoiders with $m$
left-right maxima. Clearly, $G_0(x)=1$ and $G_1(x)=xF_T(x)$.

To write an equation for $G_m(x)$ with $m\geq2$, suppose
$\pi=i_1\pi^{(1)}i_2\pi^{(2)}\cdots i_m\pi^{(m)}\in S_n(T)$ with $m$ left-right maxima. Then $\pi$ has the form illustrated in Figure \ref{fig220}a) below where the shaded region is empty (a letter $u$ in the shaded region
implies $i_1 i_2\,u\, i_m$ is a 2314).
\begin{figure}[htp]
\begin{center}
\begin{pspicture}(-5.9,-0.6)(8,4)
\psset{xunit=.9cm}
\psset{yunit=.7cm}
\pspolygon[fillstyle=solid,fillcolor=lightgray](-4,0)(-4,1)(-3,1)(-3,2)(-2,2)(-2,3)(-1,3)(-1,0)(-4,0)
\psline(-4,1)(-5,1)(-5,0)(-4,0)
\psline(-1,3)(-1,5)(0,5)(0,0)(-1,0)
\psline(-4,1)(-4,2)(-3,2)(-3,3)(-2,3)(-2,4)
\qdisk(-5,1){2.5pt}
\qdisk(-4,2){2.5pt}
\qdisk(-3,3){2.5pt}
\qdisk(-2,4){2.5pt}
\qdisk(-1,5){2.5pt}
\rput(-5.3,1.2){\textrm{\footnotesize $i_1$}}
\rput(-4.4,2.2){\textrm{\footnotesize $i_2$}}
\rput(-3.4,3.2){\textrm{\footnotesize $i_3$}}
\rput(-2.4,4.2){\textrm{\footnotesize $i_4$}}
\rput(-1.4,5.2){\textrm{\footnotesize $i_m$}}
%
\psline[linecolor=white,linewidth=.5pt](-2,3.1)(-1,3.1)
\rput(-1.5,3.1){\textrm{\footnotesize $\dots$}}
\rput(-4.5,.5){\textrm{\footnotesize $\pi^{(1)}$}}
\rput(-3.5,1.5){\textrm{\footnotesize $\pi^{(2)}$}}
\rput(-2.5,2.5){\textrm{\footnotesize $\pi^{(3)}$}}
\rput(-0.5,2.5){\textrm{\footnotesize $\pi^{(m)}$}}
\rput(-1.5,4.7){\textrm{\footnotesize $\iddots$}}
\rput(1.4,2){\textrm{$\longrightarrow$}}
\rput(1.4,1.5){\textrm{\footnotesize if $\pi^{(m)}\not> i_1$}}
\psline(3,2)(4,2)(4,3)(3,3)(3,2)
\psline(6,.5)(7,.5)(7,2)(6,2)(6,.5)
\psline(7,3)(8,3)(8,4)(7,4)(7,3)
\rput(3.5,2.5){\textrm{\small $\pi^{(1)}$}}
\rput(6.5,1.1){\textrm{\small $\al$}}
\rput(7.5,3.5){\textrm{\small $\be$}}
\qdisk(3,3){2.5pt}
\qdisk(4,4){2.5pt}
\qdisk(6,5){2.5pt}
\qdisk(6.4,1.7){2pt}
\rput(6.7,1.7){\textrm{\footnotesize $a$}}
\psline[linecolor=lightgray](6,2)(6,5)
\psline[linecolor=lightgray](4,4)(7,4)
\psline[linecolor=lightgray](4,3)(7,3)
\psline[linecolor=lightgray](4,2)(6,2)
\psline[linecolor=lightgray](7,2)(7,3)
\rput(5,4.7){\textrm{\footnotesize $\iddots$}}
\rput(2.6,3.2){\textrm{\footnotesize $i_1$}}
\rput(3.6,4.2){\textrm{\footnotesize $i_2$}}
\rput(5.8,5.4){\textrm{\footnotesize $i_m$}}
\rput(-2.5,-.6){\textrm{\footnotesize a)}}
\rput(5.5,-.6){\textrm{\footnotesize b)}}
\end{pspicture}
\caption{A $T$-avoider with $m\ge 2$ left-right maxima}\label{fig220}
\end{center}
\end{figure}

Note that  $\pi^{(1)}$ avoids 231 (or $i_m$ is the ``4'' of a 2314).
If $\pi^{(m)}>i_1$, then we get a contribution of $xC(x)G_{m-1}(x)$ where $C(x)$ is contributed by $\pi^{(1)}$.

Now suppose $\pi^{(m)}$ has a letter $a$ smaller than $i_1$. Then $\pi^{(2)}\pi^{(3)}\cdots \pi^{(m-1)}=
\emptyset$ (if $u \in \pi^{(s)},\ 2\le s \le m-1$, then $i_{s-1}\,i_s\,u\,a$ is a 2431) and $\pi^{(1)}>a$
(if $u\in \pi^{(1)}$ with $u<a$, then $i_1\, u\, i_2\, a$ is 3142). Also, $\pi^{(m)}<i_2$. To see this,
suppose $u\in \pi^{(m)}$ with $u>i_2$. If $u$ occurs before $a$ in $\pi$, then $i_2\,i_m\,u\,a$ is a 2431,
while if $u$ occurs after $a$, then $i_1\,i_2\,a\,u$ is a 2314. So $\pi$ has the form illustrated in
Figure \ref{fig220}b), where all entries in $\al$ lie to the left of entries in $\be$
($uv$ with $u \in \be$ and $v \in \al$ implies $i_1\,i_2\,u\,v$ is a 2431).

We now consider two cases:
\begin{itemize}
\item $\beta=\emptyset$. Here, $\pi$ avoids $T$ if and only if $\pi^{(1)}$ avoids $132$ (to avoid 2431) and $231$, and $\alpha$ avoids $T$, giving a contribution of $x^m\big(F_T(x)-1\big)L(x)$, where $L(x)=\frac{1-x}{1-2x}$ is the generating function for $\{132,231\}$-avoiders.
\item $\beta\ne \emptyset$. Here, $\pi^{(1)}$ is decreasing ($uv$ in $\pi^{(1)}$ with $u<v\Rightarrow uvab$
is a 2314 for $b\in \be$), $\alpha$ avoids $231$ (to avoid 2314), and $\beta$ avoids $T$,
giving a contribution of $\frac{x^m}{1-x}\big(C(x)-1\big)\big(F_T(x)-1\big)$.
\end{itemize}

Thus, for all $m\geq2$,
$$G_m(x)=xC(x)G_{m-1}(x)+x^m(F_T(x)-1)L(x)+\frac{x^m}{1-x}\big(C(x)-1\big)\big(F_T(x)-1\big)\,.$$
Summing over $m\geq2$ and using the expressions for $G_1(x)$ and $G_0(x)$, we obtain
$$F_T(x)-1-xF_T(x)=xC(x)\big(F_T(x)-1\big)+\frac{x^2}{1-x}\left(L(x)+\frac{1}{1-x}\big(C(x)-1\big)\right)\big(F_T(x)-1\big)\,,$$
with the stated solution for $F_T(x)$.
\end{proof}

\subsection{Case 222: $\{3412,3421,1342\}$}
Here, we use the notion of generating tree  (West \cite{W}),
and consider the generating forest whose vertices are identified
with $S:=\bigcup_{n\ge 2}S_n(T)$ where
12 and 21 are the roots and each non-root $\pi \in S$ is a child of the permutation obtained from $\pi$ by deleting its largest element. We will show that it is possible to label the vertices
so that if $v_1$ and $v_2$ are any two vertices with the same label and $\ell$ is any label,
then $v_1$ and $v_2$ have the same number of children with label $\ell$.
Indeed, we will specify (i) the labels of the roots, and (ii) a set of succession
rules explaining how to derive from the label of a parent the labels of all of its children. This will determine
a labelled generating forest.

A permutation $\pi=\pi_1\pi_2\cdots\pi_n\in S_n$ determines $n+1$ positions, called \emph{sites}, between its entries. The sites are denoted $1,2,\dots,n+1$ left to right.
In particular, site $i$ is the space between $\pi_{i-1}$ and $\pi_{i}$, $2\le i \le n$.
Site $i$ in $\pi$ is said to be {\em active} (with respect to $T$) if, by inserting $n+1$ into $\pi$ in
site $i$, we get a permutation in $S_{n+1}(T)$, otherwise \emph{inactive}.
For $\pi \in S_n(T)$, sites 1 and $n+1$ are always active, and if $\pi_n=n$, then site $n$ is active.

For $\pi\in S_n(T)$, define $A(\pi)$ to be the set of all active sites for $\pi$ and
$L(\pi)$ to be the set of active sites lying to the left of $n$. For example, $L(13254)=\{1,2,4\}$ since there are 4 possible sites to insert 6 to the left of $n=5$
and, of these insertions, only $136254 $ is not in $S_6(T)$.

\begin{lemma}\label{lemmaActiven}
For $\pi\in S_n(T)$, we have $A(\pi)=L(\pi)\cup\{n+1\}$ unless $\pi_n<n$ and site $n$ is active, in which case $A(\pi)=L(\pi)\cup\{n,n+1\}$
\end{lemma}
\begin{proof}
No site to the right of $n$ is active except (possibly) site $n$ and (definitely) site $n+1$ for if $n+1$ is inserted after $n$ in a site $\le n-1$, then $n\,(n+1)\,\pi_{n-1}\,\pi_n$ is a 3412 or a 3421, both forbidden.
\end{proof}

If site $n$ is inactive, then 1 and $n+1$ are the only active sites iff $\pi_1=n$. In particular, there are at least 3 active sites unless site $n$ is inactive and $\pi_1=n$.

For $n\ge 2$, say $\pi \in S_n$ is \emph{special} if it has the form $\pi=n(n-1)\dots (j+1)\pi'j$ for some $j$ with $2\le j \le n$, where $j=n$ means $\pi=\pi'n$.

We now assign labels. Suppose $n\ge 2$ and $\pi \in S_n(T)$ has $k$ active sites. Then $\pi$ is labelled $k,\,\overline{k},\,\overline{\overline{k}}$ according to whether site $n$ is active and whether $\pi$ is special as follows. If site $n$ is inactive then label $\pi$ by $\overline{\overline{k}}$. Otherwise, if $\pi$ is special, then label $\pi$ by $\overline{k}$, and if $\pi$ is not special, then label it by $k$.

For instance, all 3 sites are active for both $12$ and $21$ and only the former is special, so their labels are $\bar{3}$ and 3 respectively; $12$ has three children $312$, $132$ and $123$ with active sites $\{1,3,4\},\,\{1,2,4\}$ and $\{1,2,3,4\}$, respectively, hence labels $\bar{3}$, $\bar{\bar{3}}$ and $\bar{4}$ because only the first and third are special; $21$ has three children $321$, $231$ and $213$ with active sites $\{1,3,4\},\,\{1,2,3,4\}$ and $\{1,2,3,4\}$, respectively, hence labels $3$, $4$ and $\bar{4}$ because only the last is special.

To establish the succession rules, we have the following proposition. The proof is left to the reader.
\begin{proposition}
Fix $n\ge 2$. Suppose $\pi \in S_n(T)$ has $k$ active sites and site $n$ is active so that
\[
A(\pi)=\{1=L_1<L_2< \cdots <L_{k-1}=n\} \cup \{n+1\}\,.
\]

If $\pi$ is special, then $A(\pi^{L_1})=\{L_1,n+1,n+2\}$ and $A(\pi^{L_i})=\{L_1,\dots,L_i,n+2\}$ for $2 \le i \le k-1$.

If $\pi$ is not special, then $A(\pi^{L_i})=\{L_1,\dots,L_i,n+1,n+2\}$ for $1\le i  \le k-1$.

In both cases, $A(\pi^{n+1})=\{L_1,\dots,L_{k-1},n+1,n+2\}$.

Next, suppose $\pi \in S_n(T)$ has $k$ active sites and site $n$ is inactive so that
\[
A(\pi)=\{1=L_1<L_2< \cdots <L_{k-1}\le n-1\} \cup \{n+1\}\,.
\]
Then $A(\pi^{L_i})=\{L_1,\dots,L_i,n+2\}$ for $1\le i \le  k-1$, and $A(\pi^{n+1})=\{L_1,\dots,L_{k-1},n+1,n+2\}$.
\end{proposition}

An immediate consequence is

\begin{corollary}\label{cor222a1}
The labelled generating forest $\mathcal{F}$ is given by
\[
\begin{array}{lll}
\mbox{\bf Roots: }&3,\bar{3}\\
\mbox{\bf Rules: }&k\rightsquigarrow 3,4,\ldots,k,k+1,\overline{k+1}  & \textrm{\quad for $k\ge 3$,}\\
&\bar{k}\rightsquigarrow\bar{3},\bar{\bar{3}},\bar{\bar{4}},\ldots,\bar{\bar{k}},\overline{k+1}  & \textrm{\quad for $k\ge 3$,}\\
&\bar{\bar{k}}\rightsquigarrow\bar{\bar{2}},\bar{\bar{3}},\ldots,\bar{\bar{k}},\overline{k+1}  & \textrm{\quad for $k\ge 2$.}
\end{array}
\] \qed
\end{corollary}

\begin{theorem}\label{th222a}
Let $T=\{3412,3421,1342\}$. Then
\[
F_T(x)=\frac{2 - 11 x + 13 x^2 -  6 x^3 + (1 - x) x (1 - 6 x + 4 x^2) (1 - 4 x)^{-1/2}}{2 (1 - 6 x + 8 x^2 - 4 x^3)}\, .
\]
\end{theorem}
\begin{proof}
Let $a_k(x)$, $b_k(x)$ and $c_k(x)$ denote the generating functions for the number of permutations in the $n$th level of the labelled generating forest $\mathcal{F}$ with label $k$, $\bar{k}$ and $\bar{\bar{k}}$, respectively.
By Corollary \ref{cor222a1}, we have
\begin{align*}
a_k(x)&=x\sum_{j\geq k-1}a_j(x)\,,\\
b_k(x)&=x\big(a_{k-1}(x)+b_{k-1}(x)+c_{k-1}(x)\big)\,,\\
c_k(x)&=x\sum_{j\geq k}\big(b_j(x)+c_j(x)\big)\,,
\end{align*}
with $a_3(x)=x^2+x\sum_{j\geq3}a_j(x)$, $b_3(x)=x^3+xc_2(x)+x\sum_{j\geq3}b_j(x)$, and $c_2(x)=x\sum_{j\geq2}c_j(x)$.

Now let $A(x,v)=\sum_{k\geq3}a_k(x)v^k$, $B(x,v)=\sum_{k\geq2}b_k(x)v^k$ and $C(x,v)=\sum_{k\geq3}c_k(x)v^k$.
The above recurrences can then be written as
\begin{align}
A(x,v)&=a_3(x)v^3+\frac{x}{1-v}(v^4A(x,1)-v^2A(x,v)),\label{eq222ax1}\\
B(x,v)&=b_3(x)v^3+xv(v^3A(x,v)+B(x,v)+C(x,v)-c_2(x)v^2),\label{eq222ax2}\\
C(x,v)&=c_2(x)v^2+\frac{x}{1-v}(v^3B(x,1)-vB(x,v))+\frac{x}{1-v}(v^3C(x,1)-vC(x,v)),\label{eq222ax3}
\end{align}
where $a_3(x)=x^2+xA(x,1)$, $b_3(x)=x^2+xB(x,1)+x^2C(x,1)$ and $c_2(x)=xC(x,1)$.

By finding $A(x,v)$ from \eqref{eq222ax1} and $B(x,v)$ from \eqref{eq222ax2} and then substituting them into \eqref{eq222ax3}, we obtain
\begin{align*}
&\frac{(1-v+xv^2)^2}{(1-xv)(1-v)^2}C(x,t)\\
&\qquad+\frac{xv^2(1-v+xv^2)(2v^2x^2-2x^2v+2x-1)}{(1-v)^2(1-2x)(1-xv)}C(x,1)\\
&\qquad-\frac{(2v^2x^2-2xv+1)x^2v^3}{(1-xv)(1-2x)(1-v)}A(x,1)-\frac{(2v^2x^2-v+1)v^3x^3}{(1-xv)(1-2x)(1-v)}=0
\end{align*}
To solve the preceding functional equation, we apply the kernel method and take $v=C(x)$. This gives
$$A(x,1)=xC(x)-x.$$
Now, by differentiating the functional equation with respect to $t$ and then substituting $t=C(x)$ and $A(x,1)=xC(x)-x$, we obtain
$$C(x,1)=-\frac{1-7x+12x^2-8x^3}{2(1-6x+8x^2-4x^3)}+\frac{1-9x+24x^2-20x^3+8x^4}{2(1-6x+8x^2-4x^3)\sqrt{1-4x}}.$$
Thus, by \eqref{eq222ax2}, we have
$$B(x,1)=\frac{x^3(1-2x-\sqrt{1-4x})}{(1-x)(1-4x)^2-(1-2x)(1-5x+2x^2)\sqrt{1-4x}}.$$
Since $F_T(x)= A(x,1)+B(x,1)+C(x,1)$, the result follows by adding the last three displayed expressions and simplifying.
\end{proof}

\subsection{Case 223: $\{1243,1342,2413\}$}
To find an explicit formula for $F_T(x)$, we define $A_m(x)$ (resp. $B_m(x)$) to be the generating function for $T$-avoiders $\pi=i_1\pi^{(1)}\cdots i_m\pi^{(m)}$ ($i_1,\dots,i_m$ are the left-right maxima) such that $\pi^{(s)}<i_1$ for all $s\neq2$ and $\pi^{(2)}<i_1$ (resp. $\pi^{(2)}$ has a letter greater than $i_1$). Also, we define $G_m(x)$ to be the generating function for $T$-avoiders with $m$ left-right maxima. Clearly, $G_0(x)=1$ and $G_1(x)=xF_T(x)$. Recall $L(x):=\frac{1-x}{1-2x}$ is the generating function for $\{132,231\}$-avoiders \cite{SiS}.

\begin{lemma}\label{lem223a1}
$A_1(x)=G_1(x)$ and for all $m\geq2$,
$$A_m(x)=xA_{m-1}(x)+x\sum_{j\geq m}G_j(x)\,.$$
\end{lemma}
\begin{proof}
Clearly, $A_1(x)=G_1(x)$. To find an equation for $A_m(x)$ with $m\geq2$, suppose $\pi=i_1\pi^{(1)}\cdots i_m\pi^{(m)}\in S_n(T)$ with $m$ left-right maxima satisfies $\pi^{(s)}<i_1$ for all $s$. If $\pi^{(1)}=\emptyset$ we have a contribution of $xA_{m-1}(x)$. Otherwise, assume that $\pi^{(1)}$ has $d\ge 1$ left-right maxima.
Since $\pi$ avoids $1342$ and $1243$, we see that $\pi$ can be written as
$$\pi=i_1j_1\alpha^{(1)}j_2\alpha^{(2)}\cdots j_d\alpha^{(d)}i_2\pi^{(2)}\cdots i_m\pi^{(m)}$$
with $\alpha^{(s)}<j_1$ for all $s=1,3,\ldots,d$, $\alpha^{(2)}<j_2$, and $\pi^{(s)}<j_1$ for all $s=2,3,\ldots,m$.
Thus, we have a contribution of $xG_{m+d-1}(x)$. Summing over all the contributions, we obtain
$$A_m(x)=xA_{m-1}(x)+x\sum_{d\geq1}G_{m+d-1}(x),$$
as required.
\end{proof}

\begin{lemma}\label{lem223a2}
For all $m\geq2$,
$$B_m(x)=x\big(L(x)-1\big)A_{m-1}(x)+\frac{x^3}{(1-2x)^2}A_{m-1}(x)\,.$$
\end{lemma}
\begin{proof}
Let us write an equation for $B_m(x)$, $m\geq2$. Suppose $\pi=i_1\pi^{(1)}\cdots i_m\pi^{(m)}\in S_n(T)$ and that $\pi^{(1)}$ contains $d$ letters. Note that $\pi^{(1)}$ is decreasing ($\pi$ avoids $1243$) and $\pi^{(2)}$ has the form $\alpha\beta$ with $\alpha>i_1$ and $\beta<i_1$.

If $d=0$, then $\pi$ avoids $T$ if and only
if $\alpha$ is nonempty and avoids $132$ and $231$, and $$i_2\beta i_3\pi^{(3)}\cdots i_m\pi^{(m)}$$ avoids $T$. Thus, we have a contribution of $x\big(L(x)-1\big)A_{m-1}(x)$.

For the case $d\geq1$, since $\pi$ avoids $T$, $\pi$ has the form
$$\pi=i_1 j_1j_2\cdots j_di_2\alpha^{(0)}\alpha^{(1)}\cdots\alpha^{(d)}\beta  i_3\pi^{(3)}\cdots i_m\pi^{(m)}$$
such that $j_0:=i_1>j_1>\cdots>j_d\geq1$, $j_{s-1}>\alpha^{(s)}>j_s$ for all $s=0,1,\ldots,d$ with $j_{-1}=i_2$, and $\beta,\pi^{(3)},\ldots,\pi^{(m)}<j_d$. Here, we consider three cases:
\begin{itemize}
\item $\alpha^{(s)}$ is decreasing for all $s=0,1,\ldots,d-1$. The contribution is $\frac{x^{d+2}}{(1-x)^d}L(x)A_{m-1}(x)$.
\item $\alpha^{(0)}$ is not decreasing (has a rise). Then $\alpha^{(s)}=\emptyset$ for all $s=1,2,\ldots,d$. Hence, we have a contribution of $x^{d+1}\big(L(x)-1/(1-x)\big)A_{m-1}(x)$.
\item there is a minimal $p\in [1,d-1]$ such that $\alpha^{(p)}$ is not decreasing. Then $\alpha^{(s)}$ is decreasing for all $s=0,1,\ldots,p-1$ and $\alpha^{(s)}=\emptyset$ for all $s=p+1,\ldots,d$. Thus, we have a contribution of $\frac{x^{d+2}}{(1-x)^p}\big(L(x)-\frac{1}{1-x}\big)A_{m-1}(x)$.
\end{itemize}
Hence,
\begin{align*}
B_m(x)&=x\big(L(x)-1\big)A_{m-1}(x)\\
&+\sum_{d\geq1}\left(\frac{x^{d+2}}{(1-x)^d}L(x)+\left(x^{d+1}+\sum_{p=1}^{d-1}\frac{x^{d+2}}{(1-x)^p}\right)\big(L(x)-1/(1-x)\big)\right)A_{m-1}(x)\\
&=x\big(L(x)-1\big)A_{m-1}(x)+\frac{x^3}{(1-2x)^2}A_{m-1}(x),
\end{align*}
as required.
\end{proof}

Now, we are ready to find a formula for $F_T(x)$.
\begin{theorem}\label{th223a}
Let $T=\{1243,1342,2413\}$. Then
$$F_T(x)=\frac{(1-2x)\big(1-2x-\sqrt{1-8x+20x^2-20x^3+4x^4}\,\big)}{2x(1-4x+5x^2-x^3)}\,.$$
\end{theorem}
\begin{proof}
Define 
\begin{align*}
G(x,y)&=\sum_{m\geq0}G_m(x)y^m, &&A(x,y)=\sum_{m\geq1}A_m(x)y^m,\\ B(x,y)&=\sum_{m\geq2}B_m(x)y^m.
\end{align*}
From the definitions of $A_m$ and $B_m$, we have $G_m(x)=A_m(x)+B_m(x)$ for all $m\geq1$ while $G_0(x)=1$ and $B_1(x)=0$, which implies
\begin{align}
G(x,y)=1+A(x,y)+B(x,y)\,.\label{eq223a1}
\end{align}
Lemma \ref{lem223a1} asserts that
$$A_m(x)=xA_{m-1}(x)+x\sum_{j\geq m}G_j(x),\quad A_1(x)=G_1(x).$$
Multiplying by $y^m$ and summing over $m\geq2$ and using the fact that $G(x,1)=F_T(x)$, we obtain
\begin{align}
(1-xy)A(x,y)=xy+\frac{xy}{1-y}(F_T(x)-G(x,y))\,.\label{eq223a2}
\end{align}
Similarly, Lemma \ref{lem223a2} yields
\begin{align}
B(x,y)=\frac{x^2(1-x)y}{(1-2x)^2}A(x,y)\,.\label{eq223a3}
\end{align}
From \eqref{eq223a1}, \eqref{eq223a2}, and \eqref{eq223a3}, we obtain
$$(1-xy)G(x,y)=1-xy+xy\left(1+\frac{x^2(1-x)y}{(1-2x)^2}\right)\left(1+\frac{F_T(x)-G(x,y)}{1-y}\right).$$
This equation can be solved by the kernel method, giving the stated result.
\end{proof}

\subsection{Case 224: $\{4132,1342,1423\}$}
Here, we consider (right-left) cell decompositions, which allow a useful characterization of $R$-avoiders,
where $R:=\{1342,1423\}$ is a subset of $T$. So suppose
$$\pi=\pi^{(m)}i_m\pi^{(m-1)}i_{m-1}\cdots\pi^{(1)}i_1\in S_n$$
has $m\geq2$ right-left maxima $n=i_m>i_{m-1}>\cdots >i_1\geq1$. The right-left maxima determine a
{\em cell decomposition} of the matrix diagram of $\pi$ as illustrated in Figure \ref{fig224a} for $m=4$.
There are $\binom{m+1}{2}$ cells $C_{ij},\ i,j \ge 1,\,i+j \le m+1,$ indexed by $(x,y)$ coordinates, for example, $C_{21}$ and $C_{32}$ are shown.
\begin{figure}[htp]
\begin{center}
\begin{pspicture}(-2,0)(4,3)
\psset{unit=.7cm}
\psline(0,4)(0,0)(4,0)
\psline(0,4)(1,4)(1,0)
\psline(0,3)(2,3)(2,0)
\psline(0,2)(3,2)(3,0)
\psline(0,1)(4,1)(4,0)
\rput(1.5,0.5){\textrm{{\footnotesize $C_{21}$}}}
\rput(2.5,1.5){\textrm{{\footnotesize $C_{32}$}}}
\rput(1.3,4.2){\textrm{\footnotesize $i_4$}}
\rput(2.3,3.2){\textrm{\footnotesize $i_3$}}
\rput(3.3,2.2){\textrm{\footnotesize $i_2$}}
\rput(4.3,1.2){\textrm{\footnotesize $i_1$}}
\end{pspicture}
\caption{Cell decomposition}\label{fig224a}
\end{center}
\end{figure}
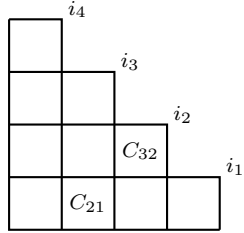

Cells with $i=1$ or $j=1$ are \emph{boundary} cells, the others are \emph{interior}. A cell is \emph{occupied} if it contains at least one letter of $\pi$, otherwise \emph{empty}. Let $\alpha_{ij}$ denote the subpermutation of entries in $C_{ij}$.

The reader may check the following characterization of $R$-avoiders in terms of the cell decomposition.
A permutation $\pi$ is an $R$ avoider if and only if
\begin{enumerate}
\item For each occupied cell $C$, all cells that lie both strictly east and strictly north of $C$ are empty.
\item For each pair of occupied cells $C,D$ with $D$ directly north of $C$ (same column), all entries in $C$ lie to the right of all entries in $D$.
\item For each pair of occupied cells $C,D$ with $D$ directly east of $C$ (same row), all entries in $C$ are larger than all entries in $D$.
\item $\alpha_{ij}$ avoids $R$ for all $i,j$.
\end{enumerate}
Condition (1)  imposes restrictions on occupied cells as follows. A \emph{major} cell for $\pi$ is an interior cell $C$ that is occupied and such that all cells directly north or directly east of $C$ are empty. The set of major cells (possibly empty) determines a Dyck path of semilength $m-1$ such that cells in the first column correspond to vertices in the first ascent of the path and major cells correspond to valley vertices as illustrated in Figure \ref{fig224b}. (If there are no major cells, the Dyck path covers the boundary cells and has no valleys.)
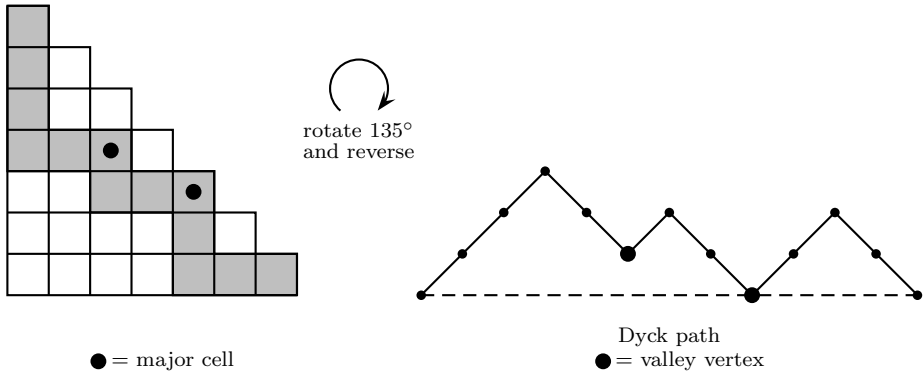
\begin{figure}[htp]
\begin{center}
\begin{pspicture}(-6,-1.2)(8,4)
\psset{unit=.55cm}
\pspolygon[fillstyle=solid,fillcolor=lightgray](-10,7)(-10,3)(-8,3)(-8,2)(-6,2)(-6,0)(-3,0)(-3,1)(-5,1)(-5,3)(-7,3)(-7,4)(-9,4)(-9,7)(-10,7)
\psline(-10,7)(-10,0)(-3,0)
\psline(-10,7)(-9,7)(-9,0)
\psline(-10,6)(-8,6)(-8,0)
\psline(-10,5)(-7,5)(-7,0)
\psline(-10,4)(-6,4)(-6,0)
\psline(-10,3)(-5,3)(-5,0)
\psline(-10,2)(-4,2)(-4,0)
\psline(-10,1)(-3,1)(-3,0)
\qdisk(-7.5,3.5){3pt}
\qdisk(-5.5,2.5){3pt}
\psdots(0,0)(1,1)(2,2)(3,3)(4,2)(5,1)(6,2)(7,1)(8,0)(9,1)(10,2)(11,1)(12,0)
\psline(0,0)(3,3)(5,1)(6,2)(8,0)(10,2)(12,0)
\psline[linestyle=dashed](0,0)(12,0)
\qdisk(8,0){3pt}
\qdisk(5,1){3pt}
\psarcn[fillcolor=white,arrows=->,arrowsize=3pt 3](-1.5,5){.7}{230}{310}
\rput(-1.5,4){\textrm{\footnotesize rotate 135${}^\circ$}}
\rput(-1.5,3.5){\textrm{\footnotesize and reverse}}
\rput(-6,-1.6){\textrm{\footnotesize = major cell}}
\qdisk(-7.8,-1.6){3pt}
\rput(6,-1){\textrm{\footnotesize Dyck path}}
\rput(6.5,-1.6){\textrm{\footnotesize = valley vertex}}
\qdisk(4.4,-1.6){3pt}
\end{pspicture}
\caption{Dyck path from cell diagram}\label{fig224b}
\end{center}
\end{figure}
If $\pi$ avoids $R$, then condition (1) implies that all cells not on the Dyck path are empty, and condition (4) implies St($\alpha_{ij}$) is an $R$ avoider for all $i,j$. Conversely, if $n=i_m>i_{m-1}> \dots > i_1 \ge 1$ are given and we have a Dyck path in the associated cell diagram,
and an $R$-avoider $\pi_C$ is specified for each cell $C$ on the Dyck path, with the additional proviso $\pi_C \ne \emptyset$ for valley cells, then conditions (2) and (3) imply that an $R$-avoider with this Dyck path is uniquely determined.

Now, we can find an equation for the generating function $L_m(x)$ for $T$-avoiders
with exactly $m+1$ right-left maxima.  Clearly, $L_1(x)=xF_T(x)$. So assume $m\geq2$.
If an $R$-avoider also avoids 4132 then all cells not in the leftmost column avoid 132 and all
cells in the leftmost column below the topmost nonempty cell also avoid 132.
Also the associated Dyck path has no valleys above the $x$-axis.
Thus we can assume that the Dyck path has the form
$$P=U^{a_1}D^{a_1}U^{a_2}D^{a_2}\cdots U^{a_{s+1}}D^{a_{s+1}}\, ,$$
with $s\ge 0$ valleys, all on the $x$-axis.
By the cell decomposition each valley contributes $C(x)-1$.

Let us consider the first $a_1+1$ cells, say $C_1,C_2,\ldots,C_{a_1+1}$, from top to bottom in the leftmost
column of the diagram:
\begin{itemize}
\item if $C_1=\cdots=C_{a_1}=\emptyset$, then we have a contribution of $x^mF_T(x)C(x)\big(C(x)-1\big)^{s}$;
\item if $C_1=\cdots=C_{j-1}=\emptyset$ and $C_j\neq\emptyset$, then $C_{j+2}=\cdots=C_{a_1+1}=\emptyset$, which gives a contribution $x^m\big(F_T(x)-1\big)C(x)^2\big(C(x)-1\big)^{s}$.
\end{itemize}
Summing over all Dyck paths of form $P$ with fixed $m$, $a_1$ and $s$, we find that the generating function for  $T$-avoiders with a fixed diagram associated with a Dyck path of $2m-2$ steps, no valleys
above $x$-axis, first ascent of length $a_1$ steps, and $s$ valleys is given by
$$\left\{\begin{array}{ll}
x^mF_T(x)C(x)+(m-1)x^m \big(F_T(x)-1\big)C(x)^2,& a_1=m-1\textrm{ (\,i.e., } s=0)\\
x^mF_T(x)C(x)\big(C(x)-1\big)^s+a_1x^m \big(F_T(x)-1\big)C(x)^2\big(C(x)-1\big)^s\,,& 1\leq a_1\leq m-2\,.
\end{array}\right.$$
Thus, by summing over all $s,a_1$, we obtain
\begin{align*}
L_m(x)&=\sum_{s=0}^{m-2}\binom{m-2}{s}x^mF_T(x)C(x)\big(C(x)-1\big)^{s}\\
&+\sum_{s=1}^{m-1}\binom{m-1}{s}x^m\big(F_T(x)-1\big)C^2(x)\big(C(x)-1\big)^{s-1}\\
&=x^mC^{m-1}(x)F_T(x)+\frac{x^mC(x)^{m+1}\big(F_T(x)-1\big)}{C(x)-1}-x^m\big(F_T(x)-1\big)C(x)^2
\,,
\end{align*}
which implies
\begin{align*}
L_m(x)&=x^m \big(xC^mF_T(x)+(C^m(x)-1)(F_T(x)-1)\big)\,.
\end{align*}
Summing over $m\geq1$ and using $L_0(x)=1$, we obtain
$$F_T(x)=1+xC(x)F_T(x)+\left(C(x)-\frac{1}{1-x}\right)(F_T(x)-1)\,.$$
Solving for $F_T(x)$ and simplifying leads to the following result.

\begin{theorem}\label{th224a}
Let $T=\{4132,1342,1423\}$. Then
$$F_T(x)=\frac{2 - 10 x + 9 x^2 - 3 x^3 +  x (1 - x) (2 - x)
 \sqrt{1 - 4 x}}{2 (1 - 5 x + 4 x^2 - x^4)}\,.$$
\end{theorem}

\subsection{Case 226: $\{1342,2143,2413\}$}
Define $H_k(x)$ to be the generating function for $T$-avoiders whose left-right maxima are $n+1-k,n+2-k,\ldots,n$, and refine $H_k(x)$ to $H_{k,d}(x)$, the generating function for $T$-avoiders $(n+1-k)\pi'(n+2-k)\pi''$ with $k$ left-right maxima $n+1-k,n+2-k,\ldots,n$ where $\pi'$ has $d$ left-right maxima. Clearly, $H_1(x)=xF_T(x)$.
Recall $K(x):=\frac{1-2x}{1-3x+x^2}$ is the generating function for the $\{231,2143\}$-avoiders \cite[Seq. A001519]{Sl}.
\begin{lemma}\label{lem226a}
We have
$$\sum_{k\geq2}H_k(x)=x\big(F_T(x)-1\big)C\left(x+\frac{x^2}{1-x}\big(K(x)-1\big)\right)\, .$$
\end{lemma}
\begin{proof}
Let $\pi=(n+1-k)\pi^{(1)}(n+2-k)\pi^{(2)}\cdots n\pi^{(k)}\in S_n(T)$ with exactly $k$ left-right maxima. If $\pi^{(1)}=\emptyset$ then we have a contribution to $H_k(x)$ of $xH_{k-1}(x)$. Thus,
\begin{align}
H_k(x)=xH_{k-1}(x)+\sum_{j\geq1}H_{m,j}(x)\, .\label{eq226a1}
\end{align}

To write an equation for $H_{k,d}(x)$ with $d\geq1$, suppose $\pi^{(1)}$ has $d$ left-right maxima $j_1,j-2,\dots,j_d$, and write $\pi$ as
$$\pi=(n+1-k)j_1\alpha^{(1)}\cdots j_d\alpha^{(d)}(n+2-k)\alpha^{(d+1)}\cdots n\alpha^{(d+k-1)}\,.$$
Then $\alpha^{(s)}<j_1$ for all $s=d+2,d+3,\ldots,d+k-1$ (or else $j_1\,(n+2-k)\,(n+3-k)$ form the 134 of a 1342). Now consider two cases:
\begin{itemize}
\item $\alpha^{(s)}<j_1$ for all $s=2,3,\ldots,d+1$. Then the $T$-avoiding property is undisturbed by deletion of the first letter and so, by definition of $H_k$, we have a contribution of $xH_{k+d-1}(x)$.
\item Otherwise, there is a minimal $p\in [2,d+1]$ such that $\alpha^{(p)}$ has a letter greater than $j_1$.
In fact, this letter is greater than
$j_{p-1}$ (to avoid 1342). Since $\pi$ avoids $2413$ and $2143$, we see that $\alpha^{(s)}<j_1$ for all $s\neq p$. Moreover $\alpha^{(p)}$ has the form $\beta\gamma$ with $\beta>j_{p-1}$ and $\gamma<j_1$.
Note that $\pi$ avoids $T$ if and only if $j_p\gamma j_{p+1}\alpha^{(p+1)}\cdots j_d\alpha^{(d)}(n+2-k)\alpha^{(d+1)}\cdots n\alpha^{(d+k-1)}$ avoids $T$ and $\beta$ avoids $231$ and $2143$.
Thus, we have a contribution of $x^p\big(K(x)-1\big)H_{k+d-p}(x)$.
\end{itemize}
Hence,
$$H_{k,d}(x)=xH_{k+d-1}(x)+\sum_{p=2}^{d+1}x^p\big(K(x)-1\big)H_{k+d-p}(x).$$
Summing over $d\geq1$ and using \eqref{eq226a1}, we obtain
$$H_k(x)-xH_{k-1}(x)=\frac{x^2}{1-x}\big(K(x)-1\big)H_{k-1}(x)+\left(x+\frac{x^2}{1-x}(K(x)-1)\right)\sum_{j\geq k}H_j(x),$$
which leads to
\begin{align}
H_k(x)=\left(x+\frac{x^2}{1-x}\big(K(x)-1\big)\right)\sum_{j\geq k-1} H_j(x).\label{eq226a2}
\end{align}
Also, by deletion of first letter, $H_2(x)=x(F_T(x)-1)$.

Define $H(x,y)=\sum_{k\geq2}H_k(x)y^k$. Multiplying \eqref{eq226a2} by $y^k$ and summing over $k\geq3$, we obtain $$H(x,y)=x\big(F_T(x)-1\big)y^2+\left(x+\frac{x^2}{1-x}\big(K(x)-1\big)\right)\frac{y^2}{1-y}\big(yH(x,1)-H(x,y)\big).$$
This equation can be solved by the kernel method, taking $y=C\left(x+\frac{x^2}{1-x}(K(x)-1)\right)$ where $C(x)$ is the Catalan \gf. We find that
$$H(x,1)=x\big(F_T(x)-1\big)C\left(x+\frac{x^2}{1-x}\big(K(x)-1\big)\right),$$
as required.
\end{proof}

\begin{theorem}\label{th226a}
Let $T=\{1342,2143,2413\}$. Then
$$F_T(x)=\frac{1 - 3 x + x^2 - \sqrt{(1 - 7 x + 13 x^2 - 8 x^3) (1 - 3 x +
    x^2)}}{2 x (1 - x) (1 - 2 x)}\,.$$
\end{theorem}
\begin{proof}
Let $G_m(x)$ be the generating function for $T$-avoiders with $m$
left-right maxima. Clearly, $G_0(x)=1$ and $G_1(x)=xF_T(x)$.

To write an equation for $G_m(x)$ with $m\geq2$, suppose $\pi\in S_n(T)$ has exactly $m$ left-right maxima.
Since $\pi$ avoids $2413$, we can write $\pi$ as
$$\pi=i_1\alpha^{(1)}i_2\beta^{(2)}\alpha^{(2)}\cdots i_m\beta^{(m)}\alpha^{(m)}$$
where $\alpha^{(s)}<i_1$ for all $s=1,2,\ldots,m$ and $\beta^{(s)}>i_{s-1}$ for all $s=2,3,\ldots,m$.
Now, consider the following two cases:
\begin{itemize}
\item $\beta^{(2)}\cdots\beta^{(m)}=\emptyset$. By definition of $H_k$, we have a contribution of $H_m(x)$.
\item there exists $j$ such that $\beta^{(j)}\neq\emptyset$. Since $\pi$ avoids $2143$, we see that
$\alpha^{(s)}=\emptyset$ for all $s=1,2,\ldots,j-1$, and $\beta^{(s)}=\emptyset$ for all $s\neq j$. Note that $\pi$ avoids $T$ if and only if $i_j\alpha^{(j)}i_{j_1+1}\alpha^{(j+1)}\cdots i_m\alpha^{(m)}$ avoids $T$ and $\beta^{(j)}$ avoids $231$ and $2143$. Thus, we have a contribution of $x^{j-1}\big(K(x)-1\big)H_{m+1-j}(x)$.
\end{itemize}
Hence, $G_m(x)=H_m(x)+\sum_{j=2}^mx^{j-1}\big(K(x)-1\big)H_{m+1-j}(x)$.
Summing over $m\geq2$ and using the expressions for $G_0(x)$ an $G_1(x)$, we obtain
$$F_T(x)-1-xF_T(x)=\sum_{m\geq2}G_m(x)=\sum_{m\geq2}H_m(x)+\frac{x\big(K(x)-1\big)}{1-x}\left(xF_T(x)+\sum_{m\geq2}H_m(x)\right).$$
Hence, by Lemma \ref{lem226a},
\begin{align*}
F_T(x)&=1+xF_T(x)+x\big(F_T(x)-1\big)\,C\!\left(x+\frac{x^2}{1-x}\big(K(x)-1\big)\right)\\
&+\frac{x\big(K(x)-1\big)}{1-x}\left(xF_T(x)+x\big(F_T(x)-1\big)\,C\!\left(x+\frac{x^2}{1-x}\big(K(x)-1\big)\right)\right)\,.
\end{align*}
Solve for $F_T(x)$ and simplify to complete the proof.
\end{proof}

\subsection{Case 232: $\{1234,1342,2341\}$} Let $c(n;i)$ denote the number of permutations in $S_n(\{123\})$ whose first letter is $i, \ 1\le i \le n$.
\begin{lemma}\label{lem232a1}
Let $a(n;i_1,i_2,\ldots,i_m)$ be the number of permutations in $S_n(T)$ whose first $m$ letters are $i_1i_2\cdots i_m$
and set $a_n=|S_n(T)|$. Then $a(n;n)=a_{n-1}$ for $n\ge 1$, and $a(n;n-1)=a_{n-1}$
 for $n\ge 2$, $a(n;i,n)=a(n-1;i)$ for $i=1,2,\ldots,n-1$,
$$a(n;i,j)=a(n-1;j),\quad 1\leq j<i\leq n,$$
and
$$a(n;i,j)=c(j-1;i),\quad 1\leq i<j\leq n-1.$$
\end{lemma}
\begin{proof}
We prove only the displayed equations; the other statements are clear since no pattern in $T$ has first letter 3 or 4, or second letter equal to 4.
First, suppose $\pi=ij\pi'\in S_n(T)$ with $1\leq i<j\leq n-1$. Since $\pi$ avoids $T$, we see that $\pi=ij\pi''n(n-1)\cdots(j+1)$ and $\pi$ avoids $T$ if and only if $i\pi''$ avoids $123$. Hence, $a(n;i,j)=c(j-1;i)$.

Second, suppose $\pi=ij\pi'\in S_n(T)$ with $1\leq j<i\leq n$. Note that if $j\pi'$ contains a pattern in $T$ then $\pi$ contains the same pattern. On other hand, if $\pi$ contains $1234$ (resp. $1342$) as $iabc$ with $i<a<b<c$ (resp. $i<c<a<b$) then $j\pi'$ contains $1234$ (resp. $1342$). If $\pi$ contains $2341$ as $iabc$ with $c<i<a<b$, then $j\pi$ contains either $j<c<a<b$ or $c<j<a<b$, which form a $1342$ or $2341$, respectively. Thus, $\pi$ avoids $T$ if and only if $j\pi'$ avoids $T$, which implies that $a(n;i,j)=a(n-1;j)$ for all $1\leq j<i\leq n$.
\end{proof}
Define $A(n;v)=\sum_{i=1}^na(n;i)v^{i-1}$ and $C(n;v)=\sum_{i=1}^nc(n;i)v^{i-1}$ for $n \ge 1$ and set $A(0;v)=C(0;v)=1$.

From Lemma \ref{lem232a1} and its notation, we have
$$a(n;i)=\sum_{j=1}^{i-1}a(n-1;j)+\sum_{j=i+1}^{n-1}c(j-1;i)$$
with $a(n;n)=a(n;n-1)=a_{n-1}$.  Multiplying this recurrence by $v^{i-1}$ and summing over $i=1,2,\ldots,n-2$, we obtain
$$A(n;v)=\frac{1}{1-v}(A(n-1;v)-v^nA(n-1;1))+\sum_{j=1}^{n-2}C(j;v), \quad n\geq2,$$
while $A(0;v)=A(1;v)=1$.

Now define $A(x,v)=\sum_{n\geq0}A(n;v)x^n$ and $C(x,v)=\sum_{n\geq0}C(n;v)x^n$. Then the last recurrence can be written as
$$A(x,v)=1+\frac{x}{1-v}\big(A(x,v)-vA(xv,1)\big)+\frac{x^2}{1-x}\big(C(x,v)-1\big)\,.$$

From \cite[Section 4.1.1]{FM} we have
$$C(x;v)=1+\frac{x}{1-v}\big(C(x,v)-vC(xv,1)\big).$$
Now, set $K(x,v)=1-\frac{x}{1-v}$. Then
\begin{align*}
K(x,v)A(x,v)-\frac{x^2}{1-x}C(x,v)&=1-\frac{x^2}{1-x}-\frac{xv}{1-v}A(xv,1),\\
K(x,v)C(x,v)&=1-\frac{xv}{1-v}C(xv,1)\,.
\end{align*}
Since $C(x,1)=C(x)$, the Catalan \gf, these equations imply
\begin{align*}
K(x/v,v)^2A(v/x,v)&=K(x/v,v)-\frac{x^2}{v(v-x)}\big(K(x/v,v)-1\big)\\
&-\frac{x}{1-v}K(x/v,v)A(x,1)-\frac{x^3}{v(1-v)(v-x)}C(x),
\end{align*}
By differentiating respect to $v$, we have
\begin{align*}
\frac{d}{dv}(K(x/v,v)^2A(v/x,v))&=K'(x/v,v)-\frac{x^2(2v-x)}{v^2(v-x)^2}(K(x/v,v)-1)-\frac{x^2}{v(v-x)}K'(x/v,v)\\
&-\frac{x}{(1-v)^2}K(x/v,v)A(x,1)-\frac{x}{1-v}K'(x/v,v)A(x,1)\\
&-\frac{x^3(x-2v+3v^2-2xv)}{v^2(1-v)^2(v-x)^2}C(x),
\end{align*}
This equation can be solved by the kernel method using $v=v_0=1/\big(xC(x)\big)$,  $K(x/v_0,v_0)=0$, leading to the following result.
\begin{theorem}\label{th232a}
Let $T=\{1234,1342,2341\}$. Then
\[
F_T(x)=\frac{1-4x+2x^2 - (1 -6x+9x^2) C(x)}{x (1 - 4 x)}\, .
\]
\end{theorem}

\subsection{Case 242: $\{2341,2431,3241\}$}
\begin{theorem}\label{th242a}
Let $T=\{2341,2431,3241\}$. Then $F_T(x)$ satisfies
$$F_T(x)=1+\frac{xF_T(x)}{1-xF_T^2(x)}.$$
Explicitly,
$$F_T(x)=1+\sum_{n\geq1}x^n\left(\sum_{i=1}^n\frac{1}{i}\binom{n-1}{i-1}\binom{2n-i}{i-1}\right).$$
\end{theorem}
\begin{proof}
Let $G_m(x)$ be the generating function for $T$-avoiders with $m$
left-right maxima. Clearly, $G_0(x)=1$ and $G_1(x)=xF_T(x)$.
Now let $m\geq2$. To find an equation for $G_m(x)$,
write $\pi\in S_n(T)$ with $m$ left-right maxima as $i_1\pi^{(1)}i_2\pi^{(2)}\cdots i_m\pi^{(m)}$. Since $\pi$ avoids $2341$,  $i_j<\pi^{(j+2)}$ for $j=1,2,\ldots,m-2$. Since $\pi$ avoids $2431$ and $3241$, $\pi$ is further restricted to have the form illustrated in Figure \ref{figc242a} for $m=5$ (blank regions empty).
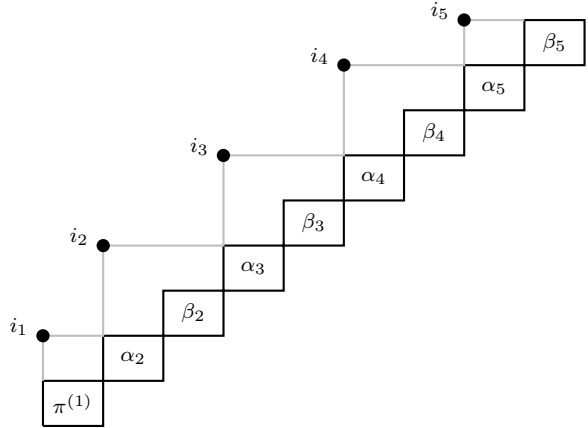
\begin{figure}[htp]
\begin{center}
\begin{pspicture}(1,5.8)
\psset{xunit=.8cm}
\psset{yunit=.6cm}
\psline(-4,0)(-4,1)(-2,1)(-2,3)(0,3)(0,5)(2,5)(2,7)(4,7)(4,9)(5,9)(5,8)(3,8)(3,6)(1,6)(1,4)(-1,4)(-1,2)(-3,2)(-3,0)(-4,0)
\psline[linecolor=lightgray](-4,1)(-4,2)(-3,2)(-3,4)(-1,4)(-1,6)(1,6)(1,8)(3,8)(3,9)(4,9)
\psdots[linewidth=1.5pt](-4,2)(-3,4)(-1,6)(1,8)(3,9)
\rput(-3.5,0.5){\textrm{{\footnotesize $\pi^{(1)}$}}}
\rput(-1.5,2.5){\textrm{{\footnotesize $\be_2$}}}
\rput(0.5,4.5){\textrm{{\footnotesize $\be_3$}}}
\rput(2.5,6.5){\textrm{{\footnotesize $\be_4$}}}
\rput(4.5,8.5){\textrm{{\footnotesize $\be_5$}}}
\rput(-2.5,1.5){\textrm{{\footnotesize $\al_2$}}}
\rput(-0.5,3.5){\textrm{{\footnotesize $\al_3$}}}
\rput(1.5,5.5){\textrm{{\footnotesize $\al_4$}}}
\rput(3.5,7.5){\textrm{{\footnotesize $\al_5$}}}
\rput(-4.4,2.2){\textrm{{\footnotesize $i_1$}}}
\rput(-3.4,4.2){\textrm{{\footnotesize $i_2$}}}
\rput(-1.4,6.2){\textrm{{\footnotesize $i_3$}}}
\rput(0.6,8.2){\textrm{{\footnotesize $i_4$}}}
\rput(2.6,9.2){\textrm{{\footnotesize $i_5$}}}
\end{pspicture}
\caption{A $T$-avoider with $m=5$}\label{figc242a}
\end{center}
\end{figure}
Conversely, in a permutation of this form, if all $2m-1$ labelled regions are $T$-avoiders, then so is the permutation. Hence, $G_m(x)=x^mF_T^{2m-1}(x)$, and this formula also holds for $m=1$. Summing  over $m\geq0$ yields the equation for $F_T(x)$.

Define a function $g(x,y)$ via the equation $g(x,y)=\frac{xy(g(x,y)+1)}{1-x(g(x,y)+1)^2}$, where $g(x,1)=F_T(x)-1$.
The Lagrange Inversion formula \cite[Sec. 5.1]{wilf} yields
$$g(x,y)=\sum_{i\geq1}y^i\sum_{j\geq0}\frac{x^{i+j}}{i}\binom{i-1+j}{i-1}\binom{i+2j}{i-1},$$
and picking out the coefficient of $x^n$ in $g(x,1)$  completes the proof.
\end{proof}

For other combinatorial objects with this counting sequence,
see \cite[Seq. A106228]{Sl}.

\end{document}